\documentclass{preprint}

\hypersetup{
  pdftitle={Towards a more robust algorithm for computing the restricted
    singular value decomposition},
  pdfauthor={Ian N. Zwaan},
  pdfsubject={65F15; 65F22; 65F30; 65F50; 65R30; 65R32},
  pdfkeywords={Restricted singular value decomposition; RSVD;
    Kogbetliantz; numerically stable; rank decisions; numerical linear
    algebra},
  pdflang={en-US},
  unicode=true,
}

\usepackage[USenglish]{babel}
\PassOptionsToPackage{USenglish}{babel}
\usepackage[space,noadjust,nocompress]{cite}

\ifx\siamartcls\undefined
  \newcommand*\aligncr{&}
  \newcommand*\compressedmatrices{}
\else
  \newcommand*\aligncr{\\}
  \newcommand*\qedhere{}
  \newcommand*\compressedmatrices{%
    \arraycolsep=0.5\arraycolsep
    \renewcommand*\arraystretch{1}}
  \newsiamremark{remark}{Remark}
\fi

\ifx\preprintcls\undefined
  \usepackage{amsmath}
  \usepackage{mathtools}
\else
  \usepackage[fleqn,leqno]{amsmath}
  \usepackage[fleqn,leqno]{mathtools}
\fi
\usepackage{amssymb}
\usepackage{bm}
\renewcommand*\vec\bm

\usepackage{array}
\usepackage{arydshln}

\usepackage{booktabs}
\renewcommand\arraystretch{1.25}

\usepackage{rotating}

\makeatletter
\newcommand*\centerfloat{%
  \parindent \z@
  \leftskip \z@ \@plus 1fil \@minus \textwidth
  \rightskip\leftskip
  \parfillskip \z@skip}
\makeatother

\usepackage{pgfplots}
\pgfplotsset{
  compat = newest,
  filter discard warning = false,
  legend cell align = left,
  every axis plot/.append style = {
    black,
    mark = none,
    line width = 1pt
  },
  cycle list name = linestyles*,
  /pgf/number format/1000 sep = {},
  ticklabel style = {font = \footnotesize},
  label style = {font = \footnotesize},
  every axis legend/.append style = {
    font = \footnotesize
  }
}

\usepackage{pgfplotstable}
\pgfplotstableset{
  comment chars = {\#},
  string replace = {NaN}{},
}

\ifx\preprintcls\undefined\else
  \theoremstyle{definition} 
  \newtheorem{myalgorithm}{Algorithm}
  \newcommand*\tabwidth\leftmargini
  \newcommand*\tab[1][]{\makebox[\tabwidth][l]{#1}\ignorespaces}
\fi

\ifx\siamartcls\undefined\else
  \theoremstyle{plain}
  \theoremheaderfont{\normalfont\sc}
  \theorembodyfont{\normalfont}
  \theoremseparator{.}
  \theoremsymbol{}
  \newtheorem{myalgorithm}[theorem]{Algorithm}

  \newcommand*\tabwidth{2em}
  \newcommand*\tab[1][]{\makebox[\tabwidth][l]{#1}\ignorespaces}
\fi

\newcommand*\bbR{\mathbb{R}}

\DeclareMathOperator\adj{adj}
\DeclareMathOperator\diag{diag} 
\DeclareMathOperator\fl{fl}
\DeclareMathOperator\hypot{hypot}

\DeclareMathOperator\rank{rank} 
\DeclareMathOperator\rot{rot}

\newcommand*\iter[2]{#1^{(#2)}}
\newcommand*\nonzero\underline
\newcommand*\minimat[4]{\left[%
  \begin{smallmatrix} #1 & #2 \\ #3 & #4 \end{smallmatrix}%
    \right]}

\makeatletter
\newif\if@borderstar
\def\bordermatrix{\@ifnextchar*{%
\@borderstartrue\@bordermatrix@i}{\@borderstarfalse\@bordermatrix@i*}%
}
\def\@bordermatrix@i*{\@ifnextchar[{\@bordermatrix@ii}{\@bordermatrix@ii[()]}}
\def\@bordermatrix@ii[#1]#2{%
\begingroup
\m@th\@tempdima8.75\p@\setbox\z@\vbox{%
\def\cr{\crcr\noalign{\kern 2\p@\global\let\cr\endline }}%
\ialign {$##$\hfil\kern 2\p@\kern\@tempdima & \thinspace %
\hfil $##$\hfil && \quad\hfil $##$\hfil\crcr\omit\strut %
\hfil\crcr\noalign{\kern -\baselineskip}#2\crcr\omit %
\strut\cr}}%
\setbox\tw@\vbox{\unvcopy\z@\global\setbox\@ne\lastbox}%
\setbox\tw@\hbox{\unhbox\@ne\unskip\global\setbox\@ne\lastbox}%
\setbox\tw@\hbox{%
$\kern\wd\@ne\kern -\@tempdima\left\@firstoftwo#1%
\if@borderstar\kern2pt\else\kern -\wd\@ne\fi%
\global\setbox\@ne\vbox{\box\@ne\if@borderstar\else\kern 2\p@\fi}%
\vcenter{\if@borderstar\else\kern -\ht\@ne\fi%
\unvbox\z@\kern-\if@borderstar2\fi\baselineskip}%
\if@borderstar\kern-2\@tempdima\kern2\p@\else\,\fi\right\@secondoftwo#1 $%
}\null \;\vbox{\kern\ht\@ne\box\tw@}%
\endgroup
}
\makeatother

\pgfplotstableread{
Lo         Up         VCQ01      PAQ01      PBU01      KP01       QH01       Max01      VCQerr     PAQerr     PBUerr     KPerr      QHerr      Maxerr
-2.40e+01  -2.38e+01  +1.53e-02  +3.10e-02  +2.35e-02  +2.71e-02  +3.10e-02  +2.16e-02  +8.88e-03  +1.03e-02  +9.12e-03  +4.67e-03  +4.68e-03  +2.28e-03
-2.38e+01  -2.35e+01  +2.24e-02  +3.11e-02  +2.70e-02  +3.01e-02  +3.19e-02  +2.54e-02  +8.88e-03  +1.03e-02  +9.11e-03  +4.70e-03  +4.69e-03  +2.31e-03
-2.35e+01  -2.32e+01  +3.00e-02  +3.37e-02  +3.20e-02  +3.36e-02  +3.45e-02  +3.06e-02  +8.88e-03  +1.03e-02  +9.12e-03  +4.71e-03  +4.70e-03  +2.33e-03
-2.32e+01  -2.30e+01  +3.17e-02  +3.42e-02  +3.29e-02  +3.46e-02  +3.48e-02  +3.19e-02  +8.89e-03  +1.03e-02  +9.12e-03  +4.72e-03  +4.72e-03  +2.35e-03
-2.30e+01  -2.28e+01  +3.26e-02  +3.44e-02  +3.35e-02  +3.48e-02  +3.51e-02  +3.26e-02  +8.90e-03  +1.03e-02  +9.11e-03  +4.74e-03  +4.74e-03  +2.37e-03
-2.28e+01  -2.25e+01  +3.31e-02  +3.44e-02  +3.36e-02  +3.47e-02  +3.49e-02  +3.29e-02  +8.92e-03  +1.03e-02  +9.11e-03  +4.75e-03  +4.76e-03  +2.40e-03
-2.25e+01  -2.22e+01  +3.30e-02  +3.40e-02  +3.36e-02  +3.44e-02  +3.46e-02  +3.28e-02  +8.92e-03  +1.03e-02  +9.08e-03  +4.77e-03  +4.77e-03  +2.42e-03
-2.22e+01  -2.20e+01  +3.29e-02  +3.37e-02  +3.33e-02  +3.40e-02  +3.41e-02  +3.26e-02  +8.92e-03  +1.03e-02  +9.09e-03  +4.78e-03  +4.79e-03  +2.44e-03
-2.20e+01  -2.18e+01  +3.25e-02  +3.34e-02  +3.29e-02  +3.36e-02  +3.37e-02  +3.23e-02  +8.93e-03  +1.03e-02  +9.09e-03  +4.80e-03  +4.80e-03  +2.47e-03
-2.18e+01  -2.15e+01  +3.23e-02  +3.31e-02  +3.26e-02  +3.32e-02  +3.33e-02  +3.22e-02  +8.94e-03  +1.03e-02  +9.08e-03  +4.82e-03  +4.82e-03  +2.49e-03
-2.15e+01  -2.12e+01  +3.22e-02  +3.28e-02  +3.24e-02  +3.29e-02  +3.29e-02  +3.21e-02  +8.93e-03  +1.03e-02  +9.08e-03  +4.84e-03  +4.83e-03  +2.52e-03
-2.12e+01  -2.10e+01  +3.22e-02  +3.25e-02  +3.22e-02  +3.26e-02  +3.26e-02  +3.20e-02  +8.95e-03  +1.03e-02  +9.09e-03  +4.85e-03  +4.85e-03  +2.54e-03
-2.10e+01  -2.08e+01  +3.20e-02  +3.21e-02  +3.20e-02  +3.23e-02  +3.22e-02  +3.18e-02  +8.95e-03  +1.03e-02  +9.07e-03  +4.88e-03  +4.87e-03  +2.57e-03
-2.08e+01  -2.05e+01  +3.19e-02  +3.18e-02  +3.17e-02  +3.19e-02  +3.17e-02  +3.16e-02  +8.97e-03  +1.03e-02  +9.06e-03  +4.89e-03  +4.89e-03  +2.60e-03
-2.05e+01  -2.02e+01  +3.17e-02  +3.14e-02  +3.15e-02  +3.15e-02  +3.14e-02  +3.14e-02  +8.97e-03  +1.03e-02  +9.07e-03  +4.92e-03  +4.92e-03  +2.62e-03
-2.02e+01  -2.00e+01  +3.15e-02  +3.11e-02  +3.12e-02  +3.12e-02  +3.09e-02  +3.12e-02  +8.98e-03  +1.03e-02  +9.06e-03  +4.96e-03  +4.96e-03  +2.66e-03
-2.00e+01  -1.98e+01  +3.15e-02  +3.08e-02  +3.10e-02  +3.09e-02  +3.07e-02  +3.12e-02  +9.01e-03  +1.03e-02  +9.07e-03  +5.00e-03  +5.01e-03  +2.71e-03
-1.98e+01  -1.95e+01  +3.15e-02  +3.06e-02  +3.10e-02  +3.06e-02  +3.05e-02  +3.11e-02  +9.03e-03  +1.03e-02  +9.09e-03  +5.09e-03  +5.08e-03  +2.77e-03
-1.95e+01  -1.92e+01  +3.14e-02  +3.02e-02  +3.08e-02  +3.03e-02  +3.02e-02  +3.10e-02  +9.07e-03  +1.04e-02  +9.12e-03  +5.21e-03  +5.21e-03  +2.86e-03
-1.92e+01  -1.90e+01  +3.13e-02  +2.99e-02  +3.06e-02  +2.99e-02  +2.97e-02  +3.08e-02  +9.16e-03  +1.04e-02  +9.19e-03  +5.43e-03  +5.42e-03  +3.01e-03
-1.90e+01  -1.88e+01  +3.11e-02  +2.96e-02  +3.02e-02  +2.97e-02  +2.93e-02  +3.06e-02  +9.30e-03  +1.06e-02  +9.29e-03  +5.79e-03  +5.78e-03  +3.24e-03
-1.88e+01  -1.85e+01  +3.08e-02  +2.92e-02  +3.01e-02  +2.93e-02  +2.90e-02  +3.04e-02  +9.56e-03  +1.08e-02  +9.52e-03  +6.44e-03  +6.40e-03  +3.65e-03
-1.85e+01  -1.82e+01  +3.07e-02  +2.88e-02  +2.97e-02  +2.89e-02  +2.85e-02  +3.02e-02  +1.00e-02  +1.13e-02  +9.94e-03  +7.56e-03  +7.51e-03  +4.37e-03
-1.82e+01  -1.80e+01  +3.06e-02  +2.86e-02  +2.95e-02  +2.87e-02  +2.81e-02  +3.01e-02  +1.08e-02  +1.21e-02  +1.07e-02  +9.55e-03  +9.48e-03  +5.64e-03
-1.80e+01  -1.78e+01  +3.05e-02  +2.81e-02  +2.94e-02  +2.83e-02  +2.77e-02  +3.00e-02  +1.23e-02  +1.37e-02  +1.22e-02  +1.31e-02  +1.30e-02  +7.92e-03
-1.78e+01  -1.75e+01  +3.04e-02  +2.78e-02  +2.91e-02  +2.78e-02  +2.74e-02  +2.98e-02  +1.50e-02  +1.66e-02  +1.48e-02  +1.94e-02  +1.91e-02  +1.21e-02
-1.75e+01  -1.72e+01  +3.03e-02  +2.76e-02  +2.88e-02  +2.76e-02  +2.70e-02  +2.97e-02  +1.99e-02  +2.19e-02  +1.96e-02  +3.05e-02  +3.02e-02  +1.97e-02
-1.72e+01  -1.70e+01  +3.01e-02  +2.74e-02  +2.87e-02  +2.72e-02  +2.68e-02  +2.96e-02  +2.88e-02  +3.16e-02  +2.85e-02  +5.04e-02  +4.98e-02  +3.41e-02
-1.70e+01  -1.68e+01  +3.04e-02  +2.74e-02  +2.88e-02  +2.73e-02  +2.67e-02  +2.98e-02  +4.49e-02  +4.91e-02  +4.44e-02  +8.56e-02  +8.47e-02  +6.23e-02
-1.68e+01  -1.65e+01  +3.16e-02  +2.80e-02  +2.97e-02  +2.81e-02  +2.74e-02  +3.08e-02  +7.24e-02  +8.00e-02  +7.22e-02  +1.48e-01  +1.47e-01  +1.21e-01
-1.65e+01  -1.62e+01  +3.88e-02  +3.30e-02  +3.59e-02  +3.31e-02  +3.23e-02  +3.71e-02  +1.10e-01  +1.27e-01  +1.13e-01  +2.58e-01  +2.55e-01  +2.48e-01
-1.62e+01  -1.60e+01  +3.80e-02  +3.31e-02  +3.60e-02  +3.34e-02  +3.28e-02  +3.76e-02  +1.64e-01  +1.86e-01  +1.64e-01  +2.18e-01  +2.18e-01  +2.69e-01
-1.60e+01  -1.58e+01  +3.46e-03  +5.23e-03  +4.73e-03  +5.99e-03  +5.95e-03  +5.00e-03  +1.92e-01  +1.77e-01  +1.89e-01  +3.08e-02  +3.58e-02  +9.17e-02
-1.58e+01  -1.55e+01  +1.21e-04  +2.18e-04  +2.28e-04  +4.19e-04  +3.99e-04  +2.52e-04  +1.18e-01  +4.53e-02  +1.17e-01  +1.83e-02  +1.94e-02  +6.36e-02
-1.55e+01  -1.52e+01  +0.00e+00  +0.00e+00  +0.00e+00  +0.00e+00  +5.04e-07  +6.59e-08  +3.41e-03  +8.72e-04  +3.50e-03  +1.48e-03  +1.45e-03  +3.27e-03
-1.52e+01  -1.50e+01  +0.00e+00  +0.00e+00  +0.00e+00  +0.00e+00  +0.00e+00  +0.00e+00  +8.46e-06  +1.60e-07  +9.06e-06  +5.39e-06  +5.49e-06  +1.04e-05
-1.50e+01  -1.48e+01  +0.00e+00  +0.00e+00  +0.00e+00  +0.00e+00  +0.00e+00  +0.00e+00  +0.00e+00  +0.00e+00  +0.00e+00  +0.00e+00  +0.00e+00  +0.00e+00
-1.48e+01  -1.45e+01  +0.00e+00  +0.00e+00  +0.00e+00  +0.00e+00  +0.00e+00  +0.00e+00  +0.00e+00  +0.00e+00  +0.00e+00  +0.00e+00  +0.00e+00  +0.00e+00
-1.45e+01  -1.42e+01  +0.00e+00  +0.00e+00  +0.00e+00  +0.00e+00  +0.00e+00  +0.00e+00  +0.00e+00  +0.00e+00  +0.00e+00  +0.00e+00  +0.00e+00  +0.00e+00
-1.42e+01  -1.40e+01  +0.00e+00  +0.00e+00  +0.00e+00  +0.00e+00  +0.00e+00  +0.00e+00  +0.00e+00  +0.00e+00  +0.00e+00  +0.00e+00  +0.00e+00  +0.00e+00
-1.40e+01  -1.38e+01  +0.00e+00  +0.00e+00  +0.00e+00  +0.00e+00  +0.00e+00  +0.00e+00  +0.00e+00  +0.00e+00  +0.00e+00  +0.00e+00  +0.00e+00  +0.00e+00
-1.38e+01  -1.35e+01  +0.00e+00  +0.00e+00  +0.00e+00  +0.00e+00  +0.00e+00  +0.00e+00  +0.00e+00  +0.00e+00  +0.00e+00  +0.00e+00  +0.00e+00  +0.00e+00
-1.35e+01  -1.32e+01  +0.00e+00  +0.00e+00  +0.00e+00  +0.00e+00  +0.00e+00  +0.00e+00  +0.00e+00  +0.00e+00  +0.00e+00  +0.00e+00  +0.00e+00  +0.00e+00
-1.32e+01  -1.30e+01  +0.00e+00  +0.00e+00  +0.00e+00  +0.00e+00  +0.00e+00  +0.00e+00  +0.00e+00  +0.00e+00  +0.00e+00  +0.00e+00  +0.00e+00  +0.00e+00
-1.30e+01  -1.28e+01  +0.00e+00  +0.00e+00  +0.00e+00  +0.00e+00  +0.00e+00  +0.00e+00  +0.00e+00  +0.00e+00  +0.00e+00  +0.00e+00  +0.00e+00  +0.00e+00
-1.28e+01  -1.25e+01  +0.00e+00  +0.00e+00  +0.00e+00  +0.00e+00  +0.00e+00  +0.00e+00  +0.00e+00  +0.00e+00  +0.00e+00  +0.00e+00  +0.00e+00  +0.00e+00
-1.25e+01  -1.22e+01  +0.00e+00  +0.00e+00  +0.00e+00  +0.00e+00  +0.00e+00  +0.00e+00  +0.00e+00  +0.00e+00  +0.00e+00  +0.00e+00  +0.00e+00  +0.00e+00
-1.22e+01  -1.20e+01  +0.00e+00  +0.00e+00  +0.00e+00  +0.00e+00  +0.00e+00  +0.00e+00  +0.00e+00  +0.00e+00  +0.00e+00  +0.00e+00  +0.00e+00  +0.00e+00
-1.20e+01  -1.18e+01  +0.00e+00  +0.00e+00  +0.00e+00  +0.00e+00  +0.00e+00  +0.00e+00  +0.00e+00  +0.00e+00  +0.00e+00  +0.00e+00  +0.00e+00  +0.00e+00
-1.18e+01  -1.15e+01  +0.00e+00  +0.00e+00  +0.00e+00  +0.00e+00  +0.00e+00  +0.00e+00  +0.00e+00  +0.00e+00  +0.00e+00  +0.00e+00  +0.00e+00  +0.00e+00
-1.15e+01  -1.12e+01  +0.00e+00  +0.00e+00  +0.00e+00  +0.00e+00  +0.00e+00  +0.00e+00  +0.00e+00  +0.00e+00  +0.00e+00  +0.00e+00  +0.00e+00  +0.00e+00
-1.12e+01  -1.10e+01  +0.00e+00  +0.00e+00  +0.00e+00  +0.00e+00  +0.00e+00  +0.00e+00  +0.00e+00  +0.00e+00  +0.00e+00  +0.00e+00  +0.00e+00  +0.00e+00
-1.10e+01  -1.08e+01  +0.00e+00  +0.00e+00  +0.00e+00  +0.00e+00  +0.00e+00  +0.00e+00  +0.00e+00  +0.00e+00  +0.00e+00  +0.00e+00  +0.00e+00  +0.00e+00
-1.08e+01  -1.05e+01  +0.00e+00  +0.00e+00  +0.00e+00  +0.00e+00  +0.00e+00  +0.00e+00  +0.00e+00  +0.00e+00  +0.00e+00  +0.00e+00  +0.00e+00  +0.00e+00
-1.05e+01  -1.02e+01  +0.00e+00  +0.00e+00  +0.00e+00  +0.00e+00  +0.00e+00  +0.00e+00  +0.00e+00  +0.00e+00  +0.00e+00  +0.00e+00  +0.00e+00  +0.00e+00
-1.02e+01  -1.00e+01  +0.00e+00  +0.00e+00  +0.00e+00  +0.00e+00  +0.00e+00  +0.00e+00  +0.00e+00  +0.00e+00  +0.00e+00  +0.00e+00  +0.00e+00  +0.00e+00
-1.00e+01  -9.75e+00  +0.00e+00  +0.00e+00  +0.00e+00  +0.00e+00  +0.00e+00  +0.00e+00  +0.00e+00  +0.00e+00  +0.00e+00  +0.00e+00  +0.00e+00  +0.00e+00
-9.75e+00  -9.50e+00  +0.00e+00  +0.00e+00  +0.00e+00  +0.00e+00  +0.00e+00  +0.00e+00  +0.00e+00  +0.00e+00  +0.00e+00  +0.00e+00  +0.00e+00  +0.00e+00
-9.50e+00  -9.25e+00  +0.00e+00  +0.00e+00  +0.00e+00  +0.00e+00  +0.00e+00  +0.00e+00  +0.00e+00  +0.00e+00  +0.00e+00  +0.00e+00  +0.00e+00  +0.00e+00
-9.25e+00  -9.00e+00  +0.00e+00  +0.00e+00  +0.00e+00  +0.00e+00  +0.00e+00  +0.00e+00  +0.00e+00  +0.00e+00  +0.00e+00  +0.00e+00  +0.00e+00  +0.00e+00
-9.00e+00  -8.75e+00  +0.00e+00  +0.00e+00  +0.00e+00  +0.00e+00  +0.00e+00  +0.00e+00  +0.00e+00  +0.00e+00  +0.00e+00  +0.00e+00  +0.00e+00  +0.00e+00
-8.75e+00  -8.50e+00  +0.00e+00  +0.00e+00  +0.00e+00  +0.00e+00  +0.00e+00  +0.00e+00  +0.00e+00  +0.00e+00  +0.00e+00  +0.00e+00  +0.00e+00  +0.00e+00
-8.50e+00  -8.25e+00  +0.00e+00  +0.00e+00  +0.00e+00  +0.00e+00  +0.00e+00  +0.00e+00  +0.00e+00  +0.00e+00  +0.00e+00  +0.00e+00  +0.00e+00  +0.00e+00
-8.25e+00  -8.00e+00  +0.00e+00  +0.00e+00  +0.00e+00  +0.00e+00  +0.00e+00  +0.00e+00  +0.00e+00  +0.00e+00  +0.00e+00  +0.00e+00  +0.00e+00  +0.00e+00
-8.00e+00  -7.75e+00  +0.00e+00  +0.00e+00  +0.00e+00  +0.00e+00  +0.00e+00  +0.00e+00  +0.00e+00  +0.00e+00  +0.00e+00  +0.00e+00  +0.00e+00  +0.00e+00
-7.75e+00  -7.50e+00  +0.00e+00  +0.00e+00  +0.00e+00  +0.00e+00  +0.00e+00  +0.00e+00  +0.00e+00  +0.00e+00  +0.00e+00  +0.00e+00  +0.00e+00  +0.00e+00
-7.50e+00  -7.25e+00  +0.00e+00  +0.00e+00  +0.00e+00  +0.00e+00  +0.00e+00  +0.00e+00  +0.00e+00  +0.00e+00  +0.00e+00  +0.00e+00  +0.00e+00  +0.00e+00
-7.25e+00  -7.00e+00  +0.00e+00  +0.00e+00  +0.00e+00  +0.00e+00  +0.00e+00  +0.00e+00  +0.00e+00  +0.00e+00  +0.00e+00  +0.00e+00  +0.00e+00  +0.00e+00
-7.00e+00  -6.75e+00  +0.00e+00  +0.00e+00  +0.00e+00  +0.00e+00  +0.00e+00  +0.00e+00  +0.00e+00  +0.00e+00  +0.00e+00  +0.00e+00  +0.00e+00  +0.00e+00
-6.75e+00  -6.50e+00  +0.00e+00  +0.00e+00  +0.00e+00  +0.00e+00  +0.00e+00  +0.00e+00  +0.00e+00  +0.00e+00  +0.00e+00  +0.00e+00  +0.00e+00  +0.00e+00
-6.50e+00  -6.25e+00  +0.00e+00  +0.00e+00  +0.00e+00  +0.00e+00  +0.00e+00  +0.00e+00  +0.00e+00  +0.00e+00  +0.00e+00  +0.00e+00  +0.00e+00  +0.00e+00
-6.25e+00  -6.00e+00  +0.00e+00  +0.00e+00  +0.00e+00  +0.00e+00  +0.00e+00  +0.00e+00  +0.00e+00  +0.00e+00  +0.00e+00  +0.00e+00  +0.00e+00  +0.00e+00
-6.00e+00  -5.75e+00  +0.00e+00  +0.00e+00  +0.00e+00  +0.00e+00  +0.00e+00  +0.00e+00  +0.00e+00  +0.00e+00  +0.00e+00  +0.00e+00  +0.00e+00  +0.00e+00
-5.75e+00  -5.50e+00  +0.00e+00  +0.00e+00  +0.00e+00  +0.00e+00  +0.00e+00  +0.00e+00  +0.00e+00  +0.00e+00  +0.00e+00  +0.00e+00  +0.00e+00  +0.00e+00
-5.50e+00  -5.25e+00  +0.00e+00  +0.00e+00  +0.00e+00  +0.00e+00  +0.00e+00  +0.00e+00  +0.00e+00  +0.00e+00  +0.00e+00  +0.00e+00  +0.00e+00  +0.00e+00
-5.25e+00  -5.00e+00  +0.00e+00  +0.00e+00  +0.00e+00  +0.00e+00  +0.00e+00  +0.00e+00  +0.00e+00  +0.00e+00  +0.00e+00  +0.00e+00  +0.00e+00  +0.00e+00
-5.00e+00  -4.75e+00  +0.00e+00  +0.00e+00  +0.00e+00  +0.00e+00  +0.00e+00  +0.00e+00  +0.00e+00  +0.00e+00  +0.00e+00  +0.00e+00  +0.00e+00  +0.00e+00
-4.75e+00  -4.50e+00  +0.00e+00  +0.00e+00  +0.00e+00  +0.00e+00  +0.00e+00  +0.00e+00  +0.00e+00  +0.00e+00  +0.00e+00  +0.00e+00  +0.00e+00  +0.00e+00
-4.50e+00  -4.25e+00  +0.00e+00  +0.00e+00  +0.00e+00  +0.00e+00  +0.00e+00  +0.00e+00  +0.00e+00  +0.00e+00  +0.00e+00  +0.00e+00  +0.00e+00  +0.00e+00
-4.25e+00  -4.00e+00  +0.00e+00  +0.00e+00  +0.00e+00  +0.00e+00  +0.00e+00  +0.00e+00  +0.00e+00  +0.00e+00  +0.00e+00  +0.00e+00  +0.00e+00  +0.00e+00
-4.00e+00  -3.75e+00  +0.00e+00  +0.00e+00  +0.00e+00  +0.00e+00  +0.00e+00  +0.00e+00  +0.00e+00  +0.00e+00  +0.00e+00  +0.00e+00  +0.00e+00  +0.00e+00
-3.75e+00  -3.50e+00  +0.00e+00  +0.00e+00  +0.00e+00  +0.00e+00  +0.00e+00  +0.00e+00  +0.00e+00  +0.00e+00  +0.00e+00  +0.00e+00  +0.00e+00  +0.00e+00
-3.50e+00  -3.25e+00  +0.00e+00  +0.00e+00  +0.00e+00  +0.00e+00  +0.00e+00  +0.00e+00  +0.00e+00  +0.00e+00  +0.00e+00  +0.00e+00  +0.00e+00  +0.00e+00
-3.25e+00  -3.00e+00  +0.00e+00  +0.00e+00  +0.00e+00  +0.00e+00  +0.00e+00  +0.00e+00  +0.00e+00  +0.00e+00  +0.00e+00  +0.00e+00  +0.00e+00  +0.00e+00
-3.00e+00  -2.75e+00  +0.00e+00  +0.00e+00  +0.00e+00  +0.00e+00  +0.00e+00  +0.00e+00  +0.00e+00  +0.00e+00  +0.00e+00  +0.00e+00  +0.00e+00  +0.00e+00
-2.75e+00  -2.50e+00  +0.00e+00  +0.00e+00  +0.00e+00  +0.00e+00  +0.00e+00  +0.00e+00  +0.00e+00  +0.00e+00  +0.00e+00  +0.00e+00  +0.00e+00  +0.00e+00
-2.50e+00  -2.25e+00  +0.00e+00  +0.00e+00  +0.00e+00  +0.00e+00  +0.00e+00  +0.00e+00  +0.00e+00  +0.00e+00  +0.00e+00  +0.00e+00  +0.00e+00  +0.00e+00
-2.25e+00  -2.00e+00  +0.00e+00  +0.00e+00  +0.00e+00  +0.00e+00  +0.00e+00  +0.00e+00  +0.00e+00  +0.00e+00  +0.00e+00  +0.00e+00  +0.00e+00  +0.00e+00
-2.00e+00  -1.75e+00  +0.00e+00  +0.00e+00  +0.00e+00  +0.00e+00  +0.00e+00  +0.00e+00  +0.00e+00  +0.00e+00  +0.00e+00  +0.00e+00  +0.00e+00  +0.00e+00
-1.75e+00  -1.50e+00  +0.00e+00  +0.00e+00  +0.00e+00  +0.00e+00  +0.00e+00  +0.00e+00  +0.00e+00  +0.00e+00  +0.00e+00  +0.00e+00  +0.00e+00  +0.00e+00
-1.50e+00  -1.25e+00  +0.00e+00  +0.00e+00  +0.00e+00  +0.00e+00  +0.00e+00  +0.00e+00  +0.00e+00  +0.00e+00  +0.00e+00  +0.00e+00  +0.00e+00  +0.00e+00
-1.25e+00  -1.00e+00  +0.00e+00  +0.00e+00  +0.00e+00  +0.00e+00  +0.00e+00  +0.00e+00  +0.00e+00  +0.00e+00  +0.00e+00  +0.00e+00  +0.00e+00  +0.00e+00
-1.00e+00  -7.50e-01  +0.00e+00  +0.00e+00  +0.00e+00  +0.00e+00  +0.00e+00  +0.00e+00  +0.00e+00  +0.00e+00  +0.00e+00  +0.00e+00  +0.00e+00  +0.00e+00
-7.50e-01  -5.00e-01  +0.00e+00  +0.00e+00  +0.00e+00  +0.00e+00  +0.00e+00  +0.00e+00  +0.00e+00  +0.00e+00  +0.00e+00  +0.00e+00  +0.00e+00  +0.00e+00
-5.00e-01  -2.50e-01  +0.00e+00  +0.00e+00  +0.00e+00  +0.00e+00  +0.00e+00  +0.00e+00  +0.00e+00  +0.00e+00  +0.00e+00  +0.00e+00  +0.00e+00  +0.00e+00
-2.50e-01  +0.00e+00  +0.00e+00  +0.00e+00  +0.00e+00  +0.00e+00  +0.00e+00  +0.00e+00  +0.00e+00  +0.00e+00  +0.00e+00  +0.00e+00  +0.00e+00  +0.00e+00
+0.00e+00  +2.50e-01  +0.00e+00  +0.00e+00  +0.00e+00  +0.00e+00  +0.00e+00  +0.00e+00  +0.00e+00  +0.00e+00  +0.00e+00  +0.00e+00  +0.00e+00  +0.00e+00
+2.50e-01  +5.00e-01  +0.00e+00  +0.00e+00  +0.00e+00  +0.00e+00  +0.00e+00  +0.00e+00  +0.00e+00  +0.00e+00  +0.00e+00  +0.00e+00  +0.00e+00  +0.00e+00
+5.00e-01  +7.50e-01  +0.00e+00  +0.00e+00  +0.00e+00  +0.00e+00  +0.00e+00  +0.00e+00  +0.00e+00  +0.00e+00  +0.00e+00  +0.00e+00  +0.00e+00  +0.00e+00
+7.50e-01  +1.00e+00  +0.00e+00  +0.00e+00  +0.00e+00  +0.00e+00  +0.00e+00  +0.00e+00  +0.00e+00  +0.00e+00  +0.00e+00  +0.00e+00  +0.00e+00  +0.00e+00
}\tableSwapOnErr

\pgfplotstableread{
Lo         Up         eta_g1     eta_h1     eta_k1     eta_l1     eta_max1   eta_g2     eta_h2     eta_k2     eta_l2     eta_max2   best_g     best_h     best_k     best_l     best_max
+0.00e+00  +1.00e-01  +9.86e-01  +9.40e-01  +9.84e-01  +9.64e-01  +9.99e-01  +9.79e-01  +9.93e-01  +9.62e-01  +9.97e-01  +9.99e-01  +0.00e+00  +9.99e-01  +9.99e-01  +0.00e+00  +1.00e+00
+1.00e-01  +2.00e-01  +4.89e-04  +1.71e-03  +7.92e-04  +9.97e-04  +2.89e-04  +5.83e-04  +5.00e-04  +1.09e-03  +1.79e-04  +2.05e-04  +0.00e+00  +1.90e-04  +1.89e-04  +0.00e+00  +4.68e-07
+2.00e-01  +3.00e-01  +3.19e-04  +1.14e-03  +4.89e-04  +6.74e-04  +1.62e-04  +3.94e-04  +2.99e-04  +7.22e-04  +1.09e-04  +1.14e-04  +0.00e+00  +1.07e-04  +1.06e-04  +0.00e+00  +1.89e-07
+3.00e-01  +4.00e-01  +2.52e-04  +9.15e-04  +3.67e-04  +5.46e-04  +1.09e-04  +3.19e-04  +2.17e-04  +5.81e-04  +8.04e-05  +7.65e-05  +0.00e+00  +7.18e-05  +7.14e-05  +0.00e+00  +1.24e-07
+4.00e-01  +5.00e-01  +2.16e-04  +7.96e-04  +2.98e-04  +4.80e-04  +7.84e-05  +2.80e-04  +1.71e-04  +5.06e-04  +6.43e-05  +5.57e-05  +0.00e+00  +5.21e-05  +5.17e-05  +0.00e+00  +5.70e-08
+5.00e-01  +6.00e-01  +1.94e-04  +7.27e-04  +2.57e-04  +4.44e-04  +5.88e-05  +2.57e-04  +1.41e-04  +4.63e-04  +5.44e-05  +4.22e-05  +0.00e+00  +3.89e-05  +3.89e-05  +0.00e+00  +3.00e-08
+6.00e-01  +7.00e-01  +1.80e-04  +6.80e-04  +2.28e-04  +4.19e-04  +4.56e-05  +2.40e-04  +1.22e-04  +4.33e-04  +4.74e-05  +3.21e-05  +0.00e+00  +2.98e-05  +2.97e-05  +0.00e+00  +2.60e-08
+7.00e-01  +8.00e-01  +1.70e-04  +6.54e-04  +2.07e-04  +4.02e-04  +3.49e-05  +2.31e-04  +1.07e-04  +4.14e-04  +4.25e-05  +2.52e-05  +0.00e+00  +2.29e-05  +2.30e-05  +0.00e+00  +1.10e-08
+8.00e-01  +9.00e-01  +1.62e-04  +6.30e-04  +1.91e-04  +3.90e-04  +2.75e-05  +2.25e-04  +9.61e-05  +4.01e-04  +3.87e-05  +1.96e-05  +0.00e+00  +1.81e-05  +1.82e-05  +0.00e+00  +1.00e-08
+9.00e-01  +1.00e+00  +1.57e-04  +6.13e-04  +1.80e-04  +3.82e-04  +2.15e-05  +2.20e-04  +8.73e-05  +3.91e-04  +3.56e-05  +1.54e-05  +0.00e+00  +1.42e-05  +1.43e-05  +0.00e+00  +3.00e-09
+1.00e+00  +1.10e+00  +1.53e-04  +6.03e-04  +1.71e-04  +3.75e-04  +1.69e-05  +2.16e-04  +8.11e-05  +3.83e-04  +3.35e-05  +1.20e-05  +1.00e+00  +1.13e-05  +1.10e-05  +1.00e+00  +3.00e-09
+1.10e+00  +1.20e+00  +1.50e-04  +5.94e-04  +1.65e-04  +3.73e-04  +1.36e-05  +2.12e-04  +7.59e-05  +3.78e-04  +3.18e-05  +9.66e-06  +0.00e+00  +8.78e-06  +8.72e-06  +0.00e+00  +1.00e-09
+1.20e+00  +1.30e+00  +1.47e-04  +5.88e-04  +1.60e-04  +3.67e-04  +1.06e-05  +2.11e-04  +7.24e-05  +3.75e-04  +3.06e-05  +7.69e-06  +0.00e+00  +7.08e-06  +6.95e-06  +0.00e+00  +0.00e+00
+1.30e+00  +1.40e+00  +1.45e-04  +5.83e-04  +1.55e-04  +3.66e-04  +8.45e-06  +2.09e-04  +6.92e-05  +3.72e-04  +2.98e-05  +5.93e-06  +0.00e+00  +5.53e-06  +5.70e-06  +0.00e+00  +3.00e-09
+1.40e+00  +1.50e+00  +1.44e-04  +5.82e-04  +1.52e-04  +3.64e-04  +6.74e-06  +2.09e-04  +6.65e-05  +3.68e-04  +2.87e-05  +4.75e-06  +0.00e+00  +4.36e-06  +4.39e-06  +0.00e+00  +0.00e+00
+1.50e+00  +1.60e+00  +1.43e-04  +5.80e-04  +1.49e-04  +3.62e-04  +5.24e-06  +2.06e-04  +6.40e-05  +3.66e-04  +2.78e-05  +3.70e-06  +0.00e+00  +3.43e-06  +3.52e-06  +0.00e+00  +0.00e+00
+1.60e+00  +1.70e+00  +1.43e-04  +5.78e-04  +1.47e-04  +3.61e-04  +4.15e-06  +2.06e-04  +6.27e-05  +3.65e-04  +2.74e-05  +2.96e-06  +0.00e+00  +2.74e-06  +2.74e-06  +0.00e+00  +2.00e-09
+1.70e+00  +1.80e+00  +1.41e-04  +5.74e-04  +1.46e-04  +3.59e-04  +3.31e-06  +2.05e-04  +6.17e-05  +3.65e-04  +2.69e-05  +2.34e-06  +0.00e+00  +2.19e-06  +2.16e-06  +0.00e+00  +0.00e+00
+1.80e+00  +1.90e+00  +1.41e-04  +5.73e-04  +1.45e-04  +3.59e-04  +2.61e-06  +2.04e-04  +6.07e-05  +3.63e-04  +2.69e-05  +1.89e-06  +0.00e+00  +1.72e-06  +1.74e-06  +0.00e+00  +0.00e+00
+1.90e+00  +2.00e+00  +1.40e-04  +5.73e-04  +1.43e-04  +3.58e-04  +2.06e-06  +2.05e-04  +5.99e-05  +3.62e-04  +2.62e-05  +1.55e-06  +0.00e+00  +1.34e-06  +1.34e-06  +0.00e+00  +1.00e-09
+2.00e+00  +2.10e+00  +1.40e-04  +5.73e-04  +1.43e-04  +3.59e-04  +1.72e-06  +2.03e-04  +5.91e-05  +3.62e-04  +2.60e-05  +1.16e-06  +0.00e+00  +1.07e-06  +1.08e-06  +0.00e+00  +0.00e+00
+2.10e+00  +2.20e+00  +1.40e-04  +5.72e-04  +1.42e-04  +3.57e-04  +1.30e-06  +2.05e-04  +5.85e-05  +3.62e-04  +2.58e-05  +9.19e-07  +0.00e+00  +8.80e-07  +8.45e-07  +0.00e+00  +0.00e+00
+2.20e+00  +2.30e+00  +1.39e-04  +5.71e-04  +1.42e-04  +3.57e-04  +1.07e-06  +2.04e-04  +5.85e-05  +3.62e-04  +2.58e-05  +8.02e-07  +0.00e+00  +6.88e-07  +6.71e-07  +0.00e+00  +0.00e+00
+2.30e+00  +2.40e+00  +1.39e-04  +5.73e-04  +1.42e-04  +3.57e-04  +8.46e-07  +2.04e-04  +5.74e-05  +3.62e-04  +2.53e-05  +5.67e-07  +0.00e+00  +4.98e-07  +5.52e-07  +0.00e+00  +0.00e+00
+2.40e+00  +2.50e+00  +1.39e-04  +5.71e-04  +1.41e-04  +3.56e-04  +6.55e-07  +2.04e-04  +5.71e-05  +3.61e-04  +2.52e-05  +5.01e-07  +0.00e+00  +4.13e-07  +4.45e-07  +0.00e+00  +0.00e+00
+2.50e+00  +2.60e+00  +1.38e-04  +5.69e-04  +1.41e-04  +3.57e-04  +4.99e-07  +2.04e-04  +5.66e-05  +3.61e-04  +2.52e-05  +3.67e-07  +0.00e+00  +3.14e-07  +3.35e-07  +0.00e+00  +0.00e+00
+2.60e+00  +2.70e+00  +1.38e-04  +5.72e-04  +1.40e-04  +3.55e-04  +4.24e-07  +2.03e-04  +5.65e-05  +3.61e-04  +2.51e-05  +2.65e-07  +0.00e+00  +2.74e-07  +2.66e-07  +0.00e+00  +0.00e+00
+2.70e+00  +2.80e+00  +1.37e-04  +5.72e-04  +1.39e-04  +3.55e-04  +3.21e-07  +2.03e-04  +5.64e-05  +3.62e-04  +2.51e-05  +2.25e-07  +0.00e+00  +1.92e-07  +2.32e-07  +0.00e+00  +0.00e+00
+2.80e+00  +2.90e+00  +1.38e-04  +5.73e-04  +1.40e-04  +3.56e-04  +2.48e-07  +2.02e-04  +5.63e-05  +3.60e-04  +2.48e-05  +1.85e-07  +0.00e+00  +1.43e-07  +1.67e-07  +0.00e+00  +0.00e+00
+2.90e+00  +3.00e+00  +1.38e-04  +5.70e-04  +1.41e-04  +3.55e-04  +2.00e-07  +2.02e-04  +5.58e-05  +3.59e-04  +2.47e-05  +1.55e-07  +0.00e+00  +1.23e-07  +1.33e-07  +0.00e+00  +0.00e+00
+3.00e+00  +3.10e+00  +1.37e-04  +5.70e-04  +1.41e-04  +3.55e-04  +1.64e-07  +2.02e-04  +5.58e-05  +3.60e-04  +2.46e-05  +1.27e-07  +0.00e+00  +1.35e-07  +1.16e-07  +0.00e+00  +0.00e+00
+3.10e+00  +3.20e+00  +1.37e-04  +5.72e-04  +1.40e-04  +3.54e-04  +1.30e-07  +2.02e-04  +5.56e-05  +3.61e-04  +2.45e-05  +1.04e-07  +0.00e+00  +7.90e-08  +9.50e-08  +0.00e+00  +0.00e+00
+3.20e+00  +3.30e+00  +1.37e-04  +5.74e-04  +1.40e-04  +3.54e-04  +1.02e-07  +2.03e-04  +5.54e-05  +3.61e-04  +2.43e-05  +9.30e-08  +0.00e+00  +6.70e-08  +6.40e-08  +0.00e+00  +0.00e+00
+3.30e+00  +3.40e+00  +1.37e-04  +5.72e-04  +1.40e-04  +3.54e-04  +8.00e-08  +2.02e-04  +5.53e-05  +3.60e-04  +2.44e-05  +7.20e-08  +0.00e+00  +5.20e-08  +6.20e-08  +0.00e+00  +0.00e+00
+3.40e+00  +3.50e+00  +1.36e-04  +5.74e-04  +1.40e-04  +3.56e-04  +6.60e-08  +2.03e-04  +5.55e-05  +3.61e-04  +2.42e-05  +4.20e-08  +0.00e+00  +4.50e-08  +3.60e-08  +0.00e+00  +0.00e+00
+3.50e+00  +3.60e+00  +1.37e-04  +5.73e-04  +1.40e-04  +3.54e-04  +4.50e-08  +2.02e-04  +5.53e-05  +3.59e-04  +2.44e-05  +2.80e-08  +0.00e+00  +2.50e-08  +3.80e-08  +0.00e+00  +0.00e+00
+3.60e+00  +3.70e+00  +1.36e-04  +5.73e-04  +1.41e-04  +3.54e-04  +3.00e-08  +2.01e-04  +5.52e-05  +3.60e-04  +2.43e-05  +3.50e-08  +0.00e+00  +2.90e-08  +2.30e-08  +0.00e+00  +0.00e+00
+3.70e+00  +3.80e+00  +1.36e-04  +5.74e-04  +1.40e-04  +3.53e-04  +3.40e-08  +2.02e-04  +5.44e-05  +3.61e-04  +2.39e-05  +2.40e-08  +0.00e+00  +2.00e-08  +2.30e-08  +0.00e+00  +0.00e+00
+3.80e+00  +3.90e+00  +1.36e-04  +5.75e-04  +1.40e-04  +3.53e-04  +2.80e-08  +2.01e-04  +5.46e-05  +3.60e-04  +2.38e-05  +2.10e-08  +0.00e+00  +1.40e-08  +1.40e-08  +0.00e+00  +0.00e+00
+3.90e+00  +4.00e+00  +1.36e-04  +5.73e-04  +1.41e-04  +3.54e-04  +1.80e-08  +2.02e-04  +5.46e-05  +3.61e-04  +2.39e-05  +1.90e-08  +0.00e+00  +8.00e-09  +1.60e-08  +0.00e+00  +0.00e+00
+4.00e+00  +4.10e+00  +1.36e-04  +5.74e-04  +1.40e-04  +3.52e-04  +2.20e-08  +2.02e-04  +5.43e-05  +3.60e-04  +2.37e-05  +1.30e-08  +0.00e+00  +9.00e-09  +1.50e-08  +0.00e+00  +0.00e+00
+4.10e+00  +4.20e+00  +1.36e-04  +5.75e-04  +1.40e-04  +3.52e-04  +1.50e-08  +2.03e-04  +5.43e-05  +3.61e-04  +2.38e-05  +8.00e-09  +0.00e+00  +9.00e-09  +3.00e-09  +0.00e+00  +0.00e+00
+4.20e+00  +4.30e+00  +1.36e-04  +5.75e-04  +1.40e-04  +3.54e-04  +1.10e-08  +2.01e-04  +5.41e-05  +3.61e-04  +2.37e-05  +5.00e-09  +0.00e+00  +6.00e-09  +6.00e-09  +0.00e+00  +0.00e+00
+4.30e+00  +4.40e+00  +1.35e-04  +5.75e-04  +1.40e-04  +3.53e-04  +1.00e-08  +2.02e-04  +5.38e-05  +3.61e-04  +2.39e-05  +3.00e-09  +0.00e+00  +6.00e-09  +8.00e-09  +0.00e+00  +0.00e+00
+4.40e+00  +4.50e+00  +1.35e-04  +5.74e-04  +1.40e-04  +3.52e-04  +8.00e-09  +2.01e-04  +5.36e-05  +3.60e-04  +2.36e-05  +4.00e-09  +0.00e+00  +1.00e-09  +3.00e-09  +0.00e+00  +0.00e+00
+4.50e+00  +4.60e+00  +1.34e-04  +5.75e-04  +1.40e-04  +3.53e-04  +5.00e-09  +2.02e-04  +5.33e-05  +3.60e-04  +2.31e-05  +3.00e-09  +0.00e+00  +3.00e-09  +6.00e-09  +0.00e+00  +0.00e+00
+4.60e+00  +4.70e+00  +1.35e-04  +5.76e-04  +1.40e-04  +3.52e-04  +6.00e-09  +2.01e-04  +5.32e-05  +3.61e-04  +2.33e-05  +1.00e-09  +0.00e+00  +3.00e-09  +4.00e-09  +0.00e+00  +0.00e+00
+4.70e+00  +4.80e+00  +1.34e-04  +5.75e-04  +1.40e-04  +3.52e-04  +2.00e-09  +2.02e-04  +5.35e-05  +3.62e-04  +2.35e-05  +0.00e+00  +0.00e+00  +4.00e-09  +3.00e-09  +0.00e+00  +0.00e+00
+4.80e+00  +4.90e+00  +1.34e-04  +5.75e-04  +1.40e-04  +3.52e-04  +2.00e-09  +2.02e-04  +5.33e-05  +3.61e-04  +2.34e-05  +0.00e+00  +0.00e+00  +1.00e-09  +0.00e+00  +0.00e+00  +0.00e+00
+4.90e+00  +5.00e+00  +1.35e-04  +5.78e-04  +1.40e-04  +3.51e-04  +5.00e-09  +2.01e-04  +5.36e-05  +3.60e-04  +2.35e-05  +3.00e-09  +0.00e+00  +1.00e-09  +1.00e-09  +0.00e+00  +0.00e+00
+5.00e+00  +5.10e+00  +1.35e-04  +5.78e-04  +1.41e-04  +3.51e-04  +1.00e-09  +2.02e-04  +5.37e-05  +3.61e-04  +2.32e-05  +2.00e-09  +0.00e+00  +1.00e-09  +0.00e+00  +0.00e+00  +0.00e+00
+5.10e+00  +5.20e+00  +1.34e-04  +5.76e-04  +1.40e-04  +3.49e-04  +2.00e-09  +2.02e-04  +5.29e-05  +3.60e-04  +2.30e-05  +1.00e-09  +0.00e+00  +0.00e+00  +1.00e-09  +0.00e+00  +0.00e+00
+5.20e+00  +5.30e+00  +1.34e-04  +5.78e-04  +1.40e-04  +3.51e-04  +2.00e-09  +2.01e-04  +5.25e-05  +3.60e-04  +2.28e-05  +1.00e-09  +0.00e+00  +0.00e+00  +1.00e-09  +0.00e+00  +0.00e+00
+5.30e+00  +5.40e+00  +1.35e-04  +5.78e-04  +1.40e-04  +3.52e-04  +2.00e-09  +2.02e-04  +5.26e-05  +3.61e-04  +2.29e-05  +2.00e-09  +0.00e+00  +0.00e+00  +1.00e-09  +0.00e+00  +0.00e+00
+5.40e+00  +5.50e+00  +1.34e-04  +5.79e-04  +1.40e-04  +3.50e-04  +1.00e-09  +2.02e-04  +5.24e-05  +3.61e-04  +2.30e-05  +0.00e+00  +0.00e+00  +1.00e-09  +1.00e-09  +0.00e+00  +0.00e+00
+5.50e+00  +5.60e+00  +1.33e-04  +5.78e-04  +1.40e-04  +3.51e-04  +0.00e+00  +2.01e-04  +5.22e-05  +3.61e-04  +2.27e-05  +0.00e+00  +0.00e+00  +0.00e+00  +0.00e+00  +0.00e+00  +0.00e+00
+5.60e+00  +5.70e+00  +1.34e-04  +5.79e-04  +1.40e-04  +3.50e-04  +0.00e+00  +2.02e-04  +5.18e-05  +3.59e-04  +2.24e-05  +1.00e-09  +0.00e+00  +0.00e+00  +0.00e+00  +0.00e+00  +0.00e+00
+5.70e+00  +5.80e+00  +1.33e-04  +5.78e-04  +1.40e-04  +3.51e-04  +0.00e+00  +2.01e-04  +5.21e-05  +3.60e-04  +2.26e-05  +0.00e+00  +0.00e+00  +0.00e+00  +1.00e-09  +0.00e+00  +0.00e+00
+5.80e+00  +5.90e+00  +1.33e-04  +5.78e-04  +1.40e-04  +3.51e-04  +0.00e+00  +2.01e-04  +5.21e-05  +3.60e-04  +2.26e-05  +0.00e+00  +0.00e+00  +0.00e+00  +0.00e+00  +0.00e+00  +0.00e+00
+5.90e+00  +6.00e+00  +1.34e-04  +5.79e-04  +1.40e-04  +3.52e-04  +0.00e+00  +2.01e-04  +5.15e-05  +3.62e-04  +2.23e-05  +0.00e+00  +0.00e+00  +0.00e+00  +0.00e+00  +0.00e+00  +0.00e+00
+6.00e+00  +6.10e+00  +1.32e-04  +5.81e-04  +1.39e-04  +3.50e-04  +0.00e+00  +2.01e-04  +5.19e-05  +3.60e-04  +2.26e-05  +0.00e+00  +0.00e+00  +0.00e+00  +0.00e+00  +0.00e+00  +0.00e+00
+6.10e+00  +6.20e+00  +1.33e-04  +5.79e-04  +1.40e-04  +3.50e-04  +0.00e+00  +2.01e-04  +5.17e-05  +3.63e-04  +2.25e-05  +0.00e+00  +0.00e+00  +1.00e-09  +0.00e+00  +0.00e+00  +0.00e+00
+6.20e+00  +6.30e+00  +1.32e-04  +5.80e-04  +1.40e-04  +3.50e-04  +0.00e+00  +2.01e-04  +5.14e-05  +3.62e-04  +2.24e-05  +0.00e+00  +0.00e+00  +0.00e+00  +0.00e+00  +0.00e+00  +0.00e+00
+6.30e+00  +6.40e+00  +1.33e-04  +5.82e-04  +1.40e-04  +3.51e-04  +0.00e+00  +2.00e-04  +5.16e-05  +3.63e-04  +2.24e-05  +0.00e+00  +0.00e+00  +0.00e+00  +0.00e+00  +0.00e+00  +0.00e+00
+6.40e+00  +6.50e+00  +1.32e-04  +5.83e-04  +1.41e-04  +3.50e-04  +0.00e+00  +2.00e-04  +5.12e-05  +3.60e-04  +2.20e-05  +0.00e+00  +0.00e+00  +0.00e+00  +0.00e+00  +0.00e+00  +0.00e+00
+6.50e+00  +6.60e+00  +1.33e-04  +5.80e-04  +1.40e-04  +3.49e-04  +0.00e+00  +2.01e-04  +5.10e-05  +3.61e-04  +2.22e-05  +0.00e+00  +0.00e+00  +1.00e-09  +0.00e+00  +0.00e+00  +0.00e+00
+6.60e+00  +6.70e+00  +1.32e-04  +5.81e-04  +1.41e-04  +3.49e-04  +0.00e+00  +2.01e-04  +5.08e-05  +3.61e-04  +2.19e-05  +0.00e+00  +0.00e+00  +0.00e+00  +0.00e+00  +0.00e+00  +0.00e+00
+6.70e+00  +6.80e+00  +1.33e-04  +5.81e-04  +1.40e-04  +3.48e-04  +0.00e+00  +2.00e-04  +5.05e-05  +3.61e-04  +2.18e-05  +0.00e+00  +0.00e+00  +0.00e+00  +0.00e+00  +0.00e+00  +0.00e+00
+6.80e+00  +6.90e+00  +1.32e-04  +5.81e-04  +1.40e-04  +3.49e-04  +0.00e+00  +2.01e-04  +5.09e-05  +3.61e-04  +2.22e-05  +0.00e+00  +0.00e+00  +0.00e+00  +0.00e+00  +0.00e+00  +0.00e+00
+6.90e+00  +7.00e+00  +1.32e-04  +5.82e-04  +1.40e-04  +3.48e-04  +0.00e+00  +2.00e-04  +5.03e-05  +3.60e-04  +2.19e-05  +0.00e+00  +0.00e+00  +0.00e+00  +0.00e+00  +0.00e+00  +0.00e+00
+7.00e+00  +7.10e+00  +1.32e-04  +5.82e-04  +1.40e-04  +3.50e-04  +0.00e+00  +2.00e-04  +5.01e-05  +3.61e-04  +2.16e-05  +0.00e+00  +0.00e+00  +0.00e+00  +0.00e+00  +0.00e+00  +0.00e+00
+7.10e+00  +7.20e+00  +1.31e-04  +5.82e-04  +1.40e-04  +3.49e-04  +0.00e+00  +2.00e-04  +5.04e-05  +3.62e-04  +2.19e-05  +0.00e+00  +0.00e+00  +0.00e+00  +0.00e+00  +0.00e+00  +0.00e+00
+7.20e+00  +7.30e+00  +1.31e-04  +5.81e-04  +1.40e-04  +3.49e-04  +0.00e+00  +2.00e-04  +5.02e-05  +3.62e-04  +2.14e-05  +0.00e+00  +0.00e+00  +0.00e+00  +0.00e+00  +0.00e+00  +0.00e+00
+7.30e+00  +7.40e+00  +1.31e-04  +5.84e-04  +1.40e-04  +3.47e-04  +0.00e+00  +2.00e-04  +4.99e-05  +3.62e-04  +2.16e-05  +0.00e+00  +0.00e+00  +0.00e+00  +0.00e+00  +0.00e+00  +0.00e+00
+7.40e+00  +7.50e+00  +1.31e-04  +5.84e-04  +1.40e-04  +3.49e-04  +0.00e+00  +1.99e-04  +4.98e-05  +3.62e-04  +2.13e-05  +0.00e+00  +0.00e+00  +0.00e+00  +0.00e+00  +0.00e+00  +0.00e+00
+7.50e+00  +7.60e+00  +1.31e-04  +5.82e-04  +1.40e-04  +3.48e-04  +0.00e+00  +2.00e-04  +4.97e-05  +3.61e-04  +2.14e-05  +0.00e+00  +0.00e+00  +0.00e+00  +0.00e+00  +0.00e+00  +0.00e+00
+7.60e+00  +7.70e+00  +1.31e-04  +5.85e-04  +1.40e-04  +3.49e-04  +0.00e+00  +1.99e-04  +4.99e-05  +3.62e-04  +2.14e-05  +0.00e+00  +0.00e+00  +0.00e+00  +0.00e+00  +0.00e+00  +0.00e+00
+7.70e+00  +7.80e+00  +1.31e-04  +5.83e-04  +1.40e-04  +3.48e-04  +0.00e+00  +2.00e-04  +4.99e-05  +3.61e-04  +2.14e-05  +0.00e+00  +0.00e+00  +0.00e+00  +0.00e+00  +0.00e+00  +0.00e+00
+7.80e+00  +7.90e+00  +1.31e-04  +5.83e-04  +1.40e-04  +3.47e-04  +0.00e+00  +2.00e-04  +4.97e-05  +3.61e-04  +2.11e-05  +0.00e+00  +0.00e+00  +0.00e+00  +0.00e+00  +0.00e+00  +0.00e+00
+7.90e+00  +8.00e+00  +1.30e-04  +5.85e-04  +1.40e-04  +3.47e-04  +0.00e+00  +2.00e-04  +4.93e-05  +3.61e-04  +2.12e-05  +0.00e+00  +0.00e+00  +0.00e+00  +0.00e+00  +0.00e+00  +0.00e+00
+8.00e+00  +8.10e+00  +1.29e-04  +5.83e-04  +1.40e-04  +3.47e-04  +0.00e+00  +2.00e-04  +4.91e-05  +3.61e-04  +2.11e-05  +0.00e+00  +0.00e+00  +0.00e+00  +0.00e+00  +0.00e+00  +0.00e+00
+8.10e+00  +8.20e+00  +1.30e-04  +5.85e-04  +1.40e-04  +3.48e-04  +0.00e+00  +2.00e-04  +4.90e-05  +3.61e-04  +2.08e-05  +0.00e+00  +0.00e+00  +0.00e+00  +0.00e+00  +0.00e+00  +0.00e+00
+8.20e+00  +8.30e+00  +1.30e-04  +5.88e-04  +1.39e-04  +3.48e-04  +0.00e+00  +1.99e-04  +4.87e-05  +3.61e-04  +2.09e-05  +0.00e+00  +0.00e+00  +0.00e+00  +0.00e+00  +0.00e+00  +0.00e+00
+8.30e+00  +8.40e+00  +1.30e-04  +5.84e-04  +1.40e-04  +3.47e-04  +0.00e+00  +1.99e-04  +4.86e-05  +3.62e-04  +2.08e-05  +0.00e+00  +0.00e+00  +0.00e+00  +0.00e+00  +0.00e+00  +0.00e+00
+8.40e+00  +8.50e+00  +1.30e-04  +5.87e-04  +1.41e-04  +3.48e-04  +0.00e+00  +1.99e-04  +4.82e-05  +3.62e-04  +2.02e-05  +0.00e+00  +0.00e+00  +0.00e+00  +0.00e+00  +0.00e+00  +0.00e+00
+8.50e+00  +8.60e+00  +1.29e-04  +5.88e-04  +1.39e-04  +3.47e-04  +0.00e+00  +2.00e-04  +4.86e-05  +3.62e-04  +2.08e-05  +0.00e+00  +0.00e+00  +0.00e+00  +0.00e+00  +0.00e+00  +0.00e+00
+8.60e+00  +8.70e+00  +1.29e-04  +5.86e-04  +1.40e-04  +3.47e-04  +0.00e+00  +1.99e-04  +4.87e-05  +3.62e-04  +2.10e-05  +0.00e+00  +0.00e+00  +0.00e+00  +0.00e+00  +0.00e+00  +0.00e+00
+8.70e+00  +8.80e+00  +1.29e-04  +5.86e-04  +1.40e-04  +3.46e-04  +0.00e+00  +1.99e-04  +4.84e-05  +3.61e-04  +2.05e-05  +0.00e+00  +0.00e+00  +0.00e+00  +0.00e+00  +0.00e+00  +0.00e+00
+8.80e+00  +8.90e+00  +1.29e-04  +5.87e-04  +1.40e-04  +3.47e-04  +0.00e+00  +1.99e-04  +4.84e-05  +3.61e-04  +2.05e-05  +0.00e+00  +0.00e+00  +0.00e+00  +0.00e+00  +0.00e+00  +0.00e+00
+8.90e+00  +9.00e+00  +1.28e-04  +5.87e-04  +1.40e-04  +3.46e-04  +0.00e+00  +1.99e-04  +4.78e-05  +3.61e-04  +2.05e-05  +0.00e+00  +0.00e+00  +0.00e+00  +0.00e+00  +0.00e+00  +0.00e+00
+9.00e+00  +9.10e+00  +1.28e-04  +5.88e-04  +1.39e-04  +3.47e-04  +0.00e+00  +1.99e-04  +4.77e-05  +3.61e-04  +2.00e-05  +0.00e+00  +0.00e+00  +0.00e+00  +0.00e+00  +0.00e+00  +0.00e+00
+9.10e+00  +9.20e+00  +1.28e-04  +5.88e-04  +1.40e-04  +3.45e-04  +0.00e+00  +1.99e-04  +4.76e-05  +3.61e-04  +2.01e-05  +0.00e+00  +0.00e+00  +0.00e+00  +0.00e+00  +0.00e+00  +0.00e+00
+9.20e+00  +9.30e+00  +1.29e-04  +5.89e-04  +1.40e-04  +3.45e-04  +0.00e+00  +1.98e-04  +4.77e-05  +3.61e-04  +2.03e-05  +0.00e+00  +0.00e+00  +0.00e+00  +0.00e+00  +0.00e+00  +0.00e+00
+9.30e+00  +9.40e+00  +1.27e-04  +5.89e-04  +1.39e-04  +3.44e-04  +0.00e+00  +1.99e-04  +4.71e-05  +3.62e-04  +2.02e-05  +0.00e+00  +0.00e+00  +0.00e+00  +0.00e+00  +0.00e+00  +0.00e+00
+9.40e+00  +9.50e+00  +1.27e-04  +5.88e-04  +1.39e-04  +3.46e-04  +0.00e+00  +2.00e-04  +4.74e-05  +3.62e-04  +1.99e-05  +0.00e+00  +0.00e+00  +0.00e+00  +0.00e+00  +0.00e+00  +0.00e+00
+9.50e+00  +9.60e+00  +1.27e-04  +5.89e-04  +1.39e-04  +3.45e-04  +0.00e+00  +1.99e-04  +4.70e-05  +3.61e-04  +1.97e-05  +0.00e+00  +0.00e+00  +0.00e+00  +0.00e+00  +0.00e+00  +0.00e+00
+9.60e+00  +9.70e+00  +1.28e-04  +5.89e-04  +1.39e-04  +3.44e-04  +0.00e+00  +1.99e-04  +4.71e-05  +3.62e-04  +2.00e-05  +0.00e+00  +0.00e+00  +0.00e+00  +0.00e+00  +0.00e+00  +0.00e+00
+9.70e+00  +9.80e+00  +1.28e-04  +5.89e-04  +1.39e-04  +3.45e-04  +0.00e+00  +2.00e-04  +4.65e-05  +3.61e-04  +1.97e-05  +0.00e+00  +0.00e+00  +0.00e+00  +0.00e+00  +0.00e+00  +0.00e+00
+9.80e+00  +9.90e+00  +1.27e-04  +5.90e-04  +1.40e-04  +3.44e-04  +0.00e+00  +1.98e-04  +4.72e-05  +3.61e-04  +1.99e-05  +0.00e+00  +0.00e+00  +0.00e+00  +0.00e+00  +0.00e+00  +0.00e+00
+9.90e+00  +1.00e+01  +1.27e-04  +5.88e-04  +1.41e-04  +3.45e-04  +0.00e+00  +1.98e-04  +4.65e-05  +3.61e-04  +1.96e-05  +0.00e+00  +0.00e+00  +0.00e+00  +0.00e+00  +0.00e+00  +0.00e+00
}\tableSwapOnEta

\pgfplotstableread{
Lo         Up         VCQ01      PAQ01      PBU01      KP01       QH01       Max01      VCQerr     PAQerr     PBUerr     KPerr      QHerr      Maxerr
-2.40e+01  -2.38e+01  +1.54e-02  +3.12e-02  +2.37e-02  +2.74e-02  +3.09e-02  +2.17e-02  +8.84e-03  +1.01e-02  +9.09e-03  +4.67e-03  +4.67e-03  +2.26e-03
-2.38e+01  -2.35e+01  +2.24e-02  +3.15e-02  +2.70e-02  +3.02e-02  +3.21e-02  +2.55e-02  +8.85e-03  +1.01e-02  +9.10e-03  +4.69e-03  +4.69e-03  +2.28e-03
-2.35e+01  -2.32e+01  +3.00e-02  +3.37e-02  +3.20e-02  +3.38e-02  +3.45e-02  +3.06e-02  +8.84e-03  +1.01e-02  +9.08e-03  +4.71e-03  +4.71e-03  +2.31e-03
-2.32e+01  -2.30e+01  +3.17e-02  +3.44e-02  +3.30e-02  +3.46e-02  +3.51e-02  +3.20e-02  +8.87e-03  +1.02e-02  +9.08e-03  +4.72e-03  +4.72e-03  +2.33e-03
-2.30e+01  -2.28e+01  +3.28e-02  +3.45e-02  +3.35e-02  +3.49e-02  +3.50e-02  +3.27e-02  +8.87e-03  +1.02e-02  +9.07e-03  +4.74e-03  +4.74e-03  +2.35e-03
-2.28e+01  -2.25e+01  +3.32e-02  +3.45e-02  +3.37e-02  +3.48e-02  +3.48e-02  +3.29e-02  +8.88e-03  +1.02e-02  +9.08e-03  +4.75e-03  +4.76e-03  +2.37e-03
-2.25e+01  -2.22e+01  +3.30e-02  +3.42e-02  +3.35e-02  +3.45e-02  +3.46e-02  +3.28e-02  +8.88e-03  +1.02e-02  +9.07e-03  +4.77e-03  +4.77e-03  +2.40e-03
-2.22e+01  -2.20e+01  +3.29e-02  +3.38e-02  +3.32e-02  +3.41e-02  +3.41e-02  +3.27e-02  +8.90e-03  +1.02e-02  +9.07e-03  +4.78e-03  +4.79e-03  +2.42e-03
-2.20e+01  -2.18e+01  +3.26e-02  +3.35e-02  +3.29e-02  +3.35e-02  +3.36e-02  +3.24e-02  +8.91e-03  +1.02e-02  +9.07e-03  +4.80e-03  +4.80e-03  +2.44e-03
-2.18e+01  -2.15e+01  +3.24e-02  +3.32e-02  +3.27e-02  +3.34e-02  +3.33e-02  +3.22e-02  +8.90e-03  +1.02e-02  +9.06e-03  +4.82e-03  +4.82e-03  +2.47e-03
-2.15e+01  -2.12e+01  +3.22e-02  +3.29e-02  +3.23e-02  +3.29e-02  +3.30e-02  +3.21e-02  +8.92e-03  +1.02e-02  +9.06e-03  +4.84e-03  +4.84e-03  +2.49e-03
-2.12e+01  -2.10e+01  +3.21e-02  +3.26e-02  +3.22e-02  +3.27e-02  +3.27e-02  +3.20e-02  +8.92e-03  +1.02e-02  +9.06e-03  +4.86e-03  +4.85e-03  +2.52e-03
-2.10e+01  -2.08e+01  +3.21e-02  +3.23e-02  +3.20e-02  +3.23e-02  +3.23e-02  +3.19e-02  +8.95e-03  +1.02e-02  +9.06e-03  +4.87e-03  +4.87e-03  +2.54e-03
-2.08e+01  -2.05e+01  +3.19e-02  +3.18e-02  +3.17e-02  +3.19e-02  +3.18e-02  +3.16e-02  +8.94e-03  +1.02e-02  +9.06e-03  +4.90e-03  +4.89e-03  +2.57e-03
-2.05e+01  -2.02e+01  +3.17e-02  +3.14e-02  +3.15e-02  +3.14e-02  +3.14e-02  +3.14e-02  +8.95e-03  +1.02e-02  +9.05e-03  +4.92e-03  +4.92e-03  +2.60e-03
-2.02e+01  -2.00e+01  +3.16e-02  +3.12e-02  +3.12e-02  +3.12e-02  +3.10e-02  +3.13e-02  +8.98e-03  +1.02e-02  +9.04e-03  +4.95e-03  +4.96e-03  +2.63e-03
-2.00e+01  -1.98e+01  +3.14e-02  +3.08e-02  +3.10e-02  +3.08e-02  +3.07e-02  +3.11e-02  +8.99e-03  +1.02e-02  +9.08e-03  +5.02e-03  +5.01e-03  +2.68e-03
-1.98e+01  -1.95e+01  +3.14e-02  +3.05e-02  +3.09e-02  +3.05e-02  +3.04e-02  +3.11e-02  +9.02e-03  +1.02e-02  +9.08e-03  +5.09e-03  +5.08e-03  +2.74e-03
-1.95e+01  -1.92e+01  +3.14e-02  +3.02e-02  +3.08e-02  +3.03e-02  +3.00e-02  +3.10e-02  +9.06e-03  +1.03e-02  +9.11e-03  +5.21e-03  +5.21e-03  +2.83e-03
-1.92e+01  -1.90e+01  +3.13e-02  +2.98e-02  +3.05e-02  +3.00e-02  +2.98e-02  +3.08e-02  +9.15e-03  +1.03e-02  +9.17e-03  +5.43e-03  +5.42e-03  +2.98e-03
-1.90e+01  -1.88e+01  +3.10e-02  +2.95e-02  +3.03e-02  +2.95e-02  +2.93e-02  +3.06e-02  +9.29e-03  +1.05e-02  +9.30e-03  +5.80e-03  +5.77e-03  +3.21e-03
-1.88e+01  -1.85e+01  +3.08e-02  +2.91e-02  +3.01e-02  +2.93e-02  +2.90e-02  +3.04e-02  +9.54e-03  +1.07e-02  +9.50e-03  +6.44e-03  +6.41e-03  +3.62e-03
-1.85e+01  -1.82e+01  +3.06e-02  +2.88e-02  +2.98e-02  +2.89e-02  +2.84e-02  +3.02e-02  +1.00e-02  +1.12e-02  +9.94e-03  +7.56e-03  +7.51e-03  +4.33e-03
-1.82e+01  -1.80e+01  +3.06e-02  +2.85e-02  +2.96e-02  +2.85e-02  +2.83e-02  +3.01e-02  +1.08e-02  +1.21e-02  +1.07e-02  +9.57e-03  +9.48e-03  +5.60e-03
-1.80e+01  -1.78e+01  +3.06e-02  +2.81e-02  +2.93e-02  +2.82e-02  +2.77e-02  +3.00e-02  +1.23e-02  +1.36e-02  +1.22e-02  +1.31e-02  +1.30e-02  +7.85e-03
-1.78e+01  -1.75e+01  +3.04e-02  +2.78e-02  +2.91e-02  +2.78e-02  +2.73e-02  +2.98e-02  +1.50e-02  +1.65e-02  +1.48e-02  +1.94e-02  +1.91e-02  +1.20e-02
-1.75e+01  -1.72e+01  +3.02e-02  +2.74e-02  +2.88e-02  +2.75e-02  +2.71e-02  +2.96e-02  +2.00e-02  +2.18e-02  +1.96e-02  +3.05e-02  +3.02e-02  +1.95e-02
-1.72e+01  -1.70e+01  +3.02e-02  +2.73e-02  +2.86e-02  +2.72e-02  +2.67e-02  +2.95e-02  +2.89e-02  +3.16e-02  +2.85e-02  +5.03e-02  +4.98e-02  +3.38e-02
-1.70e+01  -1.68e+01  +3.04e-02  +2.71e-02  +2.88e-02  +2.73e-02  +2.66e-02  +2.97e-02  +4.50e-02  +4.93e-02  +4.45e-02  +8.55e-02  +8.46e-02  +6.19e-02
-1.68e+01  -1.65e+01  +3.15e-02  +2.78e-02  +2.97e-02  +2.79e-02  +2.73e-02  +3.07e-02  +7.26e-02  +8.02e-02  +7.24e-02  +1.48e-01  +1.46e-01  +1.20e-01
-1.65e+01  -1.62e+01  +3.87e-02  +3.26e-02  +3.58e-02  +3.30e-02  +3.23e-02  +3.69e-02  +1.11e-01  +1.28e-01  +1.14e-01  +2.58e-01  +2.55e-01  +2.48e-01
-1.62e+01  -1.60e+01  +3.79e-02  +3.27e-02  +3.57e-02  +3.33e-02  +3.27e-02  +3.74e-02  +1.65e-01  +1.87e-01  +1.65e-01  +2.18e-01  +2.18e-01  +2.69e-01
-1.60e+01  -1.58e+01  +3.46e-03  +5.16e-03  +4.81e-03  +5.92e-03  +5.94e-03  +4.99e-03  +1.92e-01  +1.78e-01  +1.89e-01  +3.10e-02  +3.62e-02  +9.35e-02
-1.58e+01  -1.55e+01  +1.22e-04  +2.12e-04  +2.46e-04  +4.06e-04  +4.13e-04  +2.52e-04  +1.17e-01  +4.57e-02  +1.16e-01  +1.83e-02  +1.96e-02  +6.42e-02
-1.55e+01  -1.52e+01  +0.00e+00  +0.00e+00  +0.00e+00  +0.00e+00  +0.00e+00  +0.00e+00  +3.42e-03  +9.05e-04  +3.51e-03  +1.49e-03  +1.45e-03  +3.31e-03
-1.52e+01  -1.50e+01  +0.00e+00  +0.00e+00  +0.00e+00  +0.00e+00  +0.00e+00  +0.00e+00  +9.22e-06  +1.60e-06  +9.78e-06  +5.60e-06  +5.58e-06  +1.09e-05
-1.50e+01  -1.48e+01  +0.00e+00  +0.00e+00  +0.00e+00  +0.00e+00  +0.00e+00  +0.00e+00  +3.84e-07  +7.99e-07  +5.13e-07  +3.63e-08  +5.03e-08  +2.13e-07
-1.48e+01  -1.45e+01  +0.00e+00  +0.00e+00  +0.00e+00  +0.00e+00  +0.00e+00  +0.00e+00  +1.97e-07  +3.67e-07  +2.08e-07  +1.41e-08  +2.82e-08  +9.39e-08
-1.45e+01  -1.42e+01  +0.00e+00  +0.00e+00  +0.00e+00  +0.00e+00  +0.00e+00  +0.00e+00  +1.47e-07  +2.82e-07  +1.60e-07  +1.41e-08  +2.01e-08  +7.10e-08
-1.42e+01  -1.40e+01  +0.00e+00  +0.00e+00  +0.00e+00  +0.00e+00  +0.00e+00  +0.00e+00  +6.57e-08  +1.13e-07  +6.41e-08  +2.01e-09  +8.06e-09  +3.01e-08
-1.40e+01  -1.38e+01  +0.00e+00  +0.00e+00  +0.00e+00  +0.00e+00  +0.00e+00  +0.00e+00  +5.05e-08  +9.87e-08  +5.87e-08  +6.04e-09  +4.03e-09  +2.53e-08
-1.38e+01  -1.35e+01  +0.00e+00  +0.00e+00  +0.00e+00  +0.00e+00  +0.00e+00  +0.00e+00  +1.52e-08  +4.70e-08  +3.74e-08  +2.01e-09  +0.00e+00  +1.20e-08
-1.35e+01  -1.32e+01  +0.00e+00  +0.00e+00  +0.00e+00  +0.00e+00  +0.00e+00  +0.00e+00  +1.01e-08  +1.41e-08  +5.34e-09  +0.00e+00  +2.01e-09  +3.61e-09
-1.32e+01  -1.30e+01  +0.00e+00  +0.00e+00  +0.00e+00  +0.00e+00  +0.00e+00  +0.00e+00  +1.52e-08  +2.35e-08  +1.07e-08  +0.00e+00  +2.01e-09  +6.02e-09
-1.30e+01  -1.28e+01  +0.00e+00  +0.00e+00  +0.00e+00  +0.00e+00  +0.00e+00  +0.00e+00  +5.05e-09  +9.40e-09  +5.34e-09  +0.00e+00  +0.00e+00  +2.41e-09
-1.28e+01  -1.25e+01  +0.00e+00  +0.00e+00  +0.00e+00  +0.00e+00  +0.00e+00  +0.00e+00  +0.00e+00  +4.70e-09  +5.34e-09  +0.00e+00  +0.00e+00  +1.20e-09
-1.25e+01  -1.22e+01  +0.00e+00  +0.00e+00  +0.00e+00  +0.00e+00  +0.00e+00  +0.00e+00  +0.00e+00  +0.00e+00  +0.00e+00  +0.00e+00  +0.00e+00  +0.00e+00
-1.22e+01  -1.20e+01  +0.00e+00  +0.00e+00  +0.00e+00  +0.00e+00  +0.00e+00  +0.00e+00  +0.00e+00  +4.70e-09  +5.34e-09  +0.00e+00  +0.00e+00  +1.20e-09
-1.20e+01  -1.18e+01  +0.00e+00  +0.00e+00  +0.00e+00  +0.00e+00  +0.00e+00  +0.00e+00  +0.00e+00  +0.00e+00  +0.00e+00  +0.00e+00  +0.00e+00  +0.00e+00
-1.18e+01  -1.15e+01  +0.00e+00  +0.00e+00  +0.00e+00  +0.00e+00  +0.00e+00  +0.00e+00  +0.00e+00  +0.00e+00  +0.00e+00  +0.00e+00  +0.00e+00  +0.00e+00
-1.15e+01  -1.12e+01  +0.00e+00  +0.00e+00  +0.00e+00  +0.00e+00  +0.00e+00  +0.00e+00  +0.00e+00  +0.00e+00  +0.00e+00  +0.00e+00  +0.00e+00  +0.00e+00
-1.12e+01  -1.10e+01  +0.00e+00  +0.00e+00  +0.00e+00  +0.00e+00  +0.00e+00  +0.00e+00  +0.00e+00  +0.00e+00  +0.00e+00  +0.00e+00  +0.00e+00  +0.00e+00
-1.10e+01  -1.08e+01  +0.00e+00  +0.00e+00  +0.00e+00  +0.00e+00  +0.00e+00  +0.00e+00  +0.00e+00  +0.00e+00  +0.00e+00  +0.00e+00  +0.00e+00  +0.00e+00
-1.08e+01  -1.05e+01  +0.00e+00  +0.00e+00  +0.00e+00  +0.00e+00  +0.00e+00  +0.00e+00  +0.00e+00  +0.00e+00  +0.00e+00  +0.00e+00  +0.00e+00  +0.00e+00
-1.05e+01  -1.02e+01  +0.00e+00  +0.00e+00  +0.00e+00  +0.00e+00  +0.00e+00  +0.00e+00  +0.00e+00  +0.00e+00  +0.00e+00  +0.00e+00  +0.00e+00  +0.00e+00
-1.02e+01  -1.00e+01  +0.00e+00  +0.00e+00  +0.00e+00  +0.00e+00  +0.00e+00  +0.00e+00  +0.00e+00  +0.00e+00  +0.00e+00  +0.00e+00  +0.00e+00  +0.00e+00
-1.00e+01  -9.75e+00  +0.00e+00  +0.00e+00  +0.00e+00  +0.00e+00  +0.00e+00  +0.00e+00  +0.00e+00  +0.00e+00  +0.00e+00  +0.00e+00  +0.00e+00  +0.00e+00
-9.75e+00  -9.50e+00  +0.00e+00  +0.00e+00  +0.00e+00  +0.00e+00  +0.00e+00  +0.00e+00  +0.00e+00  +0.00e+00  +0.00e+00  +0.00e+00  +0.00e+00  +0.00e+00
-9.50e+00  -9.25e+00  +0.00e+00  +0.00e+00  +0.00e+00  +0.00e+00  +0.00e+00  +0.00e+00  +0.00e+00  +0.00e+00  +0.00e+00  +0.00e+00  +0.00e+00  +0.00e+00
-9.25e+00  -9.00e+00  +0.00e+00  +0.00e+00  +0.00e+00  +0.00e+00  +0.00e+00  +0.00e+00  +0.00e+00  +0.00e+00  +0.00e+00  +0.00e+00  +0.00e+00  +0.00e+00
-9.00e+00  -8.75e+00  +0.00e+00  +0.00e+00  +0.00e+00  +0.00e+00  +0.00e+00  +0.00e+00  +0.00e+00  +0.00e+00  +0.00e+00  +0.00e+00  +0.00e+00  +0.00e+00
-8.75e+00  -8.50e+00  +0.00e+00  +0.00e+00  +0.00e+00  +0.00e+00  +0.00e+00  +0.00e+00  +0.00e+00  +0.00e+00  +0.00e+00  +0.00e+00  +0.00e+00  +0.00e+00
-8.50e+00  -8.25e+00  +0.00e+00  +0.00e+00  +0.00e+00  +0.00e+00  +0.00e+00  +0.00e+00  +0.00e+00  +0.00e+00  +0.00e+00  +0.00e+00  +0.00e+00  +0.00e+00
-8.25e+00  -8.00e+00  +0.00e+00  +0.00e+00  +0.00e+00  +0.00e+00  +0.00e+00  +0.00e+00  +0.00e+00  +0.00e+00  +0.00e+00  +0.00e+00  +0.00e+00  +0.00e+00
-8.00e+00  -7.75e+00  +0.00e+00  +0.00e+00  +0.00e+00  +0.00e+00  +0.00e+00  +0.00e+00  +0.00e+00  +0.00e+00  +0.00e+00  +0.00e+00  +0.00e+00  +0.00e+00
-7.75e+00  -7.50e+00  +0.00e+00  +0.00e+00  +0.00e+00  +0.00e+00  +0.00e+00  +0.00e+00  +0.00e+00  +0.00e+00  +0.00e+00  +0.00e+00  +0.00e+00  +0.00e+00
-7.50e+00  -7.25e+00  +0.00e+00  +0.00e+00  +0.00e+00  +0.00e+00  +0.00e+00  +0.00e+00  +0.00e+00  +0.00e+00  +0.00e+00  +0.00e+00  +0.00e+00  +0.00e+00
-7.25e+00  -7.00e+00  +0.00e+00  +0.00e+00  +0.00e+00  +0.00e+00  +0.00e+00  +0.00e+00  +0.00e+00  +0.00e+00  +0.00e+00  +0.00e+00  +0.00e+00  +0.00e+00
-7.00e+00  -6.75e+00  +0.00e+00  +0.00e+00  +0.00e+00  +0.00e+00  +0.00e+00  +0.00e+00  +0.00e+00  +0.00e+00  +0.00e+00  +0.00e+00  +0.00e+00  +0.00e+00
-6.75e+00  -6.50e+00  +0.00e+00  +0.00e+00  +0.00e+00  +0.00e+00  +0.00e+00  +0.00e+00  +0.00e+00  +0.00e+00  +0.00e+00  +0.00e+00  +0.00e+00  +0.00e+00
-6.50e+00  -6.25e+00  +0.00e+00  +0.00e+00  +0.00e+00  +0.00e+00  +0.00e+00  +0.00e+00  +0.00e+00  +0.00e+00  +0.00e+00  +0.00e+00  +0.00e+00  +0.00e+00
-6.25e+00  -6.00e+00  +0.00e+00  +0.00e+00  +0.00e+00  +0.00e+00  +0.00e+00  +0.00e+00  +0.00e+00  +0.00e+00  +0.00e+00  +0.00e+00  +0.00e+00  +0.00e+00
-6.00e+00  -5.75e+00  +0.00e+00  +0.00e+00  +0.00e+00  +0.00e+00  +0.00e+00  +0.00e+00  +0.00e+00  +0.00e+00  +0.00e+00  +0.00e+00  +0.00e+00  +0.00e+00
-5.75e+00  -5.50e+00  +0.00e+00  +0.00e+00  +0.00e+00  +0.00e+00  +0.00e+00  +0.00e+00  +0.00e+00  +0.00e+00  +0.00e+00  +0.00e+00  +0.00e+00  +0.00e+00
-5.50e+00  -5.25e+00  +0.00e+00  +0.00e+00  +0.00e+00  +0.00e+00  +0.00e+00  +0.00e+00  +0.00e+00  +0.00e+00  +0.00e+00  +0.00e+00  +0.00e+00  +0.00e+00
-5.25e+00  -5.00e+00  +0.00e+00  +0.00e+00  +0.00e+00  +0.00e+00  +0.00e+00  +0.00e+00  +0.00e+00  +0.00e+00  +0.00e+00  +0.00e+00  +0.00e+00  +0.00e+00
-5.00e+00  -4.75e+00  +0.00e+00  +0.00e+00  +0.00e+00  +0.00e+00  +0.00e+00  +0.00e+00  +0.00e+00  +0.00e+00  +0.00e+00  +0.00e+00  +0.00e+00  +0.00e+00
-4.75e+00  -4.50e+00  +0.00e+00  +0.00e+00  +0.00e+00  +0.00e+00  +0.00e+00  +0.00e+00  +0.00e+00  +0.00e+00  +0.00e+00  +0.00e+00  +0.00e+00  +0.00e+00
-4.50e+00  -4.25e+00  +0.00e+00  +0.00e+00  +0.00e+00  +0.00e+00  +0.00e+00  +0.00e+00  +0.00e+00  +0.00e+00  +0.00e+00  +0.00e+00  +0.00e+00  +0.00e+00
-4.25e+00  -4.00e+00  +0.00e+00  +0.00e+00  +0.00e+00  +0.00e+00  +0.00e+00  +0.00e+00  +0.00e+00  +0.00e+00  +0.00e+00  +0.00e+00  +0.00e+00  +0.00e+00
-4.00e+00  -3.75e+00  +0.00e+00  +0.00e+00  +0.00e+00  +0.00e+00  +0.00e+00  +0.00e+00  +0.00e+00  +0.00e+00  +0.00e+00  +0.00e+00  +0.00e+00  +0.00e+00
-3.75e+00  -3.50e+00  +0.00e+00  +0.00e+00  +0.00e+00  +0.00e+00  +0.00e+00  +0.00e+00  +0.00e+00  +0.00e+00  +0.00e+00  +0.00e+00  +0.00e+00  +0.00e+00
-3.50e+00  -3.25e+00  +0.00e+00  +0.00e+00  +0.00e+00  +0.00e+00  +0.00e+00  +0.00e+00  +0.00e+00  +0.00e+00  +0.00e+00  +0.00e+00  +0.00e+00  +0.00e+00
-3.25e+00  -3.00e+00  +0.00e+00  +0.00e+00  +0.00e+00  +0.00e+00  +0.00e+00  +0.00e+00  +0.00e+00  +0.00e+00  +0.00e+00  +0.00e+00  +0.00e+00  +0.00e+00
-3.00e+00  -2.75e+00  +0.00e+00  +0.00e+00  +0.00e+00  +0.00e+00  +0.00e+00  +0.00e+00  +0.00e+00  +0.00e+00  +0.00e+00  +0.00e+00  +0.00e+00  +0.00e+00
-2.75e+00  -2.50e+00  +0.00e+00  +0.00e+00  +0.00e+00  +0.00e+00  +0.00e+00  +0.00e+00  +0.00e+00  +0.00e+00  +0.00e+00  +0.00e+00  +0.00e+00  +0.00e+00
-2.50e+00  -2.25e+00  +0.00e+00  +0.00e+00  +0.00e+00  +0.00e+00  +0.00e+00  +0.00e+00  +0.00e+00  +0.00e+00  +0.00e+00  +0.00e+00  +0.00e+00  +0.00e+00
-2.25e+00  -2.00e+00  +0.00e+00  +0.00e+00  +0.00e+00  +0.00e+00  +0.00e+00  +0.00e+00  +0.00e+00  +0.00e+00  +0.00e+00  +0.00e+00  +0.00e+00  +0.00e+00
-2.00e+00  -1.75e+00  +0.00e+00  +0.00e+00  +0.00e+00  +0.00e+00  +0.00e+00  +0.00e+00  +0.00e+00  +0.00e+00  +0.00e+00  +0.00e+00  +0.00e+00  +0.00e+00
-1.75e+00  -1.50e+00  +0.00e+00  +0.00e+00  +0.00e+00  +0.00e+00  +0.00e+00  +0.00e+00  +0.00e+00  +0.00e+00  +0.00e+00  +0.00e+00  +0.00e+00  +0.00e+00
-1.50e+00  -1.25e+00  +0.00e+00  +0.00e+00  +0.00e+00  +0.00e+00  +0.00e+00  +0.00e+00  +0.00e+00  +0.00e+00  +0.00e+00  +0.00e+00  +0.00e+00  +0.00e+00
-1.25e+00  -1.00e+00  +0.00e+00  +0.00e+00  +0.00e+00  +0.00e+00  +0.00e+00  +0.00e+00  +0.00e+00  +0.00e+00  +0.00e+00  +0.00e+00  +0.00e+00  +0.00e+00
-1.00e+00  -7.50e-01  +0.00e+00  +0.00e+00  +0.00e+00  +0.00e+00  +0.00e+00  +0.00e+00  +0.00e+00  +0.00e+00  +0.00e+00  +0.00e+00  +0.00e+00  +0.00e+00
-7.50e-01  -5.00e-01  +0.00e+00  +0.00e+00  +0.00e+00  +0.00e+00  +0.00e+00  +0.00e+00  +0.00e+00  +0.00e+00  +0.00e+00  +0.00e+00  +0.00e+00  +0.00e+00
-5.00e-01  -2.50e-01  +0.00e+00  +0.00e+00  +0.00e+00  +0.00e+00  +0.00e+00  +0.00e+00  +0.00e+00  +0.00e+00  +0.00e+00  +0.00e+00  +0.00e+00  +0.00e+00
-2.50e-01  +0.00e+00  +0.00e+00  +0.00e+00  +0.00e+00  +0.00e+00  +0.00e+00  +0.00e+00  +0.00e+00  +0.00e+00  +0.00e+00  +0.00e+00  +0.00e+00  +0.00e+00
+0.00e+00  +2.50e-01  +0.00e+00  +0.00e+00  +0.00e+00  +0.00e+00  +0.00e+00  +0.00e+00  +0.00e+00  +0.00e+00  +0.00e+00  +0.00e+00  +0.00e+00  +0.00e+00
+2.50e-01  +5.00e-01  +0.00e+00  +0.00e+00  +0.00e+00  +0.00e+00  +0.00e+00  +0.00e+00  +0.00e+00  +0.00e+00  +0.00e+00  +0.00e+00  +0.00e+00  +0.00e+00
+5.00e-01  +7.50e-01  +0.00e+00  +0.00e+00  +0.00e+00  +0.00e+00  +0.00e+00  +0.00e+00  +0.00e+00  +0.00e+00  +0.00e+00  +0.00e+00  +0.00e+00  +0.00e+00
+7.50e-01  +1.00e+00  +0.00e+00  +0.00e+00  +0.00e+00  +0.00e+00  +0.00e+00  +0.00e+00  +0.00e+00  +0.00e+00  +0.00e+00  +0.00e+00  +0.00e+00  +0.00e+00
}\tableSwapOffErr

\pgfplotstableread{
Lo         Up         eta_g1     eta_h1     eta_k1     eta_l1     eta_max1   eta_g2     eta_h2     eta_k2     eta_l2     eta_max2   best_g     best_h     best_k     best_l     best_max
+0.00e+00  +1.00e-01  +9.86e-01  +9.40e-01  +9.84e-01  +9.64e-01  +9.99e-01  +9.79e-01  +9.93e-01  +9.62e-01  +9.97e-01  +9.99e-01  +0.00e+00  +9.99e-01  +9.99e-01  +0.00e+00  +1.00e+00
+1.00e-01  +2.00e-01  +4.89e-04  +1.71e-03  +7.92e-04  +9.97e-04  +2.89e-04  +5.85e-04  +4.99e-04  +1.09e-03  +1.79e-04  +2.04e-04  +0.00e+00  +1.90e-04  +1.91e-04  +0.00e+00  +4.88e-07
+2.00e-01  +3.00e-01  +3.19e-04  +1.14e-03  +4.91e-04  +6.73e-04  +1.62e-04  +3.94e-04  +2.98e-04  +7.26e-04  +1.09e-04  +1.14e-04  +0.00e+00  +1.06e-04  +1.07e-04  +0.00e+00  +2.06e-07
+3.00e-01  +4.00e-01  +2.52e-04  +9.13e-04  +3.65e-04  +5.47e-04  +1.08e-04  +3.20e-04  +2.17e-04  +5.82e-04  +8.03e-05  +7.74e-05  +0.00e+00  +7.15e-05  +7.14e-05  +0.00e+00  +9.60e-08
+4.00e-01  +5.00e-01  +2.16e-04  +7.96e-04  +2.98e-04  +4.81e-04  +7.88e-05  +2.79e-04  +1.71e-04  +5.06e-04  +6.46e-05  +5.60e-05  +0.00e+00  +5.14e-05  +5.15e-05  +0.00e+00  +5.60e-08
+5.00e-01  +6.00e-01  +1.94e-04  +7.28e-04  +2.56e-04  +4.43e-04  +5.91e-05  +2.56e-04  +1.41e-04  +4.64e-04  +5.40e-05  +4.17e-05  +0.00e+00  +3.88e-05  +3.86e-05  +0.00e+00  +2.90e-08
+6.00e-01  +7.00e-01  +1.79e-04  +6.84e-04  +2.27e-04  +4.19e-04  +4.52e-05  +2.41e-04  +1.21e-04  +4.34e-04  +4.74e-05  +3.21e-05  +0.00e+00  +2.97e-05  +2.98e-05  +0.00e+00  +2.30e-08
+7.00e-01  +8.00e-01  +1.69e-04  +6.51e-04  +2.07e-04  +4.00e-04  +3.53e-05  +2.31e-04  +1.06e-04  +4.15e-04  +4.24e-05  +2.48e-05  +0.00e+00  +2.30e-05  +2.31e-05  +0.00e+00  +1.60e-08
+8.00e-01  +9.00e-01  +1.62e-04  +6.31e-04  +1.91e-04  +3.90e-04  +2.79e-05  +2.25e-04  +9.58e-05  +4.00e-04  +3.89e-05  +1.96e-05  +0.00e+00  +1.80e-05  +1.82e-05  +0.00e+00  +5.00e-09
+9.00e-01  +1.00e+00  +1.57e-04  +6.13e-04  +1.81e-04  +3.82e-04  +2.18e-05  +2.19e-04  +8.77e-05  +3.91e-04  +3.61e-05  +1.51e-05  +0.00e+00  +1.43e-05  +1.43e-05  +0.00e+00  +8.00e-09
+1.00e+00  +1.10e+00  +1.53e-04  +6.03e-04  +1.72e-04  +3.75e-04  +1.72e-05  +2.16e-04  +8.14e-05  +3.83e-04  +3.35e-05  +1.22e-05  +1.00e+00  +1.14e-05  +1.13e-05  +1.00e+00  +1.00e-09
+1.10e+00  +1.20e+00  +1.50e-04  +5.95e-04  +1.65e-04  +3.72e-04  +1.36e-05  +2.13e-04  +7.60e-05  +3.78e-04  +3.19e-05  +9.42e-06  +0.00e+00  +8.81e-06  +8.82e-06  +0.00e+00  +1.00e-09
+1.20e+00  +1.30e+00  +1.47e-04  +5.88e-04  +1.59e-04  +3.68e-04  +1.05e-05  +2.11e-04  +7.22e-05  +3.73e-04  +3.06e-05  +7.54e-06  +0.00e+00  +6.91e-06  +6.93e-06  +0.00e+00  +2.00e-09
+1.30e+00  +1.40e+00  +1.46e-04  +5.87e-04  +1.55e-04  +3.65e-04  +8.38e-06  +2.09e-04  +6.94e-05  +3.71e-04  +2.98e-05  +6.04e-06  +0.00e+00  +5.69e-06  +5.58e-06  +0.00e+00  +0.00e+00
+1.40e+00  +1.50e+00  +1.44e-04  +5.83e-04  +1.53e-04  +3.65e-04  +6.64e-06  +2.09e-04  +6.67e-05  +3.69e-04  +2.88e-05  +4.79e-06  +0.00e+00  +4.35e-06  +4.42e-06  +0.00e+00  +0.00e+00
+1.50e+00  +1.60e+00  +1.43e-04  +5.80e-04  +1.50e-04  +3.63e-04  +5.30e-06  +2.06e-04  +6.44e-05  +3.66e-04  +2.82e-05  +3.80e-06  +0.00e+00  +3.43e-06  +3.51e-06  +0.00e+00  +0.00e+00
+1.60e+00  +1.70e+00  +1.43e-04  +5.76e-04  +1.48e-04  +3.61e-04  +4.29e-06  +2.07e-04  +6.27e-05  +3.65e-04  +2.76e-05  +3.00e-06  +0.00e+00  +2.78e-06  +2.71e-06  +0.00e+00  +0.00e+00
+1.70e+00  +1.80e+00  +1.41e-04  +5.74e-04  +1.46e-04  +3.59e-04  +3.41e-06  +2.07e-04  +6.15e-05  +3.65e-04  +2.70e-05  +2.30e-06  +0.00e+00  +2.24e-06  +2.27e-06  +0.00e+00  +0.00e+00
+1.80e+00  +1.90e+00  +1.40e-04  +5.73e-04  +1.45e-04  +3.58e-04  +2.55e-06  +2.05e-04  +6.05e-05  +3.65e-04  +2.65e-05  +1.86e-06  +0.00e+00  +1.82e-06  +1.78e-06  +0.00e+00  +0.00e+00
+1.90e+00  +2.00e+00  +1.40e-04  +5.70e-04  +1.44e-04  +3.57e-04  +2.07e-06  +2.04e-04  +6.01e-05  +3.63e-04  +2.65e-05  +1.50e-06  +0.00e+00  +1.35e-06  +1.32e-06  +0.00e+00  +0.00e+00
+2.00e+00  +2.10e+00  +1.40e-04  +5.72e-04  +1.43e-04  +3.57e-04  +1.66e-06  +2.03e-04  +5.94e-05  +3.61e-04  +2.63e-05  +1.15e-06  +0.00e+00  +1.11e-06  +1.13e-06  +0.00e+00  +0.00e+00
+2.10e+00  +2.20e+00  +1.39e-04  +5.72e-04  +1.43e-04  +3.57e-04  +1.32e-06  +2.05e-04  +5.89e-05  +3.61e-04  +2.60e-05  +8.98e-07  +0.00e+00  +8.57e-07  +8.48e-07  +0.00e+00  +0.00e+00
+2.20e+00  +2.30e+00  +1.39e-04  +5.70e-04  +1.42e-04  +3.57e-04  +1.05e-06  +2.04e-04  +5.80e-05  +3.61e-04  +2.54e-05  +7.20e-07  +0.00e+00  +6.82e-07  +6.66e-07  +0.00e+00  +0.00e+00
+2.30e+00  +2.40e+00  +1.39e-04  +5.71e-04  +1.41e-04  +3.56e-04  +8.41e-07  +2.04e-04  +5.77e-05  +3.60e-04  +2.57e-05  +5.92e-07  +0.00e+00  +5.47e-07  +5.51e-07  +0.00e+00  +0.00e+00
+2.40e+00  +2.50e+00  +1.38e-04  +5.71e-04  +1.41e-04  +3.55e-04  +6.41e-07  +2.03e-04  +5.73e-05  +3.59e-04  +2.53e-05  +4.75e-07  +0.00e+00  +4.18e-07  +4.60e-07  +0.00e+00  +0.00e+00
+2.50e+00  +2.60e+00  +1.38e-04  +5.71e-04  +1.40e-04  +3.55e-04  +5.27e-07  +2.03e-04  +5.65e-05  +3.61e-04  +2.51e-05  +3.36e-07  +0.00e+00  +3.26e-07  +3.48e-07  +0.00e+00  +0.00e+00
+2.60e+00  +2.70e+00  +1.38e-04  +5.72e-04  +1.41e-04  +3.55e-04  +4.07e-07  +2.04e-04  +5.67e-05  +3.61e-04  +2.48e-05  +3.20e-07  +0.00e+00  +2.96e-07  +2.75e-07  +0.00e+00  +0.00e+00
+2.70e+00  +2.80e+00  +1.38e-04  +5.71e-04  +1.40e-04  +3.56e-04  +3.54e-07  +2.03e-04  +5.61e-05  +3.60e-04  +2.51e-05  +2.51e-07  +0.00e+00  +2.10e-07  +1.92e-07  +0.00e+00  +0.00e+00
+2.80e+00  +2.90e+00  +1.38e-04  +5.72e-04  +1.41e-04  +3.55e-04  +2.50e-07  +2.02e-04  +5.61e-05  +3.60e-04  +2.48e-05  +1.97e-07  +0.00e+00  +1.76e-07  +1.88e-07  +0.00e+00  +0.00e+00
+2.90e+00  +3.00e+00  +1.37e-04  +5.71e-04  +1.40e-04  +3.55e-04  +2.13e-07  +2.02e-04  +5.60e-05  +3.60e-04  +2.48e-05  +1.41e-07  +0.00e+00  +1.23e-07  +1.46e-07  +0.00e+00  +0.00e+00
+3.00e+00  +3.10e+00  +1.37e-04  +5.72e-04  +1.40e-04  +3.55e-04  +1.66e-07  +2.03e-04  +5.57e-05  +3.59e-04  +2.46e-05  +1.03e-07  +0.00e+00  +1.03e-07  +1.11e-07  +0.00e+00  +0.00e+00
+3.10e+00  +3.20e+00  +1.37e-04  +5.73e-04  +1.40e-04  +3.56e-04  +1.22e-07  +2.03e-04  +5.59e-05  +3.62e-04  +2.47e-05  +9.70e-08  +0.00e+00  +8.40e-08  +9.30e-08  +0.00e+00  +0.00e+00
+3.20e+00  +3.30e+00  +1.37e-04  +5.73e-04  +1.40e-04  +3.55e-04  +1.18e-07  +2.03e-04  +5.52e-05  +3.60e-04  +2.43e-05  +6.10e-08  +0.00e+00  +6.50e-08  +7.00e-08  +0.00e+00  +0.00e+00
+3.30e+00  +3.40e+00  +1.37e-04  +5.73e-04  +1.41e-04  +3.54e-04  +7.50e-08  +2.03e-04  +5.53e-05  +3.62e-04  +2.46e-05  +6.50e-08  +0.00e+00  +4.30e-08  +5.30e-08  +0.00e+00  +0.00e+00
+3.40e+00  +3.50e+00  +1.36e-04  +5.74e-04  +1.40e-04  +3.54e-04  +4.60e-08  +2.03e-04  +5.50e-05  +3.61e-04  +2.45e-05  +4.30e-08  +0.00e+00  +4.30e-08  +5.10e-08  +0.00e+00  +0.00e+00
+3.50e+00  +3.60e+00  +1.37e-04  +5.72e-04  +1.40e-04  +3.54e-04  +5.10e-08  +2.04e-04  +5.50e-05  +3.61e-04  +2.42e-05  +4.00e-08  +0.00e+00  +3.00e-08  +3.30e-08  +0.00e+00  +0.00e+00
+3.60e+00  +3.70e+00  +1.36e-04  +5.73e-04  +1.40e-04  +3.52e-04  +4.10e-08  +2.03e-04  +5.52e-05  +3.61e-04  +2.41e-05  +4.00e-08  +0.00e+00  +3.00e-08  +2.80e-08  +0.00e+00  +0.00e+00
+3.70e+00  +3.80e+00  +1.36e-04  +5.73e-04  +1.40e-04  +3.53e-04  +2.50e-08  +2.02e-04  +5.51e-05  +3.60e-04  +2.39e-05  +2.20e-08  +0.00e+00  +1.80e-08  +1.60e-08  +0.00e+00  +0.00e+00
+3.80e+00  +3.90e+00  +1.36e-04  +5.74e-04  +1.40e-04  +3.54e-04  +2.50e-08  +2.02e-04  +5.48e-05  +3.61e-04  +2.43e-05  +2.10e-08  +0.00e+00  +1.70e-08  +1.40e-08  +0.00e+00  +0.00e+00
+3.90e+00  +4.00e+00  +1.35e-04  +5.73e-04  +1.40e-04  +3.53e-04  +2.00e-08  +2.03e-04  +5.42e-05  +3.63e-04  +2.40e-05  +1.40e-08  +0.00e+00  +1.50e-08  +1.00e-08  +0.00e+00  +0.00e+00
+4.00e+00  +4.10e+00  +1.36e-04  +5.73e-04  +1.40e-04  +3.52e-04  +1.70e-08  +2.02e-04  +5.41e-05  +3.60e-04  +2.39e-05  +1.00e-08  +0.00e+00  +1.50e-08  +1.10e-08  +0.00e+00  +0.00e+00
+4.10e+00  +4.20e+00  +1.36e-04  +5.73e-04  +1.41e-04  +3.52e-04  +1.20e-08  +2.01e-04  +5.40e-05  +3.60e-04  +2.37e-05  +7.00e-09  +0.00e+00  +9.00e-09  +5.00e-09  +0.00e+00  +0.00e+00
+4.20e+00  +4.30e+00  +1.35e-04  +5.72e-04  +1.40e-04  +3.51e-04  +8.00e-09  +2.01e-04  +5.44e-05  +3.61e-04  +2.38e-05  +3.00e-09  +0.00e+00  +5.00e-09  +6.00e-09  +0.00e+00  +0.00e+00
+4.30e+00  +4.40e+00  +1.35e-04  +5.75e-04  +1.40e-04  +3.52e-04  +9.00e-09  +2.01e-04  +5.36e-05  +3.59e-04  +2.37e-05  +9.00e-09  +0.00e+00  +4.00e-09  +6.00e-09  +0.00e+00  +0.00e+00
+4.40e+00  +4.50e+00  +1.35e-04  +5.76e-04  +1.40e-04  +3.52e-04  +8.00e-09  +2.01e-04  +5.41e-05  +3.60e-04  +2.37e-05  +2.00e-09  +0.00e+00  +7.00e-09  +2.00e-09  +0.00e+00  +0.00e+00
+4.50e+00  +4.60e+00  +1.35e-04  +5.75e-04  +1.40e-04  +3.51e-04  +5.00e-09  +2.02e-04  +5.39e-05  +3.60e-04  +2.36e-05  +4.00e-09  +0.00e+00  +3.00e-09  +5.00e-09  +0.00e+00  +0.00e+00
+4.60e+00  +4.70e+00  +1.35e-04  +5.76e-04  +1.40e-04  +3.53e-04  +1.00e-09  +2.01e-04  +5.34e-05  +3.61e-04  +2.29e-05  +5.00e-09  +0.00e+00  +4.00e-09  +2.00e-09  +0.00e+00  +0.00e+00
+4.70e+00  +4.80e+00  +1.35e-04  +5.77e-04  +1.40e-04  +3.53e-04  +0.00e+00  +2.02e-04  +5.29e-05  +3.60e-04  +2.32e-05  +1.00e-09  +0.00e+00  +4.00e-09  +1.00e-09  +0.00e+00  +0.00e+00
+4.80e+00  +4.90e+00  +1.34e-04  +5.77e-04  +1.40e-04  +3.52e-04  +2.00e-09  +2.02e-04  +5.34e-05  +3.61e-04  +2.34e-05  +1.00e-09  +0.00e+00  +1.00e-09  +1.00e-09  +0.00e+00  +0.00e+00
+4.90e+00  +5.00e+00  +1.34e-04  +5.77e-04  +1.40e-04  +3.52e-04  +2.00e-09  +2.01e-04  +5.28e-05  +3.61e-04  +2.31e-05  +1.00e-09  +0.00e+00  +1.00e-09  +2.00e-09  +0.00e+00  +0.00e+00
+5.00e+00  +5.10e+00  +1.35e-04  +5.77e-04  +1.40e-04  +3.52e-04  +3.00e-09  +2.02e-04  +5.29e-05  +3.60e-04  +2.31e-05  +0.00e+00  +0.00e+00  +3.00e-09  +1.00e-09  +0.00e+00  +0.00e+00
+5.10e+00  +5.20e+00  +1.34e-04  +5.78e-04  +1.40e-04  +3.53e-04  +1.00e-09  +2.01e-04  +5.29e-05  +3.60e-04  +2.31e-05  +3.00e-09  +0.00e+00  +2.00e-09  +0.00e+00  +0.00e+00  +0.00e+00
+5.20e+00  +5.30e+00  +1.33e-04  +5.77e-04  +1.40e-04  +3.51e-04  +0.00e+00  +2.01e-04  +5.24e-05  +3.62e-04  +2.31e-05  +0.00e+00  +0.00e+00  +3.00e-09  +0.00e+00  +0.00e+00  +0.00e+00
+5.30e+00  +5.40e+00  +1.34e-04  +5.77e-04  +1.40e-04  +3.51e-04  +2.00e-09  +2.02e-04  +5.26e-05  +3.61e-04  +2.28e-05  +0.00e+00  +0.00e+00  +1.00e-09  +1.00e-09  +0.00e+00  +0.00e+00
+5.40e+00  +5.50e+00  +1.33e-04  +5.78e-04  +1.39e-04  +3.51e-04  +1.00e-09  +2.01e-04  +5.25e-05  +3.62e-04  +2.30e-05  +1.00e-09  +0.00e+00  +0.00e+00  +0.00e+00  +0.00e+00  +0.00e+00
+5.50e+00  +5.60e+00  +1.34e-04  +5.78e-04  +1.41e-04  +3.51e-04  +0.00e+00  +2.02e-04  +5.24e-05  +3.60e-04  +2.27e-05  +0.00e+00  +0.00e+00  +0.00e+00  +0.00e+00  +0.00e+00  +0.00e+00
+5.60e+00  +5.70e+00  +1.33e-04  +5.77e-04  +1.39e-04  +3.50e-04  +0.00e+00  +2.01e-04  +5.18e-05  +3.60e-04  +2.27e-05  +0.00e+00  +0.00e+00  +0.00e+00  +1.00e-09  +0.00e+00  +0.00e+00
+5.70e+00  +5.80e+00  +1.34e-04  +5.78e-04  +1.40e-04  +3.52e-04  +0.00e+00  +2.01e-04  +5.18e-05  +3.61e-04  +2.23e-05  +1.00e-09  +0.00e+00  +0.00e+00  +0.00e+00  +0.00e+00  +0.00e+00
+5.80e+00  +5.90e+00  +1.33e-04  +5.77e-04  +1.40e-04  +3.51e-04  +0.00e+00  +2.00e-04  +5.17e-05  +3.61e-04  +2.25e-05  +0.00e+00  +0.00e+00  +0.00e+00  +0.00e+00  +0.00e+00  +0.00e+00
+5.90e+00  +6.00e+00  +1.33e-04  +5.78e-04  +1.40e-04  +3.50e-04  +0.00e+00  +2.01e-04  +5.20e-05  +3.61e-04  +2.25e-05  +0.00e+00  +0.00e+00  +0.00e+00  +0.00e+00  +0.00e+00  +0.00e+00
+6.00e+00  +6.10e+00  +1.33e-04  +5.78e-04  +1.40e-04  +3.51e-04  +0.00e+00  +2.01e-04  +5.18e-05  +3.61e-04  +2.25e-05  +0.00e+00  +0.00e+00  +0.00e+00  +0.00e+00  +0.00e+00  +0.00e+00
+6.10e+00  +6.20e+00  +1.33e-04  +5.79e-04  +1.40e-04  +3.49e-04  +0.00e+00  +2.01e-04  +5.16e-05  +3.60e-04  +2.21e-05  +0.00e+00  +0.00e+00  +0.00e+00  +0.00e+00  +0.00e+00  +0.00e+00
+6.20e+00  +6.30e+00  +1.32e-04  +5.82e-04  +1.40e-04  +3.51e-04  +0.00e+00  +2.00e-04  +5.13e-05  +3.60e-04  +2.22e-05  +1.00e-09  +0.00e+00  +0.00e+00  +0.00e+00  +0.00e+00  +0.00e+00
+6.30e+00  +6.40e+00  +1.33e-04  +5.81e-04  +1.40e-04  +3.50e-04  +0.00e+00  +2.01e-04  +5.12e-05  +3.62e-04  +2.24e-05  +0.00e+00  +0.00e+00  +1.00e-09  +0.00e+00  +0.00e+00  +0.00e+00
+6.40e+00  +6.50e+00  +1.32e-04  +5.79e-04  +1.40e-04  +3.50e-04  +0.00e+00  +2.00e-04  +5.12e-05  +3.61e-04  +2.20e-05  +0.00e+00  +0.00e+00  +0.00e+00  +0.00e+00  +0.00e+00  +0.00e+00
+6.50e+00  +6.60e+00  +1.33e-04  +5.80e-04  +1.40e-04  +3.49e-04  +0.00e+00  +2.01e-04  +5.10e-05  +3.61e-04  +2.20e-05  +0.00e+00  +0.00e+00  +0.00e+00  +0.00e+00  +0.00e+00  +0.00e+00
+6.60e+00  +6.70e+00  +1.32e-04  +5.81e-04  +1.40e-04  +3.48e-04  +0.00e+00  +2.01e-04  +5.04e-05  +3.61e-04  +2.17e-05  +0.00e+00  +0.00e+00  +0.00e+00  +0.00e+00  +0.00e+00  +0.00e+00
+6.70e+00  +6.80e+00  +1.32e-04  +5.82e-04  +1.40e-04  +3.50e-04  +0.00e+00  +2.01e-04  +5.10e-05  +3.62e-04  +2.20e-05  +0.00e+00  +0.00e+00  +0.00e+00  +0.00e+00  +0.00e+00  +0.00e+00
+6.80e+00  +6.90e+00  +1.31e-04  +5.80e-04  +1.40e-04  +3.48e-04  +0.00e+00  +2.00e-04  +5.08e-05  +3.61e-04  +2.19e-05  +0.00e+00  +0.00e+00  +0.00e+00  +0.00e+00  +0.00e+00  +0.00e+00
+6.90e+00  +7.00e+00  +1.31e-04  +5.81e-04  +1.40e-04  +3.48e-04  +0.00e+00  +2.00e-04  +5.02e-05  +3.61e-04  +2.15e-05  +0.00e+00  +0.00e+00  +0.00e+00  +0.00e+00  +0.00e+00  +0.00e+00
+7.00e+00  +7.10e+00  +1.31e-04  +5.83e-04  +1.40e-04  +3.49e-04  +0.00e+00  +2.01e-04  +5.05e-05  +3.60e-04  +2.16e-05  +0.00e+00  +0.00e+00  +0.00e+00  +0.00e+00  +0.00e+00  +0.00e+00
+7.10e+00  +7.20e+00  +1.31e-04  +5.81e-04  +1.40e-04  +3.49e-04  +0.00e+00  +2.01e-04  +5.02e-05  +3.61e-04  +2.14e-05  +0.00e+00  +0.00e+00  +0.00e+00  +0.00e+00  +0.00e+00  +0.00e+00
+7.20e+00  +7.30e+00  +1.31e-04  +5.83e-04  +1.40e-04  +3.49e-04  +0.00e+00  +2.00e-04  +4.99e-05  +3.62e-04  +2.16e-05  +0.00e+00  +0.00e+00  +0.00e+00  +0.00e+00  +0.00e+00  +0.00e+00
+7.30e+00  +7.40e+00  +1.31e-04  +5.85e-04  +1.41e-04  +3.49e-04  +0.00e+00  +2.00e-04  +5.00e-05  +3.60e-04  +2.16e-05  +0.00e+00  +0.00e+00  +0.00e+00  +0.00e+00  +0.00e+00  +0.00e+00
+7.40e+00  +7.50e+00  +1.30e-04  +5.83e-04  +1.39e-04  +3.47e-04  +0.00e+00  +2.00e-04  +4.99e-05  +3.61e-04  +2.15e-05  +0.00e+00  +0.00e+00  +0.00e+00  +0.00e+00  +0.00e+00  +0.00e+00
+7.50e+00  +7.60e+00  +1.31e-04  +5.85e-04  +1.40e-04  +3.47e-04  +0.00e+00  +2.00e-04  +4.98e-05  +3.61e-04  +2.14e-05  +0.00e+00  +0.00e+00  +0.00e+00  +0.00e+00  +0.00e+00  +0.00e+00
+7.60e+00  +7.70e+00  +1.31e-04  +5.82e-04  +1.40e-04  +3.48e-04  +0.00e+00  +2.00e-04  +4.95e-05  +3.62e-04  +2.12e-05  +0.00e+00  +0.00e+00  +0.00e+00  +0.00e+00  +0.00e+00  +0.00e+00
+7.70e+00  +7.80e+00  +1.30e-04  +5.83e-04  +1.40e-04  +3.47e-04  +0.00e+00  +2.01e-04  +4.97e-05  +3.62e-04  +2.12e-05  +0.00e+00  +0.00e+00  +0.00e+00  +0.00e+00  +0.00e+00  +0.00e+00
+7.80e+00  +7.90e+00  +1.31e-04  +5.85e-04  +1.40e-04  +3.47e-04  +0.00e+00  +2.00e-04  +4.93e-05  +3.61e-04  +2.11e-05  +0.00e+00  +0.00e+00  +0.00e+00  +0.00e+00  +0.00e+00  +0.00e+00
+7.90e+00  +8.00e+00  +1.30e-04  +5.85e-04  +1.39e-04  +3.46e-04  +0.00e+00  +2.00e-04  +4.96e-05  +3.62e-04  +2.11e-05  +0.00e+00  +0.00e+00  +0.00e+00  +0.00e+00  +0.00e+00  +0.00e+00
+8.00e+00  +8.10e+00  +1.31e-04  +5.83e-04  +1.40e-04  +3.47e-04  +0.00e+00  +1.99e-04  +4.94e-05  +3.60e-04  +2.12e-05  +0.00e+00  +0.00e+00  +0.00e+00  +0.00e+00  +0.00e+00  +0.00e+00
+8.10e+00  +8.20e+00  +1.29e-04  +5.86e-04  +1.40e-04  +3.47e-04  +0.00e+00  +1.99e-04  +4.89e-05  +3.61e-04  +2.09e-05  +0.00e+00  +0.00e+00  +0.00e+00  +0.00e+00  +0.00e+00  +0.00e+00
+8.20e+00  +8.30e+00  +1.30e-04  +5.84e-04  +1.40e-04  +3.48e-04  +0.00e+00  +1.99e-04  +4.91e-05  +3.62e-04  +2.09e-05  +0.00e+00  +0.00e+00  +0.00e+00  +0.00e+00  +0.00e+00  +0.00e+00
+8.30e+00  +8.40e+00  +1.30e-04  +5.87e-04  +1.40e-04  +3.48e-04  +0.00e+00  +2.00e-04  +4.85e-05  +3.62e-04  +2.08e-05  +0.00e+00  +0.00e+00  +0.00e+00  +0.00e+00  +0.00e+00  +0.00e+00
+8.40e+00  +8.50e+00  +1.30e-04  +5.87e-04  +1.41e-04  +3.47e-04  +0.00e+00  +2.00e-04  +4.86e-05  +3.63e-04  +2.06e-05  +0.00e+00  +0.00e+00  +0.00e+00  +0.00e+00  +0.00e+00  +0.00e+00
+8.50e+00  +8.60e+00  +1.29e-04  +5.86e-04  +1.39e-04  +3.48e-04  +0.00e+00  +1.99e-04  +4.87e-05  +3.60e-04  +2.09e-05  +0.00e+00  +0.00e+00  +0.00e+00  +0.00e+00  +0.00e+00  +0.00e+00
+8.60e+00  +8.70e+00  +1.30e-04  +5.86e-04  +1.40e-04  +3.46e-04  +0.00e+00  +2.00e-04  +4.85e-05  +3.62e-04  +2.06e-05  +0.00e+00  +0.00e+00  +0.00e+00  +0.00e+00  +0.00e+00  +0.00e+00
+8.70e+00  +8.80e+00  +1.30e-04  +5.87e-04  +1.40e-04  +3.44e-04  +0.00e+00  +2.00e-04  +4.85e-05  +3.61e-04  +2.05e-05  +0.00e+00  +0.00e+00  +0.00e+00  +0.00e+00  +0.00e+00  +0.00e+00
+8.80e+00  +8.90e+00  +1.29e-04  +5.87e-04  +1.41e-04  +3.47e-04  +0.00e+00  +1.98e-04  +4.79e-05  +3.60e-04  +2.05e-05  +0.00e+00  +0.00e+00  +0.00e+00  +0.00e+00  +0.00e+00  +0.00e+00
+8.90e+00  +9.00e+00  +1.28e-04  +5.86e-04  +1.40e-04  +3.46e-04  +0.00e+00  +1.98e-04  +4.81e-05  +3.61e-04  +2.03e-05  +0.00e+00  +0.00e+00  +0.00e+00  +0.00e+00  +0.00e+00  +0.00e+00
+9.00e+00  +9.10e+00  +1.28e-04  +5.88e-04  +1.40e-04  +3.45e-04  +0.00e+00  +1.99e-04  +4.81e-05  +3.61e-04  +2.04e-05  +0.00e+00  +0.00e+00  +0.00e+00  +0.00e+00  +0.00e+00  +0.00e+00
+9.10e+00  +9.20e+00  +1.29e-04  +5.89e-04  +1.40e-04  +3.45e-04  +0.00e+00  +1.99e-04  +4.76e-05  +3.61e-04  +2.02e-05  +0.00e+00  +0.00e+00  +0.00e+00  +0.00e+00  +0.00e+00  +0.00e+00
+9.20e+00  +9.30e+00  +1.28e-04  +5.89e-04  +1.40e-04  +3.45e-04  +0.00e+00  +1.99e-04  +4.73e-05  +3.61e-04  +2.01e-05  +0.00e+00  +0.00e+00  +0.00e+00  +0.00e+00  +0.00e+00  +0.00e+00
+9.30e+00  +9.40e+00  +1.29e-04  +5.90e-04  +1.40e-04  +3.45e-04  +0.00e+00  +1.99e-04  +4.73e-05  +3.62e-04  +2.02e-05  +0.00e+00  +0.00e+00  +0.00e+00  +0.00e+00  +0.00e+00  +0.00e+00
+9.40e+00  +9.50e+00  +1.28e-04  +5.87e-04  +1.40e-04  +3.45e-04  +0.00e+00  +1.99e-04  +4.75e-05  +3.61e-04  +2.02e-05  +0.00e+00  +0.00e+00  +0.00e+00  +0.00e+00  +0.00e+00  +0.00e+00
+9.50e+00  +9.60e+00  +1.28e-04  +5.88e-04  +1.40e-04  +3.45e-04  +0.00e+00  +1.98e-04  +4.70e-05  +3.62e-04  +1.97e-05  +0.00e+00  +0.00e+00  +0.00e+00  +0.00e+00  +0.00e+00  +0.00e+00
+9.60e+00  +9.70e+00  +1.28e-04  +5.89e-04  +1.40e-04  +3.46e-04  +0.00e+00  +1.98e-04  +4.73e-05  +3.63e-04  +2.00e-05  +0.00e+00  +0.00e+00  +0.00e+00  +0.00e+00  +0.00e+00  +0.00e+00
+9.70e+00  +9.80e+00  +1.28e-04  +5.88e-04  +1.40e-04  +3.45e-04  +0.00e+00  +1.99e-04  +4.67e-05  +3.62e-04  +1.97e-05  +0.00e+00  +0.00e+00  +0.00e+00  +0.00e+00  +0.00e+00  +0.00e+00
+9.80e+00  +9.90e+00  +1.27e-04  +5.88e-04  +1.40e-04  +3.45e-04  +0.00e+00  +1.99e-04  +4.70e-05  +3.62e-04  +1.99e-05  +0.00e+00  +0.00e+00  +0.00e+00  +0.00e+00  +0.00e+00  +0.00e+00
+9.90e+00  +1.00e+01  +1.27e-04  +5.88e-04  +1.40e-04  +3.44e-04  +0.00e+00  +1.98e-04  +4.66e-05  +3.62e-04  +1.97e-05  +0.00e+00  +0.00e+00  +0.00e+00  +0.00e+00  +0.00e+00  +0.00e+00
}\tableSwapOffEta

\pgfplotstableread{
Lo         Up         VCQ01      PAQ01      PBU01      KP01       QH01       Max01      VCQerr     PAQerr     PBUerr     KPerr      QHerr      Maxerr
-2.40e+01  -2.38e+01  +1.24e-02  +7.44e-03  +4.70e-02  +8.65e-03  +3.39e-02  +8.48e-03  +8.81e-03  +5.50e-03  +4.52e-03  +4.45e-03  +4.68e-03  +2.02e-03
-2.38e+01  -2.35e+01  +2.04e-02  +7.41e-03  +5.19e-02  +8.59e-03  +3.36e-02  +9.28e-03  +8.80e-03  +5.51e-03  +4.54e-03  +4.47e-03  +4.69e-03  +2.04e-03
-2.35e+01  -2.32e+01  +2.93e-02  +7.48e-03  +5.71e-02  +8.60e-03  +3.45e-02  +1.03e-02  +8.81e-03  +5.52e-03  +4.52e-03  +4.48e-03  +4.71e-03  +2.07e-03
-2.32e+01  -2.30e+01  +3.14e-02  +7.51e-03  +5.50e-02  +8.63e-03  +3.47e-02  +1.06e-02  +8.81e-03  +5.53e-03  +4.53e-03  +4.50e-03  +4.73e-03  +2.08e-03
-2.30e+01  -2.28e+01  +3.29e-02  +7.53e-03  +5.47e-02  +8.63e-03  +3.48e-02  +1.08e-02  +8.83e-03  +5.52e-03  +4.52e-03  +4.51e-03  +4.74e-03  +2.10e-03
-2.28e+01  -2.25e+01  +3.34e-02  +7.52e-03  +5.26e-02  +8.62e-03  +3.45e-02  +1.08e-02  +8.83e-03  +5.53e-03  +4.52e-03  +4.52e-03  +4.76e-03  +2.12e-03
-2.25e+01  -2.22e+01  +3.33e-02  +7.48e-03  +5.17e-02  +8.57e-03  +3.42e-02  +1.08e-02  +8.82e-03  +5.52e-03  +4.52e-03  +4.54e-03  +4.77e-03  +2.14e-03
-2.22e+01  -2.20e+01  +3.32e-02  +7.46e-03  +4.90e-02  +8.54e-03  +3.38e-02  +1.07e-02  +8.84e-03  +5.53e-03  +4.51e-03  +4.55e-03  +4.79e-03  +2.16e-03
-2.20e+01  -2.18e+01  +3.28e-02  +7.43e-03  +4.62e-02  +8.50e-03  +3.32e-02  +1.07e-02  +8.84e-03  +5.53e-03  +4.53e-03  +4.56e-03  +4.80e-03  +2.19e-03
-2.18e+01  -2.15e+01  +3.26e-02  +7.42e-03  +4.43e-02  +8.49e-03  +3.30e-02  +1.06e-02  +8.84e-03  +5.53e-03  +4.52e-03  +4.58e-03  +4.82e-03  +2.21e-03
-2.15e+01  -2.12e+01  +3.24e-02  +7.40e-03  +4.24e-02  +8.46e-03  +3.27e-02  +1.06e-02  +8.85e-03  +5.53e-03  +4.53e-03  +4.60e-03  +4.83e-03  +2.23e-03
-2.12e+01  -2.10e+01  +3.25e-02  +7.36e-03  +4.11e-02  +8.42e-03  +3.22e-02  +1.06e-02  +8.86e-03  +5.52e-03  +4.52e-03  +4.61e-03  +4.86e-03  +2.26e-03
-2.10e+01  -2.08e+01  +3.25e-02  +7.36e-03  +3.87e-02  +8.42e-03  +3.20e-02  +1.06e-02  +8.86e-03  +5.54e-03  +4.52e-03  +4.62e-03  +4.87e-03  +2.28e-03
-2.08e+01  -2.05e+01  +3.23e-02  +7.31e-03  +3.66e-02  +8.34e-03  +3.14e-02  +1.05e-02  +8.88e-03  +5.53e-03  +4.53e-03  +4.65e-03  +4.89e-03  +2.30e-03
-2.05e+01  -2.02e+01  +3.20e-02  +7.30e-03  +3.50e-02  +8.32e-03  +3.10e-02  +1.05e-02  +8.87e-03  +5.54e-03  +4.52e-03  +4.67e-03  +4.92e-03  +2.33e-03
-2.02e+01  -2.00e+01  +3.19e-02  +7.26e-03  +3.32e-02  +8.29e-03  +3.05e-02  +1.04e-02  +8.88e-03  +5.54e-03  +4.54e-03  +4.71e-03  +4.95e-03  +2.36e-03
-2.00e+01  -1.98e+01  +3.17e-02  +7.26e-03  +3.12e-02  +8.28e-03  +3.04e-02  +1.04e-02  +8.91e-03  +5.55e-03  +4.53e-03  +4.74e-03  +5.00e-03  +2.40e-03
-1.98e+01  -1.95e+01  +3.17e-02  +7.23e-03  +2.89e-02  +8.22e-03  +3.02e-02  +1.04e-02  +8.94e-03  +5.55e-03  +4.54e-03  +4.81e-03  +5.08e-03  +2.46e-03
-1.95e+01  -1.92e+01  +3.16e-02  +7.22e-03  +2.73e-02  +8.21e-03  +3.01e-02  +1.04e-02  +8.98e-03  +5.56e-03  +4.55e-03  +4.92e-03  +5.21e-03  +2.54e-03
-1.92e+01  -1.90e+01  +3.16e-02  +7.17e-03  +2.54e-02  +8.15e-03  +2.94e-02  +1.03e-02  +9.07e-03  +5.60e-03  +4.56e-03  +5.10e-03  +5.41e-03  +2.67e-03
-1.90e+01  -1.88e+01  +3.13e-02  +7.15e-03  +2.30e-02  +8.12e-03  +2.92e-02  +1.03e-02  +9.19e-03  +5.65e-03  +4.62e-03  +5.41e-03  +5.77e-03  +2.89e-03
-1.88e+01  -1.85e+01  +3.10e-02  +7.12e-03  +2.16e-02  +8.09e-03  +2.87e-02  +1.02e-02  +9.45e-03  +5.76e-03  +4.68e-03  +5.94e-03  +6.41e-03  +3.25e-03
-1.85e+01  -1.82e+01  +3.09e-02  +7.10e-03  +1.92e-02  +8.06e-03  +2.84e-02  +1.02e-02  +9.89e-03  +5.94e-03  +4.84e-03  +6.89e-03  +7.52e-03  +3.89e-03
-1.82e+01  -1.80e+01  +3.08e-02  +7.06e-03  +1.75e-02  +8.02e-03  +2.82e-02  +1.01e-02  +1.07e-02  +6.28e-03  +5.11e-03  +8.57e-03  +9.48e-03  +5.02e-03
-1.80e+01  -1.78e+01  +3.08e-02  +7.06e-03  +1.56e-02  +8.00e-03  +2.76e-02  +1.01e-02  +1.22e-02  +6.93e-03  +5.60e-03  +1.15e-02  +1.30e-02  +7.04e-03
-1.78e+01  -1.75e+01  +3.07e-02  +7.02e-03  +1.39e-02  +7.95e-03  +2.72e-02  +1.01e-02  +1.48e-02  +8.10e-03  +6.56e-03  +1.68e-02  +1.92e-02  +1.07e-02
-1.75e+01  -1.72e+01  +3.04e-02  +7.01e-03  +1.16e-02  +7.94e-03  +2.70e-02  +1.00e-02  +1.97e-02  +1.02e-02  +8.28e-03  +2.62e-02  +3.02e-02  +1.74e-02
-1.72e+01  -1.70e+01  +3.04e-02  +6.98e-03  +9.95e-03  +7.91e-03  +2.66e-02  +1.00e-02  +2.85e-02  +1.42e-02  +1.15e-02  +4.28e-02  +4.99e-02  +2.99e-02
-1.70e+01  -1.68e+01  +3.06e-02  +6.96e-03  +7.93e-03  +7.87e-03  +2.66e-02  +1.00e-02  +4.43e-02  +2.13e-02  +1.72e-02  +7.23e-02  +8.49e-02  +5.40e-02
-1.68e+01  -1.65e+01  +3.20e-02  +6.99e-03  +6.29e-03  +7.90e-03  +2.73e-02  +1.02e-02  +7.17e-02  +3.39e-02  +2.74e-02  +1.25e-01  +1.47e-01  +1.03e-01
-1.65e+01  -1.62e+01  +3.96e-02  +7.24e-03  +3.85e-03  +8.16e-03  +3.17e-02  +1.13e-02  +1.09e-01  +5.35e-02  +4.29e-02  +2.17e-01  +2.56e-01  +2.05e-01
-1.62e+01  -1.60e+01  +3.78e-02  +7.26e-03  +4.90e-04  +8.21e-03  +3.30e-02  +1.13e-02  +1.62e-01  +7.69e-02  +6.06e-02  +1.82e-01  +2.18e-01  +2.10e-01
-1.60e+01  -1.58e+01  +2.02e-03  +5.74e-03  +0.00e+00  +6.66e-03  +7.66e-03  +5.50e-03  +1.93e-01  +7.23e-02  +7.87e-02  +2.21e-02  +3.48e-02  +5.88e-02
-1.58e+01  -1.55e+01  +0.00e+00  +5.40e-03  +0.00e+00  +6.15e-03  +6.94e-04  +4.73e-03  +1.25e-01  +1.89e-02  +6.59e-02  +1.34e-02  +1.92e-02  +4.93e-02
-1.55e+01  -1.52e+01  +0.00e+00  +5.35e-03  +0.00e+00  +6.08e-03  +0.00e+00  +4.66e-03  +3.67e-03  +1.71e-03  +3.18e-03  +1.77e-03  +1.27e-03  +2.78e-03
-1.52e+01  -1.50e+01  +0.00e+00  +5.34e-03  +0.00e+00  +6.05e-03  +0.00e+00  +4.65e-03  +1.02e-05  +1.42e-03  +1.73e-03  +9.17e-04  +4.04e-06  +5.97e-04
-1.50e+01  -1.48e+01  +0.00e+00  +5.34e-03  +0.00e+00  +6.04e-03  +0.00e+00  +4.64e-03  +2.72e-07  +1.41e-03  +1.71e-03  +9.18e-04  +4.03e-08  +5.86e-04
-1.48e+01  -1.45e+01  +0.00e+00  +5.30e-03  +0.00e+00  +5.99e-03  +0.00e+00  +4.61e-03  +1.95e-07  +1.41e-03  +1.71e-03  +9.18e-04  +2.42e-08  +5.84e-04
-1.45e+01  -1.42e+01  +0.00e+00  +5.30e-03  +0.00e+00  +6.00e-03  +0.00e+00  +4.62e-03  +1.00e-07  +1.40e-03  +1.70e-03  +9.16e-04  +1.41e-08  +5.82e-04
-1.42e+01  -1.40e+01  +0.00e+00  +5.28e-03  +0.00e+00  +5.96e-03  +0.00e+00  +4.59e-03  +2.78e-08  +1.40e-03  +1.70e-03  +9.17e-04  +6.04e-09  +5.82e-04
-1.40e+01  -1.38e+01  +0.00e+00  +5.26e-03  +0.00e+00  +5.94e-03  +0.00e+00  +4.58e-03  +3.33e-08  +1.40e-03  +1.69e-03  +9.19e-04  +8.06e-09  +5.78e-04
-1.38e+01  -1.35e+01  +0.00e+00  +5.25e-03  +0.00e+00  +5.93e-03  +0.00e+00  +4.57e-03  +2.78e-08  +1.39e-03  +1.69e-03  +9.18e-04  +6.04e-09  +5.76e-04
-1.35e+01  -1.32e+01  +0.00e+00  +5.25e-03  +0.00e+00  +5.90e-03  +0.00e+00  +4.57e-03  +1.11e-08  +1.39e-03  +1.68e-03  +9.20e-04  +4.03e-09  +5.74e-04
-1.32e+01  -1.30e+01  +0.00e+00  +5.21e-03  +0.00e+00  +5.86e-03  +0.00e+00  +4.53e-03  +2.22e-08  +1.38e-03  +1.68e-03  +9.18e-04  +6.04e-09  +5.72e-04
-1.30e+01  -1.28e+01  +0.00e+00  +5.20e-03  +0.00e+00  +5.86e-03  +0.00e+00  +4.53e-03  +0.00e+00  +1.38e-03  +1.67e-03  +9.19e-04  +0.00e+00  +5.71e-04
-1.28e+01  -1.25e+01  +0.00e+00  +5.18e-03  +0.00e+00  +5.83e-03  +0.00e+00  +4.51e-03  +5.56e-09  +1.37e-03  +1.67e-03  +9.16e-04  +0.00e+00  +5.70e-04
-1.25e+01  -1.22e+01  +0.00e+00  +5.17e-03  +0.00e+00  +5.80e-03  +0.00e+00  +4.50e-03  +0.00e+00  +1.37e-03  +1.66e-03  +9.20e-04  +0.00e+00  +5.67e-04
-1.22e+01  -1.20e+01  +0.00e+00  +5.15e-03  +0.00e+00  +5.78e-03  +0.00e+00  +4.48e-03  +0.00e+00  +1.36e-03  +1.65e-03  +9.18e-04  +0.00e+00  +5.65e-04
-1.20e+01  -1.18e+01  +0.00e+00  +5.15e-03  +0.00e+00  +5.77e-03  +0.00e+00  +4.48e-03  +0.00e+00  +1.36e-03  +1.65e-03  +9.22e-04  +0.00e+00  +5.62e-04
-1.18e+01  -1.15e+01  +0.00e+00  +5.13e-03  +0.00e+00  +5.75e-03  +0.00e+00  +4.46e-03  +0.00e+00  +1.35e-03  +1.64e-03  +9.19e-04  +0.00e+00  +5.60e-04
-1.15e+01  -1.12e+01  +0.00e+00  +5.11e-03  +0.00e+00  +5.70e-03  +0.00e+00  +4.45e-03  +0.00e+00  +1.34e-03  +1.63e-03  +9.20e-04  +0.00e+00  +5.58e-04
-1.12e+01  -1.10e+01  +0.00e+00  +5.10e-03  +0.00e+00  +5.71e-03  +0.00e+00  +4.44e-03  +0.00e+00  +1.34e-03  +1.63e-03  +9.22e-04  +0.00e+00  +5.57e-04
-1.10e+01  -1.08e+01  +0.00e+00  +5.08e-03  +0.00e+00  +5.67e-03  +0.00e+00  +4.42e-03  +0.00e+00  +1.34e-03  +1.63e-03  +9.21e-04  +0.00e+00  +5.56e-04
-1.08e+01  -1.05e+01  +0.00e+00  +5.07e-03  +0.00e+00  +5.66e-03  +0.00e+00  +4.41e-03  +0.00e+00  +1.33e-03  +1.62e-03  +9.22e-04  +0.00e+00  +5.54e-04
-1.05e+01  -1.02e+01  +0.00e+00  +5.05e-03  +0.00e+00  +5.63e-03  +0.00e+00  +4.40e-03  +0.00e+00  +1.33e-03  +1.62e-03  +9.22e-04  +0.00e+00  +5.52e-04
-1.02e+01  -1.00e+01  +0.00e+00  +5.04e-03  +0.00e+00  +5.62e-03  +0.00e+00  +4.39e-03  +0.00e+00  +1.32e-03  +1.62e-03  +9.22e-04  +0.00e+00  +5.51e-04
-1.00e+01  -9.75e+00  +0.00e+00  +5.01e-03  +0.00e+00  +5.59e-03  +0.00e+00  +4.36e-03  +0.00e+00  +1.32e-03  +1.61e-03  +9.23e-04  +0.00e+00  +5.50e-04
-9.75e+00  -9.50e+00  +0.00e+00  +5.01e-03  +0.00e+00  +5.58e-03  +0.00e+00  +4.36e-03  +0.00e+00  +1.32e-03  +1.61e-03  +9.22e-04  +0.00e+00  +5.49e-04
-9.50e+00  -9.25e+00  +0.00e+00  +5.00e-03  +0.00e+00  +5.55e-03  +0.00e+00  +4.35e-03  +0.00e+00  +1.31e-03  +1.61e-03  +9.25e-04  +0.00e+00  +5.46e-04
-9.25e+00  -9.00e+00  +0.00e+00  +4.99e-03  +0.00e+00  +5.53e-03  +0.00e+00  +4.34e-03  +0.00e+00  +1.31e-03  +1.61e-03  +9.25e-04  +0.00e+00  +5.48e-04
-9.00e+00  -8.75e+00  +0.00e+00  +4.97e-03  +0.00e+00  +5.51e-03  +0.00e+00  +4.32e-03  +0.00e+00  +1.30e-03  +1.62e-03  +9.25e-04  +0.00e+00  +5.49e-04
-8.75e+00  -8.50e+00  +0.00e+00  +4.96e-03  +0.00e+00  +5.49e-03  +0.00e+00  +4.31e-03  +0.00e+00  +1.30e-03  +1.62e-03  +9.26e-04  +0.00e+00  +5.47e-04
-8.50e+00  -8.25e+00  +0.00e+00  +4.94e-03  +0.00e+00  +5.47e-03  +0.00e+00  +4.30e-03  +0.00e+00  +1.29e-03  +1.61e-03  +9.24e-04  +0.00e+00  +5.45e-04
-8.25e+00  -8.00e+00  +0.00e+00  +4.91e-03  +0.00e+00  +5.38e-03  +0.00e+00  +4.28e-03  +0.00e+00  +1.28e-03  +1.54e-03  +9.25e-04  +0.00e+00  +5.30e-04
-8.00e+00  -7.75e+00  +0.00e+00  +4.91e-03  +0.00e+00  +5.27e-03  +0.00e+00  +4.27e-03  +0.00e+00  +1.28e-03  +1.49e-03  +9.25e-04  +0.00e+00  +5.21e-04
-7.75e+00  -7.50e+00  +0.00e+00  +4.89e-03  +0.00e+00  +5.21e-03  +0.00e+00  +4.26e-03  +0.00e+00  +1.28e-03  +1.52e-03  +9.25e-04  +0.00e+00  +5.25e-04
-7.50e+00  -7.25e+00  +0.00e+00  +4.86e-03  +0.00e+00  +5.17e-03  +0.00e+00  +4.23e-03  +0.00e+00  +1.27e-03  +1.54e-03  +9.27e-04  +0.00e+00  +5.28e-04
-7.25e+00  -7.00e+00  +0.00e+00  +4.85e-03  +0.00e+00  +5.17e-03  +0.00e+00  +4.22e-03  +0.00e+00  +1.27e-03  +1.54e-03  +9.27e-04  +0.00e+00  +5.28e-04
-7.00e+00  -6.75e+00  +0.00e+00  +4.84e-03  +0.00e+00  +5.14e-03  +0.00e+00  +4.21e-03  +0.00e+00  +1.26e-03  +1.54e-03  +9.28e-04  +0.00e+00  +5.26e-04
-6.75e+00  -6.50e+00  +0.00e+00  +4.82e-03  +0.00e+00  +5.10e-03  +0.00e+00  +4.20e-03  +0.00e+00  +1.26e-03  +1.54e-03  +9.27e-04  +0.00e+00  +5.24e-04
-6.50e+00  -6.25e+00  +0.00e+00  +4.81e-03  +0.00e+00  +5.10e-03  +0.00e+00  +4.19e-03  +0.00e+00  +1.25e-03  +1.53e-03  +9.28e-04  +0.00e+00  +5.24e-04
-6.25e+00  -6.00e+00  +0.00e+00  +4.79e-03  +0.00e+00  +5.07e-03  +0.00e+00  +4.17e-03  +0.00e+00  +1.25e-03  +1.53e-03  +9.28e-04  +0.00e+00  +5.21e-04
-6.00e+00  -5.75e+00  +0.00e+00  +4.78e-03  +0.00e+00  +5.05e-03  +0.00e+00  +4.16e-03  +0.00e+00  +1.24e-03  +1.52e-03  +9.28e-04  +0.00e+00  +5.19e-04
-5.75e+00  -5.50e+00  +0.00e+00  +4.77e-03  +0.00e+00  +5.04e-03  +0.00e+00  +4.15e-03  +0.00e+00  +1.24e-03  +1.51e-03  +9.28e-04  +0.00e+00  +5.17e-04
-5.50e+00  -5.25e+00  +0.00e+00  +4.76e-03  +0.00e+00  +5.01e-03  +0.00e+00  +4.14e-03  +0.00e+00  +1.23e-03  +1.51e-03  +9.29e-04  +0.00e+00  +5.15e-04
-5.25e+00  -5.00e+00  +0.00e+00  +4.72e-03  +0.00e+00  +4.99e-03  +0.00e+00  +4.11e-03  +0.00e+00  +1.23e-03  +1.51e-03  +9.29e-04  +0.00e+00  +5.14e-04
-5.00e+00  -4.75e+00  +0.00e+00  +4.71e-03  +0.00e+00  +4.97e-03  +0.00e+00  +4.10e-03  +0.00e+00  +1.23e-03  +1.50e-03  +9.30e-04  +0.00e+00  +5.13e-04
-4.75e+00  -4.50e+00  +0.00e+00  +4.70e-03  +0.00e+00  +4.95e-03  +0.00e+00  +4.09e-03  +0.00e+00  +1.22e-03  +1.49e-03  +9.30e-04  +0.00e+00  +5.10e-04
-4.50e+00  -4.25e+00  +0.00e+00  +4.69e-03  +0.00e+00  +4.93e-03  +0.00e+00  +4.08e-03  +0.00e+00  +1.22e-03  +1.49e-03  +9.33e-04  +0.00e+00  +5.10e-04
-4.25e+00  -4.00e+00  +0.00e+00  +4.67e-03  +0.00e+00  +4.90e-03  +0.00e+00  +4.07e-03  +0.00e+00  +1.21e-03  +1.48e-03  +9.31e-04  +0.00e+00  +5.06e-04
-4.00e+00  -3.75e+00  +0.00e+00  +4.65e-03  +0.00e+00  +4.88e-03  +0.00e+00  +4.05e-03  +0.00e+00  +1.20e-03  +1.48e-03  +9.32e-04  +0.00e+00  +5.04e-04
-3.75e+00  -3.50e+00  +0.00e+00  +4.63e-03  +0.00e+00  +4.86e-03  +0.00e+00  +4.03e-03  +0.00e+00  +1.20e-03  +1.47e-03  +9.34e-04  +0.00e+00  +5.02e-04
-3.50e+00  -3.25e+00  +0.00e+00  +4.63e-03  +0.00e+00  +4.84e-03  +0.00e+00  +4.03e-03  +0.00e+00  +1.20e-03  +1.47e-03  +9.33e-04  +0.00e+00  +5.02e-04
-3.25e+00  -3.00e+00  +0.00e+00  +4.61e-03  +0.00e+00  +4.82e-03  +0.00e+00  +4.01e-03  +0.00e+00  +1.19e-03  +1.46e-03  +9.34e-04  +0.00e+00  +4.98e-04
-3.00e+00  -2.75e+00  +0.00e+00  +4.60e-03  +0.00e+00  +4.79e-03  +0.00e+00  +4.00e-03  +0.00e+00  +1.19e-03  +1.45e-03  +9.33e-04  +0.00e+00  +4.96e-04
-2.75e+00  -2.50e+00  +0.00e+00  +4.58e-03  +0.00e+00  +4.78e-03  +0.00e+00  +3.98e-03  +0.00e+00  +1.18e-03  +1.45e-03  +9.37e-04  +0.00e+00  +4.95e-04
-2.50e+00  -2.25e+00  +0.00e+00  +4.56e-03  +0.00e+00  +4.75e-03  +0.00e+00  +3.97e-03  +0.00e+00  +1.18e-03  +1.44e-03  +9.40e-04  +0.00e+00  +4.92e-04
-2.25e+00  -2.00e+00  +0.00e+00  +4.55e-03  +0.00e+00  +4.73e-03  +0.00e+00  +3.96e-03  +0.00e+00  +1.19e-03  +1.44e-03  +9.39e-04  +0.00e+00  +4.91e-04
-2.00e+00  -1.75e+00  +0.00e+00  +4.54e-03  +0.00e+00  +4.71e-03  +0.00e+00  +3.95e-03  +0.00e+00  +1.19e-03  +1.44e-03  +9.44e-04  +0.00e+00  +4.92e-04
-1.75e+00  -1.50e+00  +0.00e+00  +4.53e-03  +0.00e+00  +4.69e-03  +0.00e+00  +3.94e-03  +0.00e+00  +1.21e-03  +1.43e-03  +9.49e-04  +0.00e+00  +4.87e-04
-1.50e+00  -1.25e+00  +0.00e+00  +4.52e-03  +0.00e+00  +4.68e-03  +0.00e+00  +3.93e-03  +0.00e+00  +1.23e-03  +1.42e-03  +9.61e-04  +0.00e+00  +4.85e-04
-1.25e+00  -1.00e+00  +0.00e+00  +4.51e-03  +0.00e+00  +4.67e-03  +0.00e+00  +3.92e-03  +0.00e+00  +1.28e-03  +1.43e-03  +9.77e-04  +0.00e+00  +4.87e-04
-1.00e+00  -7.50e-01  +0.00e+00  +4.56e-03  +0.00e+00  +4.72e-03  +0.00e+00  +3.97e-03  +0.00e+00  +1.39e-03  +1.43e-03  +1.02e-03  +0.00e+00  +4.87e-04
-7.50e-01  -5.00e-01  +0.00e+00  +4.73e-03  +0.00e+00  +4.88e-03  +0.00e+00  +4.12e-03  +0.00e+00  +1.54e-03  +1.42e-03  +1.09e-03  +0.00e+00  +4.85e-04
-5.00e-01  -2.50e-01  +0.00e+00  +5.50e-03  +0.00e+00  +5.67e-03  +0.00e+00  +4.79e-03  +0.00e+00  +1.86e-03  +1.34e-03  +1.30e-03  +0.00e+00  +4.58e-04
-2.50e-01  +0.00e+00  +0.00e+00  +1.14e-01  +0.00e+00  +1.15e-01  +0.00e+00  +9.93e-02  +0.00e+00  +4.18e-02  +1.43e-03  +2.99e-02  +0.00e+00  +4.90e-04
+0.00e+00  +2.50e-01  +0.00e+00  +3.42e-01  +0.00e+00  +2.81e-01  +0.00e+00  +2.98e-01  +0.00e+00  +4.27e-01  +4.67e-01  +6.39e-02  +0.00e+00  +1.60e-01
+2.50e-01  +5.00e-01  +0.00e+00  +0.00e+00  +0.00e+00  +0.00e+00  +0.00e+00  +0.00e+00  +0.00e+00  +0.00e+00  +0.00e+00  +0.00e+00  +0.00e+00  +0.00e+00
+5.00e-01  +7.50e-01  +0.00e+00  +0.00e+00  +0.00e+00  +0.00e+00  +0.00e+00  +0.00e+00  +0.00e+00  +0.00e+00  +0.00e+00  +0.00e+00  +0.00e+00  +0.00e+00
+7.50e-01  +1.00e+00  +0.00e+00  +0.00e+00  +0.00e+00  +0.00e+00  +0.00e+00  +0.00e+00  +0.00e+00  +0.00e+00  +0.00e+00  +0.00e+00  +0.00e+00  +0.00e+00
}\tableZhaErr

\title{Towards a more robust algorithm for computing the restricted
  singular value decomposition\thanks{Version: February 10, 2020.
  \funding{This work was supported in part by the Deutsche
  Forschungsgemeinschaft through the collaborative research centre
  SFB-TRR55.}}}
\author{Ian~N.~Zwaan\thanks{Faculty of Mathematics and Natural Sciences,
  Bergische Universit\"at Wuppertal, \url{ianzwaan.com}.}}
\date{}

\ifx\etnacls\undefined\else
  \shorttitle{A more robust RSVD algorithm}
  \shortauthor{I.~N.~Zwaan}
\fi

\ifx\siamartcls\undefined\else
  \headers{A more robust RSVD algorithm}{I.~N.~Zwaan}
\fi

\begin{document}

\maketitle


\begin{abstract}
  A new algorithm to compute the restricted singular value decomposition
  of dense matrices is presented. Like Zha's method~\cite{Zha92},
  the new algorithm uses an implicit Kogbetliantz iteration, but with
  four major innovations.  The first innovation is a useful quasi-upper
  triangular generalized Schur form that just requires orthonormal
  transformations to compute.  Depending on the application, this Schur
  form can be used instead of the full decomposition.  The second
  innovation is a new preprocessing phase that requires fewer rank
  determinations than previous methods.  The third innovation is a
  numerically stable RSVD algorithm for $2\times 2$ upper-triangular
  matrices, which forms a key component of the implicit Kogbetliantz
  iteration.  The fourth innovation is an alternative scaling for the
  restricted singular triplets that results in elegant formulas for
  their computation.  Beyond these four innovations, the qualitative
  (numerical) characteristics of the algorithm are discussed
  extensively. Some numerical challenges in the (optional)
  postprocessing phase are considered too; though, their solutions
  require further research.  Numerical tests and examples confirm the
  effectiveness of the method.
\end{abstract}









\section{Introduction}\label{sec:int}

The restricted singular value decomposition (RSVD) is a generalization
of the ordinary singular value decomposition (SVD or OSVD) to matrix
triplets.  Applications of the RSVD include, for example, rank
minimization of structured perturbations, unitarily invariant norm
minimization with rank constraints, low rank approximation of
partitioned matrices, restricted total least squares, generalized
Gauss--Markov models, etc.  See, e.g., Zha~\cite{Zha91,Zha92},
De~Moor and Golub~\cite{MG91}, and their references for more
information. The problem is that computing the RSVD accurately and
robustly is challenging, and avoiding numerical pitfalls is hard.  The
goal of the ideas and algorithms presented in this work is to improve
upon existing computation methods in these areas, even though some
numerical challenges remain.

The RSVD is a little-known generalization of the OSVD. Better-known is
``the'' generalized singular value decomposition (GSVD) for matrix
pairs; see, e.g., Bai~\cite{ZB92gsvd}. For now, it suffices to think
of the RSVD as a GSVD for three matrices instead of two, with a formal
definition following in Section~\ref{sec:th}. Since the RSVD and GSVD
are just two out of infinitely many generalizations of the
OSVD~\cite{MZ91}, a more appropriate name for the GSVD is the
``quotient singular value decomposition'' (QSVD). Thus, we adopt 
the mnemonics O-Q-R-SVD as standardized nomenclature for the rest
of the text, as suggested by De~Moor and Golub~\cite{MG89gsvds}.
For more information on the relation between the SVD, QSVD, and RSVD,
see, for example, De~Moor and Golub~\cite[Sec.~2.2.4]{MG91}.

One illustration of the fact that computing the RSVD in a numerically
sound way is not straightforward is Zha's constructive
proof~\cite[Thm.~3.2]{Zha91}, which he describes as unsuitable for
computation. This is because it uses transformations with potentially
ill-conditioned matrices in intermediate steps. Zha addresses this issue
by deriving an implicit Kogbetliantz algorithm~\cite{Zha92}, but this
algorithm lacks a (numerically) stable method for computing $2\times 2$
RSVDs (cf.\ the $2\times 2$ QSVD from Bai and Demmel~\cite{BD93}).
Furthermore, the preprocessing phase of his implicit Kogbetliantz
algorithm requires a sequence of up to four rank decisions, where each
depends on the previous one.  These dependencies, and the fact that rank
determination is an ill-posed problem in floating-point arithmetic, make
the preprocessing prone to errors. For example, it would be
straightforward to construct a matrix triplet where we should have a
clear gap in the singular values for each rank decision in exact
arithmetic, but no longer have any gap (or a gap in the right place) for
the fourth, or even third, rank decision in floating-point arithmetic.
These faulty rank decisions can even show up if we use OSVD instead of,
e.g., QR with pivoting, for the rank decisions.

Chu, De~Lathauwer, and De~Moor~\cite{CLM00} present a QR based
method which does not require a numerically stable $2\times 2$ RSVD.
Still, their method requires a sequence of up to five mutually dependent
rank decisions, and it may also require nonorthonormal transformations
in the preprocessing phase.

Another algorithm to compute the restricted singular values (RSVs) is
due to Drma\v{c}~\cite{Drma00}, who uses both a Jacobi-type iteration
and nonorthonormal transformations.  Despite the latter, the algorithm
still has favorable numerical properties, such as independence of
certain types of diagonal scaling. Drma\v{c} also provides a bound on
the backward error of his method, and discusses when one can expect the
computed singular values to have high relative accuracy. A potential
downside of this method is that it does not compute the ``full'' RSVD.\@
Another issue is that the algorithm requires the assumption that one of
the input matrices is nonsingular, which needs not be true in the
general case.

Like Zha's algorithm, the algorithm in this work centers around an
implicit Kogbetliantz iteration, but with four main innovations over the
existing algorithms.  The first main innovation is a generalized
Schur-form RSVD consisting of a triplet of quasi upper-triangular
matrices that we can compute just with orthonormal transformations. In
particular, this Schur form allows us to skip the postprocessing
necessary to get the full decomposition, while still being useful for
certain applications; see Section~\ref{sec:th} for details.  The second
main innovation is a new preprocessing phase, discussed in
Section~\ref{sec:pre}, that uses fewer transformations and rank
decisions, and has fewer dependencies between the rank decisions. The
third main innovation is a numerically stable $2\times 2$ RSVD algorithm
like Bai and Demmel's backward stable $2\times 2$ QSVD
algorithm~\cite{BD93}. This $2\times 2$ algorithm is a crucial part of
the implicit Kogbetliantz iteration, and we investigate its numerical
properties in exact and floating-point arithmetic in
Sections~\ref{sec:rsvdTwoExact}, \ref{sec:rsvdTwoFloat},
and~\ref{sec:bwdA}. The fourth main innovation is primarily discussed in
Section~\ref{sec:ext} and consists of an alternative scaling of the
restricted singular value triplets. This new scaling  leads to
mathematically and numerically elegant formulas for the computation of
the triplets.

Since rank revealing decompositions do not necessarily need to use
orthonormal transformations, we can combine ideas from Drma\v{c}'s
nonorthogonal algorithm with the new preprocessing phase and the
implicit Kogbetliantz iteration from this work.
Section~\ref{sec:nonorth} contains an overview of how this hybrid
algorithm would work. While the use of nonorthogonal transformations in
the earlier phases is optional, the postprocessing phase generally
requires nonorthogonal transformations, as we will see in
Section~\ref{sec:post}.

The numerical tests in Section~\ref{sec:num} consist of three parts. The
first part is dedicated to verifying the numerical properties of the
$2\times 2$ RSVD algorithm. The second part focuses on the rate of
convergence of the implicit Kogbetliantz iteration. The third part
compares the accuracy of the new RSVD method with existing methods, and
also compares the effect of different implementations of the
preprocessing phase on the accuracy.  The results show that the new
$2\times 2$ RSVD is numerically stable, and that the implicit
Kogbetliantz iteration typically converges rapidly and can compute the
RSVs with high accuracy. In fact, for ill-conditioned matrices the
accuracy the new method can exceed that of existing methods by several
orders of magnitude.

Throughout this work we use uppercase letters for matrices, lowercase
letters for their elements and for scalars, and bold lowercase letters
for vectors. The matrix $I$ is always an identity matrix, $\vec e_j$ the
$j$th canonical basis vector, $0$ a zero matrix or scalar, and $\times$
an arbitrary matrix or scalar that can be nonzero.  These quantities
always have a size that is appropriate for the context in which they are
used. In some places we use Matlab notation when stacking block matrices
vertically; for example, $[A;\; C] = [A^T\; C^T]^T$.  As usual,
$\|\cdot\|_p$ denotes the (induced) $p$-norm for $1 \le p \le \infty$,
and we sometimes drop the index for $p=2$ when no other norms are used
in the same context. Other norms that we use are the Frobenius norm
$\|\cdot\|_F$ and the max norm $\|\cdot\|_{\max}$, where the latter
equals the largest magnitude of any element in the matrix.  Finally, the
absolute value notation $|\cdot|$ acts elementwise on matrices.



\section{Background and theory}\label{sec:th}

The definition of the RSVD given by the theorem below combines the ones
from Zha~\cite[Lem.~4.1]{Zha91} and De~Moor and
Golub~\cite[Thm.~1]{MG91}, but with some small changes. In particular,
some of the blocks in \eqref{eq:sigmas} are in a different position,
which helps with the computation of the decomposition, and some of the
trivial triplets are counted differently.  Furthermore, the theorem
below and the theory and algorithms in the rest of this work focus on
the real case for simplicity and clarity, although we can compute the
RSVD of a triplet of complex matrices too.


\newlength{\DIndexWidth}
\settowidth{\DIndexWidth}{\hbox{$D_\alpha$}}
\newlength{\ZeroWidth}
\settowidth{\ZeroWidth}{\hbox{$0$}}

\begin{theorem}[RSVD --- Diagonal Form]\label{thm:rsvd}
  Let $A \in \bbR^{p \times q}$, $B \in \bbR^{p \times m}$, and $C \in
  \bbR^{n \times q}$, and define $r_{A} = \rank A$, $r_{B} = \rank B$,
  $r_{C} = \rank C$, $r_{AB} = \rank \left[ A\; B \right]$, $r_{AC} =
  \rank \left[ A;\; C \right]$, and $r_{ABC} = \rank
  \minimat{A}{B}{C}{0}$. Then the triplet of matrices $(A,B,C)$ can be
  factorized as $A = X^{-T} \Sigma_\alpha Y^{-1}$, $B = X^{-T}
  \Sigma_\beta U^T$, and $C = V \Sigma_\gamma Y^{-1}$, where $X \in
  \bbR^{p\times p}$ and $Y \in \bbR^{q\times q}$ are nonsingular, and $U
  \in \bbR^{m\times m}$ and $V \in \bbR^{n\times n}$ are orthonormal.
  Furthermore, $\Sigma_\alpha$, $\Sigma_\beta$, and $\Sigma_\gamma$ are
  quasi-diagonal\footnote{A quasi-diagonal matrix, in this work, is a
  matrix that is diagonal after removing all zero rows and columns.}
  with nonnegative entries, and are such that
  {\scriptsize $\arraycolsep=0.25\arraycolsep
    \def\arraystretch{1}
    \left[\begin{array}{c|c}
      \Sigma_\alpha & \Sigma_\beta \\ \hline \Sigma_\gamma
    \end{array}\right]$}
  can be written as
  \begin{equation}\label{eq:sigmas}
    \begin{array}{cccccccccccc}
      &
      \multicolumn{10}{c}{\hspace{-2.0\arraycolsep}
        \begin{array}{cccccccccc}
          \text{\makebox[\ZeroWidth]{\small $q_1$}} &
          \text{\makebox[\ZeroWidth]{\small $q_2$}} &
          \text{\makebox[\DIndexWidth]{\small $q_3$}} &
          \text{\makebox[\ZeroWidth]{\small $q_4$}} &
          \text{\makebox[\ZeroWidth]{\small $q_5$}} &
          \text{\makebox[\ZeroWidth]{\small $q_6$}} &
          \text{\makebox[\DIndexWidth]{\small $m_1$}} &
          \text{\makebox[\ZeroWidth]{\small $m_2$}} &
          \text{\makebox[\ZeroWidth]{\small $m_3$}} &
          \text{\makebox[\ZeroWidth]{\small $m_4$}}
        \end{array}
      }
      \\
      \begin{array}{c}
        \text{\small $p_1$} \\
        \text{\small $p_2$} \\
        \text{\small $p_3$} \\
        \text{\small $p_4$} \\
        \text{\small $p_5$} \\
        \text{\small $p_6$} \\
        \text{\small $n_1$} \\
        \text{\small $n_2$} \\
        \text{\small $n_3$} \\
        \text{\small $n_4$}
      \end{array}
      &
      \multicolumn{10}{c}{\hspace{-2.0\arraycolsep}\left[
        \begin{array}{cc:cccc|ccc:c}
          0 & 0 & D_\alpha & 0 & 0 & 0 & D_\beta & 0 & 0 & 0 \\
          0 & 0 & 0        & I & 0 & 0 &       0 & I & 0 & 0 \\
          0 & 0 & 0        & 0 & I & 0 &       0 & 0 & 0 & 0 \\
          0 & 0 & 0        & 0 & 0 & I &       0 & 0 & 0 & 0 \\
          \hdashline
          0 & 0 & 0        & 0 & 0 & 0 &       0 & 0 & 0 & I \\
          0 & 0 & 0        & 0 & 0 & 0 &       0 & 0 & 0 & 0 \\
          \hline
          0 & I & 0        & 0 & 0 & 0 \\\cdashline{1-6}
          0 & 0 & D_\gamma & 0 & 0 & 0 \\
          0 & 0 & 0        & 0 & I & 0 \\
          0 & 0 & 0        & 0 & 0 & 0 \\
        \end{array}
      \right]}
      &\hspace{-2.0\arraycolsep}
      \begin{array}{l}
        \text{\small $p_1 = q_3 = r_{ABC} + r_{A} - r_{AB} - r_{AC}$} \\
        \text{\small $p_2 = q_4 = r_{AC} + r_{B} - r_{ABC}$} \\
        \text{\small $p_3 = q_5 = r_{AB} + r_{C} - r_{ABC}$} \\
        \text{\small $p_4 = q_6 = r_{ABC} - r_{B} - r_{C}$} \\
        \text{\small $p_5 = r_{AB} - r_{A}$, $q_2 = r_{AC} - r_{A}$} \\
        \text{\small $p_6 = p - r_{AB}$, $q_1 = q - r_{AC}$} \\
        \text{\small $n_1 = q_2$, $m_4 = p_5$} \\
        \text{\small $n_2 = m_1 = p_1 = q_3$} \\
        \text{\small $n_3 = p_3 = q_5$, $m_2 = p_2 = q_4$} \\
        \text{\small $n_4 = n - r_{C}$, $m_3 = m - r_{B}$},
      \end{array}
    \end{array}
  \end{equation}
  where $D_\alpha = \diag(\alpha_1, \dots, \alpha_{p_1})$, $D_\beta =
  \diag(\beta_1, \dots, \beta_{p_1})$, and $D_\gamma = \diag(\gamma_1,
  \dots, \gamma_{p_1})$. Moreover, $\alpha_j$, $\beta_j$, and $\gamma_j$
  are scaled such that $\alpha_j^2 + \beta_j^2 \gamma_j^2 = 1$ for $i =
  1$, \dots, $p_1$. Besides the $p_1$ triplets $(\alpha_j,
  \beta_j, \gamma_j)$, there are $p_2$ triplets $(1,1,0)$, $p_3$
  triplets $(1,0,1)$, $p_4$ triplets $(1,0,0)$, and $\min \{ p_5, q_2
  \}$ triplets $(0,1,1)$.  This leads to a total of $p_1 + p_2 + p_3 +
  p_4 + \min \{ p_5, q_2 \} = r_A + \min \{ p_5, q_2 \} = \min \{
  r_{AB}, r_{AC} \}$ regular triplets of the form $(\alpha, \beta,
  \gamma)$ with $\alpha^2 + \beta^2 \gamma^2 = 1$. Each of these
  triplets corresponds to a \emph{restricted singular value} $\sigma =
  \alpha / (\beta \gamma)$, where the result is $\infty$ by convention
  if $\alpha \neq 0$ and $\beta \gamma = 0$.
  Finally, the triplet has a right (or column) trivial block of
  dimension $q_1 = \dim(\mathcal N(A) \cap \mathcal N(C))$, and
  a left (or row) trivial block of dimension $p_6 = \dim(\mathcal N(A^T)
  \cap \mathcal N(B^T))$.
\end{theorem}



\begin{remark}
  Zha~\cite[Sec.~4]{Zha91} and De~Moor and
  Golub~\cite[Sec.~2.1]{MG91} differ in the triplets that they list.
  In particular, the former does not list any of the triplets $(0,0,0)$,
  $(0,0,1)$, and $(0,1,0)$; whereas the latter do not list $(0,1,1)$,
  but instead the equivalent of $p_5$ triplets $(0,1,0)$ and $q_2$
  triplets $(0,0,1)$. Theorem~\ref{thm:rsvd} adopts Zha's definition of
  $(0,1,1)$, because of \cite[Thm.~4.2]{Zha91} and the simple example
  $(A,B,C) = (0,1,1)$, and avoids problematic definitions of trivial
  triplets by listing the left and right trivial blocks.
\end{remark}



\begin{remark}
  The typical scaling of the triplets $(\alpha_i, \beta_i, \gamma_i)$ in
  literature is such that $\alpha_i^2 + \beta_i^2 + \gamma_i^2 = 1$,
  rather than $\alpha_i^2 + \beta_i^2 \gamma_i^2 = 1$ as in the theorem
  above.  But we will see in Section~\ref{sec:ext} that the latter
  scaling has theoretical and computational benefits. Besides, if
  $\alpha_i^2 + \beta_i^2 \gamma_i^2 = 1$, then we can compute (as in
  Zha~\cite[Thm.~4.1]{Zha91}) $\widetilde \alpha_i = \alpha_i^2 (1 +
  \alpha_i^2)^{-1/2}$, $\widetilde \beta_i = \beta\gamma_i$, and
  $\widetilde \gamma_i = \alpha_i (1 + \alpha_i^2)^{-1/2}$, so that
  $\widetilde \alpha_i^2 + \widetilde \beta_i^2 + \widetilde \gamma_i^2
  = 1$.
\end{remark}



\begin{corollary}[RSVD --- Triangular Form]\label{thm:rsvdtri}
  Let $X^{-T} = PS$ and $Y^{-1} = TQ^T$, where $P \in \bbR^{p \times p}$
  and $Q \in \bbR^{q \times q}$ are ortho\-normal, and $S \in \bbR^{p
  \times p}$ and $T \in \bbR^{q \times q}$ are nonsingular and upper
  triangular. Then the triplet $(A,B,C)$ can be factorized as $A = P(S
  \Sigma_\alpha T)Q^T$, $B = P(S \Sigma_\beta) U^T$, and $C = V
  (\Sigma_\gamma T)Q^T$.
\end{corollary}


Suppose that $A$, $B$, and $C$ are nonsingular and have compatible
sizes; then the restricted singular values of the triplet $(A,B,C)$ are
the ordinary singular values of $B^{-1} A C^{-1}$. Just like Drma\v{c}'s
algorithm~\cite{Drma00}, the method described in this work is to look at
the singular values of $CA^{-1}B$ instead. The benefit for more general
matrices is that it suffices to ``extract'' a triplet with a nonsingular
$A$ during the preprocessing, rather than having to extract a triplet
with nonsingular $B$ and $C$ (cf. Zha's algorithm~\cite{Zha92}). It
turns out that this alternative extraction requires fewer
transformations, and more importantly, fewer rank decisions. To see why
we can change our perspective like this, we first need the following
definition of the regular RSVs.


\begin{theorem}[{Zha~\cite[Def.~2.1]{Zha91}}]\label{thm:rsvRank}
  The regular restricted singular values of the matrix triplet $A$, $B$,
  and $C$ can be characterized as
  \begin{equation*}
    \sigma_i = \min_{D} \{ \|D\| : \rank(A + BDC) \le i - 1 \}
    \qquad(i = 1, \dots, r_A + \min \{ p_5, q_2 \}).
  \end{equation*}
  The value $\sigma_i = \infty$ corresponds to the situation that we
  cannot find any matrix $D$ to make the rank of $A + BDC$ less than or
  equal to $i - 1$.
\end{theorem}


Now we can prove the following proposition, which formalizes the idea of
working with $CA^{-1}B$ instead of $B^{-1} A C^{-1}$ for general
matrices.


\begin{proposition}\label{thm:invrsv}
  Suppose that the matrices $A$, $B$, and $C$ have compatible sizes and
  that the restricted singular values of the triplet $(A,B,C)$ are
  defined as in Theorem~\ref{thm:rsvRank}. Then, for the nonzero
  restricted singular values it holds that
  \begin{equation}\label{eq:invRSV}
    \sigma_i^{-1} = \min_{D} \{ \|D\|
      \mid \rank (D + CA^{\dagger}B) \le r_A - i \}
      \qquad(i = 1, \dots, r_A),
  \end{equation}
  where $A^\dagger$ denotes the Moore--Penrose pseudoinverse of $A$ and
  $\infty^{-1} = 0$ by convention.  The remaining $\min \{ p_5, q_2 \}$
  regular RSVs can be characterized as $0^{-1} = \infty$.
\end{proposition}
\begin{proof}
  Suppose that $A$, $B$, and $C$ are decomposed as in
  Theorem~\ref{thm:rsvd} and let $E = V^TDU$; then $V^T (CA^{\dagger}B)
  U = \Sigma_\gamma \Sigma_{\alpha}^{\dagger} \Sigma_\beta$ and $\rank(D
  + CA^{\dagger}B) = \rank(E + \Sigma_\gamma \Sigma_{\alpha}^{\dagger}
  \Sigma_\beta)$. By using the definition of the $\Sigma$s and the fact
  that $\|D\| = \|E\|$, it follows that the minimization
  in~\eqref{eq:invRSV} is equivalent to
  \begin{equation*}
    \min_{E} \left\{ \|E\| \mid \rank
      \left[\begin{smallmatrix}
        E_{11} & E_{12} \\
        E_{21} + D_\gamma D_\alpha^{-1} D_\beta & E_{22} \\
        E_{31} & E_{32}
      \end{smallmatrix}\right]
      \le r_A - i \right\},
  \end{equation*}
  which equals $\sigma_i^{-1}$ for the nonzero RSVs.  In particular
  $\sigma_i^{-1} = 0$ for $i = 1$, \dots, $p_2 + p_3 + p_4$, and
  $\sigma_i^{-1} > 0$ for $i = p_2 + p_3 + p_4 + 1, \dots, r_A$.
\end{proof}


We can interpret the proposition above as a generalization of
Zha~\cite[Cor~4.1]{Zha91}, but it is also related to the analysis of
generalized Schur complements in De~Moor and
Golub~\cite[Sec.~3.2.1]{MG91}. An important observation is that we do
not need the full RSVD to compute $CA^\dagger B$. In fact, the outputs
of the new algorithm after the preprocessing phase and the implicit
Kogbetliantz iteration are the matrices $P$, $Q$, $U$, and $V$, and the
products $P^T\!AQ$, $P^T\!BU$, and $V^T\!CQ$, which are such that
\begin{equation*}
  (V^T\!CQ) (P^T\!AQ)^\dagger (P^T\!BU)
  = \Sigma_\gamma \Sigma_\alpha^\dagger \Sigma_\beta
\end{equation*}
is quasi-diagonal and easily determined.  The postprocessing is only
necessary to get the individual factors $S$ and $T$, and it depends on
the application if we need those. This suggest the following
decomposition, which can be thought of as kind of generalized Schur
decomposition like the QZ decomposition for generalized eigenvalue
problems.


\begin{theorem}[RSVD --- Generalized Schur form]\label{thm:schurform}
  Let $A \in \mathbb{R}^{p \times q}$, $B \in \mathbb{R}^{p \times m}$,
  and $C \in \mathbb{R}^{n \times q}$; then there exist orthonormal
  matrices $P$, $Q$, $U$, and $V$, such that
  \begingroup
  \allowdisplaybreaks
  \begin{align}
    \compressedmatrices
      P^T\!AQ &=\;
      \bordermatrix[{[]}]{%
          & q_1 & q_2 & q_3    & q_4    & q_5    \cr
      p_1 & 0   & 0   & A_{13} & A_{14} & A_{15} \cr
      p_2 & 0   & 0   & 0      & A_{24} & A_{25} \cr
      p_3 & 0   & 0   & 0      & 0      & A_{35} \cr
      p_4 & 0   & 0   & 0      & 0      & 0      \cr
      p_5 & 0   & 0   & 0      & 0      & 0      \cr
      },\label{eq:schurforma}
      \\
      P^T\!BU &=\;
      \bordermatrix[{[]}]{%
          & m_1 & m_2    & m_3    & m_4    \cr
      p_1 & 0   & B_{12} & B_{13} & B_{14} \cr
      p_2 & 0   & 0      & B_{23} & B_{24} \cr
      p_3 & 0   & 0      & 0      & B_{34} \cr
      p_4 & 0   & 0      & 0      & B_{44} \cr
      p_5 & 0   & 0      & 0      & 0      \cr
      },\label{eq:schurformb}
      \\
      V^T\!CQ &=\;
      \bordermatrix[{[]}]{%
          & q_1 & q_2    & q_3    & q_4    & q_5    \cr
      n_1 & 0   & C_{12} & C_{13} & C_{14} & C_{15} \cr
      n_2 & 0   & 0      & 0      & C_{24} & C_{25} \cr
      n_3 & 0   & 0      & 0      & 0      & C_{35} \cr
      n_4 & 0   & 0      & 0      & 0      & 0      \cr
      },\label{eq:schurformc}
  \end{align}
  \endgroup
  where $A_{13}$, $A_{24}$, $A_{34}$, $B_{44}$, and $C_{12}$ are
  nonsingular and upper triangular; $B_{23}$ and $C_{24}$ are square and
  upper triangular; and $B_{12}$ and $C_{35}$ are upper trapezoidal with
  $p_1 \ge m_2$ and $n_3 \le q_5$, respectively. Here,
  upper trapezoidal with the given dimensions means that $B_{12}$ and
  $C_{35}$ are structured as
  \begin{equation*}
    \compressedmatrices
    \begin{bmatrix}
      \times & \cdots & \times \\
      \vdots &        & \vdots \\
      \times & \cdots & \times \\
            & \ddots & \vdots \\
            &        & \times
    \end{bmatrix}
    \quad\text{and}\quad
    \begin{bmatrix}
      \times & \cdots & \times & \cdots & \times \\
            & \ddots & \vdots &        & \vdots \\
            &        & \times & \cdots & \times
    \end{bmatrix},
  \end{equation*}
  respectively.  Furthermore, the matrices $A_{24}$, $B_{23}$, and
  $C_{24}$ are such that $C_{24} A_{24}^{-1} B_{23} = \Sigma$, where
  $\Sigma$ is diagonal with nonnegative entries.
\end{theorem}
\begin{proof}
  The structure of the matrices $P^T\!AQ$, $P^T\!BU$, and $V^T\!CQ$
  follows from the preprocessing phase from Section~\ref{sec:pre}.
  %
  For the claim that $C_{24} A_{24}^{-1} B_{23} = \Sigma$, we can use
  the following limit argument.  Suppose that $\widehat P$, $\widehat
  Q$, $\widehat U$, and $\widehat V$ are such that $\widehat P^T\! A
  \widehat Q$, $\widehat P^T\! B \widehat U$, and $\widehat V^T\! C
  \widehat Q$ have the structure from
  \eqref{eq:schurforma}--\eqref{eq:schurformb}, but that $C_{24}
  A_{24}^{-1} B_{23}$ is not yet diagonal. Then let $\{ B_k \}$ be a
  bounded sequence of nonsingular upper-triangular matrices that
  converge to $B_{23}$. For each $k$, let $\widetilde V_k^T (C_{24}
  A_{24}^{-1} B_k) \widetilde U_k = \Sigma_k$ be an SVD of $C_{24}
  A_{24}^{-1} B_k$.  Furthermore, let $\widetilde P_k$ and $\widetilde
  Q_k$ be orthonormal matrices such that $\widetilde P_k^T B_k
  \widetilde U_k$ and $\widetilde P_k^T A_{23} \widetilde Q_k$ are upper
  triangular, respectively. Then $C_{24} = \Sigma_k (\widetilde P_k^T
  B_k \widetilde U_k)^{-1} (\widetilde P_k^T A_{24} \widetilde Q_k)$ is
  also upper triangular. Using the Bolzano--Weierstrass theorem, we know
  that the bounded sequence $\{(\widetilde P_k, \widetilde Q_k,
  \widetilde U_k, \widetilde V_k, \Sigma_k)\}$ has a converging
  subsequence
  \begin{equation*}
    \lim_{i \to \infty} (\widetilde P_{k_i}, \widetilde Q_{k_i},
      \widetilde U_{k_i}, \widetilde V_{k_i}, \Sigma_{k_i})
    = (\widetilde P, \widetilde Q, \widetilde U, \widetilde V, \Sigma).
  \end{equation*}
  It is easy to show that $\widetilde P$, $\widetilde Q$, $\widetilde
  U$, and $\widetilde V$ are orthonormal and that $\widetilde P^T\!
  A_{24} \widetilde Q$, $\widetilde P^T\! B_{23} \widetilde U$, and
  $\widetilde V^T\! C_{24} \widetilde Q$ are upper triangular, and
  satisfy $\widetilde V^T\!C_{24} A_{24}^{-1} B_{23} \widetilde U =
  \Sigma$, where $\Sigma$ is a diagonal matrix with nonnegative entries.
  Hence, the products
  \begin{equation*}
    \begin{split}
      P &= \widehat P \diag(I, \widetilde P, I, I, I), \aligncr
      U &= \widehat U \diag(I, I, \widetilde U, I), \\
      Q &= \widehat Q \diag(I, I, I, \widetilde Q, I), \aligncr
      V &= \widehat V \diag(I, \widetilde V, I, I),
    \end{split}
  \end{equation*}
  are the sought after orthonormal transformations.
  %
\end{proof}


The $P$, $Q$, $U$, $V$, $p_i$, $q_i$, $m_i$, and $n_i$ from the above
theorem are not necessarily equal to their counterparts from
Theorem~\ref{thm:rsvd} and Corollary~\ref{thm:rsvdtri}.  The main
benefit of the Schur-form RSVD is that we can compute it with only
orthonormal transformations and at most three rank decisions, while we
can still use it to compute, e.g., $CA^\dagger B$ and the RSVs. The
computation of the Schur form is the subject of later sections, but to
see how we can use it, consider the following proposition first.


\begin{proposition}\label{thm:schursparse}
  Let $A$, $B$, $C$, $P$, $Q$, $U$, and $V$ be as in
  Theorem~\ref{thm:schurform}; then there exist nonsingular
  upper-triangular matrices $S$ and $T$ so that $S^{-1} (P^T\!AQ)
  T^{-1}$, $S^{-1} (P^T\!BU)$, and $(V^T\!CQ) T^{-1}$ have the form
  \begin{equation}\label{eq:schursparse}
    \compressedmatrices
    \begin{bmatrix}
      0 & 0 & I & 0      & 0 \\
      0 & 0 & 0 & A_{24} & 0 \\
      0 & 0 & 0 & 0      & I \\
      0 & 0 & 0 & 0      & 0 \\
      0 & 0 & 0 & 0      & 0
    \end{bmatrix},
    \qquad
    \begin{bmatrix}
      0 & B_{12} & B_{13} & 0 \\
      0 & 0      & B_{23} & 0 \\
      0 & 0      & 0      & 0 \\
      0 & 0      & 0      & I \\
      0 & 0      & 0      & 0
    \end{bmatrix},
    \quad\text{and}\quad
    \begin{bmatrix}
      0 & I & 0 & 0      & 0      \\
      0 & 0 & 0 & C_{24} & C_{25} \\
      0 & 0 & 0 & 0      & C_{35} \\
      0 & 0 & 0 & 0      & 0
    \end{bmatrix},
  \end{equation}
  respectively.
\end{proposition}
\begin{proof}
  Let
  \begin{equation}\label{eq:schursparseST}
    \compressedmatrices
    S =
    \begin{bmatrix}
      I & & \frac{1}{2} A_{15} & B_{14} \\
      & I & A_{25} & B_{24} \\
      & & A_{35} & B_{34} \\
      & & & B_{44} \\
      & & & & I
    \end{bmatrix}
    \quad\text{and}\quad
    T =
    \begin{bmatrix}
      I \\
      & C_{12} & C_{13} & C_{14} & C_{15} \\
      & & A_{13} & A_{14} & \frac{1}{2} A_{15} \\
      & & & I \\
      & & & & I
    \end{bmatrix};
  \end{equation}
  then direct verification concludes the proof.
\end{proof}


Now, with the Schur-form RSVD and Proposition~\ref{thm:schursparse}, we
see that
\begin{equation*}
  \compressedmatrices
  \begin{split}
    V^T\!CA^{\dagger}BU
    &=
    \begin{bmatrix}
      0 & 0                         & 0 & 0 \\
      0 & C_{24} A_{24}^{-1} B_{23} & 0 & 0 \\
      0 & 0                         & 0 & 0 \\
      0 & 0                         & 0 & 0
    \end{bmatrix}
    =
    \begin{bmatrix}
      0 & 0      & 0 & 0 \\
      0 & \Sigma & 0 & 0 \\
      0 & 0      & 0 & 0 \\
      0 & 0      & 0 & 0
    \end{bmatrix}
  \end{split}
\end{equation*}
is indeed easily determined.  Likewise, if $D$ is a $4\times 4$ block
matrix of compatible dimensions, then
\begin{equation*}
  \compressedmatrices
  \begin{split}
    \rank(A + BDC) &= \rank
    \begin{bmatrix}
      0 & \times & I & \times & \times \\
      0 & B_{23} D_{31} & 0 & A_{24} + B_{23} D_{32} C_{24} & \times \\
      0 & 0 & 0 & 0 & I \\
      0 & D_{41} & 0 & D_{42} C_{24} & \times \\
      0 & 0 & 0 & 0 & 0
    \end{bmatrix}
    \\ &= p_1 + p_3
    {} + \rank \left(
    \begin{bmatrix}
      0 & A_{24} \\
      0 & 0
    \end{bmatrix}
    {} +
    \begin{bmatrix}
      B_{23} \\ & I
    \end{bmatrix}
    \begin{bmatrix}
      D_{31} & D_{32} \\
      D_{41} & D_{42}
    \end{bmatrix}
    \begin{bmatrix}
      I \\ & C_{24}
    \end{bmatrix}
    \right).
  \end{split}
\end{equation*}
This yields $p_1 + p_3$ RSVs at $\infty$, $\min \{p_4, q_2\}$ at $0$,
and $p_2$ reciprocals from the diagonal elements of $C_{24} A_{24}^{-1}
B_{23} = \Sigma$ (where $0^{-1} = \infty$ by convention).



\section{The preprocessing phase}\label{sec:pre}

Zha's algorithm for computing the RSVD starts with a preprocessing phase
to extract a triplet $(A', B', C')$ from a given matrix triplet $(A, B,
C)$, where $A'$ is square and upper triangular, and $B'$ and $C'$ are
nonsingular and upper triangular. The goal of the preprocessing phase
described in this section is similar to Zha's in that it extracts square
and upper triangular matrices. But unlike Zha's approach, we only need
$A'$ to be nonsingular, rather than both $B'$ and $C'$. Moreover, this
triplet should correspond to the nonzero regular RSVs of $(A, B, C)$ so
that we can (implicitly) apply the Kogbetliantz iteration to the product
$C' (A')^{-1} B'$ in the next phase.  The result is a triplet of
matrices with the structure from
\eqref{eq:schurforma}--\eqref{eq:schurformc}.

Two key procedures in Zha's preprocessing phase are so called \emph{row
compressions} and \emph{column compressions}.  For simplicity, we only
use the combined action of both compressions, which leads to the
definition below.


\begin{definition}
  Let $M$ be a real matrix; then we refer to the URV decomposition
  $U^T\! M V = \minimat{0}{R}{0}{0}$ as the \emph{compression} of $M$
  if $R$ is nonsingular and $U$ and $V$ are orthonormal matrices.
\end{definition}


One way to compress a matrix is to use the SVD, although any
rank-revealing URV decomposition would work. For example, the QR
decomposition with column pivoting is a popular and fast alternative to
the SVD in this context.  See, e.g., Fierro, Hansen, and
Hansen~\cite{UTVtools} for an overview of the qualitative differences
between different URV decompositions in floating-point arithmetic.

We use compressions, along with other orthonormal transformations, in
the preprocessing phase to compute a sequence of orthonormal matrices
$\iter{P}{\ell}$, $\iter{Q}{\ell}$, $\iter{U}{\ell}$, and
$\iter{V}{\ell}$. These matrices are such that if we start with
$\iter{A}{1} = A$, $\iter{B}{1} = B$, and $\iter{C}{1} = C$, and compute
\begin{equation}\label{eq:update}
  \iter{A}{\ell+1} = \iter{P}{\ell}{}^T \iter{A}{\ell} \iter{Q}{\ell},
  \
  \iter{B}{\ell+1} = \iter{P}{\ell}{}^T \iter{B}{\ell} \iter{U}{\ell},
  \ \text{and}\
  \iter{C}{\ell+1} = \iter{V}{\ell}{}^T \iter{C}{\ell} \iter{Q}{\ell};
\end{equation}
we will ultimately get partitioned matrices $\iter{A}{i}$,
$\iter{B}{i}$, and $\iter{C}{i}$ from which we can take specific blocks
as the sought-after triplet. That is, the matrices $\iter{A}{i}$,
$\iter{B}{i}$, $\iter{C}{i}$ have blocks that are square and upper
triangular, with $\iter{A}{i}$ nonsingular, that correspond to the
nonzero RSVs of $(A, B, C)$.

Since partitioned matrices will be important, both in the preprocessing
phase and the postprocessing phase, the following implicit notations
help to simplify the presentation. First, all blocks have indices that
mark their respective positions in their matrix and are of no further
importance.  Second, if $M$ is a partitioned matrix and if $M_{ij}$
denotes the block at the $i$th block row and $j$th block column, then
any transformation of $M_{ij}$ into $X_{ii}^T M_{ij} Y_{jj}$ has a
corresponding transformation of $M$ into $X^T\!MY$.  In particular, $X =
\diag(I, X_{ii}, I)$ and $Y = \diag(I, Y_{jj}, I)$ so that ${(X^T\! M
Y)}_{ij} = X_{ii}^T M_{ij} Y_{jj}$, unless stated otherwise.  Third, the
latter notation is understood to work recursively, so that if
$\iter{M}{\ell}$ is a submatrix of $\iter{M}{\ell-1}$, then any
transformation applied to $\iter{M}{\ell}$, or a block of
$\iter{M}{\ell}$, has a corresponding transformation applied to
$\iter{M}{\ell-1}$.

Now for the first step of the preprocessing phase, let $\iter{P}{1}$ and
$\iter{Q}{1}$ compress $\iter{A}{1}$, and take $\iter{U}{1} = I$ and
$\iter{V}{1} = I$, then
\begin{equation*}
  \compressedmatrices
  \iter{A}{2} =
  \begin{bmatrix}
    0 & \iter{A_{12}}{2} \\
    0 & 0
  \end{bmatrix},
  \qquad
  \iter{B}{2} =
  \begin{bmatrix}
    \iter{B_{11}}{2} \\
    \iter{B_{21}}{2}
  \end{bmatrix},
  \quad\text{and}\quad
  \iter{C}{2} =
  \begin{bmatrix}
    \iter{C_{11}}{2} & \iter{C_{12}}{2}
  \end{bmatrix}.
\end{equation*}
Let $\iter{P_{22}}{2}$ and $\iter{U_{11}}{2}$ compress
$\iter{C_{11}}{2}$, and let $\iter{V_{11}}{2}$ and $\iter{Q_{11}}{2}$ to
compress $\iter{B_{21}}{2}$, so that
\begin{equation*}
  \compressedmatrices
  \iter{A}{3} =
  \begin{bmatrix}
      0 & 0 & \iter{A_{13}}{3} \\
      0 & 0 & 0 \\
      0 & 0 & 0      
  \end{bmatrix},
  \qquad
  \iter{B}{3} =
  \begin{bmatrix}
    \iter{B_{11}}{3} & \iter{B_{12}}{3} \\
    0 & \iter{B_{22}}{3} \\
    0 & 0
  \end{bmatrix},
  \quad\text{and}\quad
  \iter{C}{3} =
  \begin{bmatrix}
    0 & \iter{C_{12}}{3} & \iter{C_{13}}{3} \\
    0 & 0      & \iter{C_{23}}{3}
  \end{bmatrix}.
\end{equation*}
The compressions of $\iter{A}{1}$, $\iter{B_{21}}{2}$, and
$\iter{C_{11}}{2}$ contain the only three rank decisions necessary in
the preprocessing phase. Now, by plugging $\iter{C}{3}
\iter{A}{3}{}^\dagger \iter{B}{3}$ into Theorem~\ref{thm:rsvRank}, we
see that we may focus on the triplet $(\iter{A_{13}}{3},
\iter{B_{11}}{3}, \iter{C_{23}}{3})$, where $\iter{A_{13}}{3}$ is
nonsingular.  The matrices $\iter{B_{11}}{3}$ and $\iter{C_{23}}{3}$ are
not necessarily square or upper triangular at this point, so we are not
yet finished.  For convenience set $\iter{A}{4} = \iter{A_{13}}{3}$,
$\iter{B}{4} = \iter{B_{11}}{3}$, and $\iter{C}{4} = \iter{C_{23}}{3}$,
and suppose that $\iter{A}{4}$ is $p' \times p'$, $\iter{B}{4}$ is $p'
\times m'$, and $\iter{C}{4}$ is $n' \times p'$, and let $l = \min\{m',
n', p'\}$. Then there are three possibilities to consider (where the
choice is free when there is overlap).
\begin{enumerate}
  \item If $m', n' \ge p' = l$, then take $\iter{P}{4} = \iter{Q}{4} =
    I$ and $\iter{A}{5} = \iter{A}{4}$. Furthermore, use a QR
    decomposition to compute $\iter{V}{4}$ such that $\iter{V}{4}{}^T
    \iter{C}{4} = [\iter{C_{11}}{5};\; 0]$, where $\iter{C_{11}}{5}$ is
    upper triangular; and use an RQ decomposition\footnote{The RQ
    decomposition of a real $m \times n$ matrix $A$ is like a QR
    decomposition, but with the factors in the opposite order. That is,
    if $m > n$, then $A = RQ^T$ for some upper-trapezoidal matrix $R$
    and some $Q$ satisfying $Q^TQ = I$.  If $m \ge n$, then $A = [0\;
    R]Q^T$ or $A = [R\; 0]Q^T$, where $Q$ is as before and $R$ is upper
    triangular. See, e.g., the routine \textsf{xGERQF} in
    LAPACK~\cite[Sec.~2.4.2.5]{LAUG}.} to compute $\iter{U}{4}$ such
    that $\iter{B}{4} \iter{U}{4} = [0\; \iter{B_{12}}{5}]$, where
    $\iter{B_{12}}{5}$ is upper triangular.  Since
    \begin{equation*}
      \compressedmatrices
      \begin{bmatrix} \iter{C_{11}}{5} \\ 0 \end{bmatrix}
      (\iter{A}{5})^{-1}
      \begin{bmatrix} 0 & \iter{B_{12}}{5} \end{bmatrix}
      =
      \begin{bmatrix}
        0 & \iter{C_{11}}{5} (\iter{A}{5})^{-1} \iter{B_{12}}{5} \\
        0 & 0
      \end{bmatrix}
    \end{equation*}
    we see that we can restrict our attention to the triplet
    $(\iter{A}{5}, \iter{B_{12}}{5}, \iter{C_{11}}{5})$.

  \item If $r = \min \{n', p'\} \ge m' = l$. First use a QR
    decomposition to compute $\iter{P}{4}$ such that $\iter{P}{4}{}^T
    \iter{B}{4} = [\iter{B_{11}}{5};\; 0]$, where $\iter{B_{11}}{5}$ is
    an $m'\times m'$ upper-triangular matrix. Next, compute
    $\iter{Q}{4}$ with an RQ decomposition so that $\iter{A}{5} =
    \iter{P}{4}{}^T\!  \iter{A}{4} \iter{Q}{4}$ is upper triangular.
    Then use a QR decomposition to compute $\iter{V}{4}$ such that
    $\iter{C}{5} = \iter{V}{4}{}^T \iter{C}{4} \iter{Q}{4}$ has the form
    \begin{equation*}
      \compressedmatrices
      \bordermatrix[{[]}]{%
                           & m'               & p' - m'          \cr
        m'                 & \iter{C_{11}}{5} & \iter{C_{12}}{5} \cr
        r - m'             & 0                & \iter{C_{22}}{5} \cr
        \max \{0, n' - r\} & 0                & 0
      },
    \end{equation*}
    where $\iter{C_{11}}{5}$ is upper triangular and $\iter{C_{22}}{5}$
    is upper trapezoidal with $r - m' \le p' - m'$.  Assuming
    $\iter{A}{5}$ is partitioned conformally,
    \begin{equation*}
      \iter{C}{5} (\iter{A}{5})^{-1} \iter{B}{5} =
      \begin{bmatrix}
        \iter{C_{11}}{5} (\iter{A_{11}}{5})^{-1} \iter{B_{11}}{5} \\
        0 \\
        0
      \end{bmatrix},
    \end{equation*}
    which shows that we can restrict our attention to
    $(\iter{A_{11}}{5}, \iter{B_{11}}{5}, \iter{C_{11}}{5})$. We can
    save work if we just want to compute the restricted singular
    triplets, because then it suffices to only compute
    $\iter{C_{11}}{5}$ with a QR QR decomposition of the $m'$ left-most
    columns of $\iter{C}{4} \iter{Q}{4}$.

  \item If $r = \min \{m', p'\} \ge n' = l$. First use an RQ
    decomposition to compute $\iter{Q}{4}$ such that $\iter{C}{4}
    \iter{Q}{4} = [0\; \iter{C_{12}}{5}]$, where $\iter{C_{12}}{5}$ is a
    $n'\times n'$ upper-triangular matrix. Next, compute $\iter{P}{4}$
    such that $\iter{A_{14}}{5} = \iter{P}{4}{}^T\!  \iter{A}{4}
    \iter{Q}{4}$ is upper triangular. Then use an RQ decomposition to
    compute $\iter{U}{4}$ such that $\iter{B}{5} = \iter{P}{4}{}^T
    \iter{B}{4} \iter{U}{4}$ has the form
    \begin{equation*}
      \compressedmatrices
      \bordermatrix[{[]}]{%
                & \max\{0, m' - r\} & r - n' & n' \cr
        p' - n' & 0 & \iter{B_{12}}{5} & \iter{B_{13}}{5} \cr
        n'      & 0 & 0                & \iter{B_{23}}{5}
      },
    \end{equation*}
    where $\iter{B_{23}}{5}$ is upper triangular and $\iter{B_{12}}{5}$
    is upper trapezoidal with $ p' - n' \ge r - n'$.  Assuming
    $\iter{A}{5}$ is partitioned conformally,
    \begin{equation*}
      \iter{C}{5} (\iter{A}{5})^{-1} \iter{B}{5} =
      \begin{bmatrix}
        0 & 0 & \iter{C_{12}}{5} (\iter{A_{22}}{5})^{-1} \iter{B_{23}}{5}
      \end{bmatrix},
    \end{equation*}
    which shows that we can restrict our attention to
    $(\iter{A_{22}}{5}, \iter{B_{23}}{5}, \iter{C_{12}}{5})$. We can
    save work if we just want to compute the restricted singular
    triplets, because then it suffices to only compute
    $\iter{B_{23}}{5}$ with an RQ decomposition of the bottom $n'$ rows
    of $\iter{P}{4}{}^T \iter{B}{4}$.
\end{enumerate}

The Schur form from Theorem~\ref{thm:schurform} corresponds to the above
three cases with
\begin{enumerate}
  \item $p_1 = q_3 = 0$ and $p_3 = q_5 = 0$ and $m_1 = n_4 = 0$ (or
    $m_2 = n_3 = 0$),
  \item $p_1 = q_3 = 0$ and $m_1 = m_2 = 0$,
  \item $p_3 = q_5 = 0$ and $n_3 = n_4 = 0$,
\end{enumerate}
respectively, and $m_3 = n_2 = p_2 = q_4 = l$. In principle, we may
assume that we always have the first or second case, because we can
transform the input triplet $(A, B, C)$ to $(\Pi_r A^T \Pi_c, \Pi_r C^T
\Pi_c, \Pi_r B^T \Pi_c)$,  where $\Pi_r$ and $\Pi_c$ are the
antidiagonal permutation matrices that reverse the order of the rows and
columns.  In any case, we can compute the nonzero restricted singular
triplets of $(A, B, C)$ from specific square and upper-triangular blocks
of $\iter{A}{5}$, $\iter{B}{5}$, and $\iter{B}{5}$, where the block
coming from $\iter{A}{5}$ is nonsingular.  These three blocks correspond
to the $A_{24}$, $B_{23}$, and $C_{24}$ from
\eqref{eq:schurforma}--\eqref{eq:schurformc}, respectively, and have
exactly the form we need for the implicit Kogbetliantz iteration
described in the next section.



\section{The Kogbetliantz phase}\label{sec:kog}


\subsection{The implicit Kogbetliantz method}

For a given triplet of upper-tri\-an\-gu\-lar $l\times l$ matrices $A$,
$B$, and $C$, where $A$ is nonsingular, the goal of the Kogbetliantz
phase is to find orthonormal matrices $P$, $Q$, $U$, and $V$, so that
$P^T\!AQ$, $P^T\!BU$, and $V^T\!CQ$ are upper-triangular and
$V^T\!CA^{-1}BU$ is diagonal. The essence of this phase is to implicitly
apply a Kogbetliantz-type iteration to $M = CA^{-1}B$; that is, to
compute the SVD of $M$ without forming $M$ or computing $A^{-1}$. This
is different from Zha's approach~\cite{Zha91}, who implicitly applies
the iteration to the product $B^{-1} A C^{-1}$.  A description of the
new procedure follows below; for more background and details see, e.g.,
Bai and Demmel~\cite{BD93}, Charlier, Vanbegin, and
Van~Dooren~\cite{CVV88}, Forsythe and Henrici~\cite{FH60},
Hansen~\cite{Han63}, Heath et al.~\cite{HLPW86}, Paige~\cite{Pai86}, and
Zha~\cite{Zha92}, and their references.

The implicit Kogbetliantz method iterates over pairs $(i,j)$ with $i < j
\le l$, and for each pair applies rotations to the $i$th and $j$th rows
and columns of $A$, $B$, and $C$. This is done in such a way that
$a_{ij}$, $b_{ij}$, $c_{ij}$, and also $m_{ij}$ become zero, while the
corresponding $(j,i)$th elements (may) become nonzero. We refer to this
as annihilating the $(i,j)$th elements.  A sequence of iterations over
all $n(n-1)/2$ pairs $(i,j)$ is called a cycle, and a cycle can progress
through the pairs in different orderings. Some of these orderings, but
not all, are proven to lead to converging methods~\cite{Han63}. That is,
$M$ converges to a diagonal matrix after sufficiently many cycles.  A
common ordering, and the one that we will focus on, is the row-cyclic
ordering $(1,2)$, $(1,3)$, \dots, $(1,l)$, $(2, 3)$, $(2,4)$, \dots,
$(l-1,l)$. A row sweep is what we call a series of transformations that
annihilate all the off-diagonal elements in a single row. During a cycle
of sweeps in a row-cyclic ordering, the $i$th row sweep produces fill-in
in the $i$th column, so that a full cycle turns the initially
upper-triangular matrices into lower-triangular matrices.  In the
following cycle, we effectively consider the triplet $(A^T, C^T, B^T)$
as the input, which recovers the upper-triangular structure of the
matrices. This leads to a sequence of alternating odd and even cycles
that we repeat either until convergence, or until we reach a predefined
maximum number of cycles.

To see how we can implicitly work with $M$, suppose that $(i,j)$ is our
pivot and that we want to annihilate $m_{ij}$. At this point in a cycle
with row-cyclic ordering, $A$ and $A^{-1}$ have the form
\begin{equation*}
  \compressedmatrices
  A =
  \bordermatrix[{[]}]{%
          & i-1    & j-i+1 & l-j     \cr
    i-1   & A_{11} & 0      & 0      \cr
    j-i+1 & A_{21} & A_{22} & A_{23} \cr
    l-j   & A_{31} & 0      & A_{33} \cr
  }
  \quad\text{and}\quad
  A^{-1} =
  \begin{bmatrix}
    A_{11}^{-1}       & 0           & 0 \\
    \widetilde A_{21} & A_{22}^{-1} & \widetilde A_{23} \\
    \widetilde A_{31} & 0           & A_{33}^{-1}
  \end{bmatrix}
\end{equation*}
for appropriate $\widetilde A_{21}$, $\widetilde A_{31}$, $\widetilde
A_{23}$, where $A_{11}$ is lower triangular and $A_{33}$ upper
triangular. Furthermore, for some vector $\vec a$ with $\vec e_{j-i}^T
\vec a = 0$ and upper-triangular $R_A$ with $\vec e_{j-i}^T R_A = a_{jj}
\vec e_{j-i}^T$, we have that
\begin{equation*}
  \compressedmatrices
  A_{22} =
  \begin{bmatrix}
    a_{ii} & a_{ij} \vec e_{j-i}^T \\
    \vec a & R_A
  \end{bmatrix}
  \quad\text{and}\quad
  A_{22}^{-1} =
  \frac{1}{a_{ii} a_{jj}}
  \begin{bmatrix}
    a_{jj} & -a_{ij} \vec e_{j-i}^T \\
    -a_{jj} R_A^{-1} \vec a &
    a_{ii} a_{jj} R_A^{-1} + a_{ij} R_A^{-1} \vec a \vec e_{j-i}^T
  \end{bmatrix}.
\end{equation*}
Since $B$ and $C$ have the same structure as $A$, we can partition their
blocks identically and use a similar notation for the blocks $B_{22}$
and $C_{22}$. If we now ignore the previous subscript indices of the
matrix blocks and define $M_{ij} = \minimat{m_{ii}}{m_{ij}}{0}{m_{jj}}$
and $A_{ij}$, $B_{ij}$, and $C_{ij}$ likewise, then we can check that
\begin{equation*}
  \compressedmatrices
  \begin{split}
    M_{ij} &= \frac{1}{a_{ii} a_{jj}}
    \begin{bmatrix}
      c_{ii} & c_{ij} \vec e_{j-i}^T \\
      0 & c_{jj} \vec e_{j-i}^T
    \end{bmatrix}
    \begin{bmatrix}
      a_{jj} &
      -a_{ij} \vec e_{j-i}^T \\
      -a_{jj} R_A^{-1} \vec a &
      a_{ii} a_{jj} R_A^{-1} + a_{ij} R_A^{-1} \vec a \vec e_{j-i}^T
    \end{bmatrix}
    \begin{bmatrix}
      b_{ii} & b_{ij} \\
      \vec b & R_B \vec e_{j-i}
    \end{bmatrix}
    \\ &=
    \frac{1}{a_{ii} a_{jj}}
    \begin{bmatrix}
      c_{ii} a_{jj} b_{ii} &
      c_{ii} a_{jj} b_{ij} + (c_{ij} a_{ii} - c_{ii} a_{ij}) b_{jj} \\
      0 & c_{jj} a_{ii} b_{jj}
    \end{bmatrix}
    \\ &=
    C_{ij} A_{ij}^{-1} B_{ij}.
  \end{split}
\end{equation*}
We can even replace the inverse $A_{ij}^{-1}$ by the adjugate matrix
$\adj(A_{ij})$, because the scaling of $M_{ij}$ does not matter when
computing the rotations. Thus, we will henceforth define
\begin{equation}\label{eq:Mij}
  \compressedmatrices
  M_{ij}
  = C_{ij} \adj (A_{ij}) B_{ij}
  =
  \biggl[\begin{matrix}
    c_{ii} & c_{ij} \\ 0 & c_{jj}
  \end{matrix}\biggr]
  \biggl[\begin{matrix}
    a_{jj} & -a_{ij} \\ 0 & \phantom{-}a_{ii}
  \end{matrix}\biggr]
  \biggl[\begin{matrix}
    b_{ii} & b_{ij} \\ 0 & b_{jj}
  \end{matrix}\biggr],
\end{equation}
while stressing that this definition is only correct when annihilating
$m_{ij}$.

Computing $M_{ij}$ is the first step to computing the rotations that
annihilate the $(i,j)$th elements.  The second step is to compute an SVD
$V^T\! M_{ij}U = \diag(\sigma_1,\sigma_2)$, where $\sigma_1$ and
$\sigma_2$ are real, and $U = \rot(\phi)$ and $V = \rot(\psi)$ for
appropriate angles $\phi$, $\psi$, and $\rot(\theta)$ denotes the
rotation matrix
$\minimat{\phantom{-}\cos\theta}{\sin\theta}{-\sin\theta}{\cos\theta}$.
This SVD may be unnormalized, which means that its singular values are
not necessarily nonnegative, nor sorted by magnitude. The next step is
to compute rotations $P$ and $Q$ such that $V^T\! C_{ij} Q$, $P^T\!
A_{ij} Q$, $P^T\! B_{ij} U$ are lower triangular. For the final step,
let $P_{ij}$, $Q_{ij}$, $U_{ij}$, and $V_{ij}$ be identity matrices with
the $(i,i)$, $(i,j)$, $(j,i)$, and $(j,j)$ elements replaced by the
$(1,1)$, $(1,2)$, $(2,1)$, and $(2,2)$ elements of $P$, $Q$, $U$, and
$V$, respectively, and compute the transformations $P_{ij}^T A Q_{ij}$,
$P_{ij}^T B U_{ij}$, and $V_{ij}^T C Q_{ij}$ as in~\eqref{eq:update}.
Accumulating the matrices $P_{ij}$, $Q_{ij}$, $U_{ij}$, and $V_{ij}$ is
optional, but necessary if we need the restricted singular vectors.  See
Algorithm~\ref{alg:kog} for a summary of the procedure.


\begin{myalgorithm}[An implicit Kogbetliantz iteration for the RSVD]\label{alg:kog}
\strut\\*
\rm
  \textbf{Input:} Square and upper-triangular $l \times l$ matrices $A$,
    $B$, $C$, and $A$ nonsingular. \\
  \textbf{Output:} $P$, $Q$, $U$, $V$, $A'$, $B'$, and $C'$ such that
    $A' = P^T\!AQ$, $B' = P^T\!BU$, and $C' = V^T\!CQ$ are upper
    triangular and $V^T\!CA^{-1}BU$ is diagonal. \\
  \tab[1.] \textbf{while} \textit{\# cycles is odd} \textbf{or}
    (\textit{\# cycles} $<$ \textit{max cycles} \textbf{and}
    \textit{not converged}) \textbf{do} \\
  \tab[2.]\tab \textbf{for} $i = 1$, 2, \dots, $l-1$ \textbf{do} \\
  \tab[3.]\tab\tab \textbf{for} $j = i+1$, $i+2$, \dots, $l$ \textbf{do} \\
  \tab[4.]\tab\tab\tab Select $A_{ij}$, $B_{ij}$, and $C_{ij}$ as outlined
    in the text. \\
  \tab[5.]\tab\tab\tab In odd cycles: set $(A_{ij}, B_{ij}, C_{ij}) =
    (A_{ij}^T, C_{ij}^T, B_{ij}^T)$. \\
  \tab[6.]\tab\tab\tab Compute $P_{ij}$, $Q_{ij}$, $U_{ij}$, and $V_{ij}$
    from $A_{ij}$, $B_{ij}$, and $C_{ij}$. \\
  \tab[7.]\tab\tab\tab In odd cycles: swap $P_{ij}$ with $Q_{ij}$ and
    $U_{ij}$ with $V_{ij}$. \\
  \tab[8.]\tab\tab\tab Update $A = P_{ij}^T A Q_{ij}$, $B = P_{ij}^T B
  U_{ij}$, $C = V_{ij}^T C Q_{ij}$. \\
  \tab[9.]\tab\tab\tab Accumulate $P = PP_{ij}$, $Q = QQ_{ij}$, $U =
  UU_{ij}$, and $V = VV_{ij}$.
\end{myalgorithm}


Forsythe and Henrici~\cite{FH60} prove that row-cyclic sweeps lead
to (fast) convergence when a fixed closed interval within
$(-\pi/2,\pi/2)$ contains all angles $\phi$ and $\psi$. Since this
condition is impossible to guarantee while simultaneously diagonalizing
$M_{ij}$ exactly, Forsythe and Henrici also prove that a set of weaker
requirements suffice for linear convergence. The benefit of these weaker
requirements is that they are almost always satisfied in floating-point
arithmetic. In any case, Heath et al.~\cite[Sec.~3]{HLPW86} argue
for the use of an unnormalized SVD as it simplifies the algorithm and
they found it to be just as effective.  This observation relies on the
fact that Forsythe and Henrici's convergence proof only considers the
magnitude of the matrix entries.  In practical term this means that we
may work with $-U$ or $-V$ instead of $U$ and $V$, and thus, also with
half period shifts and angles in a fixed closed interval of $(\pi/2,3/2
\pi)$. In other words, the angles just need to stay away from an open
interval around $\pm\pi/2$.

Still, Brent, Luk, and Van~Loan~\cite[Sec.~4]{BLL83} conjecture
``that the smaller the rotation angles are the faster the procedure will
converge''. One way to adjust the angles is with a quarter period shift;
that is, by replacing $U$ and $V$ with $UJ$ and $VJ$, respectively,
where $J = \rot(\pi/2) = \minimat{\phantom{-}0}{1}{-1}{0}$.  For
example, the routine \textsf{xLAGS2} of the current version of
LAPACK\footnote{Version 3.8.0 at the time of writing.} compute the
upper-triangular $2\times 2$ SVDs with \textsf{xLASV2}, and ensures that
$\min \{ |\phi|, |\psi| \} \le \pi/4$ in essence by multiplying $U$ and
$V$ with $J$ if $|u_{11}| < |u_{12}|$ or $|v_{11}| < |v_{12}|$. Since
this condition appears suboptimal if, say, $|u_{11}| > |u_{12}|$ and
$|v_{11}| < |v_{12}|$, we will instead try to minimize the maximum
angle.  That is, we will replace $U$ and $V$ with $UJ$ and $VJ$,
respectively, if and only if $\max \{ |u_{11}|, |v_{11}| \} < \max \{
  |u_{12}|, |v_{12}| \}$.  This strategy ensures that we both have $\max
\{ |u_{11}|, |v_{11}| \} \ge 1/\sqrt{2}$ and $\max \{ |u_{11}|, |v_{11}|
\} \ge \max \{ |u_{12}|, |v_{12}| \}$, although a downside is that we
cannot always guarantee a particular ordering of the singular values
during the cycles. However, this can also not be guaranteed with other
conditions that stay away from the rotation angles $\pm \pi/2$.

A standard approach to check for convergence is to define $\rho_{ij} =
0$ if $M_{ij} = 0$ and $\rho_{ij} = |m_{ij}| / \|M_{ij}\|_{\max} \le 1$
otherwise, and to stop if all $\rho_{ij} < \tau$ for some tolerance
$\tau$. Another option, suggested by Demmel and Veseli\'{c}~\cite{DV92},
is to use $\rho_{ij} = |m_{ij}| (|m_{ii}| |m_{jj}|)^{-1/2} \le \infty$
(if $M_{ij} \neq 0$) instead. The problem for implicit
Kog\-bet\-li\-antz-type iterations with both of these definitions of
$\rho_{ij}$, is that $|m_{ij}|$ may not become ``small'' enough in
floating-point arithmetic for the stopping criterion to be fulfilled.
This unfortunate discrepancy between theory and practice exists, at
least in part, because the implicit Kogbetliantz method forms each
$M_{ij}$ on-the-fly. This means that the relative error in the computed
$|m_{ij}|$ can be of order $\mathcal O(\bm\epsilon \|A_{ij}\| \|B_{ij}\|
\|C_{ij}\|)$, where $\bm\epsilon$ is the unit roundoff, rather than
$\mathcal O(\bm\epsilon \|M_{ij}\|)$, even if the former is often
pessimistic. Hence, a $\tau$ picked based on the former may be too
large, and a $\tau$ picked based on the latter may be too small.  This
does not even take other sources of roundoff errors into account yet,
such as, for example, perturbations in the computed rotations and the
roundoff errors from the application of the rotations.

Bai and Demmel~\cite[Sec.~4]{BD93} use a different approach and
measure the parallelism between corresponding rows of two matrices $A$
and $B$. The theoretical justification is simple: when all corresponding
rows of $A$ and $B$ are parallel, then there must exist diagonal
matrices $C$ and $S$ and an upper-triangular matrix $R$ such that $A =
CR$ and $B = SR$.  This justification and the corresponding
implementation are appealing, but the generalization to matrix triplets
and the RSVD is not obvious.  A simplified approach without a similar
theoretical justification is to consider the angle between
two-dimensional vectors and if $m_{ij} \neq 0$ to take
\begin{equation}\label{eq:rhoij}
  \rho_{ij} = \max \left\{
    \frac{|\vec e_1^T C_{ij} \adj(A_{ij}) B_{ij} \vec e_2|}{%
      \|\vec e_1^T C_{ij}\| \|\adj(A_{ij}) B_{ij} \vec e_2\|},
    \frac{|\vec e_1^T C_{ij} \adj(A_{ij}) B_{ij} \vec e_2|}{%
      \|\vec e_1^T C_{ij} \adj(A_{ij}) \| \|B_{ij} \vec e_2\|}
  \right\}
\end{equation}
for each pair of $i$ and $j$. Although the relative scaling is still not
ideal because the roundoff errors in $|m_{ij}|$ may be as big as
$\mathcal O(\bm\epsilon \|\vec e_1^T C_{ij}\| \|A_{ij}\| \|B_{ij} \vec
e_2\|)$, this $\rho_{ij}$ strikes a balance that appears to work well in
our limited testing.

Regardless of the choice of $\rho_{ij}$, we may want to stop iterating
before convergence when progress is too slow and before reaching a
predefined maximum number of cycles. To decide on this, one option is to
compute $\rho = \max_{i,j} \rho_{ij}$ during each cycle, let
$\rho_{\min}$ be the smallest $\rho$ of all previous cycles, and stop
iterating (after an even number of cycles) if (after an even number of
cycles) if $\rho_{\min} \lesssim \rho \ll 1$.  That is, stop when both
$\rho$ and the improvement between cycles are small.



\subsection{The 2-by-2 RSVD in exact arithmetic}\label{sec:rsvdTwoExact}

Algorithm~\ref{alg:kog} does not tell us how to compute the RSVD of
(upper-)triangular $2$-by-$2$ matrices. But this is an easier problem to
solve than computing the RSVD of larger matrices. See, for example, the
theorem below.


\begin{proposition}\label{thm:exactrsvdtwo}
  Let $A$, $B$, and $C$ be arbitrary $2 \times 2$ upper-triangular
  matrices, and define $M = C \adj(A) B$; then there exist orthonormal
  matrices $P$, $Q$, $U$, and $V$, such that $P^T\!AQ$, $P^T\!BU$, and
  $V^T\!CQ$ are lower triangular, and $V^T\!MU = \Sigma$ is diagonal.
\end{proposition}
\begin{proof}
  If any two of the three matrices $A$, $B$, and $C$ are nonsingular,
  then the result is straightforward. For example, if $B$ and $C$ are
  nonsingular, then we can find $U$ and $V$ by computing the SVD of $M$,
  and letting $P$ and $Q$ zero the $(1,2)$ entries of $BU$ and $V^T\!C$,
  respectively.  Then $Q^T\! \adj(A) P = (V^T\!CQ)^{-1} \Sigma
  (P^T\!BU)^{-1}$ and it follows that $P^T\!AQ$ must be lower
  triangular.  By noting that $\adj(A)$ is nonsingular if and only if
  $A$ is nonsingular, we see that similar arguments hold when $A$ and
  $B$ are nonsingular or when $A$ and $C$ are nonsingular.
  
  If $B$ is singular and $A$ and $C$ are arbitrary, we can compute $P$
  and $U$ such that $P^T\!BU$ has the form $\minimat{0}{0}{0}{\times}$,
  compute $Q$ so that $P^T\!AQ$ is lower triangular, and compute $V$ so
  that $V^T\!CQ$ is lower triangular. By using the fact that $\adj(P^T\!
  AQ)$ is a scalar multiple of $P^T \adj(A) Q$, we then see that
  $V^T\!MU$ is a scalar multiple of $(V^T\!CQ) \cdot \adj(P^T\!AQ) \cdot
  (P^T\!BU)$, which is of the form $\minimat{0}{0}{0}{\times}$.

  If $C$ is singular and $A$ and $B$ are arbitrary, we can compute $V$
  and $Q$ such that $V^T\!CQ$ is of the form
  $\minimat{\times}{0}{0}{0}$, compute $P$ so that $P^T\!AQ$ is lower
  triangular, and compute $U$ so that $P^T\!BU$ is lower triangular.
  Then $V^TMU$ is a scalar multiple of $(V^T\!CQ) \cdot \adj(P^T\!AQ)
  \cdot (P^T\!BU)$, which is of the form $\minimat{\times}{0}{0}{0}$.
\end{proof}


The theorem above does not tell us anything about the angles of the
rotations, nor about the numerical stability of the computations.
Rather, the theorem shows that computing the $2\times 2$ RSVD is
possible for any triplet of upper-triangular matrices, even when $A$ is
singular.  Knowing what is possible, the question that remains is how to
do it in a numerically sound way.

Bojanczyk et al.~\cite{BELD91} propose a recursive algorithm for
accurately computing the SVD of a product of three upper-triangular
$2\times 2$ matrices that is close to the $2\times 2$ RSVD needed for
Algorithm~\ref{alg:kog}.  But their diagonalization is not guaranteed to
have high \emph{relative} accuracy, as demonstrated by Bai and
Demmel~\cite{BD93} for the QSVD.\@ Adams, Bojanczyk, and Luk address
this issue for the product of two matrices in~\cite{ABL94} with a
modified version of their algorithm that they call ``half-recursive'',
and which they show is related to Bai and Demmel's algorithm in exact
arithmetic. Though, they did not provide an improved version of their
algorithm for the product of three matrices.

We can generalize Bai and Demmel's algorithm for the $2\times 2$ QSVD to
the $2\times 2$ RSVD, as shown below in Algorithm~\ref{alg:rsvd}.
Informally, the idea of the algorithm is to apply a modified version of
Bai and Demmel's \textsf{GSVD22} to the pairs $(C,\adj(A)B)$ and
$(C\adj(A),B)$, but some of the details require further attention. For
example, what to do when $c_{11} = b_{22} = 0$, and when to replace $U$
and $V$ by $UJ$ and $VJ$. For the latter in particular, there are
qualitative differences between postmultiplying by $J$ when $c_{\max} <
s_{\max}$ or when $c_{\max} \le s_{\max}$ if $B$ or $C$ are singular,
and the choice between the two conditions is not obvious. The condition
we ultimately use in the algorithm below ensures that
Lemma~\ref{thm:algIO} and Lemma~\ref{thm:lasvtwo} hold.


\begin{myalgorithm}[$2\times 2$ upper-triangular RSVD (\textsf{RSVD22})]\label{alg:rsvd}
\strut\\* 
\rm
  \textbf{Input:} $2\times 2$ upper-triangular matrices $A$, $B$, and $C$, with
    $A$ nonsingular. \\
  \textbf{Output:} Orthonormal matrices $P$, $Q$, $U$, and $V$, and
    lower-triangular matrices $A' = P^T\!AQ$, $B' = P^T\!BU$, and $C' =
    V^T\!CQ$, such that $C' \adj(A') B'$ is diagonal. \\
  \tab[\phantom01.] \textbf{if} $c_{11} = 0$ and $b_{22} = 0$ \textbf{then} \\
  \tab[\phantom02.]\tab Compute $V$ such that $(V^T\!C)_{22} = 0$ and let $Q = J$. \\
  \tab[\phantom03.]\tab Compute $U$ such that $(BU)_{11} = 0$ and let $P = J$. \\
  \tab[\phantom04.]\tab Let $A' = P^T\!AQ$, $B' = P^T\!BU$, $C' = V^T\!CQ$,
    and $b_{21}' = c_{21}' = 0$. \\
  \tab[\phantom05.]\tab \textbf{return} \\
  \tab[\phantom06.] \textbf{endif} \\
  \tab[\phantom07.] Use \textsf{xLASV2} to compute $M = V\Sigma U^T$,
    where $M = C\adj(A)B$. \\ 
  \tab[\phantom08.] Define $c_{\max} = \max \{ |u_{11}|, |v_{11}| \}$ and
    $s_{\max} = \max \{ |u_{12}|, |v_{12}| \}$. \\
  \tab[\phantom09.] \textbf{if} $c_{11} \neq 0$ and $c_{22} \neq 0$ and
  $b_{11} \neq 0$ and $b_{22} \neq 0$ and $c_{\max} < s_{\max}$ \textbf{then} \\
  \tab[10.]\tab Let $U = UJ$ and $V = VJ$. \\
  \tab[11.] \textbf{endif} \\
  \tab[12.] Let $G = V^T\! C$ and (optionally; see text) set $g_{22}$ to
    zero if $c_{11} = 0$. \\
  \tab[13.] Let $L = BU$ and (optionally; see text) set $l_{12}$ to zero
    if $b_{22} = 0$. \\
  \tab[14.] Let $\widehat G = |V|^T |C|$, $H = \adj(A)L$, and $\widehat
    H = |\adj(A)|\, |B|\, |U|$. \\
  \tab[15.] Let $K = G\adj(A)$, $\widehat K = |V|^T |C|\, |\adj(A)|$,
    and $\widehat L = |B|\, |U|$. \\
  \tab[16.] Let $\eta_g = (\widehat g_{11} + \widehat g_{12}) /
    (|g_{11}| + |g_{12}|)$ and $\eta_h = (\widehat h_{12} + \widehat
    h_{22}) / (|h_{12}| + |h_{22}|)$. \\
  \tab[17.] Let $\eta_k = (\widehat k_{11} + \widehat k_{12}) /
    (|k_{11}| + |k_{12}|)$ and $\eta_l = (\widehat l_{12} + \widehat
    l_{22}) / (|l_{12}| + |l_{22}|)$. \\
  \tab[18.] \textbf{if} $|h_{12}| + |h_{22}| = 0$ or ( $|g_{11}| +
    |g_{12}| \neq 0$ and $\eta_g \le \eta_h$) \textbf{then} \\
  \tab[19.]\tab Use \textsf{xLARTG} to compute $Q$ such that $GQ$ is
    lower triangular. \\
  \tab[20.] \textbf{else} \\
  \tab[21.]\tab Use \textsf{xLARTG} to compute $Q$ such that $Q^TH$ is
    lower triangular. \\
  \tab[22.] \textbf{endif} \\
  \tab[23.] \textbf{if} $|k_{11}| + |k_{12}| = 0$ or ($|l_{12}| +
    |l_{22}| \neq 0$ and $\eta_l \le \eta_k$) \textbf{then} \\
  \tab[24.]\tab Use \textsf{xLARTG} to compute $P$ such that $P^T L$ is
    lower triangular. \\
  \tab[25.] \textbf{else} \\
  \tab[26.]\tab Use \textsf{xLARTG} to compute $P$ such that $KP$ is
    lower triangular. \\
  \tab[27.] \textbf{endif} \\
  \tab[28.] Let $A' = P^T\!AQ$, $B' = P^TL$, $C' = GQ$, and
    $a'_{12} = b'_{12} = c'_{12} = 0$.
\end{myalgorithm}


Since $B$ and $C$ can be singular, there may be zeros on their
diagonals. If this is the case, and if the factors $X$ and $Y$ from
Theorem~\ref{thm:rsvd} or the factors $S$ and $T$ from
Corollary~\ref{thm:rsvdtri} are desired, then we need to know the
nonzero structure of $B$ and $C$ after convergence.
Paige~\cite[Sec.~5]{Pai86} describes the nonzero structure for the QSVD
in a similar case, and has a proof which is, in his own words, ``hard
going''. The proof for the RSVD is tedious also, and is split into two
parts.  The first part is a lemma that gives the output of
\textsf{RSVD22} for a given input, and the second part is a proposition
that uses the lemma to prove what kind of nonzero structure we get for
the RSVD after a series of cycles.

In principle, we have to consider a total of 25 different cases when
investigating the nonzero structure of the outputs of
Algorithm~\ref{alg:rsvd}.  For $B$ alone, for instance, we must already
consider the following five cases:
\begin{equation*}
  \compressedmatrices
  \begin{bmatrix} \underline{b_{11}} & b_{12} \\ 0 & \underline{b_{22}} \end{bmatrix},
  \qquad
  \begin{bmatrix} \underline{b_{11}} & b_{12} \\ 0 & 0 \end{bmatrix},
  \qquad
  \begin{bmatrix} 0 & b_{12} \\ 0 & \underline{b_{22}} \end{bmatrix},
  \qquad
  \begin{bmatrix} 0 & \underline{b_{12}} \\ 0 & 0 \end{bmatrix},
  \quad\text{and}\quad
  \begin{bmatrix} 0 & 0 \\ 0 & 0 \end{bmatrix},
\end{equation*}
where the underlined entries are nonzero. Fortunately, we can treat some
of the 25 cases simultaneously and reduce them to 13 cases.


\begin{lemma}\label{thm:algIO}
  Let $A$, $B$, and $C$ be upper-triangular $2\times 2$ matrices, and
  suppose that $A$ is nonsingular. If the SVD in
  Algorithm~\ref{alg:rsvd} computes $U = V = I$ whenever $M = 0$, and is
  such that $|\sigma_1| \ge |\sigma_2|$; then the cases given below
  describe the output of Algorithm~\ref{alg:rsvd}.  Each case shows (in
  sequence) the structure of the input matrices $B$ and $C$ ($A$ is
  always upper triangular with nonzero diagonal entries), the matrix $M
  = C \adj(A) B$, and the output matrices $B'$ and $C'$ ($A'$ is always
  lower triangular with nonzero diagonal entries). Each case also shows
  $P$, $Q$, $U$, and $V$ when they take specific values.  Furthermore,
  underlined matrix entries are nonzero.
  \begin{enumerate}
    \setlength\itemsep{0.00ex}
    \item\label{it:CtriBtri}
      $\minimat{ \nonzero{c_{11}} }{c_{12}}{0}{ \nonzero{c_{22}} }$,
      $\minimat{ \nonzero{b_{11}} }{b_{12}}{0}{ \nonzero{b_{22}} }$,
      $\minimat{ \nonzero{m_{11}} }{m_{12}}{0}{ \nonzero{m_{22}} }$,
      $\minimat{ \nonzero{c_{11}'} }{0}{c_{21}'}{ \nonzero{c_{22}'} }$,
      $\minimat{ \nonzero{b_{11}'} }{0}{b_{21}'}{ \nonzero{b_{22}'} }$.

    \item\label{it:CcolBrow}\label{it:CzeroBrow}\label{it:CcolBzero}\label{it:CzeroBzero}
      $\minimat{0}{c_{12}}{0}{c_{22}}$,
      $\minimat{b_{11}}{b_{12}}{0}{0}$,
      $0$,
      $\minimat{c_{11}'}{0}{0}{0}$,
      $\minimat{0}{0}{0}{b_{22}'}$,
      $P = Q = J$.

    \item\label{it:CrowBzero}
      $\minimat{ \nonzero{c_{11}} }{c_{12}}{0}{0}$,
      $\minimat{0}{0}{0}{0}$,
      $0$,
      $\minimat{ \nonzero{c_{11}'} }{0}{0}{0}$,
      $\minimat{0}{0}{0}{0}$,
      $U = V = I$.

    \item\label{it:CtriBzero}
      $\minimat{ \nonzero{c_{11}} }{c_{12}}{0}{ \nonzero{c_{22}} }$,
      $\minimat{0}{0}{0}{0}$,
      $0$,
      $\minimat{ \nonzero{c_{11}'} }{0}{c_{21}'}{ \nonzero{c_{22}'} }$,
      $\minimat{0}{0}{0}{0}$,
      $U = V = I$.

    \item\label{it:CrowBrow}
      $\minimat{ \nonzero{c_{11}} }{c_{12}}{0}{0}$,
      $\minimat{b_{11}}{b_{12}}{0}{0} \neq 0$,
      $c_{11} a_{22} B \neq 0$,
      $\minimat{ \nonzero{c_{11}'} }{0}{0}{0}$,
      $\minimat{ \nonzero{b_{11}'} }{0}{b_{21}'}{0}$,
      and $|V| = I$.

    \item\label{it:CtriBrow}
      $\minimat{ \nonzero{c_{11}} }{c_{12}}{0}{ \nonzero{c_{22}} }$,
      $\minimat{b_{11}}{b_{12}}{0}{0} \neq 0$,
      $c_{11} a_{22} B \neq 0$,
      $\minimat{ \nonzero{c_{11}'} }{0}{c_{21}'}{ \nonzero{c_{22}}' }$,
      $\minimat{ \nonzero{b_{11}'} }{0}{b_{21}'}{0}$,
      and $|V| = I$.

    \item\label{it:CrowBtri}
      $\minimat{ \nonzero{c_{11}} }{c_{12}}{0}{0}$,
      $\minimat{ \nonzero{b_{11}} }{b_{12}}{0}{ \nonzero{b_{22}} }$,
      $\minimat{ \nonzero{m_{11}} }{m_{12}}{0}{0}$,
      $\minimat{ \nonzero{c_{11}'} }{0}{0}{0}$,
      $\minimat{ \nonzero{b_{11}'} }{0}{b_{21}'}{ \nonzero{b_{22}'} }$,
      $|U| \neq |J|$ and $|V| = I$.

    \item\label{it:CzeroBcol}
      $\minimat{0}{0}{0}{0}$,
      $\minimat{0}{b_{12}}{0}{ \nonzero{b_{22}} }$,
      $0$,
      $\minimat{0}{0}{0}{0}$,
      $\minimat{0}{0}{0}{ \nonzero{b_{22}'} }$,
      $U = V = I$.

    \item\label{it:CzeroBtri}
      $\minimat{0}{0}{0}{0}$,
      $\minimat{ \nonzero{b_{11}} }{b_{12}}{0}{ \nonzero{b_{22}} }$,
      $0$,
      $\minimat{0}{0}{0}{0}$,
      $\minimat{ \nonzero{b_{11}'} }{0}{b_{21}'}{ \nonzero{b_{22}'} }$,
      $U = V = I$.

    \item\label{it:CcolBcol}
      $\minimat{0}{c_{12}}{0}{c_{22}} \neq 0$,
      $\minimat{0}{b_{12}}{0}{ \nonzero{b_{22}} }$,
      $a_{11} b_{22} C \neq 0$,
      $\minimat{ \nonzero{c_{11}'} }{0}{0}{0}$,
      $\minimat{ \nonzero{b_{11}'} }{0}{b_{21}'}{0}$,
      and $|P| = |Q| = |U| = |J|$.

    \item\label{it:CcolBtri}
      $\minimat{0}{c_{12}}{0}{c_{22}} \neq 0$,
      $\minimat{ \nonzero{b_{11}} }{b_{12}}{0}{ \nonzero{b_{22}} }$,
      $a_{11} b_{22} C \neq 0$,
      $\minimat{\nonzero{c_{11}'}}{0}{0}{0}$,
      $\minimat{ \nonzero{b_{11}'} }{0}{b_{21}'}{ \nonzero{b_{22}'} }$,
      and $|P| = |Q| = |U| = |J|$.

    \item\label{it:CtriBcol}
      $\minimat{ \nonzero{c_{11}} }{c_{12}}{0}{ \nonzero{c_{22}} }$,
      $\minimat{0}{b_{12}}{0}{ \nonzero{b_{22}} }$,
      $\minimat{0}{m_{12}}{0}{ \nonzero{m_{22}} }$,
      $\minimat{ \nonzero{c_{11}'} }{0}{c_{21}'}{ \nonzero{c_{22}'} }$,
      $\minimat{ \nonzero{b_{11}'} }{0}{b_{21}'}{0}$,
      $|U| = |J|$ and $|V| \neq I$.

    \item\label{it:CrowBcol} 
      $\minimat{ \nonzero{c_{11}} }{c_{12}}{0}{0}$,
      $\minimat{0}{b_{12}}{0}{ \nonzero{b_{22}} }$,
      $\minimat{0}{m_{12}}{0}{0}$; if $m_{12} = 0$, then $C' =
      \minimat{\nonzero{c_{11}'}}{0}{0}{0}$ and $B' =
      \minimat{0}{0}{0}{\nonzero{b_{22}'}}$ with $U = V = I$.  If
      $m_{12} \neq 0$, then $C' = \minimat{ \nonzero{c_{11}'}
      }{0}{0}{0}$ and $B' = \minimat{ \nonzero{b_{11}'} }{0}{b_{21}}{0}$
      with $|U| = |J|$ and $|V| = I$.
  \end{enumerate}
\end{lemma}
\begin{proof}
  Following Algorithm~\ref{alg:rsvd} step-by-step for each case yields
  the desired results; see Appendix~\ref{sec:proof} for details.
\end{proof}


Cases~\ref{it:CcolBcol}, \ref{it:CcolBtri}, and~\ref{it:CtriBcol} may
violate the min-max--angle condition; however, these cases can no longer
occur in later iterations. This is a consequence of the following
proposition, which describes the nonzero structure of $A$, $B$, and $C$
after an odd and an even cycle. Furthermore, the positions of the zeros
and nonzeros in the output of Lemma~\ref{thm:algIO} are only guaranteed
in floating-point arithmetic with the optional zeroing of $g_{22}$ if
$c_{11} = 0$ and $l_{12}$ if $b_{22} = 0$ in Algorithm~\ref{alg:rsvd}.
Still, the theoretical results from later sections hold with or without
this explicit zeroing; see Section~\ref{sec:numchallenges} for further
discussion.


\begin{proposition}\label{thm:nzStructure}
  Suppose that $A$, $B$, and $C$ are square and upper-triangular $l
  \times l$ matrices, $A$ is nonsingular, and $M = CA^{-1}B =
  \minimat{M_{11}}{M_{12}}{0}{0}$, where $M_{11}$ is nonsingular and
  upper triangular. Then a pair of an odd and even cycle of
  Algorithm~\ref{alg:kog}, with Algorithm~\ref{alg:rsvd} for the
  $2\times 2$ RSVDs, transforms the structure of $A$, $B$, and $C$ into
  \begin{equation}\label{eq:nzStructureUp}
    \compressedmatrices
    \begin{bmatrix}
      A_{11} & A_{12} & A_{13} & A_{14} \\
      & A_{22} & A_{23} & A_{24} \\
      && A_{33} & A_{34} \\
      &&& A_{44} \\
    \end{bmatrix},
    \ifx\siamartcls\undefined
      \qquad
    \else
      \quad
    \fi
    \begin{bmatrix}
      B_{11} & B_{12} & B_{13} & B_{14} \\
      & B_{22} & B_{23} & B_{24} \\
      && 0 & 0 \\
      &&& 0
    \end{bmatrix},
    \ifx\siamartcls\undefined
      \quad\text{and}\quad
    \else
      \;\text{and}\;
    \fi
    \begin{bmatrix}
      C_{11} & C_{12} & C_{13} & C_{14} \\
      & 0 & 0 & C_{24} \\
      && 0 & C_{34} \\
      &&& C_{44}
    \end{bmatrix},
  \end{equation}
  where all nonzero diagonal blocks are nonsingular and upper
  triangular.
\end{proposition}
\begin{proof}
  Any triplet of $1\times 1$ matrices $A$, $B$, and $C$ satisfies
  \eqref{eq:nzStructureUp} when $A$ is nonsingular.  For larger
  matrices, the upper triangularity of the matrices is a result of the
  row-cyclic cycles. Now, let $\operatorname{kog}(A,B,C)$ denote the
  output of a single odd cycle of Algorithm~\ref{alg:kog}, and let
  $\iter{A}{0} = A$, $\iter{B}{0} = B$, and $\iter{C}{0} = C$; then we
  can write the desired pair of cycles as
  \begin{equation*}
    \begin{split}
      (\iter{A}{0.5},\iter{B}{0.5},\iter{C}{0.5})
      &= \operatorname{kog}(\iter{A}{0},\iter{B}{0},\iter{C}{0})
      \aligncr 
      (\iter{A}{1},\iter{B}{1},\iter{C}{1})
      &= (\iter{A}{0.5}{}^T,\iter{C}{0.5}{}^T,\iter{B}{0.5}{}^T) \\
      (\iter{A}{1.5},\iter{B}{1.5},\iter{C}{1.5})
      &= \operatorname{kog}(\iter{A}{1},\iter{B}{1},\iter{C}{1})
      \aligncr 
      (\iter{A}{2},\iter{B}{2},\iter{C}{2})
      &= (\iter{A}{1.5}{}^T,\iter{C}{1.5}{}^T,\iter{B}{1.5}{}^T).
    \end{split}
  \end{equation*}
  Define $\iter{M}{\ell} = \iter{C}{\ell} (\iter{A}{\ell})^{-1}
  \iter{B}{\ell}$ for $\ell = 0, 0.5$, $1$, \dots, $2$, and for any $M$,
  $j > i$, and just before annihilating $m_{ij}$, define $M_{ij} =
  \minimat{m_{ii}}{m_{ij}}{0}{m_{jj}}$ as in \eqref{eq:Mij}. Moreover,
  we assume for the rest of the proof that the indices $i$, $j$, and $k$
  are always such that $1 \le i, j, k \le l$.

  We start by proving that sweeping the first row of $M = \iter{M}{0}$
  transforms
  \begin{equation*}
    \compressedmatrices
    M =
    \begin{bmatrix}
      m_{11} & \times & \times \\
      0 & M_{22} & \times \\
      0 & 0 & 0
    \end{bmatrix}
    \quad\text{into}\quad
    \widetilde M =
    \begin{bmatrix}
      \widetilde m_{11} & 0 & 0 \\
      \times & \widetilde M_{22} & \times \\
      \times & 0 & 0
    \end{bmatrix},
  \end{equation*}
  where $m_{11} \neq 0$ and $\widetilde m_{11} \neq 0$, and $M_{22}$ and
  $\widetilde M_{22}$ are nonsingular and upper triangular and not to be
  confused with $M_{ij}$.  Now suppose that we are about to
  annihilate $m_{1j}$ for some $j > 1$; then we have the following.
  \begin{enumerate}
    \item If $m_{11} \neq 0$ and $m_{jj} \neq 0$, then both remain
      nonzero after annihilating $m_{1j}$. This follows from
      Case~\ref{it:CtriBtri} of Lemma~\ref{thm:algIO}. Moreover, the
      rotations only transform the nonsingular diagonal block of $M$,
      which preserves the desired block structure.
    \item If $m_{11} \neq 0$ and $m_{jj} = 0$, then $m_{11}$ stays
      nonzero and we get $|V| = I$. This follows from
      Cases~\ref{it:CrowBrow}, \ref{it:CtriBrow}, and \ref{it:CrowBtri}
      of Lemma~\ref{thm:algIO}. Since $|V| = I$, the fact that $m_{jk} =
      0$ for every $k \ge j$ remains true after annihilating
      $m_{1j}$; that is, the transformations do not introduce nonzeros
      in row $j$.
    \item If $m_{11} = m_{jj} = 0$, then $M = 0$ and $M$ is still the
      zero matrix after annihilating $m_{1j}$.
  \end{enumerate}
  An induction argument shows that after sweeping the $i$th row of $M$,
  the unswept trailing submatrix starting at the $(i+1,i+1)$th element
  has a block structure similar to $\iter{M}{0}$. Moreover, since
  $\iter{M}{1} = \iter{M}{0.5}{}^T$ we see that $\iter{M}{1}$ has
  the same block structure as $\iter{M}{0}$, and that the same is true
  for the structure of $\iter{M}{2}$.

  Next we will prove that after the first cycle
  \begin{equation*}
    \compressedmatrices
    \iter{B}{0.5} =
    \begin{bmatrix}
      \iter{B_{11}}{0.5} \\
      \iter{B_{21}}{0.5} & 0 \\
      \iter{B_{31}}{0.5} & 0 & 0 \\
      \iter{B_{41}}{0.5} & \iter{B_{42}}{0.5} &
        \iter{B_{43}}{0.5} & \iter{B_{44}}{0.5}
    \end{bmatrix},
  \end{equation*}
  where $\iter{B_{11}}{0.5}$ and $\iter{B_{44}}{0.5}$ are nonsingular.
  That is, if $\iter{b_{ii}}{0.5} = \iter{b_{kk}}{0.5} = 0$, then
  $\iter{b_{jj}}{0.5} = \iter{b_{ji}}{0.5} = 0$ for all $i \le j \le k$.
  Now, let us drop the superscript indices as we consider the row-cyclic
  sweeps that transform the input $\iter{A}{0}$, $\iter{B}{0}$,
  $\iter{C}{0}$, and their corresponding $\iter{M}{0}$ to the output
  $\iter{A}{0.5}$, $\iter{B}{0.5}$, $\iter{C}{0.5}$, and their
  corresponding $\iter{M}{0.5}$. Furthermore, suppose that we have swept
  $i-1$ rows, that $i$ is such that $m_{ii} = 0$, and that we are about
  to annihilate $m_{ij}$ for some $j > i$.
  \begin{enumerate}
    \item As a result of the structure of $M$, we have that $M_{ij} = 0$
      and we cannot have Cases~\ref{it:CtriBtri}, \ref{it:CrowBrow},
      \ref{it:CtriBrow}, \ref{it:CrowBtri}, \ref{it:CcolBcol},
      \ref{it:CcolBtri}, and \ref{it:CtriBcol} of Lemma~\ref{thm:algIO}.
      Moreover, Case~\ref{it:CrowBcol} can only occur with $m_{ij} = 0$.
    \item If $b_{ii} = 0$ or becomes zero after annihilating $m_{ij}$,
      then $b_{ii} = b_{ji} = 0$.  This follows from
      Cases~\ref{it:CcolBrow}, \ref{it:CrowBzero}, \ref{it:CtriBzero},
      \ref{it:CzeroBcol}, and~\ref{it:CrowBcol} of
      Lemma~\ref{thm:algIO}.
    \item If $b_{jj} \neq 0$, then $b_{jj}$ stays nonzero; that is,
      $b_{jj}$ stays nonzero at least until the $j$th row sweep. This
      follows from Cases~\ref{it:CzeroBcol} and~\ref{it:CzeroBtri} of
      Lemma~\ref{thm:algIO}.
    \item Suppose that $b_{ii} \neq 0$ at the start of the $i$th row
      sweep, and that $j$ is the first $j > i$ such that $b_{jj} = 0$.
      Then it follows from Case~\ref{it:CcolBrow} of
      Lemma~\ref{thm:algIO} that $b_{ii}$ becomes zero and $b_{jj}$
      nonzero after annihilating $m_{ij}$.  When this happens, $b_{ii}$
      remains zero for the rest of the row sweep, and thus for the rest
      of the cycle, and $b_{jj}$ remains nonzero at least until the
      $j$th row sweep. We have two possibilities before we annihilate
      $m_{ij}$ that we must consider.  Either $b_{i_0 i_0} \neq 0$ for
      every $1 \le i_0 < i$, in which case we are done since $b_{ii}$
      becomes the first zero on the diagonal of $B$, or there exists
      some $1 \le i_0 < i$ such that $b_{i_0 i_0} = 0$.  In the latter
      case, $b_{i_0 i_0}$ must have been zero at the start of the
      current cycle, or must have become zero before annihilating
      $b_{i_0 j}$.  This follows from the previous two points, which
      imply that the algorithm would otherwise have made $b_{jj}$
      nonzero when annihilating $b_{i_0 j}$ during the $i_0$th row
      sweep, and that $b_{jj}$ would have stayed nonzero at least until
      the $j$th row sweep.  Hence, we can conclude that $b_{ji_0} = 0$
      before annihilating $m_{ij}$, and since Case~\ref{it:CcolBrow}
      gives $|P| = |J|$, that $b_{ii_0} = 0$ after annihilating
      $m_{ij}$.
    \item As a result of the previous point, if $b_{ii} \neq 0$ after
      sweeping the $i$th row, then $b_{jj} \neq 0$ for all $j > i$.
  \end{enumerate}

  For the second cycle we need to prove that
  \begin{equation*}
    \compressedmatrices
    \iter{C}{1.5} =
    \begin{bmatrix}
      \iter{C_{11}}{1.5} \\
      \iter{C_{21}}{1.5} & \iter{C_{22}}{1.5} \\
      \iter{C_{31}}{1.5} & \iter{C_{32}}{1.5} & 0 \\
      \iter{C_{41}}{1.5} & \iter{C_{42}}{1.5} & 0 & 0
    \end{bmatrix},
  \end{equation*}
  where $\iter{C_{11}}{1.5}$ and $\iter{C_{22}}{1.5}$ are nonsingular.
  That is, if $\iter{c_{ii}}{1.5} = 0$, then $\iter{c_{jj}}{1.5} =
  \iter{c_{ji}}{1.5} = 0$ for all $i \le j \le l$.  Due to the previous
  sweeps we may assume that if $\iter{c_{ii}}{1} = \iter{c_{kk}}{1} =
  0$, then $\iter{c_{jj}}{1} = \iter{c_{ij}}{1} = 0$ for all $i \le j
  \le k$. Now, let us again drop the superscript indices as we consider
  the matrices during the row-cyclic sweeps, and suppose that we have
  swept $i-1$ rows, that $i$ is such that $m_{ii} = 0$, and that we are
  about to annihilate $m_{ij}$ for some $j > i$. 
  \begin{enumerate}
    \item As in the previous cycle, we have that $M_{ij} = 0$ and we
      cannot have Cases~\ref{it:CtriBtri}, \ref{it:CrowBrow},
      \ref{it:CtriBrow}, \ref{it:CrowBtri}, \ref{it:CcolBcol},
      \ref{it:CcolBtri}, and \ref{it:CtriBcol} of Lemma~\ref{thm:algIO}.
      Moreover, Case~\ref{it:CrowBcol} can only occur with $m_{ij} = 0$.
    \item If $c_{ii} \neq 0$ or becomes nonzero when sweeping row $i$,
      then it stays nonzero for the rest of the sweep.  This follows
      from Cases~\ref{it:CrowBzero}, \ref{it:CtriBzero},
      and~\ref{it:CrowBcol} of Lemma~\ref{thm:algIO}.
    \item If $c_{jj} = 0$ before annihilating $m_{ij}$, then it stays
      zero after. Furthermore, no nonzeros are introduced in the zero
      blocks of $C$.  This follows from Cases~\ref{it:CcolBrow},
      \ref{it:CrowBzero}, \ref{it:CzeroBcol}, \ref{it:CzeroBtri},
      and~\ref{it:CrowBcol} of Lemma~\ref{thm:algIO}, while noting that
      we get $|V| = I$ in each case.
    \item If $c_{ii} = 0$, then it becomes nonzero for the smallest
      integer $k > i$ such that $c_{ik} \neq 0$ or $c_{kk} \neq 0$. This
      follows from Case~\ref{it:CcolBrow} of Lemma~\ref{thm:algIO}.
      Since it holds that $c_{jj} = c_{ij} = 0$ for all $i \le j < k$
      before annihilating $m_{ik}$, and because we get $|Q| = |J|$, we
      must have $c_{ii} \neq 0$ and $c_{jj} = c_{ij} = 0$ for all $i < j
      \le k$ afterwards (note that the inequality symbols are swapped).
    \item As a result of the previous point, if $c_{ii} = 0$ after
      sweeping the $i$th row, then $c_{ij} = c_{jj} = 0$ for all $j >
      i$ and the remaining row sweeps do not introduce nonzeros in
      column $i$ below $c_{ii}$.
      \qedhere
  \end{enumerate}
\end{proof}


The block structure of $A = S \Sigma_\alpha T$, $B = S \Sigma_\beta$,
and $C = \Sigma_\gamma T$, for some upper-triangular $S$ and $T$ and
block-diagonal $\Sigma_\alpha = \diag(D_\alpha, I, I, I)$, $\Sigma_\beta
= \diag(D_\beta, I, 0, 0)$, and $\Sigma_\gamma = \diag(D_\gamma, 0, 0,
I)$ is
\begin{equation}\label{eq:betterNzStructure}
  \compressedmatrices
  \begin{bmatrix}
    A_{11} & A_{12} & A_{13} & A_{14} \\
    & A_{22} & A_{23} & A_{24} \\
    && A_{33} & A_{34} \\
    &&& A_{44} \\
  \end{bmatrix},
  \qquad
  \begin{bmatrix}
    B_{11} & B_{12} & 0 & 0 \\
    & B_{22} & 0 & 0 \\
    && 0 & 0 \\
    &&& 0
  \end{bmatrix},
  \quad\text{and}\quad
  \begin{bmatrix}
    C_{11} & C_{12} & C_{13} & C_{14} \\
    & 0 & 0 & 0 \\
    && 0 & 0 \\
    &&& C_{44}
  \end{bmatrix},
\end{equation}
where the nonzero diagonal blocks are nonsingular. Hence, if we want to
compute the factors $S$ and $T$, then we must extend the Kogbetliantz
phase to turn \eqref{eq:nzStructureUp} into
\eqref{eq:betterNzStructure}. The constructive proof of the proposition
below shows how we can do so.


\begin{proposition}\label{thm:betterNzStructure}
  Let $A$, $B$, and $C$ be structured as in \eqref{eq:nzStructureUp},
  and suppose that $CA^{-1}B$ equals $\minimat{M_{11}}{0}{0}{0}$ for
  some nonsingular $M_{11}$.  Then there exists orthonormal matrices $U$
  and $V$ such that
  \begin{equation*}
    \compressedmatrices
    BU =
    \begin{bmatrix}
      B_{11} & B_{12} B_{22}^{-1} R_B & 0 & 0 \\
      & R_B & 0 & 0 \\
      && 0 & 0 \\
      &&& 0
    \end{bmatrix}
    \quad\text{and}\quad
    V^T\!C =
    \begin{bmatrix}
      C_{11} & C_{12} & C_{13} & C_{14} \\
      & 0 & 0 & 0 \\
      && 0 & 0 \\
      &&& R_C
    \end{bmatrix}.
  \end{equation*}
  where $R_B$ and $R_C$ are nonsingular and upper triangular.
\end{proposition}
\begin{proof}
  Suppose $A = I$, then
  \begin{equation*}
    \compressedmatrices
    CA^{-1}B =
    \begin{bmatrix}
      C_{11} B_{11} &
      C_{11} B_{12} + C_{12} B_{22} &
      C_{11} B_{13} + C_{12} B_{23} &
      C_{11} B_{14} + C_{12} B_{24} \\
      & 0 & 0 & 0 \\
      && 0 & 0 \\
      &&& 0
    \end{bmatrix}
  \end{equation*}
  is block diagonal by the assumption on the structure of $CA^{-1}B$.
  It follows that $C_{12} = -C_{11} B_{12} B_{22}^{-1}$ and that $B_{1j}
  = -C_{11}^{-1} C_{12} B_{2j} = B_{12} B_{22}^{-1} B_{2j}$ for $j = 2$,
  $3$, and $4$. In other words,
  \begin{equation*}
    \compressedmatrices
    B =
    \begin{bmatrix}
      B_{11} & B_{12} B_{22}^{-1} B_{22} & B_{12} B_{22}^{-1} B_{23} &
      B_{12} B_{22}^{-1} B_{24} \\
      & B_{22} & B_{23} & B_{24} \\
      && 0 & 0 \\
      &&& 0
    \end{bmatrix}.
  \end{equation*}
  Hence, if $\widetilde U$ is such that $[B_{22}\; B_{23}\; B_{24}]
  \widetilde U = [R_B\; 0\; 0]$, then $U = \minimat{I}{0}{0}{\widetilde
  U}$ is the desired $U$.
  When $A \neq I$, the product $CA^{-1}$ has the same block structure as
  $C$ and the proof is similar. For $C$ we can use a QR decomposition to
  compute a $\widetilde V$ such that $\widetilde V^T [C_{24};\;
  C_{34};\; C_{44}]$ has the form $[0;\; 0;\; R_C]$, so that $V =
  \minimat{I}{0}{0}{\widetilde V}$ is the sought after $V$.
\end{proof}


The assumption in Proposition~\ref{thm:nzStructure} that $M =
\minimat{M_{11}}{M_{12}}{0}{0}$ for some nonsingular $M_{11}$ is not
necessarily satisfied directly after the preprocessing phase from
Section~\ref{sec:pre}. If $M$ does have this form, then $M_{12}$
converges to zero if the implicit Kogbetliantz iteration converges, so
that Proposition~\ref{thm:betterNzStructure} applies.  We suspect that a
finite number of cycles from Algorithm~\ref{alg:kog} with
Algorithm~\ref{alg:rsvd} for the $2\times 2$ RSVDs will bring $M$ into
with the desired form; however, we could not come up with a proof yet.
The reason for this suspicion is that Algorithm~\ref{alg:rsvd} computes
rotations that move nonzero entries of $M_{ij}$ to the upper-left corner
if $M_{ij}$ is singular.

In any case, we can ensure that $M$ has the desired structure with the
transformations that follow; though, this approach is only of
theoretical interest when we want the factors $S$ and $T$, and requires
(at least) two more and unwanted rank decisions in floating-point
arithmetic.  We start with compressing $B$ by comping $\iter{P}{1}$ and
$\iter{U}{1}$ such that $\iter{B}{2} = P^T\!BU =
\minimat{\iter{B_{11}}{2}}{0}{0}{0}$, where $\iter{B_{11}}{2}$ is
nonsingular. Next, we compute $\iter{Q}{1}$ such that $\iter{A}{2} =
\iter{P}{1}{}^T \!A \iter{Q}{1}$
is upper triangular, $\iter{V}{1}$ such that
$\iter{C}{2} = \iter{V}{1}{}^T\! C \iter{Q}{1}$
is upper triangular, and partition both matrices into blocks with block
sizes matching the blocks of $\iter{B}{2}$.  Then, we compress
$\iter{C_{11}}{2}$ by computing $\iter{V_{11}}{2}$ and
$\iter{Q_{11}}{2}$ such that $\iter{V_{11}}{2}{}^T\! \iter{C_{11}}{2}
\iter{Q_{11}}{2} = \minimat{\iter{C_{11}}{3}}{0}{0}{0}$.  Finally, we
compute $\iter{P_{11}}{2}$ such that $\iter{P_{11}}{2}{}^T\!
\iter{A_{11}}{2} \iter{Q_{11}}{2}$ is upper triangular, and compute
$\iter{U_{11}}{2}$ such that $\iter{P_{11}}{2}{}^T\! \iter{B_{11}}{2}
\iter{U_{11}}{2}$ is upper triangular.  We can now partition the
resulting $\iter{A}{3}$, $\iter{B}{3}$, and $\iter{C}{3}$ as
\begin{equation*}
  \compressedmatrices
  \begin{bmatrix}
    \iter{A_{11}}{3} & \iter{A_{12}}{3} & \iter{A_{13}}{3} \\
                     & \iter{A_{22}}{3} & \iter{A_{23}}{3} \\
                     &                  & \iter{A_{33}}{3}
  \end{bmatrix},
  \qquad
  \begin{bmatrix}
    \iter{B_{11}}{3} & \iter{B_{12}}{3} & 0 \\
                     & \iter{B_{22}}{3} & 0 \\
                     &                  & 0
  \end{bmatrix},
  \quad\text{and}\quad
  \begin{bmatrix}
    \iter{C_{11}}{3} & 0 & \iter{C_{13}}{3} \\
           &           0 & \iter{C_{23}}{3} \\
           &             & \iter{C_{33}}{3}
  \end{bmatrix},
\end{equation*}
respectively, from which we can see that $\iter{M}{3}$ has the desired
structure. If desired, we can even get the structure from
\eqref{eq:betterNzStructure} without the Kogbetliantz iteration by
computing
\begin{equation*}
  \begin{bmatrix}
    \iter{V_{22}}{3} & \iter{V_{23}}{3} \\
    \iter{V_{32}}{3} & \iter{V_{33}}{3}
  \end{bmatrix}^T
  \begin{bmatrix}
    \iter{C_{23}}{3} \\
    \iter{C_{33}}{3}
  \end{bmatrix}
  \iter{Q_{33}}{3}
  =
  \begin{bmatrix}
    0 & 0 \\
    0 & 0 \\
    0 & \iter{C_{44}}{4}
  \end{bmatrix},
\end{equation*}
and by computing $\iter{P_{33}}{3}$ such that $\iter{P_{33}}{3}{}^T\!
\iter{A_{33}}{3} \iter{Q_{33}}{3}$
is upper triangular.



\subsection{The 2-by-2 RSVD in floating point
arithmetic}\label{sec:rsvdTwoFloat}

Thus far, we have only considered the $2\times 2$ RSVD in exact
arithmetic.  The goal of this section is to show that
Algorithm~\ref{alg:rsvd} computes a numerically stable result in
floating-point arithmetic under the assumptions of the standard model
from, e.g., Higham~\cite[Ch.~2]{High02} or the \textit{LAPACK Users'
Guide}~\cite[Sec.~4.1.1]{LAUG}.  That is, given two floating-point
numbers $a$ and $b$, and some operation $\circ \in \{ +, -, \cdot, /
\}$, we assume that $\fl(a \circ b) = (a \circ b) (1 + \epsilon)$, where
$|\epsilon| \le \bm\epsilon$ and $\bm\epsilon$ is the unit roundoff
($2^{-53}$ in case of IEEE~754 double precision arithmetic). We
additionally assume that taking the absolute value of a floating-point
number is exact, as well as multiplying by zero or $\pm 1$. We ignore
overflow, underflow, and higher-order terms, as usual, unless stated
otherwise.  For convenience, different occurrences of $\epsilon$ and
error matrices do not need to have the same value unless they have
subscript indices.  Another convention is that overlined quantities
denote the ``computed'' version of quantities; for example, if $c = a
\circ b$, then $\overline c = \fl(a \circ b)$.

To prove the main results from this section, we first need the bounds
from the following two lemmas. The first lemma bounds a sum of elements
from the product of two particular nonnegative matrices.  The second
lemma bounds the norms of the backward perturbations in the computed
product $\fl(C\adj(A)B)$, where $A$, $B$, and $C$ are upper-triangular
$2\times 2$ matrices.


\begin{lemma}\label{thm:sqrtthree}
  Given a $2\times 2$ upper-triangular matrix $R$ and an orthonormal
  matrix $Q$, let $Z = |Q^T|\,|R|$; then $z_{11} + z_{12} \le \sqrt{3}
  \|R\|$.  Likewise, if $Z = |R| |Q|$, then $z_{12} + z_{22} \le
  \sqrt{3} \|R\|$.
\end{lemma}
\begin{proof}
  For the first result, we have for some $\alpha$ and $\beta$ satisfying
  $\alpha^2 + \beta^2 = 1$ that
  \begin{equation*}
    \compressedmatrices
    Z
    =
    \begin{bmatrix}
      |\alpha| & |\beta| \\ |\beta| & |\alpha|
    \end{bmatrix}
    \begin{bmatrix}
      |r_{11}| & |r_{12}| \\ & |r_{22}|
    \end{bmatrix}
    =
    \begin{bmatrix}
      |\alpha| |r_{11}| & |\alpha||r_{12}| + |\beta| |r_{12}| \\
      |\beta| |r_{11}| & |\beta||r_{12}| + |\alpha| |r_{12}|
    \end{bmatrix}.
  \end{equation*}
  It follows that
  \begin{multline*}
    z_{11} + z_{12} = |\alpha| (|r_{11}| + |r_{12}|) + \beta |r_{12}| \\
    \le |\alpha| \|R\|_\infty + |\beta| \|R\|_2
    \le (|\alpha| \sqrt{2} + \sqrt{1 - \alpha^2}) \|R\|_2
    \le \sqrt{3} \|R\|_2,
  \end{multline*}
  where we used the fact that the bound reaches its maximum for $\alpha
  = \pm\sqrt{2/3}$. The proof of the second result is similar.
\end{proof}



\begin{lemma}\label{thm:bwdM}
  Suppose $\overline M = \fl(C\adj(A)B)$ is computed as
  \begin{equation*}
    \compressedmatrices
    \overline M =
    \begin{bmatrix}
      \fl(\fl(c_{11} a_{22}) b_{11}) & m_{12} \\
      & \fl(c_{22} \fl(a_{11} b_{22}))
    \end{bmatrix},
  \end{equation*}
  where
    $m_{12} = \fl(\fl(\fl(\fl(c_{11} a_{22}) b_{12})
    {} + \fl(c_{12} \fl(a_{11} b_{22})))
    {} - \fl(\fl(c_{11} a_{12}) b_{22}))$.
  Then there exist small relative perturbations $\delta A_0$, $\delta
  B_0$, and $\delta C_0$ of $A$, $B$, and $C$, respectively, such that
  \begin{equation}\label{eq:rsvdTwoPer}
    \overline M = (C + \delta C_0) \adj(A + \delta A_0) (B + \delta B_0).
  \end{equation}
  Specifically, $\delta A_0$, $\delta B_0$, and $\delta C_0$ satisfy
  $\|\delta A_0\| \le 3.5 \bm\epsilon \|A\|$, $\|\delta B_0\| \le 3
  \bm\epsilon \|B\|$, and $\|\delta C_0\| \le 3 \bm\epsilon \|C\|$.
\end{lemma}
\begin{proof}
  Ignoring second order terms, we have that
  \begin{align*}
    \compressedmatrices
    \overline M &=
    \begin{bmatrix}
      (c_{11} a_{22}) b_{11} (1 + \epsilon_1 + \epsilon_2)
      & \overline m_{12} \\
      & c_{22} (a_{11} b_{22}) (1 + \epsilon_4 + \epsilon_5)
    \end{bmatrix}, \\
    \intertext{where}
    \overline m_{12} &= (c_{11} a_{22}) b_{12} (1 + \epsilon_1 + \epsilon_3 + \epsilon_9 + \epsilon_{10})
    \\
    &\qquad\qquad
    {} + c_{12} (a_{11} b_{22}) (1 + \epsilon_5 + \epsilon_6 + \epsilon_9 + \epsilon_{10})
    {} - c_{11} a_{12} b_{22} (1 + \epsilon_7 + \epsilon_8 + \epsilon_{10}),
  \end{align*}
  and $|\epsilon_i| \le \bm\epsilon$ for $i = 1$, \dots, 10. We can get
  the same $\overline M$ in exact arithmetic with the following relative
  perturbations:
  \begin{align*}
    \delta a_{11} / a_{11} &= \epsilon_5, &
    \delta a_{12} / a_{12} &= (\epsilon_7 + \epsilon_8 + \epsilon_{10}), &
    \delta a_{22} / a_{22} &= \epsilon_1,
    \\
    \delta b_{11} / b_{11} &= \epsilon_2, &
    \delta b_{12} / b_{12} &= (\epsilon_3 + \epsilon_9 + \epsilon_{10}), &
    \delta b_{22} / b_{22} &= 0,
    \\
    \delta c_{11} / c_{11} &= 0, &
    \delta c_{12} / c_{12} &= (\epsilon_6 + \epsilon_9 + \epsilon_{10}), &
    \delta c_{22} / c_{22} &= \epsilon_4,
  \end{align*}
  and $\delta a_{21} = \delta b_{21} = \delta c_{21} = 0$,
  proving~\eqref{eq:rsvdTwoPer}.  Using the equivalence of norms and the
  definition of the Frobenius norm, we get the bound:
  \begin{equation*}
    \|\delta A\|_2
    \le \|\delta A\|_F
    \le \bm\epsilon \sqrt{11} \max \{ |a_{11}|, |a_{12}|, |a_{22}| \}
    \le \bm\epsilon \sqrt{11} \|A\|_2
    < 3.5 \bm\epsilon \|A\|_2.
  \end{equation*}
  The perturbations $\delta B$ and $\delta C$ are of rank one and
  satisfy
  \begin{equation*}
    \|\delta B\| \le 3 \bm\epsilon
      (b_{11}^2 + b_{12}^2)^{1/2} \le 3 \bm\epsilon \|B\|_2
    \quad\text{and}\quad
    \|\delta C\| \le 3 \bm\epsilon
      (c_{12}^2 + c_{22}^2)^{1/2} \le 3 \bm\epsilon \|C\|_2,
  \end{equation*}
  which concludes the proof.
\end{proof}



\begin{remark}
  We can compute the product $\fl(C\adj(A)B)$ and the perturbations
  $\delta A$, $\delta B$, and $\delta C$ in different ways. Furthermore,
  the bounds in the above lemma are not the tightest possible.  Instead,
  the above perturbations and their bounds are such that we can invoke
  the lemma for the transposed and permuted triplet $(\Pi_r A^T \Pi_c,
  \Pi_r C^T \Pi_c, \Pi_r B^T \Pi_c)$ from the end of
  Section~\ref{sec:pre}, rather than for the original triplet $(A,B,C)$,
  without getting qualitative differences in the perturbations of $A$,
  $B$, and $C$.
\end{remark}


We are now ready for the main result of this section: the numerical
stability of Algorithm~\ref{alg:rsvd} in floating-point arithmetic.


\begin{theorem}\label{thm:stablersvd}
  Suppose that $\overline A'$, $\overline B'$, $\overline C'$,
  $\overline H$, $\overline K$, $\overline P$, $\overline Q$, $\overline
  U$, and $\overline V$ are computed by Algorithm~\ref{alg:rsvd} in
  floating-point arithmetic, with $\overline M$ computed as in
  Lemma~\ref{thm:bwdM}. Furthermore, define $\overline H'$ and
  $\overline K'$ as $\fl(\overline Q^T \overline H)$ and $\fl(\overline
  K \overline P)$, respectively, with their $(1,2)$ elements zeroed.
  Then the following assertions are true.
  \begin{enumerate}
    \item The matrices $\overline A'$, $\overline B'$, and $\overline
      C'$ are lower triangular.
    \item The product $\overline V^T\! \overline M \overline U$ is
      within $132 \bm\epsilon \|\overline M\|$ of being diagonal.
    \item The rows of $\overline C'$ and $\adj(\overline H')$ are within
      $86 \bm\epsilon \|C\|$ and $93.5 \bm\epsilon \|A\| \|B\|$,
      respectively, of being parallel.  Likewise, the columns of
      $\adj(\overline K')$ and $\overline B'$ are within $93.5
      \bm\epsilon \|A\| \|C\|$ and $86 \bm\epsilon \|B\|$, respectively,
      of being parallel.
    \item The matrices $\overline B'$, $\overline C'$, $\overline H'$,
      and $\overline K'$ are computed stably in the following sense.
      There exist $\delta B$, $\delta C$, $\delta H$, and $\delta K$,
      and orthonormal matrices $\mathcal P$, $\mathcal Q$, $U$, and $V$,
      such that $U^T \overline M V$ is an exact (unnormalized) SVD of
      $\overline M$, and
      \begin{equation*}
        \begin{split}
          \overline B'
          &= \mathcal P^T (B + \delta B) U,
          \aligncr
          \overline H'
          &= \mathcal Q^T (\adj(A)B + \delta H) U,
          \\
          \overline C'
          &= V^T (C + \delta C) \mathcal Q,
          \aligncr
          \overline K'
          &= V^T (C \adj(A) + \delta K) \mathcal P,
        \end{split}
      \end{equation*}
      where $\|\delta B\| \le 493 \bm\epsilon \|B\|$, $\|\delta C\| \le
      493 \bm\epsilon \|C\|$, $\|\delta H\| \le 486 \bm\epsilon \|C\|
      \|A\|$, and $\|\delta K\| \le 486 \bm\epsilon \|A\| \|B\|$.
  \end{enumerate}
\end{theorem}
\begin{proof}
  The proofs of first three assertions of the theorem follow the proof
  of Bai and Demmel for the QSVD~\cite[Thm.~3.1]{BD93},
  \textit{mutatis mutandis}. The proof of the fourth assertion deviates
  in the choice of the $\eta$s defined below, which is a difference that
  will be useful for later propositions and bounds. Due to this
  similarity, we also use the following facts from Bai and Demmel's
  proof.

  \begin{description}
    \item[Fact 1] The computed $\overline U$ and $\overline V$ from
      \textsf{xLASV2} satisfy $\overline U = U + \delta U$, $\overline V
      = V + \delta V$, where $V^T \overline M U$ is an exact
      (unnormalized) SVD of $\overline M$, and $\delta U$ and $\delta V$
      are small componentwise relative perturbations of $U$ and $V$,
      respectively, bounded by $46.5 \bm\epsilon$ in each component.
      This also implies $\|\delta U\|\le \sqrt{2} \cdot 46.5 \bm\epsilon
      < 66 \bm\epsilon$ and $\|\delta V\| < 66 \bm\epsilon$.
    \item[Fact 2] Using simple geometry, one can show that changing $f$
      to $f + \delta f$ and $g$ to $g + \delta g$ changes $c = f /
      \sqrt{f^2 + g^2}$ and $s = g / \sqrt{f^2 + g^2}$ to $c + \delta c$
      and $s + \delta s$, respectively, where $\sqrt{\delta c^2 + \delta
      s^2} \le 2((\delta f^2 + \delta g^2) / (f^2 + g^2))^{1/2}$.
    \item[Fact 3] Subroutine \textsf{xLARTG} computes $c = f / \sqrt{f^2
      + g^2}$ and $s = g / \sqrt{f^2 + g^2}$ with relative errors
      bounded by $6 \bm\epsilon$. This means that the $2\times 2$ matrix
      $\rot(c,s)$ has an error bounded in norm by $\sqrt{2} \cdot 6
      \bm\epsilon < 9 \bm\epsilon$.
    \item[Fact 4] If $X$ and $Y$ are $2\times 2$ matrices, then
      $\|\fl(XY) - XY\| \le 4 \bm\epsilon \|X\| \|Y\|$.
  \end{description}

  To prove the assertions of the theorem, first suppose that $c_{11} =
  0$ and $b_{22} = 0$. Then the first assertion holds by construction,
  and the second assertion follows from the nonzero structure of the
  matrices. The third and fourth assertions hold since $\overline P =
  \overline Q = J$ are exact, and $\overline U$ and $\overline V$ are
  computed from $B$ and $C$ with high relative accuracy by Fact~3.
  Now assume for the rest of the proof that $c_{11} \neq 0$ or $b_{22}
  \neq 0$.

      The lower-triangularity of $\overline A'$, $\overline B'$, and
      $\overline C'$ hold by construction.
      The near diagonality of $\overline V^T\! \overline M \overline
      U$ holds by the high accuracy of $\overline U$ and $\overline V$.
      Specifically, it follows from Fact~1 that
      \begin{equation*}
        \overline V^T\! \overline M \overline U
        = (V + \delta V)^T \overline M (U + \delta U)
        \approx V^T \overline M U
        {} + \delta V^T \overline MU
        {} + V^T \overline M \delta U,
      \end{equation*}
      where $V^T \overline M U$ is an exact unnormalized SVD of
      $\overline M$, and
      \begin{equation*}
        \|\delta V^T \overline MU\| + \|V^T \overline M \delta U\|
        \le (66 + 66) \bm\epsilon \|\overline M\|
        = 132 \bm\epsilon \|\overline M\|.
      \end{equation*}

      We prove the third assertion only for $\overline C'$ and
      $\adj(\overline H')$, as the and the proof for $\adj(\overline
      K')$ and $\overline B'$ is similar. We also only have to consider
      the bottom rows, since the explicitly zeroed $(1,2)$ entries make
      the top rows parallel by construction. Now, the bottom rows of
      $\overline C'$ and $\adj(\overline H')$ are identical to the
      bottom rows of $\fl(\overline G \overline Q)$ and
      $\adj(\fl(\overline Q^T \overline H))$, respectively, and the
      bottom rows of $V^T (C + \delta C_0) \mathcal Q$ and $\adj( U^T (B
      + \delta B_0) \adj(A + \delta A_0) \mathcal Q )$ are parallel by
      construction for any orthonormal matrix $\mathcal Q$.  Hence, it
      suffices to bound the distance between the former two pairs of
      matrices for a suitable choice of $\mathcal Q$, which we can do as
      follows.
      From Lemma~\ref{thm:bwdM} and Fact~4 it follows that for some
      error term $F_1$ with $\|F_1\| \le 4 \bm\epsilon \|C\|$, we have
      that
      \begin{equation*}
          \overline G
          = \fl(\overline V^T\! C)
          = \overline V^T\! C + F_1
          = V^T(C + \delta C_0) \underbrace{{} - V^T \delta C_0 + \delta
          V^T\! C + F_1}_{F_2};
      \end{equation*}
      thus, the error in $\overline G$ is bounded by
      \begin{equation*}
        \|F_2\| \le \|V^T \delta C_0\| + \|\delta V^T\! C\| + \|F_1\|
        \le (3 + 66 + 4) \bm\epsilon \|C\|
        = 73 \bm\epsilon \|C\|.
      \end{equation*}
      Using Fact~3, we see that for any $\overline Q = \mathcal Q +
      \delta \mathcal Q$ computed with \textsf{xLARTG}, we have that
      \begin{equation*}
        \begin{split}
          \fl(\overline G \overline Q)
          = \overline G \overline Q + F_2
          &= (V^T(C + \delta C_0) + F_2)(\mathcal Q + \delta \mathcal Q) + F_3 \\
          &= V^T(C + \delta C_0)\mathcal Q + \underbrace{F_2\mathcal Q +
          V^T\!C\delta \mathcal Q + F_3}_{F_4},
        \end{split}
      \end{equation*}
      with the error term bounded by
      \begin{equation*}
        \|F_4\| \le \|F_2\mathcal Q\| + \|V^T\!C\delta \mathcal Q\| + \|F_3\|
        \le (73 + 9 + 4) \bm\epsilon \|C\|
        = 86 \bm\epsilon \|C\|.
      \end{equation*}
      For some $F_5$ with $\|F_5\| \le 2\cdot 4 \bm\epsilon \|A\|
      \|B\|$, we have that
      \begin{equation*}
        \begin{split}
          \overline H
          &= \adj(A)B\overline U + F_5 \\
          &= \adj(A + \delta A_0 - \delta A_0)
            (B + \delta B_0 - \delta B_0)U + \adj(A)B \delta U + F_5 \\
          &= \adj(A + \delta A_0) (B + \delta B_0) U \underbrace{
            {} - \adj(\delta A_0) B U - \adj(A)\delta B_0 U
            {} + \adj(A)B \delta U + F_5}_{F_6},
        \end{split}
      \end{equation*}
      so that the error in $\overline H$ is bounded by
      \begin{equation*}
        \begin{split}
          \|F_6\| &\le \|\adj(\delta A_0)BU\| + \|\adj(A)\delta B_0 U\|
          {} + \|\adj(A)B \delta U\| + \|F_5\| \\
          &\le (3.5 + 3 + 66 + 2\cdot 4) \bm\epsilon \|A\| \|B\|
          = 80.5 \bm\epsilon \|A\| \|B\|.
        \end{split}
      \end{equation*}
      Hence, for some $F_6$ and $F_7$ with $\|F_6\| \le 80.5 \bm\epsilon
      \|A\| \|B\|$ and $\|F_7\| \le 4 \bm\epsilon \|A\| \|B\|$, and any
      rotation $\overline Q = \mathcal Q + \delta \mathcal Q$ computed
      with \textsf{xLARTG}, we have that
      \begin{equation}\label{eq:Herr}
        \begin{split}
          \fl(\overline Q^T \overline H)
          & = \overline Q^T \overline H + F_7 \\
          &= (\mathcal Q + \delta \mathcal Q)^T (\adj(A + \delta A_0) (B +
          \delta B_0) V + F_6) + F_7 \\
          &= \mathcal Q^T\adj(A + \delta A_0) (B + \delta B_0) V +
          \underbrace{\mathcal Q^TF_6 + \delta \mathcal Q^T\adj(A)BV +
          F_7}_{F_8},
        \end{split}
      \end{equation}
      so that the error term is bounded by $(80.5 + 9 + 4) \bm\epsilon
      \|A\|\|B\| = 93.5 \bm\epsilon \|A\| \|B\|$.

      For the fourth and final assertion, we only prove the bounds for
      $\|\delta C\|$ and $\|\delta H\|$, because bounding $\|\delta B\|$
      and $\|\delta K\|$ is similar.  The main challenge now is to
      quantify the effect of zeroing the $(1,2)$ entries at the end of
      the algorithm. Suppose first that $|\overline h_{12}| + |\overline
      h_{22}| = 0$, then the algorithm computes $Q$ from $\overline G$,
      and $\overline Q$ zeros the $(1,2)$ entry of $\overline G$ with
      high relative accuracy as a result of Fact~3.  Furthermore, in
      this case it holds for any $\overline Q$ that ${\fl(\overline Q^T
      \overline H)}_{12} = 0$. Otherwise, if $|\overline h_{12}| +
      |\overline h_{22}| \neq 0$ but $|\overline g_{11}| + |\overline
      g_{12}| = 0$, then the algorithm computes $\overline Q$ to
      accurately zero out the $(1,2)$ entry of $\overline H$, and
      ${\fl(\overline G \overline Q)}_{12} = 0$.  Now we may assume that
      $|\overline h_{12}| + |\overline h_{22}| \neq 0$ and $|\overline
      g_{11}| + |\overline g_{12}| \neq 0$ for the rest of the proof,
      and that $\overline \eta_g \leq \overline \eta_h$ so that the
      algorithm computes $Q$ from $\overline G$.  The proof is similar
      when $\overline \eta_g > \overline \eta_h$ and the algorithm
      computes $Q$ from $\overline H$, but leads to different bounds
      that we summarize at the end of the proof.

      It follows from Fact~3 that the algorithm computes $\overline Q$
      in such a way that the $(1,2)$ entry of $\overline G$ is zeroed
      with high relative precision. Bounding the effect of zeroing the
      $(1,2)$ entry of $\fl(\overline Q^T \overline H)$ to get
      $\overline H'$ is more involved.  Let $\overline Q = Q + \delta
      Q$, where $Q$ denotes the exact rotation obtained from $V^T(C +
      \delta C_0)$ in exact arithmetic (which can be bigger than just
      the error from \textsf{xLARTG} due to the errors in $\overline
      V$); then
      \begin{equation*}
        \begin{split}
          |{\fl(\overline Q^T \overline H)}_{12}|
          &= |\overline h_{12} (1 + 2\epsilon) (q_{11} + \delta q_{11})
          {} + \overline h_{22} (1 + 2\epsilon) (q_{21} + \delta q_{21})| \\
          &= |(h_{12} + 82.5 \epsilon \|A\| \|B\|) q_{11}
          {} + (1 + 2\epsilon) \overline h_{12} \delta q_{11} \\
          {} &+ (h_{22} + 82.5 \epsilon \|A\| \|B\|) q_{21}
          {} + (1 + 2\epsilon) \overline h_{22} \delta q_{21}| \\
          &\le (1 + 2\bm\epsilon) (|\overline h_{12}| |\delta q_{11}|
          {} + |\overline h_{22}| |\delta q_{21}|) + \sqrt{2} \cdot 82.5
          \bm\epsilon \|A\| \|B\|.
        \end{split}
      \end{equation*}
      Before proceeding, recall that $\overline{\widehat G} =
      \fl(|\overline V|^T |C|)$ and $\overline{\widehat H} =
      \fl(|\adj(A)| \fl(|B| |\overline U|))$, so that
      \begin{equation*}
        \bm\epsilon \overline \eta_g
        = \bm\epsilon (\overline{\widehat g}_{11} + \overline{\widehat g}_{12})
          / (|\overline g_{11}| + |\overline g_{12}|)
        \quad\text{and}\quad
        \bm\epsilon \overline \eta_h
        = \bm\epsilon (\overline{\widehat h}_{12} + \overline{\widehat h}_{22})
          / (|\overline h_{12}| + |\overline h_{22}|).
      \end{equation*}
      Furthermore, it can be verified that the entries of $\overline G$
      may have a perturbation of up to $51.5 \bm\epsilon
      \overline{\widehat g}_{ij}$, where $46.5 \bm\epsilon$ comes from
      the perturbations in $\delta U$, and $2 \bm\epsilon$ from the
      roundoff errors in the matrix-matrix multiplication, and
      $3 \bm\epsilon$ from $\delta C_0$.  Hence, using
      \begin{equation*}
        (51.5 \bm\epsilon \overline{\widehat g}_{11})^2
        {} + (51.5 \bm\epsilon \overline{\widehat g}_{12})^2
        \le (51.5 \bm\epsilon)^2 (|\overline{\widehat g}_{11}| +
        |\overline{\widehat g}_{12}|)^2
        = (51.5 \bm\epsilon \overline \eta_g)^2
        (|\overline{g}_{11}| + |\overline{g}_{12}|)^2
      \end{equation*}
      and Facts~2 and 3, we can bound $(|\delta q_{11}|^2 + |\delta
      q_{21}|^2)^{1/2}$ by
      \begin{equation*}
        9 \bm\epsilon + 2\cdot 51.5 \bm\epsilon \overline \eta_g
        \frac{|\overline{g}_{11}| + |\overline{g}_{12}|}{%
          \sqrt{\overline{g}_{11}^2 + \overline{g}_{12}^2}}
        \le 9 \bm\epsilon + 2\sqrt{2} \cdot 51.5 \bm\epsilon \overline \eta_g
        \le 155 \bm\epsilon
        \max\{ 1, \overline \eta_g \}.
      \end{equation*}
      If $\overline \eta_g \le 1$, then $|{\fl(\overline Q^T \overline
      H)}_{12}| \le \sqrt{2} (155 + 82.5) \bm\epsilon \|A\|\|B\| \le 336
      \bm\epsilon \|A\|\|B\|$; otherwise, we can use
      Lemma~\ref{thm:sqrtthree} to show that
      \begin{equation*}
        (|\overline h_{11}| + |\overline h_{12}|) \bm\epsilon
        = (\overline{\widehat h}_{12} + \overline{\widehat h}_{22})
        \bm\epsilon \overline \eta_h^{-1}
        \le \sqrt{3} \|A\| \|B\| \bm\epsilon \overline \eta_h^{-1},
      \end{equation*}
      which in turn implies the bound
      \begin{equation*}
        \begin{split}
          |\fl({(\overline Q^T \overline H)}_{12})|
          &\le (1 + 2\bm\epsilon) (|\overline h_{11}|
          {} + |\overline h_{12}|) 155 \bm\epsilon \overline \eta_g
          {} + 117 \bm\epsilon \|A\| \|B\| \\
          & \le (\sqrt{3} \cdot 155 \frac{\overline \eta_g}{\overline \eta_h} + 117)
            \bm\epsilon \|A\| \|B\|.
        \end{split}
      \end{equation*}
      Since $\overline \eta_g \le \overline \eta_h$ by assumption, it
      follows that $|{\fl(\overline Q^T \overline H)}_{12}| \le 386
      \bm\epsilon \|A\|\|B\|$. By writing the error term
      in~\eqref{eq:Herr} as $F_8$ and the explicit zeroing of the
      $(1,2)$ entry as $F_9 = -{\fl(\overline Q^T \overline H)}_{12}
      \vec e_1 \vec e_2^T$, everything can be put together to yield
      \begin{equation*}
        \begin{split}
          \overline H'
          &= \mathcal Q^T \adj(A + \delta A_0) (B + \delta B_0) U + F_8 + F_9 \\
          &= \mathcal Q^T (\adj(A)B + \adj(\delta A_0)B
          {} + \adj(A) \delta B_0 + QF_8U^T + QF_9U^T) U \\
          &= \mathcal Q^T ( \adj(A)B + \delta H) U,
        \end{split}
      \end{equation*}
      where $\|\delta H\| \le (3.5 + 3 + 93.5 + 386) \bm\epsilon \|A\|
      \|B\| = 486 \bm\epsilon \|A\| \|B\|$.

      The proof is similar when $\overline \eta_g > \overline \eta_h$
      and $\overline Q$ is computed from $\overline H$, except for the
      following differences. We get $\sqrt{2} \cdot (73 + 2) \leq 107$
      instead of $\sqrt{2} \cdot (80.5 + 2) \leq 117$, elements in
      $\overline H$ may be perturbed by up to $(46.5 + 3.5 + 3 + 2\cdot
      2) \bm\epsilon \overline{\widehat h}_{ij} = 57 \bm\epsilon
      \overline{\widehat h}_{ij}$, the quantity $(|\delta q_{11}|^2 +
      |\delta q_{21}|^2)^{1/2}$ is bounded by $9 \bm\epsilon + 2\sqrt{2}
      \cdot 57 \bm\epsilon \overline \eta_h \leq 171 \bm\epsilon \max \{
      1, \overline \eta_h \}$, the perturbation $\delta A_0$ is not part
      of $\delta C$, and $93.5 \bm\epsilon$ should be replaced by
      $\|F_4\| \le 86 \bm\epsilon$ in the final bound. Hence, the factor
      in the resulting bound is $(3 + 86 + 107 + \sqrt{3} \cdot 171)
      \bm\epsilon \le 493 \bm\epsilon$.
\end{proof}


Although the theorem above shows that Algorithm~\ref{alg:rsvd} has
favorable numerical properties, it lacks a bound on the backward error
of $\overline A'$. Moreover, we have to content ourselves with
$\overline H'$ and $\overline K'$ instead of $\adj(\overline A')
\overline B'$ and $\overline C' \adj(\overline A')$.  However, the proof
shows that we can bound the errors in $P$ and $Q$ in terms of the
$\eta$s, which in turn allows us to express the error in $\overline A'$
in terms of the $\eta$s. We can then try to ensure that the $\eta$s
remain small, so that the error in $\overline A'$ is small. These things
are the focus of the next section.




\section{The backward error of the computed
\texorpdfstring{$\mathbf{A'}$}{A}}\label{sec:bwdA}

Although the numerical results in Section~\ref{sec:num} suggest that the
relative magnitude of $\fl(\overline P^T\! A \overline Q)_{12}$ is
always small in practice, it is unclear if we can prove that $\|\delta
A\|$ is $\mathcal O(\bm\epsilon \|A\|)$ in the worst case. An
alternative is to bound the backward error of $\overline A'$ in terms of
the $\eta$s, and then to analyze the behavior of the $\eta$s. We can
simplify this analysis with the following two definitions.


\begin{definition}
  Define $\eta_g$ from Algorithm~\ref{alg:rsvd} as
  \begin{equation*}
    \eta_g = \begin{cases}
      (\widehat g_{11} + \widehat g_{12}) / (|g_{11}| + |g_{12}|) &
      \text{if $|g_{11}| + |g_{12}| \neq 0$} \\
      \infty & \text{if $|g_{11}| + |g_{12}| = 0$,}
    \end{cases}
  \end{equation*}
  and define the remaining $\eta$s and $\overline \eta$s analogously.
\end{definition}



\begin{definition}\label{def:etamax}
  Define $\eta_{\max}$ as $\eta_{\max} = \max \{ 1, \min \{ \eta_g,
  \eta_h \}, \min \{ \eta_k, \eta_l \} \}$ and define $\overline
  \eta_{\max}$ analogously.
\end{definition}


We will later see that $\eta_{\max}, \overline \eta_{\max} < \infty$,
which is important for two reasons. First, it allows us to simplify the
conditions in Algorithm~\ref{alg:rsvd} that determine whether to compute
$Q$ and $P$ from $G$ or $H$ and $L$ or $K$, respectively, by dropping
the zero checks and keeping just $\eta_g \le \eta_h$ and $\eta_l \le
\eta_k$.  Second, we can now bound the backward error of $\overline A'$
in terms of $\overline \eta_{\max}$ instead of having to consider
separate cases with separate $\overline \eta$s.


\begin{theorem}\label{thm:errA}
  Suppose $\overline P$ is obtained in a similar way as $\overline Q$,
  and $171 \bm\epsilon \overline \eta_{\max} \ll 1$; then there exists
  $\delta A$, $\mathcal P$, and $\mathcal Q$ such that $\overline A' =
  \mathcal P (A + \delta A) \mathcal Q$ and $\|\delta A\| \le (44.5 +
  342 \overline\eta_{\max}) \bm\epsilon \|A\|$.
\end{theorem}
\begin{proof}
  Since $\overline P$ is computed in a similar way as $\overline Q$, it
  follows from the proof of Theorem~\ref{thm:stablersvd} that $\overline
  P$ can be decomposed as both $\overline P = \mathcal P + \delta
  \mathcal P$ and $\overline P = P + \delta P$.  Here, $\mathcal P$ and
  $P$ are both exactly orthonormal matrices, and $\|\delta \mathcal P\|
  \le 9 \bm\epsilon$ is the error incurred by computing any $\overline
  P$ with \textsf{xLARTG} in floating-point arithmetic, and $\|\delta
  P\| \le 171 \bm\epsilon \eta_{\max}$ is the error incurred by
  computing the rotation from an approximation of $V^T (C + C_0) \adj(A
  + A_0)$ or $(B + B_0) U$ in floating-point arithmetic.  It follows
  that for some $F_1$ with $\|F_1\| \le 2 \cdot 4 \bm\epsilon$,
  \begin{equation*}
    \begin{split}
      \fl(\overline P^T\! A \overline Q)
      &= (\mathcal P + \delta \mathcal P)^T\! A (\mathcal Q + \delta \mathcal Q) + F_1 \\
      &= \mathcal P^T (A + \delta A_0) \mathcal Q \underbrace{
      {} - \mathcal P^T \delta A_0 \mathcal Q
      {} + \delta \mathcal P^T\! A \mathcal Q
      {} + \mathcal P^T\! A \delta \mathcal Q 
      {} + \delta \mathcal P^T\! A \delta \mathcal Q + F_1}_{F_2},
    \end{split}
  \end{equation*}
  where the error is bounded by
  \begin{equation*}
    \|F_2\| \le \|\mathcal P^T \delta A_0 \mathcal Q\|
    {} + \|\delta \mathcal P^T\! A \mathcal Q\|
    {} + \|\mathcal P^T\! A \delta \mathcal Q\| + \|F\|
    \le (3.5 + 9 + 9 + 2\cdot 4) \bm\epsilon \|A\|
    = 29.5 \bm\epsilon \|A\|.
  \end{equation*}
  Furthermore, for some $f_3$ with
  \begin{equation*}
    |f_3| \le 8 \bm\epsilon \|\vec e_1^T(P + \delta P)\| \|A\| \|(Q + \delta Q)\vec e_2\|
    \le 8 \bm\epsilon (1 + 171\bm\epsilon\overline\eta_{\max})^2 \|A\|
    \approx 8 \bm\epsilon \|A\|
  \end{equation*}
  we have that
  \begin{equation*}
    \begin{split}
      |\fl(\overline P^T\! A \overline Q)_{12}|
      &\le |(P + \delta P)^T A (Q + \delta Q)|_{12} + |f_3| \\
      &= |P^T (A + \delta A_0) Q - P^T \delta A_0 Q + \delta P^T\! A Q +
      P^T\! A \delta Q + \delta P^T\! A \delta Q|_{12} + |f_3| \\
      &\le |P^T \delta A_0 Q|_{12}
      {} + |\delta P^T\!  A Q|_{12}
      {} + |P^T\!A\delta Q|_{12}
      {} + |\delta P^T\! A \delta Q|_{12} + |f_{3}| \\
      &\le (3.5 \bm\epsilon + 2\cdot 171 \bm\epsilon \overline\eta_{\max}
      {} + 8 \bm\epsilon) \|A\|.
    \end{split}
  \end{equation*}
  Here, we used that $P^T(A + \delta A_0)Q = 0$, and that the assumption
  $171\bm\epsilon\overline\eta_{\max} \ll 1$ makes
  $\bm\epsilon^2\overline\eta_{\max}$ and
  $(\bm\epsilon\overline\eta_{\max})^2$ higher-order terms.  Now the
  explicit zeroing of the $(1,2)$ entry of $A'$ is the same as adding
  the error term $F_4 = -\fl(\overline P^T\! A \overline Q)_{12} \vec
  e_1 \vec e_2^T$, so that
  \begin{equation*}
    A' = \mathcal P^T (A + \delta A_0) \mathcal Q + F_2 + F_4
    = \mathcal P^T (A + \delta A_0 + \mathcal P (F_2 + F_4) \mathcal Q^T) \mathcal Q
    = \mathcal P^T (A + \delta A) \mathcal Q.
  \end{equation*}
  Thus, by combining the relevant error terms and their bounds, we get
  \begin{equation*}
    \|\delta A\|
    \le (3.5 + 29.5 + 11.5 + 342 \overline \eta_{\max}) \bm\epsilon \|A\|
    = (44.5 + 342 \overline \eta_{\max}) \bm\epsilon \|A\|,
  \end{equation*}
  which is the desired result.
\end{proof}


In essence, if $\overline \eta_{\max}$ is sufficiently small, then the
errors $\delta P$ and $\delta Q$ stay small, and
Algorithm~\ref{alg:rsvd} computes $\overline A'$ stably. Hence, the goal
is now to bound $\overline \eta_{\max}$. We start by showing that
$\overline \eta_{\max}$ is always finite, but before we can start with
the proof, we need the following properties of Algorithm~\ref{alg:rsvd}
and the routine \textsf{xLASV2}.


\begin{lemma}\label{thm:lasvtwo}
  Consider Algorithm~\ref{alg:rsvd} and assume the following: $c_{11}
  \neq 0$ or $b_{22} \neq 0$, the SVD of $\overline M$ is computed with
  \textsf{xLASV2} as given in Bai and Demmel~\cite[App.]{BD93}, and
  the columns of $\overline U$ and $\overline V$ are postmultiplied by
  $J$ if the relevant conditions in the algorithm are met. Then
  $\overline U = \overline V = I$ if $\overline M = 0$, $|\overline U| =
  |J|$ and $|\overline V| = I$ if $\overline m_{11} = \overline m_{22} =
  0$ and $\overline m_{12} \neq 0$, $|\overline V| \approx I$ if
  $b_{22} = 0$, and $|\overline U| \approx |J|$ if $c_{11} = 0$ but $C
  \neq 0$, where the zeros are exact even for the latter $\overline U$
  and $\overline V$.
\end{lemma}
\begin{proof}
  The desired results follow from the implementation and high relative
  accuracy of \textsf{xLASV2}, combined with the postmultiplication
  of $\overline U$ and $\overline V$ by $J$ when the conditions on
  Line~9 of Algorithm~\ref{alg:rsvd} are met.
\end{proof}


The preceding lemma implies that the $\overline U$ and $\overline V$
computed with Algorithm~\ref{alg:rsvd} in floating-point arithmetic
correspond to the exact $U$ and $V$ from Lemma~\ref{thm:algIO} with high
relative accuracy, at least for those cases of Lemma~\ref{thm:algIO}
that correspond to the assumption of the lemma above. With this result,
we can now prove the following proposition and corollary, which show
that $\overline \eta_{\max}$ is finite.


\begin{proposition}\label{thm:cneqZbneqZ}
  If $c_{11} \neq 0$ or $b_{22} \neq 0$, then $|\overline g_{11}| +
  |\overline g_{12}| = 0$ and $|\overline h_{12}| + |\overline h_{22}| =
  0$ cannot hold simultaneously. Likewise and under the same
  assumptions, neither $|\overline k_{11}| + |\overline k_{12}| = 0$ and
  $|\overline l_{12}| + |\overline l_{22}| = 0$, nor $|\overline g_{11}|
  + |\overline g_{12}| = 0$ and $|\overline l_{12}| + |\overline l_{22}|
  = 0$ can hold simultaneously.
\end{proposition}
\begin{proof}
  If $b_{22} = 0$, then by Lemma~\ref{thm:lasvtwo} $|\overline V|
  \approx I$, so that $g_{11} = \overline v_{11} c_{11} (1 + \epsilon)
  \neq 0$.  Conversely, if $C = 0$, then $\overline U = I$ and
  $\overline h_{22} = a_{11} b_{22} (1 + \epsilon) \neq 0$. If $c_{11} =
  0$ but $C \neq 0$, then
  \begin{equation*}
    \compressedmatrices
    \overline M = \begin{bmatrix}
      0 & c_{12} (a_{11} b_{22}) (1 + \epsilon_1) (1 + \epsilon_2) \\
      0 & c_{22} (a_{11} b_{22}) (1 + \epsilon_1) (1 + \epsilon_3) \\
    \end{bmatrix}
  \end{equation*}
  for some $|\epsilon_1|, |\epsilon_2|, |\epsilon_3| \le \bm\epsilon$.
  Using the entries of $\overline M$ and Facts~1 and~2 from the proof of
  Theorem~\ref{thm:stablersvd}, we see that $|\overline g_{12}|$ equals
  \begin{equation*}
    |v_{11} c_{12} (1 + 48.5\epsilon) + v_{21} c_{22} (1 + 48.5\epsilon)|
    = \left|
      \frac{c_{12}^2 (1 + 50.5\epsilon)}{(c_{12}^2 + c_{22}^2)^{-1/2}}
      + \frac{c_{22}^2 (1 + 50.5\epsilon)}{(c_{12}^2 + c_{22}^2)^{-1/2}}
    \right|,
  \end{equation*}
  which is within $50.5 \|C\| \bm\epsilon$ of $\|C\|$. Hence, we
  conclude that $|\overline g_{12}|$ is nonzero.

  If both $c_{11} \neq 0$ and $b_{22} \neq 0$
  but $|\overline g_{11}| + |\overline g_{12}| = 0$ and $|\overline
  h_{12}| + |\overline h_{22}| = 0$, then
  \begin{equation*}
    \begin{split}
      0
      &= |v_{11} c_{11} (1 + 47.5\epsilon)|
      {} + |v_{11} c_{12} (1 + 48.5\epsilon)
      {}     + v_{21} c_{22} (1 + 48.5\epsilon)|, \\
      0
      &= |u_{22} b_{22} a_{11} (1 + 48.5\epsilon)|
      \\ &\qquad
      {} + |u_{22} (a_{22} b_{12} (1 + 50.5\epsilon)
      {} - a_{12} b_{22} (1 + 50.5\epsilon))
      {} + u_{12} a_{22} b_{11} (1 + 50.5\epsilon)|.
      %
    \end{split}
  \end{equation*}
  The former implies that $v_{11} = 0$ and thus also that $c_{22} = 0$,
  and the latter implies that $u_{22} = 0$ and thus also that $b_{11} =
  0$.  Yet, by Lemma~\ref{thm:lasvtwo} we cannot simultaneously have
  $|\overline U| = |J|$ and $|\overline V| = |J|$ when $c_{22} = b_{11}
  = 0$ so that we have a contradiction.

  The proof for the second claim in the proposition is similar, and the
  third claim holds because $|\overline k_{11}| + |\overline k_{12}| =
  0$ if $|\overline g_{11}| + |\overline g_{12}| = 0$ and $|\overline
  h_{12}| + |\overline h_{22}| = 0$ if $|\overline l_{12}| + |\overline
  l_{22}| = 0$.
\end{proof}



\begin{corollary}\label{thm:finiteEta}
  Since neither $\overline \eta_g$ and $\overline \eta_h$, nor
  $\overline \eta_k$ and $\overline \eta_l$ are infinite simultaneously,
  $\overline \eta_{\max}$ is finite.
\end{corollary}


In exact arithmetic, we can prove an even stronger result, namely that
$\eta_g = \infty$ and $\eta_l = \infty$ if and only if $c_{11} = 0$ and
$b_{22} = 0$, respectively.  This is not the case in floating-point
arithmetic, and the $\overline \eta$s may be finite or infinite in
unexpected situations. For example, if $B =
\minimat{b_{11}}{b_{12}}{0}{0}$, then the exact $U$ should be such that
$L = BU = \minimat{l_{11}}{0}{0}{0}$ so that $\eta_l = \infty$; but the
computed $\overline U$ is typically such that $\overline L =
\fl(B\overline U) = \minimat{\overline l_{11}}{\mathcal O(\epsilon
\|B\|)}{0}{0}$ so that $\overline \eta_l < \infty$. This can be a
problem when, for example, $C = \minimat{1}{1}{0}{0}$ and $\adj(A) =
\minimat{\mu}{1 - \mu}{0}{-1}$ for some $0 < \mu \ll 1$, so that
$|\overline V| \approx I$ and $\overline \eta_k \approx (1 + |1 - \mu|)
/ |2\mu|$. Hence, if $\mu \to 0$ and we do not explicitly set $\overline
l_{12}$ to zero, then $\overline \eta_k$ can become larger than
$\overline \eta_l$ and Algorithm~\ref{alg:rsvd} will compute $P$ from
$BU$ rather than from $V^T\! C \adj(A)$.  The results from
Theorem~\ref{thm:stablersvd} still hold if this happens, but
$\overline\eta_{\max}$ will be large and we can no longer expect
$|{\fl(\overline P^T\!  A \overline Q)}_{12}|$ to be small. In this
example, the condition number of $A$ is of order $\mathcal O(\mu^{-1})$
too, and we will later see that $\kappa(A)$ plays an important role in
bounding $\eta_{\max}$.

The next proposition implies that $\overline \eta_g < \infty$ if $\eta_g
< \infty$ and that $\overline \eta_l < \infty$ if $\eta_l < \infty$,
which means that the discrepancy between the $\eta$s and the $\overline
\eta$s for $\infty$ exists only in one direction.  Furthermore, the
proposition makes it easier to compute the bounds that we want, because
we have formulae for the exact $\eta$s while the computed $\overline
\eta$s are perturbed by unknown roundoff errors.


\begin{proposition}\label{thm:etarelation}
  Suppose $\eta_g, \overline \eta_g < \infty$ and $\eta_h, \overline
  \eta_h < \infty$ are small enough; then we have the first-order
  approximations
  \begin{equation}\label{eq:etarelation}
    \overline \eta_g
    = \eta_g \frac{1 + 50.5\epsilon_1}{1 + 49.5\epsilon_2 \eta_g}
    \quad\text{and}\quad
    \overline \eta_h
    = \eta_h \frac{1 + 52.5 \epsilon_3}{1 + 51.5 \epsilon_4 \eta_h},
  \end{equation}
  respectively, where $|\epsilon_i| \le \bm\epsilon$ for $i = 1$, \dots
  $4$. A similar statement holds for $\eta_k, \overline \eta_k$ and
  $\eta_l, \overline \eta_l$.
\end{proposition}
\begin{proof}
  Since $\overline v_{11} = v_{11} (1 + 46.5\epsilon)$ and $\overline
  v_{21} = v_{21} (1 + 46.5\epsilon)$, it follows that
  \begin{equation*}
    \begin{split}
      \fl(\overline{\widehat g}_{11} + \overline{\widehat g}_{12})
      &= (|\fl(\overline v_{11} c_{11})|
      {}+ \fl(|\fl(\overline v_{11} c_{12})|
      {}+     |\fl(\overline v_{21} c_{22})|)) (1 + \epsilon) \\
      &= |v_{11} c_{11}| (1 + 48.5\epsilon)
      {}+ |v_{11} c_{12}| (1 + 49.5\epsilon)
      {}+ |v_{21} c_{22}| (1 + 49.5\epsilon) \\
      &= (\widehat g_{11} + \widehat g_{12})(1 + 49.5\epsilon)
    \end{split}
  \end{equation*}
  and
  \begin{equation*}
    \begin{split}
      \fl(|\overline g_{11}| + |\overline g_{12}|)
      &= (|\fl(\overline v_{11} c_{11})|
      {}+ |\fl(\fl(\overline v_{11} c_{12})
      {}+ \fl(\overline v_{21} c_{22}))|) (1 + \epsilon) \\
      &= |v_{11} c_{11} (1 + 48.5\epsilon)|
      {}+ |v_{11} c_{12} (1 + 49.5\epsilon)
      {}+  v_{21} c_{22} (1 + 49.5\epsilon)| \\
      &= |g_{11} + 48.5\epsilon \widehat g_{11}|
      {}+ |g_{12} + 49.5\epsilon \widehat g_{12}| \\
      &= (|g_{11}| + |g_{12}|)(1 + 49.5\epsilon \eta_g),
    \end{split}
  \end{equation*}
  so that
  \begin{equation*}
    \overline \eta_g = \fl\left( \frac{%
      (\widehat g_{11} + \widehat g_{12})(1 + 49.5\epsilon)}{%
      (|g_{11}| + |g_{12}|)(1 + 49.5\epsilon \eta_g)} \right)
    = \eta_g \frac{1 + 50.5\epsilon}{1 + 49.5\epsilon \eta_g}.
  \end{equation*}
  The derivation of the relation between $\overline \eta_h$ and $\eta_h$
  is analogous. The proof for $\eta_k, \overline \eta_k$ and $\eta_l,
  \overline \eta_l$ is similar.
\end{proof}



\begin{corollary}\label{thm:etarelation2}
  Under the same assumptions as in Proposition~\ref{thm:etarelation},
  solving \eqref{eq:etarelation} for $\eta_g$ and $\eta_h$ yields
  \begin{equation*}
    \eta_g = \overline \eta_g
      (1 + 50.5 \epsilon_1 - 49.5\epsilon_2 \overline \eta_g)^{-1}
    \quad\text{and}\quad
    \eta_h = \overline \eta_h
      (1 + 52.5 \epsilon_3 - 51.5\epsilon_4 \overline \eta_h)^{-1}.
  \end{equation*}
\end{corollary}


Proposition~\ref{thm:etarelation} shows that the computed
$\overline\eta$s approximate their exact counterparts if $\bm\epsilon
\eta_{\max} \ll 1$.  Although we generally do not know the exact $\eta$s
in practice, we still expect this result to hold if $\bm\epsilon
\overline \eta_{\max} \ll 1$.

Now that we know the relation between the $\eta$s and $\overline\eta$s,
we can use bounds for the former to inform us of the behavior of the
latter.  The next two propositions and the corollary show that bounding
the $\eta$s from below and in terms of each other is straightforward.


\begin{proposition}
  It holds that $\eta_g, \eta_h, \eta_k, \eta_l \ge 1$.
\end{proposition}
\begin{proof}
  Using the triangle inequality we see that
  \begin{equation*}
    \eta_g = \frac{|v_{11}||c_{11}| + |v_{11}||c_{12}| + |v_{21}||c_{22}|}{%
      |v_{11}||c_{11}| + |v_{11} c_{12} + v_{21} c_{22}|}
    \ge \frac{|v_{11}||c_{11}| + |v_{11}||c_{12}| + |v_{21}||c_{22}|}{%
      |v_{11}||c_{11}| + |v_{11}||c_{12}| + |v_{21}||c_{22}|}
    = 1.
  \end{equation*}
  The proof for the remaining $\eta$s is similar.
\end{proof}



\begin{lemma}
  For any $2\times 2$ upper-triangular matrix $A$, the singular values
  of $A$ equal the singular values of $|A|$.
\end{lemma}
\begin{proof}
  Compare the eigenvalues of $A^T\!A$ and $|A|^T |A|$.
\end{proof}



\begin{proposition}\label{thm:etabounds}
  If $\eta_l, \eta_h < \infty$ and $\eta_g, \eta_k < \infty$, then
  \begin{equation*}
    \frac{1}{2} \kappa(A)^{-1} \eta_l \le \eta_h \le 2\kappa(A) \eta_l
    \quad\text{and}\quad
    \frac{1}{2} \kappa(A)^{-1} \eta_g \le \eta_k \le 2\kappa(A) \eta_g,
  \end{equation*}
  respectively.
\end{proposition}
\begin{proof}
  From $\eta_h = \|\widehat H\vec e_2\|_1 / \|H\vec e_2\|_1$ and $\eta_l
  = \|\widehat L\vec e_2\|_1 / \|L\vec e_2\|_1$ we get
  \begin{equation*}
    \eta_h
    \le \sqrt{2} \frac{\|\widehat H\vec e_2\|_2}{\|H\vec e_2\|_2}
    \le \sqrt{2} \frac{\sigma_{\max}(|\adj(A)|)}{\sigma_{\min}(\adj(A))}
    \frac{\|\widehat L\vec e_2\|_2}{\|L\vec e_2\|_2}
    \le 2\kappa(A) \eta_l
  \end{equation*}
  and
  \begin{equation*}
    \eta_h
    \ge \frac{1}{\sqrt{2}} \frac{\|\widehat H\vec e_2\|_2}{\|H\vec e_2\|_2}
    \ge \frac{1}{\sqrt{2}} \frac{\sigma_{\min}(|\adj(A)|)}{\sigma_{\max}(\adj(A))}
    \frac{\|\widehat L\vec e_2\|_2}{\|L\vec e_2\|_2}
    \ge \frac{1}{2} \kappa(A)^{-1} \eta_l.
  \end{equation*}
  The proof for the bounds with $\eta_g$ and $\eta_k$ is analogous.
\end{proof}



\begin{corollary}\label{thm:etaboundscor}
  It follows from the above proposition that
  \begin{equation*}
    \frac{1}{2} \kappa(A)^{-1} \eta_h \le \eta_l \le 2\kappa(A) \eta_h
    \quad\text{and}\quad
    \frac{1}{2} \kappa(A)^{-1} \eta_k \le \eta_g \le 2\kappa(A) \eta_k.
  \end{equation*}
\end{corollary}


Despite the above bounds, we have no upper bound for $\eta_{\max}$ yet.
For example, consider $A = I$, $B =
\minimat{1}{\phantom{-}\mu}{0}{-\mu^{5}}$, and $C =
\minimat{1}{\phantom{-}\mu^{-2}}{0}{-\mu^{-4}}$ for some $0 < \mu \ll
1$; then $\eta_g \approx \mu^{-2}$, $\eta_h \approx \mu^{-4}$, $\eta_k
\approx \mu^{-2}$, $\eta_l \approx \mu^{-4}$, and $\eta_{\max} \approx
\mu^{-2}$. The key to bounding $\eta_{\max}$, is to permute the columns
of $U$ and the columns of $V$ if necessary.


\begin{lemma}\label{thm:etaBigEtaSmall}
  If $\mu^{-1} \le \eta_g < \infty$ for some $0 < \mu < 1/2$, then
  $\|\vec e_2^T \widehat G\|_1 / \|\vec e_2^T G\|_1 \le 2 / (1 - 2\mu)$.
  A similar statement holds for $\eta_l$.
\end{lemma}
\begin{proof}
  Define $\|\operatorname{vec}(G)\|_1 = |g_{11}| + |g_{12}| + |g_{21}| + |g_{22}|$,
  then
  \begin{equation*}
    \|\widehat G\|_1
    = \||V^T| V G\|
    \le \||V^T|\|_1 \|V\|_1 \|G\|_1
    \le 2 \|\operatorname{vec}(G)\|_1.
  \end{equation*}
  Now the bound $\eta_g = \|\vec e_1^T \widehat G\|_1 / \|\vec e_1^T
  G\|_1 \ge \mu^{-1}$ implies that
  \begin{equation*}
    \|\vec e_1^T G\|_1 \le \mu \|\vec e_1^T \widehat G\|_1
    \le 2\mu \|\operatorname{vec}(G)\|_1.
  \end{equation*}
  Hence,
  \begin{equation*}
    \|\vec e_2^T G\|_1
    = \|\operatorname{vec}(G)\|_1 - \|\vec e_1^T G\|_1
    \ge (1 - 2\mu) \|\operatorname{vec}(G)\|_1
  \end{equation*}
  so that $\|\vec e_2^T \widehat G\|_1 / \|\vec e_2^T G\|_1 \le
  2\|\operatorname{vec}(G)\|_1 / \|\vec e_2^T G\|_1 = 2 / (1 - 2\mu)$.
\end{proof}


The result of the lemma above implies that working with $UJ$ and $VJ$
instead of $U$ and $V$, respectively, decreases the value of $\eta_g$
($\eta_l$) when $\eta_g > 4$ ($\eta_l > 4$).  However, this
postmultiplication with $J$ may interfere with our attempt to minimize
the angles of the rotations as described in Section~\ref{sec:kog}, and
may thus lead to slower convergence of the implicit Kogbetliantz
iteration.  Hence, we should not try to minimize $\eta_{\max}$
thoughtlessly. A possible solution is to check whether $\eta_{\max}$ is
larger than some tolerance $\tau_{\eta} \ge 1$, and whether working with
$UJ$ and $VJ$ reduces $\eta_{\max}$. Otherwise, we should keep the
original $U$ and $V$.  This idea leads to the following algorithm.


\begin{myalgorithm}[$2\times 2$ upper-triangular RSVD
  (\textsf{RSVD22}-$\tau_\eta$)]\label{alg:rsvdeta}
\strut\\* 
\rm
  \textbf{Input:} $2\times 2$ upper-triangular matrices $A$, $B$, and $C$, with
    $A$ nonsingular, and tolerance $\tau_\eta \ge 1$. \\
  \textbf{Output:} Orthonormal matrices $P$, $Q$, $U$, and $V$, such
    that $P^T\!AQ$, $P^T\!BU$, and $V^T\!CQ$ are lower triangular, and
    $(V^T\!CQ) \adj(P^T\!AQ) (P^T\!BU) = \Sigma$ is diagonal. \\
  \tab[\phantom01.] Follow Lines~1 through~15 of
    Algorithm~\ref{alg:rsvd}. \\
  \tab[\phantom02.] Define $\iter{\eta_g}{1} = (\widehat g_{11} +
    \widehat g_{12}) / (|g_{11}| + |g_{12}|)$ and $\iter{\eta_h}{1} =
    (\widehat h_{12} + \widehat h_{22}) / (|h_{12}| + |h_{22}|)$. \\
  \tab[\phantom03.] Define $\iter{\eta_k}{1} = (\widehat k_{11} +
    \widehat k_{12}) / (|k_{11}| + |k_{12}|)$ and $\iter{\eta_l}{1} =
    (\widehat l_{12} + \widehat l_{22}) / (|l_{12}| + |l_{22}|)$. \\
  \tab[\phantom04.] Define $\iter{\eta_g}{2} = (\widehat g_{21} +
    \widehat g_{22}) / (|g_{21}| + |g_{22}|)$ and $\iter{\eta_h}{2} =
    (\widehat h_{11} + \widehat h_{21}) / (|h_{11}| + |h_{21}|)$. \\
  \tab[\phantom05.] Define $\iter{\eta_k}{2} = (\widehat k_{21} +
    \widehat k_{22}) / (|k_{21}| + |k_{22}|)$ and $\iter{\eta_l}{2} =
    (\widehat l_{11} + \widehat l_{21}) / (|l_{11}| + |l_{21}|)$. \\
  \tab[\phantom06.] Define $\iter{\eta_{\max}}{i} = \max \{
    \iter{\eta_g}{i}, \iter{\eta_h}{i} \}$ for $i = 1,2$. \\
  \tab[\phantom07.] \textbf{if} $\iter{\eta_{\max}}{1} \le \tau_\eta$
    \textbf{and} $\iter{\eta_{\max}}{1} \le \iter{\eta_{\max}}{2}$
    \textbf{then} \\
  \tab[\phantom08.]\tab Let $\eta_g = \iter{\eta_g}{1}$, $\eta_h =
    \iter{\eta_h}{1}$, $\eta_k = \iter{\eta_k}{1}$, and $\eta_l =
    \iter{\eta_l}{1}$. \\
  \tab[\phantom09.] \textbf{else} \\
  \tab[10.]\tab Let $\eta_g = \iter{\eta_g}{2}$, $\eta_h =
    \iter{\eta_h}{2}$, $\eta_k = \iter{\eta_k}{2}$, and $\eta_l =
    \iter{\eta_l}{2}$. \\
  \tab[11.]\tab Set $U = UJ$, $V = VJ$, $G = J^T G$, $H = HJ$, $K = J^T
    K$, and $L = LJ$. \\
  \tab[12.] \textbf{endif} \\
  \tab[13.] Follow Lines~18 through~28 of Algorithm~\ref{alg:rsvd}.
\end{myalgorithm}


Numerical tests in Section~\ref{sec:num} show the trade-off between
accuracy and performance for different values of $\tau_\eta$. For now,
the following upper bound on the smallest $\eta_{\max}$ that we get is
more important.


\begin{proposition}\label{thm:etaBound}
  Define $\iter{\eta_g}{i}$, $\iter{\eta_h}{i}$, $\iter{\eta_k}{i}$, and
  $\iter{\eta_l}{i}$ as in Algorithm~\ref{alg:rsvdeta} for $i = 1,2$,
  and define the corresponding $\iter{\eta_{\max}}{i}$ as in
  Definition~\ref{def:etamax}. Then
  \begin{equation*}
    \eta_{\max}^{\min}
    = \min \{ \iter{\eta_{\max}}{1}, \iter{\eta_{\max}}{2} \} \le
    \begin{cases}
      4\kappa(A) + 2
      &\text{if $\max \{ \iter{\eta_g}{1}, \iter{\eta_l}{1} \} < \infty$,} \\
      8\kappa(A)
      &\text{otherwise.}
    \end{cases}
  \end{equation*}
\end{proposition}
\begin{proof}
  Suppose that both $\iter{\eta_g}{1}$ and $\iter{\eta_l}{1}$ are finite
  and that $\iter{\eta_{\max}}{1} > 4\kappa(A) + 2$. Then
  $\iter{\eta_g}{1} > 4\kappa(A) + 2$ or $\iter{\eta_l}{1} > 4\kappa(A)
  + 2$, and we can assume without loss of generality that the first of
  the two bounds holds. By applying Lemma~\ref{thm:etaBigEtaSmall} we
  get the bound
  \begin{equation*}
    \iter{\eta_g}{2} \le 2 (1 - 2(4\kappa(A) + 2)^{-1})^{-1}
    = \frac{4\kappa(A) + 2}{2\kappa(A)},
  \end{equation*}
  and then from Proposition~\ref{thm:etabounds} the bound
  $\iter{\eta_k}{2} \le 4\kappa(A) + 2$. Hence, we can conclude that
  $\iter{\eta_{\max}}{2} \le 4\kappa(A) + 2$.

  Now suppose $\iter{\eta_g}{1} = \infty$ or $\iter{\eta_l}{1} = \infty$
  and $\iter{\eta_{\max}}{1} > 8\kappa(A)$, and assume without loss of
  generality that $\iter{\eta_g}{1} = \infty$; then by
  Proposition~\ref{thm:cneqZbneqZ} we have that $\iter{\eta_h}{1} <
  \infty$ and $\iter{\eta_l}{1} < \infty$. Hence, by
  Definition~\ref{def:etamax} we must have $\iter{\eta_h}{1} >
  8\kappa(A)$ or $\iter{\eta_l}{1} > 8\kappa(A)$, so that it follows
  from Corollary~\ref{thm:etaboundscor} and
  Proposition~\ref{thm:etabounds}, respectively, that $\iter{\eta_l}{1}
  > 4$. The result is that we can invoke Lemma~\ref{thm:etaBigEtaSmall}
  to see that $\iter{\eta_l}{2} \le 4$, followed by
  Proposition~\ref{thm:etabounds} to see that $\iter{\eta_h}{2} \le
  8\kappa(A)$. Thus, we can conclude that $\iter{\eta_{\max}}{2} \le
  8\kappa(A)$.
\end{proof}


By combining Theorem~\ref{thm:errA} with the proposition above, we
get the following result.


\begin{theorem}\label{thm:Astable}
  Suppose that we compute all floating-point operations in
  Algorithm~\ref{alg:rsvdeta} with a precision of at least $\mathcal
  O(\bm\epsilon\kappa(A)^{-1})$ and use the tolerance $\tau_\eta =
  8\kappa(A)$. Then the algorithm computes $\overline A'$ stably with
  respect to the precision $\bm\epsilon$.
\end{theorem}


Theorem~\ref{thm:Astable} shows how to pick the working precision to
guarantee an accurate result. But tying the working precision of the
algorithm to the condition number of $A$ is impractical and
mathematically inelegant. An alternative without a strong a priori
guarantee, is to pick a fixed working precision independent of
$\kappa(A)$, with two obvious choices. The first choice is to double the
precision, which we can motivate as follows.  If $\iter{A}{1}$ is the
$p\times q$ input matrix that we have at the beginning of the
preprocessing phase; then, with typical bounds, the first compression
sets all singular values smaller than $\bm\epsilon \max \{p,q\}
\sigma_{\max}(\iter{A}{1})$ to zero. Hence, the resulting
$\iter{A_{12}}{2}$ has a condition number bounded by $(\max \{p,q\}
\bm\epsilon)^{-1}$. Note, however, that this does not guarantee that the
condition number of the upper-triangular $2\times 2$ matrices $A_{ij}$
from the Kogbetliantz phase have the same bound. The second choice is to
not increase the working precision. As we will see in
Section~\ref{sec:num}, a large $\eta_{\max}$ is rare, even for highly
ill-conditioned $A$, and a large $\eta_{\max}^{\min}$ even rarer.
Furthermore, the numerical results show that the bound from
Theorem~\ref{thm:errA} is pessimistic, and that the relative errors do
not scale in proportion to $\eta_{\max}$.  In any case, if we fix the
working precision, then we can cheaply estimate a posteriori whether the
computed $\overline A'$ is accurate in two ways. Either by checking if
$\eta_{\max}$ or $\max\{ \tau_\eta, \eta_{\max}^{\min} \}$ is
sufficiently small, or by checking if $|\fl(\overline P^T\! A \overline
Q)_{12}|$ is sufficiently small.

This section ends with the following remarkable result for a final bit
of insight into the behavior of the $\eta$s.  Although the proof is not
obvious, it requires only elementary arithmetic and is omitted for
brevity.


\begin{proposition}\label{thm:GorL}
  In exact arithmetic $\eta_g > \eta_h$ and $\eta_l > \eta_k$ cannot
  hold simultaneously.  Furthermore, if $g_{11} \neq 0$ and $l_{22} \neq
  0$, then $\eta_g = \eta_h$ and $\eta_l = \eta_k$ can only hold at the
  same time if $a_{12} = 0$.
\end{proposition}

The consequence of the proposition above is that (in exact arithmetic)
Algorithm~\ref{alg:rsvd} computes either $P$ from $L$ or $Q$ from $G$.
In other words, $P$ and $Q$ are never computed from $K = G\adj(A)$ and
$H = \adj(A)L$ at the same time.



\section{The extraction phase}
\label{sec:ext}

In this section we consider the problem of extracting the singular
triplets $(\alpha, \beta, \gamma)$ from the upper-triangular matrices
$(A, B, C) = (S D_\alpha T, S D_\beta, D_\gamma T)$.  Without loss of
generality, we can focus on the diagonal entries and drop the indices,
and consider $a = st\alpha$, $b = s\beta$, $c = t\gamma$, and $\sigma =
a / (bc) = \alpha / (\beta\gamma)$ for unknown $s$, $t$, $\alpha$,
$\beta$, and $\gamma$. Typical treatment of the RSVD imposes the
condition $\alpha^2 + \beta^2 + \gamma^2 = 1$, but this condition alone
generally does not define the singular triplet uniquely.  For example,
for nonzero $b$ and $c$ we can swap the values of $\beta$ and $\gamma$
and adjust $s$ and $t$ accordingly.  Another example is when $a = b = c
= 1$; then we can pick any $\gamma^2 \in (0,1)$ and let
\begin{equation*}
  \alpha^2 = \gamma^2 \frac{1 - \gamma^2}{1 + \gamma^2},
  \qquad
  \beta^2 = \frac{1 - \gamma^2}{1 + \gamma^2},
  \qquad
  s = \beta^{-1},
  \quad\text{and}\quad
  t = \gamma^{-1}.
\end{equation*}
Which further conditions we should impose to make the triplet $(\alpha,
\beta, \gamma)$ well defined, are unclear. It is also unclear how to
compute the triplets in a numerically sound way.

As an alternative, we propose to impose the condition
\begin{equation}\label{eq:abccond}
  \alpha^2 + \beta^2\gamma^2 = 1
\end{equation}
for the normalization of the triplets for the following reasons.
\begin{itemize}
  \item This condition is the correct homogeneous formulation
    corresponding to the fraction $\sigma = \alpha/(\beta\gamma)$, and
    uniquely defines the pair $(\alpha, \beta\gamma)$. 
  \item We know that the RSVs correspond to the nonnegative
    eigenvalues of the pencil
    $\minimat{0}{A}{A^*}{0} - \lambda \minimat{BB^*}{0}{0}{C^*C}$;
    see, e.g., \cite[p.~193]{Zha91}.  Solving this generalized
    eigenvalue problem yields the eigenpairs $(\alpha, \pm \beta
    \gamma)$.
  \item This condition allows us to express the generalized singular
    pairs of a QSVD (i.e., an RSVD with $B = I$) in terms of restricted
    singular triplets with $\beta = 1$ (and $s = 1$).
  \item Triplets corresponding to zero and infinite singular values
    can be written as $(0,1,1)$, $(1,0,0)$, $(1,0,1)$, and $(1,1,0)$,
    and all satisfy \eqref{eq:abccond}.
  \item As shown below, we can impose a simple condition to make
    computing $(\alpha, \beta, \gamma)$, $s$, and $t$ with
    \eqref{eq:abccond} elegant and straightforward.
  \item With \eqref{eq:abccond}, the pair $(\alpha, \beta\gamma)$ is
    invariant under the scaling $(\lambda A, \lambda^p B, \lambda^q C)$
    of the matrix triplet $(A, B, C)$, where $\lambda > 0$ and $p+q=1$.
\end{itemize}

Some flexibility is still left when it comes to computing $\beta$,
$\gamma$, $s$, and $t$. One option is to take $|s| = |t|$,
or more generally $|s|^q = |t|^p$ with $p+q=1$, so that
\begin{equation*}
  a^2 + b^2c^2
  = (\alpha^2 + \beta^2\gamma^2) (st)^2
  = |s|^{2/p}
  = |t|^{2/q},
\end{equation*}
and
\begin{equation*}
  \alpha = |a| (a^2 + b^2 c^2)^{-1/2},
  \qquad
  \beta = |b| (a^2 + b^2 c^2)^{-p/2},
  \quad\text{and}\quad
  \gamma = |c| (a^2 + b^2 c^2)^{-q/2}.
\end{equation*}
Although we have some flexibility when picking $p$ and $q$, the choice
$p = q = 1/2$ is the most natural in absence of an application specific
preference.  This choice also allows us to reliably compute the triplets
$(\alpha, \beta, \gamma)$ in floating-point arithmetic for a wide range
of triplets $(a,b,c)$ with the algorithm below.
A key part of the algorithm is the function $\hypot(x,y)$, which
computes $(x^2 + y^2)^{-1/2}$ without unnecessary overflow or underflow
for $x, y \in \mathbb{R}$. The problem that the algorithm addresses, is
that we cannot use $\hypot(|a|, |b| |c|)$ directly if the product $|b|
|c|$ overflows or underflows.  Hence, the algorithm only applies
$\hypot$ to $|a|$ and $|b| |c|$ directly if the latter product is finite
and nonzero in floating-point arithmetic.  Otherwise, the algorithm
first rescales the input triplet by exploiting the scaling invariance.


\begin{myalgorithm}[Extracting restricted singular
  triplets.]\label{alg:extract}
\strut\\*
\rm
  \textbf{Input:} A triplet $(a,b,c)$, where $\max \{ |a|, |b|, |c| \} <
  \infty$ and $a \neq 0$. \\
  \textbf{Output:} A triplet $(\alpha, \beta, \gamma)$ satisfying
  $\beta \gamma / \alpha = bc/a$ and $\alpha^2 + \beta^2 \gamma^2 = 1$. \\
  \tab[\phantom01.] \textbf{if} $0 < \fl(|b| |c|) < \infty$ \textbf{then} \\
  \tab[\phantom02.]\tab Let $|st| = \hypot(|a|, |b| |c|)$. \\
  \tab[\phantom03.]\tab Let $\alpha = |a| / |st|$, $\beta = |b| / |st|^{1/2}$,
    and $\gamma = |c| / |st|^{1/2}$. \\
  \tab[\phantom04.] \textbf{else if} $|a|^{1/2} \ge \max \{ |b|, |c| \}$ \textbf{then} \\
  \tab[\phantom05.] \tab Let $b' = |b| / |a|^{1/2}$, $c' = |c| / |a|^{1/2}$, and
    $|st| = (1 + (b'c')^2)^{1/2}$. \\
  \tab[\phantom06.] \tab Let $\alpha = 1/|st|$, $\beta = b' / |st|^{1/2}$,
    and $\gamma = c' / |st|^{1/2}$. \\
  \tab[\phantom07.] \textbf{else if} $|b| \ge \max \{ |a|^{1/2}, |c| \}$ \textbf{then} \\
  \tab[\phantom08.]\tab Let $a' = (|a| / |b|) / |b|$, $c' = |c| / |b|$, and
    $|st| = \hypot(a', c')$. \\
  \tab[\phantom09.]\tab Let $\alpha = a' / |st|$, $\beta = 1/|st|^{1/2}$, and
    $\gamma = c' / |st|^{1/2}$. \\
  \tab[10.] \textbf{else if} $|c| \ge \max \{ |a|^{1/2}, |b| \}$ \textbf{then} \\
  \tab[11.]\tab Let $a' = (|a| / |c|) / |c|$, $b' = |b| / |c|$, and
    $|st| = \hypot(a', b')$. \\
  \tab[12.]\tab Let $\alpha = a' / |st|$, $\beta = b' / |st|^{1/2}$, and
    $\gamma = 1 / |st|^{1/2}$. \\
  \tab[13.] \textbf{end}
\end{myalgorithm}




\section{The postprocessing phase}\label{sec:post}

If the implicit Kogbetliantz iteration from Section~\ref{sec:kog}
converges, then we get the Schur-form RSVD from
Theorem~\ref{thm:schurform}. Combined with the extraction from
Section~\ref{sec:ext}, this form is already useful in its own right, as
explained in Section~\ref{sec:th}.  However, if we want the full
decomposition from Theorem~\ref{thm:rsvd} or
Corollary~\ref{thm:rsvdtri}, or any of the individual factors
$\Sigma_\alpha$, $\Sigma_\beta$, $\Sigma_\gamma$, $X$, $Y$, $S$, or $T$,
then further postprocessing is necessary. This necessary postprocessing
is nontrivial in the most general case, and requires that the output of
the implicit Kogbetliantz iteration is of the form described by
Proposition~\ref{thm:betterNzStructure}.  Moreover, some of the
postprocessing steps are troublesome in floating-point arithmetic, e.g.,
due to sensitivity to perturbations, which may affect their reliability.
Hence, we consider the postprocessing steps in exact arithmetic in
Section~\ref{sec:postexact}, and discuss some of the numerical
challenges in floating-point arithmetic in
Section~\ref{sec:numchallenges}.


\subsection{Postprocessing in exact arithmetic}\label{sec:postexact}

Suppose that $A$, $B$, and $C$ are as in \eqref{eq:nzStructureUp}, then
the first step of the preprocessing phase is to use the transformations
from the proof of Proposition~\ref{thm:betterNzStructure} to get the
structure from \eqref{eq:betterNzStructure}.  The next step is to
extract the nonzero restricted singular triplets $D_\alpha$, $D_\beta$,
and $D_\gamma$ with Algorithm~\ref{alg:extract}, and to let
\begin{equation*}
  \compressedmatrices
  \widetilde \Sigma_\alpha =
  \begin{bmatrix}
    D_\alpha \\
    & I \\
    && I \\
    &&& I
  \end{bmatrix},
  \qquad
  \widetilde \Sigma_\beta =
  \begin{bmatrix}
    D_\beta & 0 & 0 \\
    0       & I & 0 \\
    0       & 0 & 0 \\
    0       & 0 & 0
  \end{bmatrix},
  \quad\text{and}\quad
  \widetilde \Sigma_\gamma =
  \begin{bmatrix}
    D_\gamma & 0 & 0 & 0 \\
    0        & 0 & 0 & 0 \\
    0        & 0 & 0 & I
  \end{bmatrix}.
\end{equation*}
Using these $\Sigma$s, we can decompose our matrix triplet as
\begin{equation}\label{eq:postdecomp}
  A = S \widetilde \Sigma_\alpha T,
  \qquad
  B = S \widetilde \Sigma_\beta,
  \quad\text{and}\quad
  C = \widetilde \Sigma_\gamma T,
\end{equation}
where $S$ and $T$ are upper triangular.  In particular, let $S_{11} =
B_{11} D_\beta^{-1}$ and $T_{11} = D_\gamma^{-1} C_{11}$, so that
\begin{equation*}
  A_{11} C_{11}^{-1}
  = (S_{11} D_\alpha T_{11}) (T_{11}^{-1} D_\gamma^{-1})
  = S_{11} D_\alpha D_\gamma^{-1}
  = B_{11} D_\beta^{-1} D_\alpha D_\gamma^{-1},
\end{equation*}
and define the Schur complement $Z_{1j} = A_{1j} - A_{11} C_{11}^{-1}
C_{1j}$; then
\begin{equation}\label{eq:postST}
  \compressedmatrices
  S =
  \begin{bmatrix}
    S_{11} & B_{12} & Z_{13} & Z_{14} C_{44}^{-1} \\
           & B_{22} & A_{23} & A_{24} C_{44}^{-1} \\
           &        & A_{33} & A_{34} C_{44}^{-1} \\
           &        &        & A_{44} C_{44}^{-1}
  \end{bmatrix},
  \quad\text{and}\quad
  T =
  \begin{bmatrix}
    T_{11} & D_\gamma^{-1} C_{12} & D_\gamma^{-1} C_{13} & D_\gamma^{-1} C_{14} \\
    & B_{22}^{-1} A_{22} & 0 & 0 \\
    && I & 0 \\
    &&& C_{44}
  \end{bmatrix}.
\end{equation}
To see that these $S$ and $T$ are correct, consider the leading
principal $2\times 2$ blocks of $C A^{-1} B$, given by
\begin{equation*}
  \compressedmatrices
  \begin{split}
    \begin{bmatrix}
      C_{11} A_{11}^{-1} B_{11}
      & 0 \\
      & 0
    \end{bmatrix}
    &= \begin{bmatrix}
      C_{11} & C_{12} \\ & 0
    \end{bmatrix}
    \begin{bmatrix}
      A_{11}^{-1} & -A_{11}^{-1} A_{12} A_{22}^{-1} \\ & A_{22}^{-1}
    \end{bmatrix}
    \begin{bmatrix}
      B_{11} & B_{12} \\ & B_{22}
    \end{bmatrix}
    \\
    &= \begin{bmatrix}
      C_{11} A_{11}^{-1} B_{11}
      & C_{11} A_{11}^{-1} B_{12} - C_{11} A_{11}^{-1} A_{12} A_{22}^{-1}
      B_{22} + C_{12} A_{22}^{-1} B_{22} \\
      & 0
    \end{bmatrix}.
  \end{split}
\end{equation*}
Since the $(1,2)$ block must be zero, we have that
$A_{12} = A_{11} C_{11}^{-1} C_{12} + B_{12} B_{22}^{-1} A_{22}$, which
we can use to verify that
\begin{equation*}
  \compressedmatrices
  \begin{bmatrix}
    S_{11} & B_{12} \\ & B_{22}
  \end{bmatrix}
  \begin{bmatrix} D_\alpha \\ & I \end{bmatrix}
  \begin{bmatrix}
    T_{11} & D_\gamma^{-1} C_{12} \\
    & B_{22}^{-1} A_{22}
  \end{bmatrix}
  =
  \begin{bmatrix}
    A_{11} & A_{11} C_{11}^{-1} C_{12} + B_{12} B_{22}^{-1} A_{22} \\
    & A_{22}
  \end{bmatrix}
\end{equation*}
is equal to the leading principal $2\times 2$ blocks of $A$. The rest of the
proof that $S$ and $T$ are of the form in \eqref{eq:postST} is by direct
verification.

\newcommand*\Sa{\Sigma_\alpha}
\newcommand*\Sb{\Sigma_\beta}
\newcommand*\Sc{\Sigma_\gamma}

Now that we have the decomposition \eqref{eq:postdecomp}, we can plug it
back into the Schur-form RSVD from
\eqref{eq:schurforma}--\eqref{eq:schurformb} to get the triplet
$(\iter{A}{\ell}, \iter{B}{\ell}, \iter{C}{\ell})$. Here, we use a
similar notation as in the preprocessing phase, with a similar numbering
of the blocks. Define
\begin{equation*}
  \compressedmatrices
  (\iter{X}{\ell})^{-T} =
  \begin{bmatrix}
    I & & \frac{1}{2} A_{15} & B_{14} \\
    & S & A_{25} & B_{24} \\
    & & A_{35} & B_{34} \\
    & & & B_{44} \\
    & & & & I
  \end{bmatrix}
  \quad\text{and}\quad
  (\iter{Y}{\ell})^{-1} =
  \begin{bmatrix}
    I \\
    & C_{12} & C_{13} & C_{14} & C_{15} \\
    & & A_{13} & A_{14} & \frac{1}{2} A_{15} \\
    & & & T \\
    & & & & I
  \end{bmatrix}
\end{equation*}
cf.\ \eqref{eq:schursparseST}, where $S$ and $T$ are as in
\eqref{eq:postST}, and let
\begin{equation*}
  \iter{A}{\ell+1} = (\iter{X}{\ell})^T\! \iter{A}{\ell} \iter{Y}{\ell},
  \qquad
  \iter{B}{\ell+1} = (\iter{X}{\ell})^T\!  \iter{B}{\ell},
  \quad\text{and}\quad
  \iter{C}{\ell+1} = \iter{C}{\ell} \iter{Y}{\ell}.
\end{equation*}
Then the above three matrices are as in \eqref{eq:schursparse}, but with
$A_{24}$, $B_{23}$, and $C_{24}$ replaced by $\widetilde \Sigma_\alpha$,
$\widetilde \Sigma_\beta$, and $\widetilde \Sigma_\gamma$, respectively.
Hence, with the appropriate block permutations we get the matrices
\begingroup
\allowdisplaybreaks
\begin{align*}
  \compressedmatrices
  \iter{A}{\ell+2} &=
  \begin{bmatrix}
    0 & 0 & D_\alpha & 0 & 0      & 0 \\
    0 & 0 & 0        & I & 0      & 0 \\
    0 & 0 & 0        & 0 & \ddots & 0 \\
    0 & 0 & 0        & 0 & 0      & I \\
    0 & 0 & 0        & 0 & 0      & 0 \\
    0 & 0 & 0        & 0 & 0      & 0
  \end{bmatrix},
  \\
  \iter{B}{\ell+2} &=
  \begin{bmatrix}
    D_\beta & 0      & 0      & 0      & 0 & 0 \\
    0       & I      & 0      & 0      & 0 & 0 \\
    \iter{B_{31}}{\ell+2}  & \iter{B_{32}}{\ell+2} &
    \iter{B_{33}}{\ell+2} & \iter{B_{34}}{\ell+2} & 0 & 0 \\
    0       & 0      & 0      & 0      & 0 & 0 \\
    0       & 0      & 0      & 0      & 0 & 0 \\
    0       & 0      & 0      & 0      & 0 & 0 \\
    0       & 0      & 0      & 0      & 0 & I \\
    0       & 0      & 0      & 0      & 0 & 0
  \end{bmatrix},
  \\
  \iter{C}{\ell+2} &=
  \begin{bmatrix}
    0 & I & 0        & 0 & 0 & 0 & 0 & 0      \\
    0 & 0 & D_\gamma & 0 & 0 & 0 & 0 &
    \iter{C_{28}}{\ell+2} \\
    0 & 0 & 0        & 0 & 0 & 0 & I & \iter{C_{38}}{\ell+2} \\
    0 & 0 & 0        & 0 & 0 & 0 & 0 & \iter{C_{48}}{\ell+2} \\
    0 & 0 & 0        & 0 & 0 & 0 & 0 & \iter{C_{58}}{\ell+2} \\
    0 & 0 & 0        & 0 & 0 & 0 & 0 & 0
  \end{bmatrix}.
\end{align*}
\endgroup
Next we need to compress $[\iter{B_{33}}{\ell+2}\;
\iter{B_{34}}{\ell+2}]$ and $[\iter{C_{48}}{\ell+2};\;
\iter{C_{58}}{\ell+2}]$ (and transfer the transformations ``through''
$\iter{A}{\ell+2}$) to get $\iter{A}{\ell+3} = \iter{A}{\ell+2}$, and
$\iter{B}{\ell+3}$ and $\iter{C}{\ell+3}$ given by
\begin{equation*}
  \compressedmatrices
  \begin{bmatrix}
    D_\beta & 0      & 0   & 0 & 0 & 0 \\
    0       & I      & 0   & 0 & 0 & 0 \\
    \iter{B_{31}}{\ell+3}  & \iter{B_{32}}{\ell+3} & R_B & 0 & 0 & 0 \\
    \iter{B_{41}}{\ell+3}  & \iter{B_{42}}{\ell+3} & 0   & 0 & 0 & 0 \\
    0       & 0      & 0   & 0 & 0 & 0 \\
    \vdots  & \vdots & \vdots & \vdots & \vdots & \vdots \\
    0       & 0      & 0   & 0 & 0 & 0 \\
    0       & 0      & 0   & 0 & 0 & I \\
    0       & 0      & 0   & 0 & 0 & 0
  \end{bmatrix}
  \;\text{and}\;
  \begin{bmatrix}
    0 & I & 0        & 0 & \cdots & 0 & 0 & 0      & 0      \\
    0 & 0 & D_\gamma & 0 & \cdots & 0 & 0 &
    \iter{C_{29}}{\ell+3} & \iter{C_{2,10}}{\ell+3} \\
    0 & 0 & 0        & 0 & \cdots & 0 & I & \iter{C_{39}}{\ell+3} &
    \iter{C_{3,10}}{\ell+3} \\
    0 & 0 & 0        & 0 & \cdots & 0 & 0 & R_C    & 0      \\
    0 & 0 & 0        & 0 & \cdots & 0 & 0 & 0      & 0      \\
    0 & 0 & 0        & 0 & \cdots & 0 & 0 & 0      & 0
  \end{bmatrix},
\end{equation*}
respectively.  Next, take $(\iter{X}{\ell+3})^{-T}$ as
\begin{equation*}
  \compressedmatrices
  \begin{bmatrix}
    I & 0 & 0 & 0 & 0 & 0 & -D_\alpha D_\gamma^{-1} C_{29} R_C^{-1} &
    -D_\alpha D_\gamma^{-1} C_{2,10} & 0 & 0 \\
    0 & I & 0 & 0 & 0 & 0 & 0 & 0 & 0 & 0 \\
    B_{31} D_\beta^{-1} & B_{32} & R_{B} & 0 & 0 & 0 & 0 & 0 & 0 & 0 \\
    B_{41} D_\beta^{-1} & B_{42} & 0 & I & 0 & 0 &
    0 & 0 & 0 & 0\\
    0 & 0 & 0 & 0 & I & 0 & 0 & 0 & 0 & 0 \\
    0 & 0 & 0 & 0 & 0 & I & -C_{39} R_C^{-1} & -C_{3,10} & 0 &
    0\\
    0 & 0 & 0 & 0 & 0 & 0 & R_C^{-1} & 0 & 0 & 0 \\
    0 & 0 & 0 & 0 & 0 & 0 & 0 & I & 0 & 0 \\
    0 & 0 & 0 & 0 & 0 & 0 & 0 & 0 & I & 0 \\
    0 & 0 & 0 & 0 & 0 & 0 & 0 & 0 & 0 & I
  \end{bmatrix}
\end{equation*}
and $(\iter{Y}{\ell+3})^{-1}$ as
\begin{equation*}
  \compressedmatrices
  \begin{bmatrix}
    I & 0 & 0 & 0 & 0 & 0 & 0 & 0 & 0 & 0 \\
    0 & I & 0 & 0 & 0 & 0 & 0 & 0 & 0 & 0 \\
    0 & 0 & I & 0 & 0 & 0 & 0 & 0 & D_\gamma^{-1} C_{29} & D_\gamma^{-1} C_{2,10} \\
    0 & 0 & 0 & I & 0 & 0 & 0 & 0 & 0 & 0 \\
    0 & 0 & -R_B^{-1} B_{31} D_\beta^{-1} D_\alpha & -R_B^{-1} B_{32} &
    R_B^{-1} & 0 & 0 & 0 & -R_B^{-1} B_{31} D_\sigma^{-1} C_{29} &
    -R_B^{-1} B_{31} D_\sigma^{-1} C_{2,10} \\
    0 & 0 & -B_{41} D_\beta^{-1} D_\alpha & -B_{42} &
    0 & I & 0 & 0 & -B_{41} D_\sigma^{-1} C_{29} & -B_{41} D_\sigma^{-1} C_{2,10} \\
    0 & 0 & 0 & 0 & 0 & 0 & I & 0 & 0 & 0 \\
    0 & 0 & 0 & 0 & 0 & 0 & 0 & I & C_{39} & C_{3,10} \\
    0 & 0 & 0 & 0 & 0 & 0 & 0 & 0 & R_C & 0  \\
    0 & 0 & 0 & 0 & 0 & 0 & 0 & 0 & 0 & I
  \end{bmatrix},
\end{equation*}
where we dropped the superscript indices to save horizontal whitespace
and $D_\sigma^{-1} = D_\beta^{-1} D_\alpha D_\gamma^{-1}$, and compute
$\iter{A}{\ell+4} = \Sigma_\alpha$, $\iter{B}{\ell+4} = \Sigma_\beta$,
and $\iter{C}{\ell+4} = \Sigma_\gamma$ with a transformation like in
\eqref{eq:postdecomp}. Here, $\iter{X}{\ell+3}$ is upper triangular if
$p_1 = 0$ in Theorem~\ref{thm:schurform}, and $\iter{Y}{\ell+3}$ is
lower triangular if $q_5 = 0$ in Theorem~\ref{thm:schurform}. We can
always assume that we have the former case if we wish, by transforming
the input triplet as discussed at the end of the preprocessing phase.



\subsection{Challenges in floating-point
arithmetic}\label{sec:numchallenges}

If $B$ or $C$ is singular before the application of the implicit
Kogbetliantz iteration, then we may have $\overline m_{ii}$ after
convergence that should have been zero in exact arithmetic, but are
nonzero due to roundoff errors.  It follows that in floating-point
arithmetic $\overline M$ lacks the desired form of
Proposition~\ref{thm:nzStructure}, or at least, has more nonzeros than
it should. Consider a $3\times 3$ example, where $A = I$ and $C$, $B$,
and $M$ are
\begin{equation*}
  \compressedmatrices
  \begin{bmatrix}
    \iter{c_{11}}{1} & \iter{c_{12}}{1} & \iter{c_{13}}{1} \\
    & \iter{c_{22}}{1} & \iter{c_{23}}{1} \\
    && \iter{c_{33}}{1}
  \end{bmatrix}
  \begin{bmatrix}
    \iter{b_{11}}{1} & \iter{b_{12}}{1} & \iter{b_{13}}{1} \\
    & 0 & 0 \\
    && 0
  \end{bmatrix}
  =
  \begin{bmatrix}
    \iter{m_{11}}{1} & \iter{m_{12}}{1} & \iter{m_{13}}{1} \\
    & 0 & 0 \\
    && 0
  \end{bmatrix}.
\end{equation*}
After eliminating the $(1,2)$ element of $M$ in exact arithmetic we get
\begin{equation*}
  \compressedmatrices
  \begin{bmatrix}
    \iter{c_{11}}{2} & 0 & \iter{c_{13}}{2} \\
    \iter{c_{21}}{2} & \iter{c_{22}}{2} & \iter{c_{23}}{2} \\
    && \iter{c_{33}}{2}
  \end{bmatrix}
  \begin{bmatrix}
    \iter{b_{11}}{2} & 0 & \iter{b_{13}}{2} \\
    \iter{b_{21}}{2} & 0 & \iter{b_{23}}{2} \\
    && 0
  \end{bmatrix}
  =
  \begin{bmatrix}
    \iter{m_{11}}{2} & 0 & \iter{m_{13}}{2} \\
    & 0 & 0 \\
    && 0
  \end{bmatrix}.
\end{equation*}
Eliminating the $(1,3)$ element gives us
\begin{equation*}
  \compressedmatrices
  \begin{bmatrix}
    \iter{c_{11}}{3} & 0 & 0 \\
    \iter{c_{21}}{3} & \iter{c_{22}}{3} & \iter{c_{23}}{3} \\
    \iter{c_{31}}{3} && \iter{c_{33}}{3}
  \end{bmatrix}
  \begin{bmatrix}
    \iter{b_{11}}{3} & 0 & 0 \\
    \iter{b_{21}}{3} & 0 & 0 \\
    \iter{b_{31}}{3} && 0
  \end{bmatrix}
  =
  \begin{bmatrix}
    \iter{m_{11}}{3} & 0 & 0 \\
    & 0 & 0 \\
    && 0
  \end{bmatrix},
\end{equation*}
where $\iter{b_{23}}{2}$ is eliminated at the same time as
$\iter{b_{13}}{2}$ because the vectors $[\iter{b_{11}}{2}\;
\iter{b_{21}}{2}]$ and $[\iter{b_{13}}{2}\; \iter{b_{23}}{2}]$ are
parallel. But in floating-point arithmetic we suffer from roundoff
errors and can expect to end up with
\begin{equation*}
  \compressedmatrices
  \begin{bmatrix}
    \iter{c_{11}}{3} & 0 & 0 \\
    \iter{c_{21}}{3} & \iter{c_{22}}{3} & \iter{c_{23}}{3} \\
    \iter{c_{31}}{3} && \iter{c_{33}}{3}
  \end{bmatrix}
  \begin{bmatrix}
    \iter{b_{11}}{3} & 0 & 0 \\
    \iter{b_{21}}{3} & \epsilon & \epsilon \\
    \iter{b_{31}}{3} && \epsilon
  \end{bmatrix}
  =
  \begin{bmatrix}
    \iter{m_{11}}{3} & 0 & 0 \\
    & \epsilon & \epsilon \\
    && \epsilon
  \end{bmatrix},
\end{equation*}
where the $(1,2)$ and $(1,3)$ elements are explicitly set to zero. Now,
we can ensure that $\iter{b_{22}}{3} = \iter{b_{33}}{3} = 0$ by ensuring
that Algorithm~\ref{alg:rsvd} produces the output of
Lemma~\ref{thm:algIO} and by copying the elements of the $2\times 2$
results back to the larger matrices. The necessary changes to
Algorithm~\ref{alg:rsvd} are to explicitly set $g_{22} = 0$ whenever
$c_{11} = 0$, and to set $l_{12} = 0$ whenever $b_{22} = 0$.  Still,
this does not take care of the nonzero element $\iter{b_{23}}{3}$, and
at the end of the cycle we end up with a second nonzero on the diagonal
of $B$ and $M$. Moreover, explicitly zeroing $g_{22}$ and $l_{12}$ is
not automatically better than not doing so if we consider cases with
underflow.

For $A$ we face the opposite problem. If $A$ is severely
ill-conditioned, then one of its $2\times 2$ submatrices may become
singular by applying the rotations. Both the order of evaluation and how
the rotations are applied may affect the outcome in these cases. For
more information on the latter, see~\cite[Sec.~3]{Drma97}. This again
shows that the condition numbers of the $2$-by-$2$ matrices $A_{ij}$ are
important, as we already know from Section~\ref{sec:bwdA}.

A naive way to get rid of the unwanted nonzeros is the following: sort
the diagonal entries of $\overline M$ by magnitude after convergence by
picking $U = V = I$ or $U = V = J$ in the $2$-by-$2$ RSVD algorithm.
Then set diagonal entries, and their corresponding rows, to zero if they
are below some threshold, and follow the steps at the end of
Section~\ref{sec:rsvdTwoExact} with similar thresholding.  This strategy
appears reasonable at first sight since the $|\overline m_{ii}|$ are
computed to approximate the singular values of $CA^{-1}B$. But
$\overline M$ is a product of matrices and a rank decision based on a
simple threshold is even less reliable than usual. For example, suppose
that $A = I$ and $B = C = \diag(1,10^{-10})$; then $M = CA^{-1}B =
\diag(1,10^{-20})$ is numerically singular (in IEEE 754 double
precision) for a typical threshold like $2\bm\epsilon$, even though $B$
and $C$ are numerically nonsingular. Another example is with $A = I$, $B
= \diag(1,10^{-20})$, and $C = \diag(1,10^{20})$; now $M = I$ looks
nonsingular, while $B$ and $C$ are both numerically singular. 

The latter of the two examples above is an example of ill-conditioned
restricted singular values, e.g., a relative perturbation of $10^{-15}$
in either $B$ or $C$ may result in a relative perturbation of $10^{5}$
in $M$. Still, declaring the small relative entries of $B$ and $C$ to be
zero is not automatically reasonable, despite the unreliable entries of
$M$.

Another issue is that the implicit Kogbetliantz iteration is not rank
revealing for $B$ and $C$ in general.  For example, if $\iter{C}{0} =
\minimat{1}{10^{10}}{0}{1}$, $\iter{A}{0} = I \in R^{2 \times 2}$, and
$\iter{B}{0} = 0 \in R^{2 \times 2}$. Then in the odd cycle
$\iter{Q}{0}$ zeros the $(1,2)$ entry of $\iter{C}{0}$.
Hence, $\iter{B}{1} = (\iter{C}{0} \iter{Q}{0})^T$
so that $(\iter{P}{1})^T = \iter{Q}{0}$, and thus $\iter{C}{2} =
((\iter{P}{1})^T \iter{B}{1})^T = \iter{C}{0}$.  This example also
demonstrates that it does not suffice to just look at the diagonal
entries. The diagonal entries of $\iter{C}{0}$ are both 1, but the
condition number of $\iter{C}{0}$ is approximately $10^{20}$.

Now suppose that we have decided that both $\overline m_{ii}$ and
$\overline m_{jj}$ should be zero. If we set diagonal elements of $B$ or
$C$ to zero without doing anything else, then we do not automatically
get $\overline m_{ij} = 0$. That is, we need to be careful not to
introduce new nonzeros while zeroing elements.

Although the examples above are not exhaustive, it should be clear by
now that determining which entries of $B$, $C$, and $M$ should be zero
is a nontrivial problem. Similar problems exist for the generalized
eigenvalue problem (see, for example, Stewart and
Sun~\cite[Ch.~6]{SS90}), and their solutions are outside the scope of
this work.  Although it may appear that Zha's algorithm and the
algorithm of Chu et al.\ do not suffer from these issues, similar
problems hide inside the rank decisions in their preprocessing phases.
The difference is that we move part of these rank-decision woes to the
postprocessing that comes after the Kogbetliantz phase, and which we can
omit if the Schur-form RSVD is sufficient for our needs.  Furthermore,
we are in a better position to spot sensitivity issues after the
Kogbetliantz phase and with the help of Algorithm~\ref{alg:extract}.




\section{Nonorthogonal transformations}\label{sec:nonorth}

In the first step of the preprocessing phase we compress $A$ using
orthonormal transformations that we compute with some URV decomposition.
We can sometimes do better, for example for graded matrices, if we allow
arbitrary nonsingular transformations. Then we can replace the URV
decomposition by a rank-reveal LU or LDU decomposition, or some other
rank-revealing decomposition. This is also what Drma\v{c}'s algorithm
for the RSVD~\cite{Drma00} uses. The algorithm below summarizes the
modified algorithm for a simplified input.


\begin{myalgorithm}[Nonorthogonal RSVD]\label{alg:nonorthRSVD}
\strut\\*
\rm
  \textbf{Input:} Square and upper-triangular $k \times k$ matrices $A$,
    $B$, $C$, and $A$ nonsingular. \\
  \textbf{Output:} Nonsingular matrices $X$ and $Y$, and orthonormal
    matrices $U$ and $V$, such that $X^T\!AY$, $X^T\!BU$, and $V^T\!CY$
    are upper-triangular, and $V^T\! CA^{-1}B U$ is diagonal.  \\
  \tab[1.] Compute $D_B = \diag(\|\vec e_i^T B\|)$ and $D_C =
    \diag(\|C\vec e_i\|)$. \\
  \tab[2.] Set $A_1 = D_B^{-1} A D_C^{-1}$, $B_1 = D_B^{-1} B$, and $C_1
    = C D_C^{-1}$. \\
  \tab[3.] Compute the LDU decomposition $\Pi_r A_1 \Pi_c = L_A D_A U_A$
  with full pivoting. \\
  \tab[4.] Compute the RQ decomposition $L_A^{-1} \Pi_r B_1 = R_B Q_B^T$. \\
  \tab[5.] Compute the QR decomposition $C_1 \Pi_c U_A^{-1} = Q_C R_C$. \\
  \tab[6.] Use Algorithm~\ref{alg:kog} to compute $U$, $V$, $P$, and
    $Q$ such that $P^T D_A Q$, $P^T R_B U$, \\
  \tab and $V^T R_C Q$ are upper triangular, and $V^T R_C D_A^{-1} R_B U$ is diagonal. \\
  \tab[7.] Accumulate $U = Q_B U$, $V = Q_C V$, $X^{T} = P^T L_A^{-1}
    \Pi_r D_B^{-1}$,  and $Y = D_C^{-1} \Pi_c U_A^{-1} Q$.
\end{myalgorithm}


The benefit of Algorithm~\ref{alg:nonorthRSVD} is that it does not just
produce orthonormal $U$ and $V$ such that $V^T\! CA^{-1}B U$ is diagonal.
It also produces nonsingular $X$ and $Y$ such that $X^T\!AY$, $X^TBU$,
and $V^T\!CY$ are upper triangular. Although Drma\v{c} does not discuss
it, we can compute such $X$ and $Y$ a posteriori when using his
algorithm.  The problem then is that the necessary computations are
nontrivial when $B$ and $C$ are nonsingular, and we run into some of the
challenges from Section~\ref{sec:numchallenges}.



\section{Numerical experiments}\label{sec:num}

Our numerical testing consists of three parts. In the first part, we
test Algorithm~\ref{alg:rsvdeta} and plot the distribution of the
largest magnitudes of the computed $(1,2)$ entries, the largest relative
errors, and the values of the $\eta_{\max}$s.  We do this for both
$\tau_\eta = 1$ and $\tau_\eta = \infty$. That is, both when we always
change the columns of $U$ and $V$ if it improves $\eta_{\max}$, and when
we never change the columns. We compare the results with similar results
from Zha's method.
In the second part we consider nonsingular and upper-triangular $n
\times n$ matrix triplets, and again plot the distribution of the
largest magnitudes of $(1,2)$ elements and the $\eta_{\max}$s, but not
the relative errors. This time, we test more tolerances than just
$\tau_\eta = 1$ and $\tau_\eta = \infty$, and try to see how the value
of $\tau_\eta$ affects the accuracy and the rate of convergence of the
implicit Kogbetliantz iteration.
In the third and final part, we compare the difference in accuracy
between Algorithm~\ref{alg:rsvdeta}, Zha's method~\cite{Zha92},
Drma\v{c}'s method~\cite{Drma00}, and the method from Chu, De~Lathauwer,
and De~Moor~\cite{CLM00}. We do this just for a small class of matrices
for brevity.

For the standard linear algebra routines, such as matrix multiplication,
and matrix decompositions, such as QR with column pivoting and the
Jacobi SVD, we use Eigen~\cite{eigenweb}. For the high-precision
arithmetic we use Boost Multiprecision~\cite{mpreal}, which we can use
in combination with Eigen in a straightforward manner.


\subsection{Testing 2-by-2 RSVDs}\label{sec:numrsvdtwo}

We can test Algorithm~\ref{alg:rsvdeta} for a given matrix triplet by
computing its result both with double-precision and with high-precision
floating-point arithmetic. Then we compare the results and use the
high-precision result in place of the exact result. With this approach,
we need to be careful when dealing with $M$, because we need to ensure
that the double-precision and high-precision results approximate the
same quantities. Hence, we proceed as follows.
First, we generate $2\times 2$ upper-triangular matrices $A$, $B$, and
$C$ in double precision.  Each entry has the form $s\cdot 2^p$, where
the sign $s$ a Rademacher distributed random variable and the exponent
$p$ a uniform random variable in $[-333, 333)$. The range of $p$ is such
that the product $M = C \adj(A) B$, when computed in higher precision,
can still be represented in double precision without overflow or
underflow.  Next, we take $\overline M = \fl(M)$ and use $\overline M$
as the input for both the double and the high precision RSVD
computation. This ensures that we compute the SVD of the same matrix in
both cases, and that the high-precision product $C \adj(A) B$
approximates the double precision $\overline M$ as well as possible.
Since we generate matrices that may have extremely large or small
values, we need to ensure that no overflow, underflow, or other
numerical difficulties occur by rejecting samples that satisfy one or
more of the following conditions.
\begin{enumerate}
  \item The entry $\overline m_{12}$ is nonzero and $\min \{ |\overline
    u_{11}|, |\overline u_{12}| \} = 0$ or $\min \{ |\overline v_{11}|,
    |\overline v_{12}| \} = 0$. In this case the
    computation of the SVD underflows and the zero entries lack a
    high-relative precision, which means that Fact~1 in
    Theorem~\ref{thm:stablersvd} does not hold.
  \item The bound for the off-diagonal error
    \begin{equation*}
      \fl\left( \frac{|\overline \sigma_{12}| + |\overline
      \sigma_{21}|}{|\overline m_{11}| + |\overline m_{12}| + |\overline
      m_{22}|} \right)
      \le \frac{\sqrt{2} \|E\|_F}{\|\overline M\|_F} (1 + 4\bm\epsilon)
      \le \frac{2 \|E\|_2}{\|\overline M\|_2} (1 + 4\bm\epsilon)
      < 281 \bm\epsilon
    \end{equation*}
    does not hold, where
    \begin{equation*}
      E
      = \overline \Sigma - \Sigma
      = \fl(\overline V^T\! \overline M \overline U) - V^T\! \overline M U
      = \delta V^T\! \overline M U + V^T\! \overline M \delta U + F
    \end{equation*}
    for some $F$ containing roundoff errors. This is because we require
    $\overline V$ and $\overline U$ to diagonalize $\overline M$
    properly, and because we know the facts from the proof of
    Theorem~\ref{thm:stablersvd} imply that $\|E\| \le (2\cdot 66 +
    2\cdot 4) \bm\epsilon \|\overline M\|$ in the absence of underflow
    and overflow.  \footnote{The goal of this section is not to verify
    this claim. Moreover, the proof of this part of
    Theorem~\ref{thm:stablersvd} is straightforward and identical to the
    proof from Bai and Demmel~\cite[Thm.~3.1]{BD93}.}.
  \item The singular values of $\overline M$ are too close to each
    other, say within a relative distance of $10^{-14}$, in which case
    the columns of the double and high-precision $U$ and $V$ may be in a
    different order and the results hard to compare.
\end{enumerate}
Rejecting triplets makes the sampling nonuniform, but also allows us to
test a larger range of floating-point numbers as entries of $A$, $B$,
and $C$. 


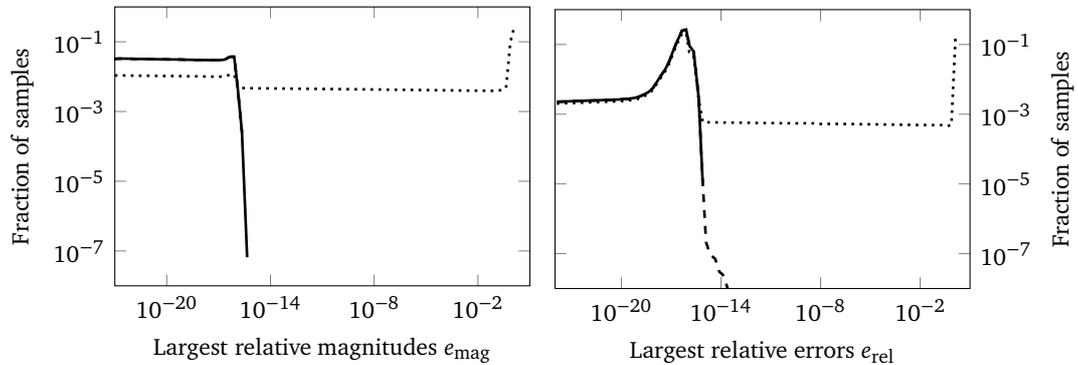
\begin{figure}[htbp!]
  \centerfloat
  \begin{tikzpicture}
    \begin{loglogaxis}[
          height=0.33\textwidth
        , width=0.44\textwidth
        , xlabel={Largest relative magnitudes $e_{\text{mag}}$}
        , ylabel={Fraction of samples}
        , every axis legend/.append style = {
              anchor = south east
            , at = {(0.98, 0.02)}
          }
        , enlarge x limits = false
        , xmax = 10
        , xmin = 1e-23
        , xtick = {1e-20, 1e-14, 1e-8, 1e-2}
        , ymax = 1
        , ymin = 1e-8
        ]

        \addplot+ table [x expr=10^((\thisrow{Lo} + \thisrow{Up})/2),y=Max01]
        {\tableSwapOnErr};

        \addplot+ table [x expr=10^((\thisrow{Lo} + \thisrow{Up})/2),y=Max01]
        {\tableSwapOffErr};

        \addplot+ table [x expr=10^((\thisrow{Lo} + \thisrow{Up})/2),y=Max01]
        {\tableZhaErr};
    \end{loglogaxis}
  \end{tikzpicture}
  ~
  \begin{tikzpicture}
    \begin{loglogaxis}[
          height=0.33\textwidth
        , width=0.44\textwidth
        , xlabel={Largest relative errors $e_{\text{rel}}$}
        , ylabel={Fraction of samples}
        , every axis legend/.append style = {
              anchor = south east
            , at = {(0.98, 0.02)}
          }
        , xmax = 10
        , xmin = 1e-24
        , xtick = {1e-20, 1e-14, 1e-8, 1e-2}
        , ymax = 1
        , ymin = 1e-8
        , yticklabel pos=right
        ]

        \addplot+ table [x expr=10^((\thisrow{Lo} + \thisrow{Up})/2),y=Maxerr]
        {\tableSwapOnErr};

        \addplot+ table [x expr=10^((\thisrow{Lo} + \thisrow{Up})/2),y=Maxerr]
        {\tableSwapOffErr};

        \addplot+ table [x expr=10^((\thisrow{Lo} + \thisrow{Up})/2),y=Maxerr]
        {\tableZhaErr};
    \end{loglogaxis}
  \end{tikzpicture}
  \caption{Normalized histograms of the $\log_{10}(e_{\text{mag}})$
    (left) and $\log_{10}(e_{\text{rel}})$ (right) of the results from
    Algorithm~\ref{alg:rsvdeta} with $\tau_\eta = 1$ (solid) and
    $\tau_\eta = \infty$ (dashed), as well as Zha's method (dotted).
    The figures do not show the individual bars for the histograms to
    avoid clutter and because each figure shows the results for three
    different algorithms.  Moreover, the histograms only include samples
    $e_{\text{mag}}, e_{\text{rel}} \ge 10^{-24}$, which means they only
    show the tail ends of the true distribution of the
    samples.}\label{fig:errors}
\end{figure}


Given the inputs $A$, $B$, $C$, and $\overline M$, we compute $\overline
A'$, $\overline B'$, $\overline C'$, $\overline H'$, and $\overline K'$
with Algorithm~\ref{alg:rsvdeta} in double precision, both with
$\tau_\eta = 1$ and with $\tau_\eta = \infty$. Then we compute $A'$,
$B'$, $C'$, $H'$, and $K'$ with 100 decimals of precision, while making
sure the high-precision computations take the same conditional branches
in the algorithm as the double-precision computations.  Given these
results, we can compute the maximum of the relative magnitudes as
\begin{equation*}
  e_{\text{mag}} = \max \left\{
    \frac{|\overline A'_{12}|}{\|A\|_F},
    \frac{|\overline B'_{12}|}{\|B\|_F},
    \frac{|\overline C'_{12}|}{\|C\|_F},
    \frac{|\overline H'_{12}|}{\|A\|_F \|B\|_F},
    \frac{|\overline K'_{12}|}{\|A\|_F \|C\|_F}
  \right\}.
\end{equation*}
We can also compute the maximum of the relative errors as
\begin{equation*}
  e_{\text{rel}} = \max \left\{
    \frac{\|\overline A' - A'\|_F}{\|A\|_F},
    \frac{\|\overline B' - B'\|_F}{\|B\|_F},
    \frac{\|\overline C' - C'\|_F}{\|C\|_F},
    \frac{\|\overline H' - H'\|_F}{\|A\|_F \|B\|_F},
    \frac{\|\overline K' - K'\|_F}{\|A\|_F \|C\|_F}
  \right\}.
\end{equation*}
Next, we run Zha's algorithm --- which, for nonsingular input, can be
thought of as always computing $\overline P$ from $\overline G$ and
$\overline Q$ from $\overline L$ in Algorithm~\ref{alg:rsvd} --- for
comparison and compute the same quantities. For each set of inputs and
outputs we pick the corresponding maxima of the $e_{\text{mag}}$s and
the $e_{\text{rel}}$s, for a total of $10^9$ generated sets, and plot
their distributions in Figure~\ref{fig:errors}. We also keep track of
the $\overline\eta_{\max}$; see Figure~\ref{fig:etas}.


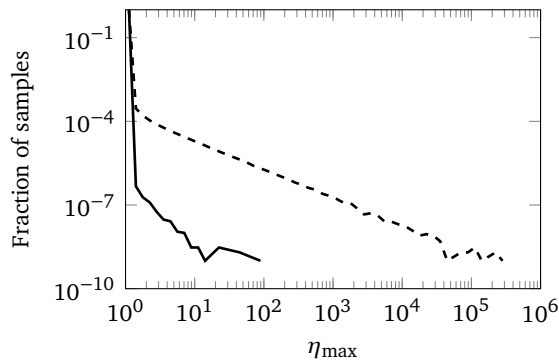
\begin{figure}[htbp!]
  \centerfloat
  \begin{tikzpicture}
    \begin{loglogaxis}[
          height=0.33\textwidth
        , width=0.44\textwidth
        , xlabel={$\eta_{\max}$}
        , ylabel={Fraction of samples}
        , enlarge x limits = false
        , xmin = 1
        , xmax = 1e6
        , ymax = 1
        , ymin = 1e-10
        ]

        \addplot+ table [x expr=10^((\thisrow{Lo} + \thisrow{Up})/2),y=best_max]
        {\tableSwapOnEta};

        \addplot+ table [x expr=10^((\thisrow{Lo} + \thisrow{Up})/2),y=eta_max1]
        {\tableSwapOffEta};
    \end{loglogaxis}
  \end{tikzpicture}
  \caption{The distribution of $\overline\eta_{\max}$ in
  Algorithm~\ref{alg:rsvdeta} with $\tau_\eta = 1$ (solid) and
  $\tau_\eta = \infty$ (dashed).}\label{fig:etas}
\end{figure}


Figure~\ref{fig:errors} shows that Algorithm~\ref{alg:rsvdeta} with
$\tau_\eta = \infty$ (which is identical to Algorithm~\ref{alg:rsvd})
zeroes the $(1,2)$ entries with high precision. In fact, even though
$\overline\eta_{\max}$ can become large, as shown in
Figure~\ref{fig:etas}, the largest $(1,2)$ entry is still $\mathcal
O(\bm\epsilon)$. For the relative errors we do see a difference between
$\tau_\eta = 1$ and $\tau_\eta = \infty$. In particular, the maximum
relative errors remain close to $\bm\epsilon$ when we always let the
algorithm change the columns of $U$ and $V$ when it improves
$\overline\eta_{\max}$, while the maximum relative error can become two
or three orders of magnitude larger when we never let the algorithm
improve $\overline\eta_{\max}$.  Still, these larger errors stay small
and they are also rare.

The numerical results are an upper bound for the errors that we see in
the $\overline A'$, which is encouraging since the error analysis for
$\overline A'$ depends on $\overline\eta_{\max}$, and none of our bounds
suggest that $\overline\eta_{\max}$ has to be small. All we know is that
$\eta_{\max} \lesssim 4\kappa(A)$ when $\tau_\eta = 1$, but $\kappa(A)$
may be over $10^{400}$ for the generated $A$s due to the numerical range
of their entries. Yet, Figure~\ref{fig:etas} shows that we can expect
$\overline\eta_{\max} \ll 4\kappa(A)$ in practice when $\tau_\eta = 1$.
Even when $\tau_\eta = \infty$, $\overline\eta_{\max}$ with values
larger than $\mathcal O(1)$ are rare.  Furthermore, even though the
largest $\overline\eta_{\max}$ we get exceeds $10^5$, the largest error
in Figure~\ref{fig:errors} is considerably smaller than the bound
$\sqrt{2} (44.5 + 342 \cdot 10^5) \bm\epsilon \approx 4.8 \cdot 10^7
\bm\epsilon$ from Theorem~\ref{thm:errA} (where the factor $\sqrt{2}$
comes from using the Frobenius norm instead of using the 2-norm).

Zha's algorithm does not take the $\overline\eta$s into account, and we
see that it may have large relative errors or fail to zero $(1,2)$
entries as a result. The reason for this is that by ignoring the
$\overline\eta$s, the algorithm effectively always computes $P$ from
$BU$ and $Q$ from $V^T\!C$. This in turn means that the rotations may be
computed from numerical noise, because the elements of the input
matrices may be extremely small or extremely large. Hence, ignoring the
$\overline\eta$s should be less of an issue if the input matrices are
well conditioned.  Moreover, Figure~\ref{fig:errors} just shows the tail
ends of the error distributions; hence, the results do not imply that
Zha's algorithm fails to find an accurate solution in, e.g., a third of
the cases.



\subsection{Testing the implicit Kogbetliantz iteration}

We can test the implicit Kogbetliantz iteration by computing the RSVs of
generated matrices with known restricted singular values and prescribed
condition numbers. Specifically, we wish to generate upper-triangular $n
\times n$ matrices
\begin{equation*}
  A = PS \Sigma_\alpha TQ^T,
  \qquad
  B = PS \Sigma_\beta U^T
  \quad\text{and}\quad
  C = V \Sigma_\gamma TQ^T;
\end{equation*}
where $P$, $Q$, $U$, and $V$ are orthonormal; $S$ and $T$ are upper
triangular, nonsingular, and such that $s_{ii} = t_{ii}$ for $i = 1$,
\dots, $n$; and $\Sigma_\alpha = \diag(\alpha_1, \dots, \alpha_n)$,
$\Sigma_\beta = \diag(\beta_1, \dots, \beta_n)$, and $\Sigma_\gamma =
\diag(\gamma_1, \dots, \gamma_n)$ with $\alpha_i, \beta_i, \gamma_i \ge
0$ and $\alpha_i^2 + \beta_i^2 \gamma_i^2 = 1$ for $i = 1$, \dots, $n$.
Furthermore, we wish to control the condition numbers of $S$ and $T$,
and control the ratios between the largest and smallest RSVs, $\alpha$s,
$\beta$s, and $\gamma$s. There exists no unique way to generate such
matrices, and we limit the discussion to matrices randomly generated by
the procedure described below.

The first step is to generate the desired RSVs $\sigma_i$ for $i = 1,
\dots, n$, and compute $\alpha_i^2 = \sigma_i^2 (\sigma_i^2 + 1)^{-1}$
and $\beta_i^2 \gamma_i^2 = (\sigma_i^2 + 1)^{-1}$.
We need to be careful when picking the $\sigma_i$, because if all
$\sigma_i \ge 1$, then all $\alpha_i$ are $\mathcal O(1)$. Likewise, if
all $\sigma_i \le 1$, then all $\beta_i \gamma_i$ are $\mathcal O(1)$.
These situations are undesirable if we want a large variation in the
range of the $\alpha_i$ and $\beta_i \gamma_i$. We can avoid this
problem by using a scaled version of the diagonals generated by LAPACKs
\textsf{xLATME}. In particular, given a condition number $\kappa$, we
randomly pick one of the following sets of $\sigma_i$:
\begin{enumerate}
  \item $\sigma_1 = \sqrt{\kappa}$ and $\sigma_2 = \dots \sigma_n =
    1/\sqrt{\kappa}$;
  \item $\sigma_1 = \dots = \sigma_{n-1} = \sqrt{\kappa}$ and
    $\sigma_n = 1/\sqrt{\kappa}$;
  \item $\sigma_i = \kappa^{1/2 - (i - 1)/(n - 1)}$;
  \item $\sigma_i = \kappa^{1/2} (1 - (i - 1) / (n - 1) \cdot (1 - 1 /
    \kappa))$;
  \item set the $\sigma_i$ to random numbers in the interval
    $(\kappa^{-1/2}, \kappa^{1/2})$ such that the $\log(\sigma_i)$ are
    uniformly distributed in the interval $(-1/2 \log(\kappa), 1/2
    \log(\kappa))$.
\end{enumerate}
While these $\sigma_i$ determine our $\alpha_i$ and the products
$\beta_i \gamma_i$, we still need to select the individual values of
$\beta_i$ and $\gamma_i$. To do so, we generate random numbers
$\delta_i$ so that the $\log(\delta_i)$ are uniformly distributed in
$(-1/8 \log(\kappa), 1/8 \log(\kappa))$, and set $\beta_i =
\sqrt{\beta_i \gamma_i} / \delta_i$ and $\gamma_i = \sqrt{\beta_i
\gamma_i} \cdot \delta_i$. The result is that $\sigma_1 / \sigma_n =
\kappa$, $\alpha_1 / \alpha_n = \beta_n \gamma_n / (\beta_1 \gamma_1) =
\sqrt{\kappa}$, and the ratios between the largest and smallest largest
and smallest $\beta_i$s and $\gamma_i$s are bounded by $\sqrt{\kappa}$.

The next step is to generate suitable $S$ and $T$. Exploratory testing
showed that \textsf{xLATME} produced severely ill-conditioned $A$, even
for small $n$. Another idea is to generate $\widetilde\sigma_i$ in the
same way as above, generate random orthonormal~\cite{Mezz07} $\widetilde
U$ and $\widetilde V$, and take $\widetilde S$ as the upper-triangular
factor of the QR decomposition of $\widetilde U \diag(\widetilde
\sigma_i) \widetilde V$. We can generate $\widetilde T$ likewise, and
then compute $S = \widetilde S D$ and $T = D^{-1} \widetilde T$, where
$D = \diag((\widetilde t_{ii} / \widetilde s_{ii})^{1/2})$. The result
is that $\widetilde s_{ii} \widetilde t_{ii} = s_{ii}^2 = t_{ii}^2$,
although the condition numbers of $S$ and $T$ are no longer exactly
equal to the condition numbers of $\widetilde S$ and $\widetilde T$,
respectively.

Once we have $S$ and $T$, we generate random orthonormal $\widetilde P$,
$\widetilde Q$, $\widetilde U$, and $\widetilde V$, and compute
$\widetilde A = \widetilde P^T S \Sigma_\alpha T \widetilde Q$,
$\widetilde B = \widetilde P S\Sigma_\beta \widetilde U^T$, and
$\widetilde C = \widetilde V \Sigma_\gamma T \widetilde Q^T$. The final
step to get the triplet $(A,B,C)$, is to run the preprocessing from
Section~\ref{sec:pre} on the triplet $(\widetilde A, \widetilde B,
\widetilde C)$. Specifically, since $A$ is nonsingular, we can get $A$
from $\widetilde A$ with a QR decomposition, $B$ from $\widetilde B$
with an RQ decomposition, and $C$ from an appropriately transformed
$\widetilde C$ with another QR decomposition.


\begin{sidewaystable}[!htbp]
  \small
  \renewcommand\tabcolsep{2pt} 
  \centering
  \caption{The relation between the swap tolerance, the number of
    iterations until convergence, and corresponding errors. Each entry
    has the form: mean (maximum). The mean (maximum) condition numbers
    of the generated matrices can be found in Table~\ref{tab:cond}.
  }\label{tab:conv}
  \begin{tabular}{@{}rrrlrrrrrrrr@{}}
    \toprule
    &&&& \multicolumn{8}{c}{\textbf{Tolerance} $\bm{\tau_\eta}$} \\
    \cmidrule{5-12}
    $\bm{n}$ & $\bm{\kappa_{ST}}$ & $\bm{\kappa_\sigma}$
    & \textbf{Measure}
    & $\bm{1}$ & $\bm{1.01}$ & $\bm{4}$ & $\bm{10}$
    & $\bm{10^2}$ & $\bm{10^4}$ & $\bm{10^8}$ & $\bm{\infty}$ \\
    \midrule
    10 & $10^{\phantom{1}}$ & $10^{4}$ & pairs of cycles       &  $8.33$    ($29$) &  $5.66$    ($17$) &  $3.67$     ($9$) &  $3.67$     ($9$) &  $3.64$     ($9$) &  $3.68$     ($9$) &  $3.65$     ($9$) &  $3.65$     ($9$) \\
       &          &          & $\log_{10}(e_{PQUV})$        & $-14.9$ ($-14.3$) & $-14.9$ ($-14.4$) & $-15.0$ ($-14.4$) & $-15.0$ ($-14.5$) & $-15.0$ ($-14.5$) & $-15.0$ ($-14.5$) & $-15.0$ ($-14.5$) & $-15.0$ ($-14.5$) \\
       &          &          & $\log_{10}(e_{ABC})$         & $-14.4$ ($-13.6$) & $-14.6$ ($-14.1$) & $-14.8$ ($-14.3$) & $-14.8$ ($-14.3$) & $-14.8$ ($-14.2$) & $-14.8$ ($-14.3$) & $-14.8$ ($-14.2$) & $-14.8$ ($-14.3$) \\
       &          &          & $\log_{10}(e_{\text{tril}})$ & $-15.2$ ($-14.7$) & $-15.2$ ($-14.8$) & $-15.2$ ($-14.8$) & $-15.3$ ($-14.8$) & $-15.3$ ($-14.8$) & $-15.2$ ($-14.8$) & $-15.3$ ($-14.8$) & $-15.3$ ($-14.8$) \\
       &          &          & $\log_{10}(e_{\chi})$        & $-8.34$ ($-0.14$) & $-15.3$ ($-2.00$) & $-15.5$ ($-14.0$) & $-15.5$ ($-14.0$) & $-15.5$ ($-14.0$) & $-15.5$ ($-13.9$) & $-15.5$ ($-13.9$) & $-15.5$ ($-14.1$) \\
    \cmidrule{2-12}
       & $10^{5}$ & $10^{4}$ & pairs of cycles                 &  $8.27$    ($36$) &  $5.41$    ($23$) &  $3.72$    ($11$) &  $3.66$    ($10$) &  $3.59$     ($8$) &  $3.58$     ($9$) &  $3.58$     ($9$) &  $3.57$    ($10$) \\
       &          &          & $\log_{10}(e_{PQUV})$        & $-14.9$ ($-14.2$) & $-14.9$ ($-14.2$) & $-15.0$ ($-14.4$) & $-15.0$ ($-14.4$) & $-15.0$ ($-14.5$) & $-15.0$ ($-14.5$) & $-15.0$ ($-14.4$) & $-15.0$ ($-14.4$) \\
       &          &          & $\log_{10}(e_{ABC})$         & $-14.4$ ($-13.6$) & $-14.6$ ($-13.8$) & $-14.7$ ($-14.3$) & $-14.7$ ($-14.3$) & $-14.7$ ($-14.3$) & $-14.7$ ($-14.3$) & $-14.7$ ($-14.2$) & $-14.7$ ($-14.3$) \\
       &          &          & $\log_{10}(e_{\text{tril}})$ & $-15.3$ ($-14.6$) & $-15.4$ ($-14.5$) & $-15.7$ ($-14.9$) & $-15.8$ ($-14.8$) & $-15.9$ ($-14.8$) & $-15.9$ ($-14.8$) & $-15.9$ ($-14.9$) & $-15.9$ ($-14.8$) \\
       &          &          & $\log_{10}(e_{\chi})$        & $-7.53$ ($-0.11$) & $-12.8$ ($-7.16$) & $-12.8$ ($-7.71$) & $-12.8$ ($-7.45$) & $-12.7$ ($-7.80$) & $-12.7$ ($-7.37$) & $-12.7$ ($-7.19$) & $-12.8$ ($-7.84$) \\
    \midrule
    50 & $10^{\phantom{1}}$ & $10^{4}$ & pairs of cycles       &  $31.6$    ($50$) &  $26.6$    ($50$) &  $4.43$    ($11$) &  $4.42$    ($10$) &  $4.47$    ($10$) &  $4.44$    ($10$) &  $4.42$    ($12$) &  $4.42$    ($11$) \\
       &          &          & $\log_{10}(e_{PQUV})$        & $-14.3$ ($-13.8$) & $-14.3$ ($-13.8$) & $-14.6$ ($-14.2$) & $-14.6$ ($-14.2$) & $-14.6$ ($-14.1$) & $-14.6$ ($-14.2$) & $-14.6$ ($-14.2$) & $-14.6$ ($-14.1$) \\
       &          &          & $\log_{10}(e_{ABC})$         & $-13.4$ ($-12.7$) & $-13.5$ ($-12.6$) & $-14.2$ ($-13.6$) & $-14.2$ ($-13.6$) & $-14.2$ ($-13.6$) & $-14.2$ ($-13.6$) & $-14.2$ ($-13.6$) & $-14.2$ ($-13.5$) \\
       &          &          & $\log_{10}(e_{\text{tril}})$ & $-14.6$ ($-14.2$) & $-14.6$ ($-14.1$) & $-14.8$ ($-14.5$) & $-14.8$ ($-14.4$) & $-14.8$ ($-14.5$) & $-14.8$ ($-14.5$) & $-14.8$ ($-14.5$) & $-14.8$ ($-14.4$) \\
       &          &          & $\log_{10}(e_{\chi})$        & $-2.85$ ($-0.02$) & $-12.8$ ($-0.87$) & $-14.8$ ($-13.9$) & $-14.8$ ($-13.8$) & $-14.8$ ($-13.9$) & $-14.8$ ($-13.9$) & $-14.8$ ($-13.9$) & $-14.8$ ($-13.9$) \\
    \cmidrule{2-12}
       & $10^{5}$ & $10^{4}$ & pairs of cycles                 &  $31.3$    ($50$) &  $24.9$    ($50$) &  $5.00$    ($21$) &  $4.55$    ($13$) &  $4.31$    ($11$) &  $4.30$    ($10$) &  $4.29$    ($11$) &  $4.29$    ($10$) \\
       &          &          & $\log_{10}(e_{PQUV})$        & $-14.3$ ($-13.5$) & $-14.3$ ($-13.5$) & $-14.6$ ($-13.9$) & $-14.6$ ($-14.1$) & $-14.6$ ($-14.1$) & $-14.6$ ($-14.1$) & $-14.6$ ($-14.1$) & $-14.6$ ($-14.0$) \\
       &          &          & $\log_{10}(e_{ABC})$         & $-13.4$ ($-12.7$) & $-13.5$ ($-12.5$) & $-14.1$ ($-13.3$) & $-14.2$ ($-13.5$) & $-14.2$ ($-13.5$) & $-14.2$ ($-13.4$) & $-14.2$ ($-13.6$) & $-14.2$ ($-13.5$) \\
       &          &          & $\log_{10}(e_{\text{tril}})$ & $-14.7$ ($-14.0$) & $-14.7$ ($-14.0$) & $-15.1$ ($-14.4$) & $-15.2$ ($-14.5$) & $-15.3$ ($-14.6$) & $-15.3$ ($-14.5$) & $-15.3$ ($-14.5$) & $-15.3$ ($-14.5$) \\
       &          &          & $\log_{10}(e_{\chi})$        & $-3.22$ ($-0.02$) & $-10.7$ ($-0.10$) & $-12.5$ ($-7.19$) & $-12.5$ ($-8.22$) & $-12.5$ ($-6.93$) & $-12.5$ ($-7.67$) & $-12.5$ ($-7.95$) & $-12.5$ ($-7.78$) \\
    \bottomrule
  \end{tabular}
\end{sidewaystable}



We generate the input matrices $A$, $B$, and $C$ in high-precision
arithmetic, and again with 100 decimals of precision. We denote the
$\kappa$ used to generate the RSVs by $\kappa_\sigma$, and the $\kappa$
used to generate $\widetilde S$ and $\widetilde T$ by $\kappa_{ST}$.
Then, we run the implicit Kogbetliantz iteration from
Algorithm~\ref{alg:kog} in double precision. We stop the iterations
after at most 50 pairs of cycles, or earlier if we detect convergence
after an even cycle as described in Section~\ref{sec:kog}. In
particular, we stop earlier when $\rho = \max_{ij} \rho_{ij}$ satisfies
$0.99\rho_{\min} < \rho < 0.01$, where $\rho_{ij}$ is as
in~\eqref{eq:rhoij}.
For solving the $2\times 2$ RSVD we use Algorithm~\ref{alg:rsvdeta}, and
consider the tolerances $\tau_\eta = 1$, $1.01$, $4$, $10$, $100$,
$10^4$, $10^8$, $\infty$. The tolerance $\tau_\eta = 4$ is of interest
because of Lemma~\ref{thm:etaBigEtaSmall}, and $\tau_\eta = 10^8$
because of the connection between Theorem~\ref{thm:Astable} and
because $10^8 \approx \bm\epsilon^{-1/2}$.

For every input triplet we record the number of cycle pairs before
stopping, and also compute the following quantities. First, the maximum
errors in the computed orthogonal matrices given by
\begin{equation*}
  e_{PQUV} = \max \{
    \|\overline P^T \overline P - I\|_F,
    \|\overline Q^T \overline Q - I\|_F,
    \|\overline U^T \overline U - I\|_F,
    \|\overline V^T \overline V - I\|_F
  \} / \sqrt{n}.
\end{equation*}
Second, the maximum errors in the transformations given by
\begin{equation*}
  e_{ABC} = \max \left\{
    \frac{\|\overline P^T\! A \overline Q - \overline A'\|_F}{\|A\|_F},
    \frac{\|\overline P^T\! B \overline U - \overline B'\|_F}{\|B\|_F},
    \frac{\|\overline V^T\! C \overline Q - \overline C'\|_F}{\|C\|_F}
  \right\},
\end{equation*}
where $\overline A'$, $\overline B'$, and $\overline C'$ are the output
matrices of Algorithm~\ref{alg:kog}. Third, $e_{\text{tril}}$, the
largest of the Frobenius norms of the strictly lower-triangular parts of
$\overline P^T\! A \overline Q$, $\overline P^T\! B \overline U$, and
$\overline V^T\! C \overline Q$.
And fourth, $e_{\chi} = \max_{i \in \{ 1, \dots, n\}} \chi(\sigma_i,
\overline\sigma_i)$, where
\begin{equation*}
  \chi(\sigma, \overline\sigma)
  = |\alpha \overline{\beta \gamma} - \overline \alpha \beta \gamma|
  = \frac{|\sigma - \overline\sigma|}{%
    \sqrt{1 + \sigma^2} \sqrt{1 + \overline \sigma^2}}
  = \frac{|\sigma^{-1} - \overline\sigma^{-1}|}{%
    \sqrt{1 + \sigma^{-2}} \sqrt{1 + \overline \sigma^{-2}}}
\end{equation*}
is the chordal metric and measures the distance between the exact and
computed RSVs; see, e.g., Stewart and Sun~\cite[Ch.~6]{SS90}.  See
Table~\ref{tab:conv} for the results.

We see that convergence is slow when $\tau_\eta = 1$, and that a small
tolerance does not improve the errors. The slow convergence is expected,
since we increase the maximum angle of the rotations whenever we
multiply $U$ and $V$ by $J$ to improve $\overline\eta_{\max}$. As a
result, we may not have convergence before the cutoff point of 50
iterations and thus also have large $e_\chi$.  That low tolerances do
not improve the remaining errors is more interesting, but may be
explained by the following two observations. First, more roundoff errors
get accumulated when we performs more cycles; second, larger values of
$\overline\eta_{\max}$ do not affect the accuracy of the results. The
latter matches with the observations from the previous section; that is,
Algorithm~\ref{alg:rsvdeta} typically computes the $2\times 2$ RSVD with
high relative accuracy, even when $\tau_\eta = \infty$.

The table also shows us that we can dramatically improve the rate of
convergence with a small increase of $\tau_\eta$. For example, we see
substantial improvements for $\tau_\eta = 1.01$ and already achieve a
near optimal rate of convergence for $\tau_\eta = 4$. The observation
that we can get fast convergence for small tolerances (larger than 1) is
expected if we look at Figure~\ref{fig:etas}. In particular,
$\overline\eta_{\max}$ is close to 1 most of the time, and large
$\overline\eta_{\max}$ are so rare that any practical difference between
the larger values of $\tau_\eta$ is would be surprising.

One caveat here is that the results from this section depend on the way
we generate the test matrices, and on the condition numbers we choose.
We consider more variations of $\kappa_{ST}$ and $\kappa_\sigma$ in the
next section, of which we only consider the pairs resulting in the best
and worst conditioned matrix triplets in this section. In any case, we
make sure to pick the $\kappa$s such that $\kappa(A)$ is never more than
$10^{-12}$; see Table~\ref{tab:cond} for the condition numbers of the
matrices generated for the results in Table~\ref{tab:conv}.


\begin{table}[!htbp]
  \small
  \centering
  \caption{The mean (maximum) condition numbers of the matrices
    generated for the tests/results in Table~\ref{tab:conv}.
  }\label{tab:cond}
  \begin{tabular}{@{}rrrrrr@{}}
    \toprule
    $\bm{n}$ & $\bm{\kappa_{ST}}$ & $\bm{\kappa_{\sigma}}$ & $\bm{{\log_{10}}(\kappa(A))}$ & $\bm{{\log_{10}}(\kappa(B))}$ & $\bm{{\log_{10}}(\kappa(C))}$ \\
    \midrule
    $10$ & $10^{\phantom{1}}$ & $10^4$ & $2.84$ ($4.00$) & $1.86$ ($3.03$) & $1.82$ ($3.20$) \\
    $10$ & $10^5$             & $10^4$ & $9.24$ ($12.0$) & $6.31$ ($9.02$) & $5.22$ ($8.69$) \\
    $50$ & $10^{\phantom{1}}$ & $10^4$ & $2.95$ ($3.99$) & $2.07$ ($2.92$) & $2.04$ ($3.05$) \\
    $50$ & $10^5$             & $10^4$ & $9.53$ ($12.0$) & $6.60$ ($8.91$) & $5.49$ ($8.49$) \\
    \bottomrule
  \end{tabular}
\end{table}




\subsection{Comparison with other methods}

In the last part we preprocessed the triplet $(\widetilde A, \widetilde
B, \widetilde C)$ in high-precision arithmetic to get the triplet $(A,
B, C)$. In this part we generate the former triplet in the same way, but
we do the preprocessing in double-precision arithmetic instead. The
purpose of this change is to try and see how different rank-revealing
decompositions of $A$ affect the accuracy of the computed RSVs (even
though $A$ is full rank). Moreover, we would like to see how these
results compare to the existing methods.

We use three different methods for the preprocessing.  The first two
methods use the approach described in Section~\ref{sec:pre}, the first
with a QR decomposition with column pivoting for the compression of $A$,
and the second with a Jacobi based SVD for the compression of $A$.  The
third method uses an LDU decomposition as described in
Algorithm~\ref{alg:nonorthRSVD}. All three methods use the implicit
Kogbetliantz iteration and Algorithm~\ref{alg:rsvdeta} with $\tau_\eta =
\infty$.  For the existing algorithms we have Drma\v{c}'s
algorithm~\cite{Drma00} and the CSD stage from the algorithm by Chu,
De~Lathauwer, and De~Moor (CLM)~\cite[Sec.~3.2]{CLM00}.  To make the
latter as accurate as possible, we implement the required CS
decomposition with a Jacobi-type SVD instead of the QR based approach
implied by the authors. We omit the results of Zha's algorithm here,
because we already have the results from Section~\ref{sec:numrsvdtwo}.
See Table~\ref{tab:methods} for an overview of the results.


\begin{table}[!htbp]
  \small
  \renewcommand\tabcolsep{2pt} 
  \centering
  \caption{The means (maxima) of $\log_{10}(e_{\chi})$ for different
  preprocessing approaches and RSVD algorithms.}\label{tab:methods}
  \begin{tabular}{@{}rrrrrrrr@{}}
    \toprule
    &&& \multicolumn{5}{c}{\textbf{Preprocessing method/RSVD algorithm}} \\
    \cmidrule{4-8}
    $\bm{n}$ & $\bm{\kappa_{ST}}$ & $\bm{\kappa_\sigma}$ & \textbf{ColPivQR} &
    \textbf{SVD} & \textbf{LDU} & \textbf{Drma\v{c}} & \textbf{CLM} \\
    \midrule
    10 & $10^{\phantom{1}}$ & $10^{4\phantom{0}}$ & $-15.4$ ($-13.7$) & $-15.2$ ($-13.3$) & $-15.4$ ($-14.1$) & $-15.5$ ($-14.2$) & $-15.2$ ($-13.7$) \\
       &                    & $10^{12}$           & $-15.0$ ($-12.1$) & $-14.8$ ($-11.5$) & $-15.0$ ($-12.5$) & $-14.1$ ($-11.8$) & $-14.5$ ($-10.8$) \\
       &                    & $10^{20}$           & $-14.3$ ($-10.4$) & $-14.0$ ($-9.24$) & $-14.2$ ($-10.6$) & $-9.93$ ($-5.34$) & $-13.3$ ($-8.07$) \\
    \cmidrule{2-8}
       & $10^{3}$           & $10^{4\phantom{0}}$ & $-13.0$ ($-9.50$) & $-12.7$ ($-8.94$) & $-13.1$ ($-9.81$) & $-13.1$ ($-9.81$) & $-12.7$ ($-9.32$) \\
       &                    & $10^{12}$           & $-13.0$ ($-8.06$) & $-12.8$ ($-6.64$) & $-13.1$ ($-8.24$) & $-12.7$ ($-8.24$) & $-12.5$ ($-7.06$) \\
    \cmidrule{2-8}
       & $10^{5}$           & $10^{4\phantom{0}}$ & $-9.44$ ($-4.38$) & $-9.14$ ($-3.83$) & $-9.54$ ($-4.71$) & $-9.54$ ($-4.71$) & $-9.35$ ($-4.41$) \\
    \midrule
    50 & $10^{\phantom{1}}$ & $10^{4\phantom{0}}$ & $-14.8$ ($-13.9$) & $-14.4$ ($-13.0$) & $-14.8$ ($-13.9$) & $-14.8$ ($-14.1$) & $-14.6$ ($-13.9$) \\
       &                    & $10^{12}$           & $-14.7$ ($-12.6$) & $-14.3$ ($-11.3$) & $-14.6$ ($-12.6$) & $-13.0$ ($-10.1$) & $-14.1$ ($-11.3$) \\
       &                    & $10^{20}$           & $-13.6$ ($-10.7$) & $-13.1$ ($-8.97$) & $-13.5$ ($-10.9$) & $-8.71$ ($-5.03$) & $-12.4$ ($-8.09$) \\
    \cmidrule{2-8}
       & $10^{3}$           & $10^{4\phantom{0}}$ & $-13.0$ ($-10.0$) & $-12.2$ ($-8.04$) & $-13.1$ ($-10.3$) & $-13.1$ ($-10.3$) & $-12.7$ ($-9.81$) \\
       &                    & $10^{12}$           & $-12.8$ ($-8.26$) & $-12.2$ ($-6.32$) & $-12.9$ ($-8.56$) & $-12.2$ ($-8.56$) & $-12.3$ ($-7.52$) \\
    \cmidrule{2-8}
       & $10^{5}$           & $10^{4\phantom{0}}$ & $-9.36$ ($-5.37$) & $-8.46$ ($-3.55$) & $-9.51$ ($-5.44$) & $-9.51$ ($-5.44$) & $-9.20$ ($-4.96$) \\
    \bottomrule
  \end{tabular}
\end{table}


The table show that the nonorthogonal preprocessing from
Algorithm~\ref{alg:nonorthRSVD} yields the most accurate results, or
close to the most accurate results. QR with pivoting is slightly behind
the former in accuracy, and the new algorithm is the least accurate when
using the SVD in the preprocessing phase.  A possible explanation is
that the SVD introduces a larger error than QR with pivoting. If true,
it would suggest an interesting trade-off between a better rank decision
with the SVD, and better performance and more accurate RSVs with the QR
decomposition.  Drma\v{c}'s method does comparatively well as long as
$\kappa_\sigma$ is not too large, but is well behind in accuracy when
$\kappa_\sigma$ is large.  The results for the method from Chu,
De~Lathauwer, and De~Moor~\cite{CLM00} (with the modified CS
decomposition) are between the QR and SVD results in the best case, and
slightly worse than the SVD result in the worst case.




\section{Conclusions}\label{sec:con}

This work introduced a new method for computing the RSVD of a matrix
triplet with an implicit Kogbetliantz-type iteration. The main
contributions consist of a new generalized Schur-form RSVD that we can
compute with orthonormal transformations only, a new preprocessing phase
that requires fewer transformations and fewer rank decisions than
existing methods, and a new $2\times 2$ triangular RSVD algorithm with
favorable numerical properties. We found that the latter is numerically
stable in the sense of Theorem~\ref{thm:stablersvd} and
Theorem~\ref{thm:Astable}.  A further contribution is a new approach to
extract restricted singular triplets.  This approach has theoretical and
practical benefits, but requires atypical scaling of the triplets.

We found that the value of $\eta_{\max}$ is critical in assessing the
accuracy of the results computed by Algorithms~\ref{alg:rsvd}
and~\ref{alg:rsvdeta}. Specifically, we can both use $\eta_{\max}$ a
priori through bounds, and a posteriori as the amplification factor of
the errors in $\overline U$ and $\overline V$.  Numerical experiments
show that we can typically keep the values of the $\eta_{\max}$ small.
In the rare cases that $\eta_{\max}$ is large, the results show that we
can still expect the $2\times 2$ RSVDs to have small backward errors. In
fact, none of the results suggests that the bounds from
Section~\ref{sec:bwdA} are sharp, and that the bounds are pessimistic in
practice.  This means that the numerical results provide empirical
evidence that we can have a numerically stable RSVD without having to
increase the working precision for the $2\times 2$ RSVD.\@

Areas where further improvements are desirable or necessary, and
potential directions for future research include the following. Better
stopping conditions, a cache friendly and parallelized implementation of
the Kogbetliantz phase, and most of all, a numerically sound
postprocessing phase. The latter in particular represents a major
deficiency of the new algorithm, although the postprocessing phase is
only necessary to compute the full RSVD. In other words, we may skip the
postprocessing phase in applications where the Schur-form RSVD suffices.

We should also note that there are techniques for and aspects of
existing Jacobi methods (for other matrix decompositions) that we
ignored in this paper. These include, for example, the scaling of the
input matrices to avoid overflow or underflow, the effects of diagonal
scaling of the input matrices on the relative accuracy of the results,
efficient implementations using blocking for better cache usage,
quasi-cycles for faster convergence, adaptive pivot strategies, whether
preconditioning is possible and useful, the (relative) accuracy of the
algorithm as a whole for arbitrary or structured matrices, etc.  See,
e.g., \cite{DV08a,DV08b,Detal99} and references therein for more
information.


\ifx\preprintcls\undefined\else

\section{Acknowledgments}

I would like to thank Andreas Frommer and Michiel Hochstenbach for
helpful discussions and their comments.


\bibliographystyle{siamplain}
\bibliography{rsvddense}
\fi

\ifx\preprintcls\undefined\else\clearpage\fi
\appendix


\section{Proof of Lemma~\ref{thm:algIO}}\label{sec:proof}


\begin{enumerate}
  \item 
    Let $A$, $B$, and $C$ be nonsingular; then the proof follows the
    proof of Proposition~\ref{thm:exactrsvdtwo}.

  \item 
    If $C = \minimat{0}{c_{12}}{0}{c_{22}}$
    and $B = \minimat{b_{11}}{b_{12}}{0}{0}$,
    then $M = \minimat{0}{0}{0}{0}$. Since $c_{11} = b_{22} = 0$, we
    take $P = Q = J$ and compute $U$ and $V$ are such that 
    ${(V^T\!C)}_{22} = (BU)_{11} = 0$. Hence
    \begin{equation*}
      V^T\!CQ = \minimat{c_{11}'}{0}{0}{0},
      \qquad
      P^T\!AQ =
      \minimat{\underline{a_{11}'}}{0}{a_{21}'}{\underline{a_{22}'}},
      \quad\text{and}\quad
      P^T\!BU = \minimat{0}{0}{0}{b_{22}'}.
    \end{equation*}

  \item 
    Let $C = \minimat{\underline{c_{11}}}{c_{12}}{0}{0}$ and $B =
    \minimat{0}{0}{0}{0}$; then $M = \minimat{0}{0}{0}{0}$ and
    \textsf{xLASV2} computes $U = V = I$. Now $1 = c_{\max} > s_{\max} =
    0$ and the algorithm does not swap the columns of $U$ and $V$.  See
    below for the computation of $P$ and $Q$, but note that $P^T\!BU$ is
    zero.

  \item 
    Let $C = \minimat{\underline{c_{11}}}{c_{12}}{0}{\underline{c_{22}}}$
    and $B = \minimat{0}{0}{0}{0}$; then $M = \minimat{0}{0}{0}{0}$ and
    \textsf{xLASV2} computes $U = V = I$. Now $1 = c_{\max} > s_{\max} =
    0$ and the algorithm does not swap the columns of $U$ and $V$.  See
    below for the computation of $P$ and $Q$, but note that $V^T\!CQ$ is
    nonsingular lower triangular.

  \item 
    Let $C = \minimat{\underline{c_{11}}}{c_{12}}{0}{0}$ and $B =
    \minimat{b_{11}}{b_{12}}{0}{0} \neq 0$; then $M = c_{11} a_{22} B
    \neq 0$ and \textsf{xLASV2} computes $|V| = I$ and $U$ such that
    ${(BU)}_{12} = 0$. Now $1 = c_{\max} \ge s_{\max}$ and the algorithm
    does not swap the columns of $U$ and $V$.  It follows that $|h_{12}|
    + |h_{22}| = |l_{12}| + |l_{22}| = 0$ ($\eta_h = \eta_l = \infty$),
    $\eta_g = 1$, and $1 \le \eta_k < \infty$. Hence, the algorithm
    computes $Q$ from $G = V^T\!C = C$ and $P$ from $K = V^T\!C\adj(A) =
    C\adj(A)$, which results in
    \begin{equation*}
      V^T\!CQ
      = \minimat{\underline{c_{11}}'}{0}{0}{0},
      \qquad
      P^T\!AQ = \minimat{\underline{a_{11}'}}{0}{a_{21}'}{\underline{a_{22}'}},
      \quad\text{and}\quad
      P^T\!BU = \minimat{\underline{b_{11}'}}{0}{b_{21}'}{0}.
    \end{equation*}
    The product $P^T\!AQ$ is lower triangular since, by construction,
    \begin{equation*}
      0
      = |{(KP)}_{12}|
      = |\vec e_1^T (V^T\!CQ) \adj(P^T\!AQ) \vec e_2|
      = |c_{11}' {(P^T\!AQ)}_{12}|.
    \end{equation*}
    Likewise, $b_{11}'$ is nonzero because $|c_{11}' a_{22}' b_{11}'|$
    equals the largest singular value of $M$.

  \item 
    Let $C = \minimat{\underline{c_{11}}}{c_{12}}{0}{\underline{c_{22}}}$
    and $B = \minimat{b_{11}}{b_{12}}{0}{0} \neq 0$; then $M = c_{11}
    a_{22} B \neq 0$ and the computation of $U$, $V$, $P$, and $Q$
    proceeds as above, except that $V^T\!CQ$ is nonsingular lower
    triangular.

  \item 
    Let $C = \minimat{ \underline{c_{11}} }{c_{12}}{0}{0}$ and $B =
    \minimat{ \underline{b_{11}} }{b_{12}}{0}{ \underline{b_{22}} }$;
    then $M = \minimat{ \underline{m_{11}} }{m_{12}}{0}{0}$ and
    \textsf{xLASV2} computes $|V| = I$ and $U$ such that ${(MU)}_{12} =
    0$.  Now $1 = c_{\max} > s_{\max} = |u_{12}|$ since $m_{11} \neq 0$
    implies that $|U| \neq |J|$, and the algorithm does not swap the
    columns of $U$ and $V$.  It follows that $\eta_g = 1$ and $1 \le
    \eta_k, \eta_h, \eta_l < \infty$, so that the algorithm always
    computes $Q$ from $G = V^T\!C = C$, but may compute $P$ from either
    $K = V^T\!C\adj(A) = C\adj(A)$ or $L = BU$.  The result is
    \begin{equation*}
      V^T\!CQ = \minimat{\underline{c_{11}'}}{0}{0}{0},
      \qquad
      P^T\!AQ = \minimat{\underline{a_{11}'}}{0}{a_{21}'}{\underline{a_{22}'}},
      \quad\text{and}\quad
      P^T\!BU = \minimat{\underline{b_{11}'}}{0}{b_{21}'}{\underline{b_{22}'}}.
    \end{equation*}
    If $P$ was computed from $K$, then $P^T\!AQ$ is lower triangular for
    the same reason as in Item~\ref{it:CrowBrow}, and $P^T\!BU$ is lower
    triangular because $0 = |{(V^T\!MU)}_{12}| = |c_{11}' a_{22}'
    {(P^T\!BU)}_{12}|$. If $P$ was computed from $L$, then $P^T\!AQ$ is
    lower triangular because $0 = |{(V^T\!MU)}_{12} = |c_{11}' b_{22}'
    {(P^T\!AQ)}_{12}|$.

  \item 
    Let $C = \minimat{0}{0}{0}{0}$ and $B = \minimat{0}{b_{12}}{0}{
    \underline{b_{22}} }$; then $M = \minimat{0}{0}{0}{0}$ and
    \textsf{xLASV2} computes $U = V = I$. Now $1 = c_{\max} > s_{\max} =
    0$ and the algorithm does not swap the columns of $U$ and $V$.
    Computing $P$ and $Q$ is the same as below, but note that $V^T\!CQ$ is
    zero. 

  \item 
    Let $C = \minimat{0}{0}{0}{0}$ and $B =
    \minimat{ \underline{b_{11}} }{b_{12}}{0}{ \underline{b_{22}} }$;
    then $M = \minimat{0}{0}{0}{0}$ and \textsf{xLASV2} computes $U = V
    = I$. Now $1 = c_{\max} > s_{\max} = 0$ and the algorithm does not
    swap the columns of $U$ and $V$. Computing $P$ and $Q$ is the same
    as below, but note that $P^T\!BU$ is nonsingular lower triangular.

  \item 
    Let $C = \minimat{0}{c_{12}}{0}{c_{22}} \neq 0$
    and $B = \minimat{0}{b_{12}}{0}{\underline{b_{22}}}$;
    then $M = a_{11} b_{22} C \neq 0$ and \textsf{xLASV2} computes $V$
    such that ${(V^T\!C)}_{22} = 0$ and $|U| = |J|$. Now $|v_{11}| =
    c_{\max} \le s_{\max} = 1$, but the $M$ is singular and no swap
    takes place. It follows that $|h_{12}| + |h_{22}| = |l_{12}| +
    |l_{22}| = 0$ ($\eta_h = \eta_l = \infty$), and $1 \le \eta_g,
    \eta_k < \infty$. As a result, $P$ and $Q$ are computed from $G$ and
    $K$, respectively.  Furthermore $|Q| = |J|$, so that $|P| = |J|$ as
    well.  The result is
    \begin{equation*}
      V^T\!CQ = \minimat{\underline{c_{11}'}}{0}{0}{0},
      \qquad
      P^T\!AQ = \minimat{\underline{a_{11}'}}{0}{a_{21}'}{\underline{a_{22}'}},
      \quad\text{and}\quad
      P^T\!BU = \minimat{\underline{b_{11}'}}{0}{b_{21}'}{0}.
    \end{equation*}

  \item 
    Let $C = \minimat{0}{c_{12}}{0}{c_{22}} \neq 0$ and $B = \minimat{
    \underline{b_{11}} }{b_{12}}{0}{ \underline{b_{22}} }$; then $M =
    a_{11} b_{22} C \neq 0$ and the computation of $P$, $Q$, $U$, and
    $V$ proceeds as above, except that $P^T\!BU$ is nonsingular lower
    triangular.

  \item 
    Let $C = \minimat{ \underline{c_{11}} }{c_{12}}{0}{ \underline{c_{22}}
    }$ and $B = \minimat{0}{b_{12}}{0}{ \underline{b_{22}} }$; then
    $M = \minimat{0}{m_{12}}{0}{ \underline{m_{22}} }$ and
    \textsf{xLASV2} computes $|U| = |J|$ and $V$ such that
    ${(V^T\!M)}_{22} = {(V^T\!M)}_{22} = 0$. Now $|v_{11}| = c_{\max} <
    s_{\max} = 1$ since $m_{22} \neq 0$ implies that $|V| \neq I$.
    However, $M$ is singular and no swap takes place.
    It follows that $1 \le \eta_g, \eta_k < \infty$ and $\eta_h = \eta_l
    = 1$, so that the algorithm always computes $Q$ from $G = V^T\!C$, and
    $P$ from $K = V^T\!C\adj(A)$.  The result is
    \begin{equation*}
      V^T\!CQ = \minimat{\underline{c_{11}'}}{0}{c_{21}'}{\underline{c_{22}'}},
      \qquad
      P^T\!AQ = \minimat{\underline{a_{11}'}}{0}{a_{21}'}{\underline{a_{22}'}},
      \quad\text{and}\quad
      P^T\!BU = \minimat{\underline{b_{11}'}}{0}{b_{21}'}{0}.
    \end{equation*}

  \item 
    Let $C = \minimat{\underline{c_{11}}}{c_{12}}{0}{0}$ and
    $B = \minimat{0}{b_{12}}{0}{\underline{b_{22}}}$; then
    $M = \minimat{0}{m_{12}}{0}{0}$. When $m_{12} = 0$, \textsf{xLASV2}
    computes $U = V = I$ and no swaps are necessary.  Furthermore, in
    this case $\eta_g = \eta_l = 1$ so that $P$ and $Q$ are
    computed from $B$ and $C$, respectively, resulting in
    \begin{equation*}
      V^T\!CQ = \minimat{\underline{c_{11}'}}{0}{0}{0},
      \qquad
      P^T\!AQ =
      \minimat{\underline{a_{11}'}}{0}{a_{21}'}{\underline{a_{22}'}},
      \quad\text{and}\quad
      P^T\!BU = \minimat{0}{0}{0}{\underline{b_{22}'}}.
    \end{equation*}
    Again, $P^T\!AQ$ is lower triangular because $0 = |V^T\!MU| =
    |b_{22}' c_{11}' {(P^T\!AQ)}_{12}|$
    If $m_{12} \neq 0$, then \textsf{xLASV2} computes $|U| = |J|$ and
    $|V| = I$. Now $c_{\max} = s_{\max} = 1$ and the algorithm does not
    swap the columns of $U$ and $V$.
    Without the swap, $\eta_g = 1$, $|h_{11}| + |h_{12}| = 0$ ($\eta_h =
    \infty$), $1 \le \eta_k < \infty$, and $|l_{12}| + |l_{22}| = 0$
    ($\eta_l = \infty$). Hence, the algorithm computes $Q$ from $G =
    V^T\!C = C$ and $P$ from $K = V^T\!C\adj(A) = C\adj(A)$, resulting in
    \begin{equation*}
      V^T\!CQ = \minimat{\underline{c_{11}'}}{0}{0}{0},
      \qquad
      P^T\!AQ = \minimat{\underline{a_{11}'}}{0}{a_{21}'}{\underline{a_{22}'}},
      \quad\text{and}\quad
      P^T\!BU = \minimat{\underline{b_{11}}'}{0}{b_{21}'}{0}.
    \end{equation*}
    See Item~\ref{it:CrowBrow} to see why $P^T\!AQ$ is lower triangular,
    and why $b_{11}' \neq 0$.
\end{enumerate}


\ifx\siamartcls\undefined\else

\section{Acknowledgments}

I would like to thank Andreas Frommer and Michiel Hochstenbach for
helpful discussions and their comments.


\bibliographystyle{siamplain}
\bibliography{rsvddense}
\fi

\end{document}